\documentclass[9pt]{amsart}
\usepackage{amsmath}
\usepackage{amsxtra}
\usepackage{amscd}
\usepackage{amsthm}
\usepackage{amsfonts}
\usepackage{amssymb}
\usepackage{eucal}
\usepackage[all]{xy}
\usepackage{graphicx}
\usepackage{hyperref}
\usepackage{tikz-cd}

\newtheorem{cor}[subsubsection]{Corollary}
\newtheorem{lem}[subsubsection]{Lemma}
\newtheorem{mainlem}[subsubsection]{Main Lemma}
\newtheorem{prop}[subsubsection]{Proposition}

\newtheorem{propconstr}[subsubsection]{Proposition-Construction}
\newtheorem{lemconstr}[subsubsection]{Lemma-Construction}

\newtheorem{thm}[subsubsection]{Theorem}

\theoremstyle{definition}

\theoremstyle{remark}
\newtheorem{rem}[subsubsection]{Remark}

\newcommand{\thmref}[1]{Theorem~\ref{#1}}

\newcommand{\secref}[1]{Sect.~\ref{#1}}
\newcommand{\lemref}[1]{Lemma~\ref{#1}}
\newcommand{\propref}[1]{Proposition~\ref{#1}}
\newcommand{\corref}[1]{Corollary~\ref{#1}}

\numberwithin{equation}{section}

\newcommand{\nc}{\newcommand}
\nc{\renc}{\renewcommand}
\nc{\ssec}{\subsection}
\nc{\sssec}{\subsubsection}
\nc{\on}{\operatorname}

\nc\ol{\overline}
\nc\wt{\widetilde}
\nc\tboxtimes{\wt{\boxt}}
\nc\tstar{\wt{\star}}
\nc{\alp}{a}

\nc{\ZZ}{{\mathbb Z}}
\nc{\NN}{{\mathbb N}}
\nc{\OO}{{\mathbb O}}
\renc{\SS}{{\mathbb S}}
\nc{\DD}{{\mathbb D}}
\nc{\GG}{{\mathbb G}}

\nc{\Fq}{{\mathbb F}_q}
\nc{\Fqb}{\ol{{\mathbb F}_q}}
\nc{\Ql}{\ol{{\mathbb Q}_\ell}}
\nc{\id}{\text{id}}
\nc\X{\mathcal X}

\nc{\Hom}{\on{Hom}}
\nc{\Lie}{\on{Lie}}
\nc{\Loc}{\on{Loc}}
\nc{\Pic}{\on{Pic}}
\nc{\Bun}{\on{Bun}}
\nc{\IC}{\on{IC}}
\nc{\Fls}{\on{Fl}^{\frac{\infty}{2}}}
\nc{\ICs}{\on{IC}^{\frac{\infty}{2}}}
\nc{\ICsl}{\on{IC}^{\lambda+\frac{\infty}{2}}}
\nc{\ICslm}{\on{IC}^{\lambda+\frac{\infty}{2},-}}
\nc{\ICsm}{\on{IC}^{\frac{\infty}{2},-}}
\nc{\Aut}{\on{Aut}}
\nc{\rk}{\on{rk}}
\nc{\Sh}{\on{Sh}}
\nc{\Perv}{\on{Perv}}
\nc{\pos}{{\on{pos}}}
\nc{\Conv}{\on{Conv}}
\nc{\Sph}{\on{Sph}}
\nc{\Sat}{\on{Sat}}
\nc{\Sym}{\on{Sym}}
\nc{\BunBb}{\overline{\Bun}_B}
\nc{\BunNb}{\overline{\Bun}_N}
\nc{\BunTb}{\overline{\Bun}_T}
\nc{\BunBbm}{\overline{\Bun}_{B^-}}
\nc{\BunBbel}{\overline{\Bun}_{B,el}}
\nc{\BunBbmel}{\overline{\Bun}_{B^-,el}}
\nc{\Buno}{\overset{o}{\Bun}}
\nc{\BunPb}{{\overline{\Bun}_P}}
\nc{\BunBM}{\Bun_{B(M)}}
\nc{\BunBMb}{\overline{\Bun}_{B(M)}}
\nc{\BunPbw}{{\widetilde{\Bun}_P}}
\nc{\BunBP}{\widetilde{\Bun}_{B,P}}
\nc{\GUb}{\overline{G/U}}
\nc{\GUPb}{\overline{G/U(P)}}

\nc{\Hhom}{\underline{\on{Hom}}}
\nc\syminfty{\on{Sym}^{\infty}}
\nc\lal{\ol{\kappa_x}}
\nc\xl{\ol{x}}
\nc\thl{\ol{\theta}}
\nc\nul{\ol{\nu}}
\nc\mul{\ol{\mu}}
\nc\Sum\Sigma
\nc{\oX}{\overset{o}{X}{}}
\nc{\hl}{\overset{\leftarrow}h{}}
\nc{\hr}{\overset{\rightarrow}h{}}
\nc{\M}{{\mathcal M}}
\nc{\N}{{\mathcal N}}
\nc{\F}{{\mathcal F}}
\nc{\D}{{\mathcal D}}
\nc{\Q}{{\mathcal Q}}
\nc{\Y}{{\mathcal Y}}
\nc{\G}{{\mathcal G}}
\nc{\E}{{\mathcal E}}
\nc{\CalC}{{\mathcal C}}
\nc\Dh{\widehat{\D}}

\nc{\C}{{\mathcal C}}
\nc{\K}{{\mathcal K}}
\renewcommand{\H}{{\mathcal H}}

\nc{\T}{{\mathcal T}}
\nc{\V}{{\mathcal V}}
\renc{\P}{{\mathcal P}}
\nc{\A}{{\mathcal A}}
\nc{\B}{{\mathcal B}}
\nc{\U}{{\mathcal U}}

\nc{\Gr}{{\on{Gr}}}

\nc{\frn}{{\check{\mathfrak u}(P)}}

\nc{\fC}{\mathfrak C}
\nc{\p}{\mathfrak p}
\nc{\q}{\mathfrak q}
\nc\f{{\mathfrak f}}

\nc{\qo}{{\mathfrak q}}
\nc{\po}{{\mathfrak p}}
\nc{\s}{{\mathfrak s}}
\nc\w{\text{w}}

\renewcommand{\mod}{{\on{-mod}}}

\nc\mathi\iota
\nc\Spec{\on{Spec}}
\nc\Mod{\on{Mod}}
\nc{\tw}{\widetilde{\mathfrak t}}
\nc{\pw}{\widetilde{\mathfrak p}}
\nc{\qw}{\widetilde{\mathfrak q}}
\nc{\jw}{\widetilde j}

\nc{\grb}{\overline{\Gr}}

\renewcommand{\j}{\mathfrak j}

\nc{\kappach}{{\check\kappa_x}}
\nc{\Lambdach}{{\check\Lambda}{}}
\nc{\much}{{\check\mu}}
\nc{\omegach}{{\check\omega}}
\nc{\nuch}{{\check\nu}}
\nc{\etach}{{\check\eta}}
\nc{\alphach}{{\checka}}
\nc{\oblvtach}{{\check\oblvta}}
\nc{\pich}{{\check\pi}}
\nc{\ch}{{\check h}}

\nc{\Hb}{\overline{\H}}

\nc{\BA}{{\mathbb{A}}}
\nc{\BC}{{\mathbb{C}}}
\nc{\BE}{{\mathbb{E}}}
\nc{\BF}{{\mathbb{F}}}
\nc{\BG}{{\mathbb{G}}}
\nc{\BM}{{\mathbb{M}}}
\nc{\BO}{{\mathbb{O}}}
\nc{\BD}{{\mathbb{D}}}
\nc{\BN}{{\mathbb{N}}}
\nc{\BP}{{\mathbb{P}}}
\nc{\BQ}{{\mathbb{Q}}}
\nc{\BR}{{\mathbb{R}}}
\nc{\BZ}{{\mathbb{Z}}}
\nc{\BS}{{\mathbb{S}}}
\nc{\BV}{{\mathbb{V}}}

\nc{\CA}{{\mathcal{A}}}
\nc{\CB}{{\mathcal{B}}}

\nc{\CE}{{\mathcal{E}}}
\nc{\CF}{{\mathcal{F}}}
\nc{\CG}{{\mathcal{G}}}
\nc{\CH}{{\mathcal{H}}}

\nc{\CL}{{\mathcal{L}}}
\nc{\CC}{{\mathcal{C}}}
\nc{\CM}{{\mathcal{M}}}
\nc{\CN}{{\mathcal{N}}}
\nc{\cCN}{\check{{\mathcal{N}}}}
\nc{\CK}{{\mathcal{K}}}
\nc{\CO}{{\mathcal{O}}}
\nc{\CP}{{\mathcal{P}}}
\nc{\CQ}{{\mathcal{Q}}}
\nc{\CR}{{\mathcal{R}}}
\nc{\CS}{{\mathcal{S}}}
\nc{\CT}{{\mathcal{T}}}
\nc{\CU}{{\mathcal{U}}}
\nc{\CV}{{\mathcal{V}}}
\nc{\CW}{{\mathcal{W}}}
\nc{\CX}{{\mathcal{X}}}
\nc{\CY}{{\mathcal{Y}}}
\nc{\CZ}{{\mathcal{Z}}}
\nc{\CI}{{\mathcal{I}}}
\nc{\CJ}{{\mathcal{J}}}

\nc{\csM}{{\check{\mathcal A}}{}}
\nc{\oM}{{\overset{\circ}{\mathcal M}}{}}
\nc{\obM}{{\overset{\circ}{\mathbf M}}{}}
\nc{\oCA}{{\overset{\circ}{\mathcal A}}{}}
\nc{\obA}{{\overset{\circ}{\mathbf A}}{}}
\nc{\ooM}{{\overset{\circ}{M}}{}}
\nc{\osM}{{\overset{\circ}{\mathsf M}}{}}
\nc{\vM}{{\overset{\bullet}{\mathcal M}}{}}
\nc{\nM}{{\underset{\bullet}{\mathcal M}}{}}
\nc{\oD}{{\overset{\circ}{\mathcal D}}{}}
\nc{\obD}{{\overset{\circ}{\mathbf D}}{}}
\nc{\oA}{{\overset{\circ}{\mathbb A}}{}}
\nc{\op}{{\overset{\bullet}{\mathbf p}}{}}
\nc{\cp}{{\overset{\circ}{\mathbf p}}{}}
\nc{\oU}{{\overset{\bullet}{\mathcal U}}{}}
\nc{\oZ}{{\overset{\circ}{\mathcal Z}}{}}
\nc{\ofZ}{{\overset{\circ}{\mathfrak Z}}{}}
\nc{\oF}{{\overset{\circ}{\fF}}}

\nc{\fa}{{\mathfrak{a}}}
\nc{\fb}{{\mathfrak{b}}}
\nc{\fd}{{\mathfrak{d}}}
\nc{\ff}{{\mathfrak{f}}}
\nc{\fg}{{\mathfrak{g}}}
\nc{\fgl}{{\mathfrak{gl}}}
\nc{\fh}{{\mathfrak{h}}}
\nc{\fj}{{\mathfrak{j}}}
\nc{\fl}{{\mathfrak{l}}}
\nc{\fm}{{\mathfrak{m}}}
\nc{\fn}{{\mathfrak{n}}}
\nc{\fu}{{\mathfrak{u}}}
\nc{\fp}{{\mathfrak{p}}}
\nc{\fr}{{\mathfrak{r}}}
\nc{\fs}{{\mathfrak{s}}}
\nc{\ft}{{\mathfrak{t}}}
\nc{\fz}{{\mathfrak{z}}}
\nc{\fsl}{{\mathfrak{sl}}}
\nc{\hsl}{{\widehat{\mathfrak{sl}}}}
\nc{\hgl}{{\widehat{\mathfrak{gl}}}}
\nc{\hg}{{\widehat{\mathfrak{g}}}}
\nc{\chg}{{\widehat{\mathfrak{g}}}{}^\vee}
\nc{\hn}{{\widehat{\mathfrak{n}}}}
\nc{\chn}{{\widehat{\mathfrak{n}}}{}^\vee}

\nc{\fA}{{\mathfrak{A}}}
\nc{\fB}{{\mathfrak{B}}}
\nc{\fD}{{\mathfrak{D}}}
\nc{\fE}{{\mathfrak{E}}}
\nc{\fF}{{\mathfrak{F}}}
\nc{\fG}{{\mathfrak{G}}}
\nc{\fK}{{\mathfrak{K}}}
\nc{\fL}{{\mathfrak{L}}}
\nc{\fM}{{\mathfrak{M}}}
\nc{\fN}{{\mathfrak{N}}}
\nc{\fP}{{\mathfrak{P}}}
\nc{\fU}{{\mathfrak{U}}}
\nc{\fV}{{\mathfrak{V}}}
\nc{\fZ}{{\mathfrak{Z}}}

\nc{\ba}{{\mathbf{a}}}
\nc{\bb}{{\mathbf{b}}}
\nc{\bc}{{\mathbf{c}}}
\nc{\bd}{{\mathbf{d}}}
\nc{\bbf}{{\mathbf{f}}}
\nc{\be}{{\mathbf{e}}}
\nc{\bi}{{\mathbf{i}}}
\nc{\bj}{{\mathbf{j}}}
\nc{\bn}{{\mathbf{n}}}
\nc{\bo}{{\mathbf{o}}}
\nc{\bp}{{\mathbf{p}}}
\nc{\bq}{{\mathbf{q}}}
\nc{\bu}{{\mathbf{u}}}
\nc{\bv}{{\mathbf{v}}}
\nc{\bx}{{\mathbf{x}}}
\nc{\bs}{{\mathbf{s}}}
\nc{\by}{{\mathbf{y}}}
\nc{\bw}{{\mathbf{w}}}
\nc{\bA}{{\mathbf{A}}}
\nc{\bK}{{\mathbf{K}}}
\nc{\bB}{{\mathbf{B}}}
\nc{\bF}{{\mathbf{F}}}
\nc{\bC}{{\mathbf{C}}}
\nc{\bG}{{\mathbf{G}}}
\nc{\bD}{{\mathbf{D}}}
\nc{\bE}{{\mathbf{E}}}
\nc{\bH}{{\mathbf{H}}}
\nc{\bI}{{\mathbf{I}}}
\nc{\bM}{{\mathbf{M}}}
\nc{\bN}{{\mathbf{N}}}
\nc{\bO}{{\mathbf{O}}}
\nc{\bV}{{\mathbf{V}}}
\nc{\bW}{{\mathbf{W}}}
\nc{\bX}{{\mathbf{X}}}
\nc{\bZ}{{\mathbf{Z}}}
\nc{\bS}{{\mathbf{S}}}

\nc{\sA}{{\mathsf{A}}}
\nc{\sB}{{\mathsf{B}}}
\nc{\sC}{{\mathsf{C}}}
\nc{\sD}{{\mathsf{D}}}
\nc{\sF}{{\mathsf{F}}}
\nc{\sK}{{\mathsf{K}}}
\nc{\sM}{{\mathsf{M}}}
\nc{\sO}{{\mathsf{O}}}
\nc{\sW}{{\mathsf{W}}}
\nc{\sQ}{{\mathsf{Q}}}
\nc{\sP}{{\mathsf{P}}}
\nc{\sZ}{{\mathsf{Z}}}
\nc{\sr}{{\mathsf{r}}}
\nc{\bk}{{\mathsf{k}}}
\nc{\sg}{{\mathsf{g}}}
\nc{\sff}{{\mathsf{f}}}
\nc{\sfe}{{\mathsf{e}}}
\nc{\sfj}{{\mathsf{j}}}
\nc{\sfb}{{\mathsf{b}}}
\nc{\sfc}{{\mathsf{c}}}
\nc{\sd}{{\mathsf{d}}}
\nc{\sv}{{\mathsf{v}}}

\nc{\BK}{{\bar{K}}}

\nc{\tA}{{\widetilde{\mathbf{A}}}}
\nc{\tB}{{\widetilde{\mathcal{B}}}}
\nc{\tg}{{\widetilde{\mathfrak{g}}}}
\nc{\tG}{{\widetilde{G}}}
\nc{\TM}{{\widetilde{\mathbb{M}}}{}}
\nc{\tO}{{\widetilde{\mathsf{O}}}{}}
\nc{\tU}{{\widetilde{\mathfrak{U}}}{}}
\nc{\TZ}{{\tilde{Z}}}
\nc{\tx}{{\tilde{x}}}
\nc{\tbv}{{\tilde{\bv}}}
\nc{\tfP}{{\widetilde{\mathfrak{P}}}{}}
\nc{\tz}{{\tilde{\zeta}}}
\nc{\tmu}{{\tilde{\mu}}}

\nc{\urho}{\underline{\pi}}
\nc{\uB}{\underline{B}}
\nc{\uC}{{\underline{\mathbb{C}}}}
\nc{\ui}{\underline{i}}
\nc{\uj}{\underline{j}}
\nc{\ofP}{{\overline{\mathfrak{P}}}}
\nc{\oB}{{\overline{\mathcal{B}}}}
\nc{\og}{{\overline{\mathfrak{g}}}}
\nc{\oI}{{\overline{I}}}

\nc{\eps}{\varepsilon}
\nc{\hrho}{{\hat{\pi}}}

\nc{\one}{{\mathbf{1}}}
\nc{\two}{{\mathbf{t}}}

\nc{\Rep}{{\mathop{\operatorname{\rm Rep}}}}
\nc{\Tot}{{\mathop{\operatorname{\rm Tot}}}}
\nc{\Ker}{{\mathop{\operatorname{\rm Ker}}}}
\nc{\Hilb}{{\mathop{\operatorname{\rm Hilb}}}}
\nc{\End}{{\mathop{\operatorname{\rm End}}}}
\nc{\Ext}{{\mathop{\operatorname{\rm Ext}}}}
\nc{\CHom}{{\mathop{\operatorname{{\mathcal{H}}\it om}}}}
\nc{\GL}{{\mathop{\operatorname{\rm GL}}}}
\nc{\gr}{{\mathop{\operatorname{\rm gr}}}}
\nc{\Ld}{{\mathop{\operatorname{\rm Id}}}}
\nc{\de}{{\mathop{\operatorname{\rm def}}}}
\nc{\length}{{\mathop{\operatorname{\rm length}}}}
\nc{\supp}{{\mathop{\operatorname{\rm supp}}}}

\nc{\Cliff}{{\mathsf{Cliff}}}
\nc{\Fl}{\on{Fl}}
\nc{\Fib}{{\mathsf{Fib}}}
\nc{\Coh}{{\mathsf{Coh}}}
\nc{\QCoh}{{\on{QCoh}}}
\nc{\LndCoh}{{\on{IndCoh}}}
\nc{\FCoh}{{\mathsf{FCoh}}}

\nc{\reg}{{\text{\rm reg}}}

\nc{\cplus}{{\mathbf{C}_+}}
\nc{\cminus}{{\mathbf{C}_-}}
\nc{\cthree}{{\mathbf{C}_*}}
\nc{\Qbar}{{\bar{Q}}}
\nc\Eis{\on{Eis}}
\nc\Eisb{\ol\Eis{}}
\nc\Eisr{\on{Eis}^{rat}{}}
\nc\wh{\widehat}
\nc{\Def}{\on{Def_{\check{\fb}}(E)}}
\nc{\barZ}{\overline{Z}{}}
\nc{\barbarZ}{\overline{\barZ}{}}
\nc{\barpi}{\overline\iota}
\nc{\barbarpi}{\overline\barpi}
\nc{\barpip}{\overline\iota{}^+}
\nc{\barpim}{\overline\iota{}^-}

\nc{\fq}{\mathfrak q}

\nc{\fqb}{\ol{\fq}{}}
\nc{\fpb}{\ol{\fp}{}}
\nc{\fpr}{{\fp^{rat}}{}}
\nc{\fqr}{{\fq^{rat}}{}}

\nc{\hattimes}{\wh\otimes}

\nc{\bh}{{\bar{h}}}
\nc{\bOmega}{{\overline{\Omega(\check \fn)}}}

\nc{\seq}[1]{\stackrel{#1}{\sim}}

\nc{\cT}{{\check{T}}}
\nc{\cG}{{\check{G}}}
\nc{\cM}{{\check{M}}}
\nc{\cB}{{\check{B}}}
\nc{\cN}{{\check{N}}}

\nc{\ct}{{\check{\mathfrak t}}}
\nc{\cg}{{\check{\fg}}}
\nc{\cb}{{\check{\fb}}}
\nc{\cn}{{\check{\fn}}}

\nc{\cLambda}{{\check\Lambda}}

\nc{\cla}{{\check\kappa_x}}
\nc{\cmu}{{\check\mu}}
\nc{\clambda}{{\check\lambda}}
\nc{\cnu}{{\check\nu}}
\nc{\ceta}{{\check\eta}}

\nc{\DefbE}{{\on{Def}_{\cB}(E_\cT)}}

\nc{\imathb}{{\ol{\imath}}}
\nc{\rlr}{\overset{\longrightarrow}{\underset{\longrightarrow}\longleftarrow}}

\nc{\KG}{K\backslash G}
\nc{\comult}{{co\text{-}mult}}
\nc{\counit}{{co\text{-}unit}}
\nc{\uHom}{{\underline{\Maps}}}
\nc{\dgSch}{\on{Sch}}
\nc{\Sch}{\on{Sch}}
\nc{\affdgSch}{\on{Sch}^{\on{aff}}}
\nc{\affSch}{\on{Sch}^{\on{aff}}}
\nc{\Groupoids}{\on{Grpd}}
\nc{\inftygroup}{\on{Spc}}
\nc{\inftyCat}{\infty\on{-Cat}}
\nc{\StinftyCat}{\inftyCat^{\on{St}}}
\nc{\MoninftyCat}{\infty\on{-Cat}^{\on{Mon}}}
\nc{\SymMoninftyCat}{\infty\on{-Cat}^{\on{SymMon}}}
\nc{\SymMonStinftyCat}{\on{DGCat}^{\on{SymMon}}}
\nc{\MonStinftyCat}{\on{DGCat}^{\on{Mon}}}
\nc{\inftystack}{\on{Stk}}
\nc{\inftystackalg}{St\sfe^{1\text{-}alg}}
\nc{\inftyprestack}{\on{PreStk}}
\nc{\inftydgnearstack}{\on{NearStk}}
\nc{\inftydgstack}{\on{Stk}}
\nc{\inftydgstackalg}{DGSt\sfe^{1\text{-}alg}}
\nc{\inftydgprestack}{\on{PreStk}}
\nc{\dgindSch}{\on{indSch}}
\nc{\indSch}{{}^{\on{cl}}\!\on{indSch}}
\nc{\infSch}{\on{infSch}}
\nc{\dr}{{\on{dR}}}

\nc{\mmod}{{\on{-}\!{\mathbf{mod}}}}
\nc{\commod}{{\on{-}\!{\mathbf{comod}}}}
\nc{\Rres}{\mathbf{Res}}

\nc{\starr}{\text{\dh}}
\nc{\Spectra}{\on{Spectra}}
\nc{\Crys}{\on{Crys}}
\nc{\oblv}{{\on{oblv}}}
\nc{\ind}{{\on{ind}}}
\nc{\coind}{{\mathbf{coind}}}
\nc{\inv}{{\mathbf{inv}}}
\nc{\triv}{{\mathbf{triv}}}
\nc{\CMaps}{{\mathcal Maps}}
\nc{\Maps}{\on{Maps}}
\nc{\bMaps}{\mathbf{Maps}}
\nc{\BMaps}{\ul{\on{Maps}}}
\nc{\Grid}{\on{Grid}}
\nc{\hGrid}{\on{Grid}^{\geq\,\on{dgnl}}}
\nc{\Diag}{\on{Diag}}
\nc{\bDelta}{\mathbf{\Delta}}
\nc{\tCateg}{(\infty\on{-2)-Cat}}
\nc{\ul}{\underline}
\nc{\Seg}{\on{Seq}}
\nc{\biSeg}{\on{bi-Seq}}
\nc{\triSeg}{\on{tri-Seq}}
\nc{\quadSeg}{\on{quad-Seq}}
\nc{\nSeg}{\on{n-Seq}}
\nc{\Segm}{\on{Seg}^{\on{mkd}}}
\nc{\fLm}{\fL^{\on{mkd}}}
\nc{\inftyCatm}{\inftyCat^{\on{mkd}}}
\nc{\Blocks}{\mathbf{Blocks}}
\nc{\Snakes}{\mathbf{Snakes}}
\nc{\bifL}{\on{bi-}\!\fL}
\nc{\Sets}{\on{Sets}}
\nc{\Ran}{{\on{Ran}}}
\nc{\Vect}{\on{Vect}}
\nc{\Shv}{\on{Shv}}
\nc{\unn}{\mathbf{union}}
\nc{\Spc}{\on{Spc}}
\nc{\ppart}{(\!(t)\!)}
\nc{\qqart}{[\![t]\!]}
\nc{\Dmod}{\on{D}}
\nc{\cD}{\mathcal D}
\nc{\ocD}{\overset{\circ}{\cD}}
\nc{\sfo}{\mathsf{o}}
\nc{\sfob}{\mathsf{ob}}
\nc{\sfp}{\mathsf{p}}
\nc{\sfq}{\mathsf{q}}
\nc{\DGCat}{\on{DGCat}}
\renc{\det}{\on{det}}
\nc{\Conf}{\on{Conf}}
\nc{\Whit}{\on{Whit}}
\nc{\Reg}{\on{Reg}}
\nc{\Res}{\on{Res}}
\nc{\BunNbox}{(\overline\Bun_N^{\omega^\rho})_{\infty\cdot x}} 
\nc{\BunNmbox}{(\overline\Bun_{N^-}^{\omega^\rho})_{\infty\cdot x}}
\nc{\bHecke}{\on{Hecke}_{\cG,\cT}}
\nc{\Hecke}{\on{Hecke}}
\nc{\bCZ}{\ol\CZ}
\nc{\oCZ}{\overset{\circ}\CZ} 
\nc{\boCZ}{\ol{\oCZ}}
\nc{\sotimes}{\overset{!}\otimes}
\nc{\semiinf}{{\frac{\infty}{2}}}
\nc{\coInd}{\on{coInd}}
\nc{\Ind}{\on{Ind}}
\nc{\bCM}{\overset{\bullet}\CM{}}
\nc{\bCF}{\overset{\bullet}\CF{}}
\nc{\SI}{\on{SI}}
\nc{\KL}{\on{KL}}
\nc{\KM}{\on{KM}}
\nc{\fSet}{\on{fSet}}
\nc{\IndCoh}{\on{IndCoh}}
\nc{\bGamma}{\mathbf{\Gamma}}
\nc{\fc}{\mathfrak c} 
\nc{\untl}{{\on{untl}}}
\nc{\disj}{{\on{disj}}}
\nc{\onfact}{{\on{fact}}}
\nc{\lax}{{\on{lax}}}

\NewDocumentCommand{\boxt}{e{_^}}{
  \mathbin{\mathop{\boxtimes}\displaylimits
    \IfValueT{#1}{_{#1}}
    \IfValueT{#2}{^{#2}}
  }
}

\begin{document}

\title[Representations of loop groups as factorization module categories]{Representations of loop groups \\ as factorization module categories}

\dedicatory{To Sasha Beilinson} 

\author{Lin Chen, Yuchen Fu, Dennis Gaitsgory and David Yang}

\date{\today}


\maketitle

\tableofcontents

\section*{Introduction}

In this paper we show that the (2-)category of categorical representations of the loop group embeds
fully faithfully into the (2-)category of factorization module categories with respect to the affine Grassmannian. 

\ssec{Why might one expect this kind of thing to be true?}

\sssec{}

Let $X$ be a smooth curve and $x_0\in X$ a point on it. Let $G$ be a reductive group. 

\medskip

A construction going back to A.~Beilinson (and probably first fleshed out in \cite{Ga0}) says that there exists a 
family parameterized by $X$, whose fiber at $x\neq x_0$ is the product $\Gr_{G,x}\times \fL(G)_{x_0}$, and
whose fiber at $x_0$ is $\fL(G)_{x_0}$, where:

\begin{itemize}

\item $\Gr_{G,x}$ is the affine Grassmannian of $G$ associated with the formal disc $\cD_x$ around $x$;

\item $\fL(G)_{x_0}$ is the loop group of $G$ associated with the formal punctured disc $\ocD_{x_0}$ around $x_0$.

\end{itemize}

One can generalize this construction slightly, and construct a similar family parameterized by $X^n$ for any
$n$.

\sssec{}

In modern language, this construction says that we can regard $\Gr_G$ as a factorization space,
and $\fL(G)_{x_0}$ as a factorization module space at $x_0$ with respect to $\Gr_G$ (see \secref{ss:fact sp},
where the relevant definitions are recalled). 

\medskip

Another insight of Beilinson's, articulated in the early 2000's, says that this factorization module space is
universal in the following (imprecise) sense: $\fL(G)_{x_0}$ should be isomorphic to the factorization
homology of $\Gr_G$ over $\ocD_{x_0}$:
\begin{equation} \label{e:integral}
\underset{\ocD_{x_0}}\int\, \Gr_G\simeq \fL(G)_{x_0},
\end{equation}
whatever this means. 

\sssec{}

The above principle can be made precise in the topological setting: 

\medskip

The affine Grassmannian $\Gr_G$ is homotopy-equivalent to $\Omega(G)$, the loop space of $G$. 
Hence, applying Lurie's non-abelian Poincar\'e duality, we obtain that the 
factorization homology of $\Gr_G$ over the circle is homotopy-equivalent to 
$$\Maps_{\on{cont}}(S^1,G),$$ 
which can be regarded as a topological counterpart of $\fL(G)_{x_0}$.

\sssec{} \label{sss:linear}

The goal of this paper is to give an articulation of this principle in algebraic geometry. We do so by
finding an appropriate linearized statement. 

\medskip

There are (at least) two ways to linearize the above principle: $0$-categorical and $1$-categorical.

\medskip

The $0$-categorical way is straightforward: it says that the factorization homology 
of $\on{C}_\cdot(\Gr_G)$ on $\ocD_{x_0}$ (which can be made precise sense of)
maps isomorphically to $\on{C}_\cdot(\fL(G)_{x_0})$. 

\medskip

This is a true statement, and it will serve as an ingredient in the proof
of our main theorem, see \secref{sss:triv intro}. 

\sssec{}

The 1-categorical linearization is richer:

\medskip

We can consider the category of D-modules on $\Gr_G$ as a factorization category, and 
we can take its factorization homology over $\ocD_{x_0}$. This is a monoidal category
that maps to $\Dmod(\fL(G)_{x_0})$ (the latter is viewed as a monoidal category under
convolution), and we can ask whether this functor is an equivalence. 

\medskip

The answer is ``no" and that is for a simple reason: factorization homology of a category
in the de Rham setting is too loose a construction; we rarely expect it to be equivalent 
to something sensible. However, we do expect that it has the ``right category" as a quotient.

\medskip

And indeed, this happens to be this case. Our main result, \thmref{t:initial}, is
equivalent to saying that the functor
$$\underset{\ocD_{x_0}}\int\, \Dmod(\Gr_G)\to \Dmod(\fL(G)_{x_0})$$
is a quotient. 

\sssec{}

Finally, we would like to draw a (loose) analogy between the above statement and 
the contractibility result of \cite{Ga2}:

\medskip

The latter says that for a complete curve $X$, the pullback functor 
$$\Dmod(\Bun_G(X))\to \Dmod(\Gr_{G,\Ran})$$
is fully faithful. 

\medskip

Note also that in the global setting, the corresponding statement in topology is that
the factorization homology of $\Gr_G\simeq \Omega^2(BG)$ over $X$ is homotopy equivalent to
$$\Maps_{\on{cont}}(X,BG),$$
which is in turn homotopy-equivalent to $\Bun_G(X)$. 

\ssec{What is actually done in this paper?}

\sssec{}

In the main body of the paper we do not talk about factorization homology of categories over the punctured disc.
Rather, we formulate our main result as follows:

\medskip

We can view $\Dmod(\fL(G)_{x_0})$ as a factorization module category at $x_0$
with respect to the factorization category $\Dmod(\Gr_G)$; as such it carries a commuting
action of $\fL(G)_{x_0}$ ``on the right". This structure allows us to construct a functor
\begin{equation} \label{e:functor Intro}
\fL(G)_{x_0}\mmod\to \Dmod(\Gr_G)\mmod^\onfact_{x_0}, \quad \bC \mapsto \bC^{\on{fact}_{x_0},\Dmod(\Gr_G)},
\end{equation}
where:

\begin{itemize}

\item $\fL(G)_{x_0}\mmod$ is the 2-category of categorical representations of $\fL(G)_{x_0}$ (see \secref{ss:cat rep loop group},
where the definition is recalled);

\medskip

\item $\Dmod(\Gr_G)\mmod^\onfact_{x_0}$ is the 2-category of factorization module categories at $x_0$ with respect to
$\Dmod(\Gr_G)$ (see \secref{ss:fact mod cat}).

\end{itemize}

\medskip

Our main result, \thmref{t:initial}, says that the functor \eqref{e:functor Intro} is fully faithful. 

\sssec{}

Concretely, \thmref{t:initial} says that for $\bC_1,\bC_2\in \fL(G)_{x_0}\mmod$, the functor
\begin{equation} \label{e:Hom Intro}
\on{Funct}_{\fL(G)_{x_0}\mmod}(\bC_1,\bC_2)\to
\on{Funct}_{\Dmod(\Gr_G)\mmod^\onfact_{x_0}}(\bC_1^{\on{fact}_{x_0},\Dmod(\Gr_G)},\bC_2^{\on{fact}_{x_0},\Dmod(\Gr_G)}), 
\end{equation}
induced by \eqref{e:functor Intro}, is an equivalence. 

\medskip

It is easy to see that when proving this statement, one can assume that the source category, i.e., $\bC_1$, is
a copy of $\Vect$, equipped with the trivial action of $\fL(G)_{x_0}$. I.e., we have to show that for $\bC\in \fL(G)_{x_0}\mmod$, the functor
\begin{equation} \label{e:Hom triv Intro}
\on{inv}_{\fL(G)_{x_0}}(\bC)\to 
\on{Funct}_{\Dmod(\Gr_G)\mmod^\onfact_{x_0}}(\Vect^{\on{fact}_{x_0},\Dmod(\Gr_G)},\bC^{\on{fact}_{x_0},\Dmod(\Gr_G)}),
\end{equation} 
induced by \eqref{e:functor Intro}, is an equivalence. 

\sssec{}

The first step in the proof of \thmref{t:initial} consists of rewriting the right-hand side of \eqref{e:Hom triv Intro} in 
terms of \emph{factorization modules} over a \emph{factorization algebra}. 

\medskip

Namely, we show that for any $\wt\bC\in \Dmod(\Gr_G)\mmod^\onfact_{x_0}$, we have a canonical equivalence
$$\on{Funct}_{\Dmod(\Gr_G)\mmod^\onfact_{x_0}}(\Vect^{\on{fact}_{x_0},\Dmod(\Gr_G)},\wt\bC)
\simeq \omega_{\Gr_G}\mod^\onfact(\wt\bC)_{x_0},$$
where:

\begin{itemize}

\item $\omega_{\Gr_G}$ is the dualizing sheaf on $\Gr_G$, viewed as a factorization algebra in the factorization category
$\Dmod(\Gr_G)$ (see \secref{sss:fact alg});

\medskip

\item $\omega_{\Gr_G}\mod^\onfact(\wt\bC)_{x_0}$ denotes the category of factorization modules at $x_0$ in $\wt\bC$
with respect to $\omega_{\Gr_G}$ (see \secref{sss:fact mod}).

\end{itemize}

\medskip

Thus, we obtain that \thmref{t:initial} is equivalent to the following statement, which appears as \thmref{t:main} in the main
body of the paper:

\medskip

\noindent {\it The functor
\begin{equation} \label{e:Hom triv Intro bis}
\on{inv}_{\fL(G)_{x_0}}(\bC)\to 
 \omega_{\Gr_G}\mod^\onfact\left(\bC^{\on{fact}_{x_0},\Dmod(\Gr_G)}\right)_{x_0},
\end{equation} 
induced by \eqref{e:functor Intro}, is an equivalence}. 

\sssec{}

We now briefly indicate the main steps involved in the proof of \thmref{t:main}. 

\medskip

The first step, which is the geometric core of the argument says that when proving \eqref{e:Hom triv Intro bis},
one can replace $\bC$ by its maximal subcategory $\on{alm-inv}_{\fL(G)_{x_0}}(\bC)$, on which the action 
of $\fL(G)_{x_0}$ is \emph{almost trivial} (see \secref{ss:alm triv loop} for what this means). 

\medskip

This step is carried out in \secref{s:proof of main: key} and it involves playing with the geometry of the fusion
construction. 

\sssec{}

Once we assume that the action of $\fL(G)_{x_0}$ is almost trivial, there is no more ``de Rham complexity" in the
game, and the idea is to try to mimic the topological argument.

\medskip

The second step in the proof of \thmref{t:main} consists of replacing ``almost trivial" by ``trivial". This is done
in \secref{ss:red to triv}, by a categorical Koszul duality type argument.

\medskip

This reduces the proof that \eqref{e:Hom triv Intro bis} is an equivalence to the case when $\bC=\Vect$. I.e.,
we have to show that the functor
\begin{equation} \label{e:Vect triv Intro}
\Vect^{\fL(G)_{x_0}}\to 
\omega_{\Gr_G}\mod^\onfact\left(\Vect^{\on{fact}_{x_0},\Dmod(\Gr_G)}\right)_{x_0}
\end{equation} 
is an equivalence. 

\sssec{}  \label{sss:triv intro}

The third step in the proof of \thmref{e:Vect triv Intro} consists of establishing the equivalence \eqref{e:Vect triv Intro}.
We show that both sides admit monadic forgetful functors to $\Vect$, and we show that the corresponding monads 
are isomorphic. This is done in \secref{s:comp}. 

\medskip

That said, one could view/prove the assertion that \eqref{e:Vect triv Intro} is an isomorphism differently:

\medskip

The left-hand side in \eqref{e:Vect triv Intro} identifies with the category of modules over $\on{C}_\cdot(\fL(G)_{x_0})$, where the structure
of associative algebra on it is induced by the group structure on $\fL(G)_{x_0}$. 

\medskip

The right-hand side in \eqref{e:Vect triv Intro} identifies with the category of factorization modules with respect to
the factorization algebra $\on{C}_\cdot(\Gr_G)$. Hence, it can be further identified with
$$\Bigl(\underset{\ocD_{x_0}}\int\, \on{C}_\cdot(\Gr_G)\Bigr)\mod.$$

\medskip

Hence, the assertion that \eqref{e:Vect triv Intro} is an isomorphism is equivalent to the assertion that the map
$$\Bigl(\underset{\ocD_{x_0}}\int\, \on{C}_\cdot(\Gr_G)\Bigr)\mod\to \on{C}_\cdot(\fL(G)_{x_0})\mod$$
is an isomorphism, which is the linearization statement from \secref{sss:linear}.

\ssec{Extensions, applications and relation to prior work}

\sssec{}

Recall that our main result, \thmref{t:initial}, has the following form: it says that a certain functor from the 2-category of modules
over a given monoidal category to the 2-category of factorization modules over a factorization category is fully faithful.

\medskip

As far as we know, this is the \emph{second-of-its-kind} result of this form. The first such result was established in \cite{Bogd}. There,
the main theorem says that a certain naturally defined functor
$$\QCoh(\on{LS}_G^{\on{restr}}(\ocD_{x_0}))\mmod \to \Rep(G)\mmod^\onfact_{x_0}$$
is fully faithful, where:

\begin{itemize}

\item $\on{LS}_G^{\on{restr}}(\ocD_{x_0})$ is the stack of local systems with restricted variation on $\ocD_{x_0}$ with respect to $G$
(defined as in \cite[Sect. 1.4]{AGKRRV}). 

\end{itemize}

\medskip

This result paves a way to questions of spectral decomposition in the \emph{restricted} local geometric Langlands theory,
see \cite[Sect. 2.6]{Ga5}.

\sssec{}

That said, one expects a stronger result to be true. Namely, we expect that (a similarly defined functor)
\begin{equation} \label{e:LocSys}
\QCoh(\on{LS}_G(\ocD_{x_0}))\mmod \to \Rep(G)\mmod^\onfact_{x_0}
\end{equation}
is fully faithful, where:

\begin{itemize}

\item $\on{LS}_G(\ocD_{x_0})$ is the stack of de Rham local systems on $\ocD_{x_0}$ with respect to $G$. 

\end{itemize}

\medskip

A result of this form would be of crucial importance for the full (i.e., unrestricted) local geometric Langlands theory.

\sssec{}

Note, however, that when $G=T$ is a torus, the functor \eqref{e:LocSys} identifies with the functor
\eqref{e:functor Intro} for the dual torus. 

\medskip

So, thanks to our \thmref{t:main}, the fully-faithfulness of
\eqref{e:LocSys} is known for tori. 

\sssec{}

For a non-commutative $G$, one can formulate the following conjecture:

\medskip

Recall that to a category $\bC$ acted on by the loop group, one can associate its Whittaker model,
$\Whit(\bC)$, see \cite[Sect. 1.3.3]{GLC2}. Moreover, this construction works in the factorization setting.

\medskip

Thus, on the one hand, we can consider 
$$\Whit(G):=\Whit(\Dmod(\Gr_G))$$ 
as a factorization category. 

\medskip

On the other hand, we can consider the monoidal category
$$\on{bi-Whit}(\fL(G)_{x_0}):=\End_{\fL(G)_{x_0}\mmod}(\Whit(\Dmod(\fL(G)_{x_0})).$$

A construction similar to \eqref{e:functor Intro} gives rise to a functor
\begin{equation} \label{e:Whit}
\on{bi-Whit}(\fL(G)_{x_0})\mmod\to \Whit(G)\mmod^\onfact_{x_0}.
\end{equation}

We conjecture that the functor \eqref{e:Whit} is fully faithful. (Note that when $G$ is a torus, the Whittaker operation
is the identity functor, and the functor \eqref{e:Whit} is just the functor \eqref{e:functor Intro}).

\sssec{}

Now, the geometric Casselman-Shalika equivalence says that we have an equivalence of factorization categories
$$\Whit(G)\simeq \Rep(\cG),$$
where $\cG$ is the Langlands dual group of $G$.

\medskip

And one of the conjectures in local geometric Langlands says that 
$$\on{bi-Whit}(\fL(G)_{x_0})\simeq \QCoh(\on{LS}_G(\ocD_{x_0}))$$
as monoidal categories.

\medskip

Under this equivalence, the functor \eqref{e:functor Intro} is supposed to correspond to the functor
\eqref{e:Whit}. This is the basis for believing that \eqref{e:Whit} is fully faithful.

\sssec{}

Up until now, we have discussed the idea that \eqref{e:integral} is an equivalence, when we linearize
our algebro-geometric objects by applying the functor $\Dmod(-)$.

\medskip

One may wonder, however, whether one could expect a similar behavior when we linearize by means of
$\QCoh(-)$ instead. 

\medskip

The answer is that an analog of \thmref{t:initial} will fail in this case. We explain a counterexample in \secref{ss:coh}.

\medskip

That said, this failure (at least, in our example) happens for subtle homological algebra reasons 
(it takes place, so to say, at the cohomological $-\infty$). It is not impossible that one could
modify the definitions around the objects involved and make an analog of \thmref{t:initial} hold.

\sssec{}

We now explain one concrete application of our \thmref{t:initial}, rather in its incarnation as \thmref{t:main},
to usual representation theory.

\medskip

Let $\kappa$ be a non-negative integral Kac-Moody level. To it we can associate a chiral algebra
$\BV^{\on{Int}}_{G,\kappa}$. 

\medskip

For example, when $G$ is semi-simple and simply-connected, $\BV^{\on{Int}}_{G,\kappa}$ 
is the ``maximal integrable quotient" of the vacuum chiral algebra $\BV_{\fg,\kappa}$. 
When $G$ is a torus, $\BV^{\on{Int}}_{G,\kappa}$ is the lattice chiral algebra.

\medskip

It is known that \emph{at the level of abelian categories}, the category $\BV^{\on{Int}}_{G,\kappa}\mod^\onfact_{x_0}$
is equivalent to the category $\Rep(\fL(G)_{x_0},\kappa)$ of \emph{integrable} Kac-Moody modules at level $\kappa$, i.e., 
representations of the central extension 
$$1\to \BG_m\to \wh\fL(G)_{\kappa,x_0} \to \fL(G)_{x_0} \to 1,$$
corresponding to $\kappa$, on which the central $\BG_m$ acts by the standard character. 

\medskip

In \thmref{t:int}, we show that this equivalence continues to hold at the derived level. 

\medskip

It is known that the category $\Rep(\fL(G)_{x_0},\kappa)$ is semi-simple (even at the derived level). So the content of
\thmref{t:int} is that there are no higher Exts between irreducible objects of $\BV^{\on{Int}}_{G,\kappa}\mod^\onfact_{x_0}$.

\ssec{Structure of the paper}

We now explain the structure of the paper section-by-section.

\sssec{}

In \secref{s:background} we supply some background in factorization, mostly borrowed from \cite[Sects. B and C]{GLC2}. 

\sssec{}

In \secref{s:functor} we construct the functor \eqref{e:functor Intro}, state our main result (\thmref{t:initial}) and
reduce it to the case when the source category is $\Vect$. 

\sssec{}

In \secref{s:dualizing} we state \thmref{t:main}, which says that \eqref{e:Hom triv Intro bis} is an equivalence. 
We prove that \thmref{t:main} is logically equivalent to \thmref{t:initial}.

\medskip

The rest of the paper (up until \secref{s:KM}) is devoted to the proof of \thmref{t:main}. 

\sssec{}

In \secref{s:almost trivial} we discuss the notion of \emph{almost trivial} action of a group (in particular, a loop group)
on a category. 

\medskip

We state \thmref{t:red to alm const case} that says that for the proof of \thmref{t:main} we can 
assume that the action of $\fL(G)_{x_0}$ on $\bC$ is trivial. 

\medskip

We show that  \thmref{t:red to alm const case} allows us to reduce \thmref{t:main} to the case 
when $\bC=\Vect$. 

\medskip

\thmref{t:red to alm const case} will be proved in Sects. \ref{s:torus}-\ref{s:proof of main-end}. 

\sssec{}

In \secref{s:proofs Kosz} we prove several technical statements formulated in \secref{s:almost trivial}. 

\sssec{}

In \secref{s:comp} we prove \thmref{t:main} for $\bC=\Vect$ by a direct calculation, which
amounts to an algebro-geometric incarnation of a particular case of Lurie's non-abelian 
Poincar\'e duality. 

\sssec{}

In \secref{s:torus} we prove \thmref{t:red to alm const case} for a torus using local geometric class field theory. 

\sssec{}

In \secref{s:proof of main: key} we supply a key geometric argument that tackles \thmref{t:red to alm const case}  
in the non-abelian case.

\sssec{}

In \secref{s:proof of main-end} we finish the proof of  \thmref{t:red to alm const case}.

\sssec{}

In \secref{s:KM} we discuss the application of our \thmref{t:main} to integrable Kac-Moody representation.
In addition, we \emph{disprove} a coherent version of \thmref{t:initial}.

\sssec{}

In \secref{s:inf type} we (re)collect some material pertaining to the theory of D-modules on algebro-geometric
objects of infinite type.

\sssec{}

In \secref{s:1-cat rep} we (re)collect some material pertaining to categorical representations of groups and
in particular, loop groups.

\sssec{}

In \secref{s:fact cat} we supply proofs of statements pertaining to factorization categories and modules. 

\ssec{Conventions and notation}

\sssec{}

Throughout the paper, we will be working over a ground field $k$, assumed algebraically closed and of characteristic $0$.

\sssec{}

We will be working with $k$-linear \emph{higher algebra}. The basic object of study for us is the $\infty$-category
of $k$-linear DG categories, denoted $\DGCat$ (see \cite[Sect. 1]{GR1}, where the relevant definitions are discussed
in detail). 

\medskip

The $\infty$-category $\DGCat$ carries a symmetric monoidal structure, the Lurie tensor product. Its unit is
the category $\Vect$ of chain complexes of $k$-vector spaces. 

\medskip

In particular, all objects of $\DGCat$ are automatically enriched over $\Vect$; for $\bC\in \DGCat$ and 
$\bc_1,\bc_2\in \DGCat$, we will denote by $\CHom_\bC(\bc_1,\bc_2)$ the corresponding object of $\Vect$. 

\medskip

The category $\DGCat$ has an internal Hom, denoted
$$\on{Funct}_{\DGCat}(\bC_1,\bC_2), \quad \bC_1,\bC_2\in \DGCat.$$

\sssec{}

As we will be interested in D-modules\footnote{Except \secref{ss:coh}.}, we can stay in the world of \emph{classical}
(as opposed to \emph{derived}) algebraic geometry. By a prestack we will mean a (presentable) functor 
$$(\Sch^{\on{aff}})^{\on{op}}\to \infty\on{-Grpds},$$
where $\Sch^{\on{aff}}$ is the category of classical affine schemes over $k$. 

\medskip

All specific classes of algebro-geometric objects (e.g., schemes, ind-schemes, algebraic stacks, etc.) are full
subcategories in the category $\on{PreStk}$ of prestacks. 

\sssec{}

The theory of D-modules on prestacks locally of finite type is built in \cite[Chapter 4]{GR2}. An extension 
to relevant algebro-geometric objects of infinite type is discussed in \secref{s:inf type}. 

\sssec{}

The material in this paper that has to do with factorization relies to a large extent on \cite{GLC2}. We make
a brief review in \secref{s:background}, and refer the reader to {\it loc. cit.} for details. 

\sssec{}

The main result of this paper, \thmref{t:initial}, talks about comparing 2-categories. 

\medskip

In this paper, by a ($k$-linear) 2-category, we will mean an $\infty$-category enriched over $\DGCat$. 
In particular, given a 2-category $\fC$, and $\fc_1,\fc_2\in \fC$, we will denote by
$$\on{Funct}_{\fC}(\fc_1,\fc_2)$$
the corresponding object of $\DGCat$. 

\sssec{}

The main source of 2-categories will be of the form 
$$\bA\mmod,$$
where $\bA$ is a monoidal DG category, i.e., an associative algebra object in $\DGCat$. 

\medskip

Objects of $\bA\mmod$ are $\bA$-module categories, and for $\bC_1,\bC_2\in \bA\mmod$,
$$\on{Funct}_{\bA\mmod}(\bC_1,\bC_2)$$
is the naturally defined category of $\bA$-linear functors. 

\sssec{}

All other conventions and notations follow ones adopted in \cite{AGKRRV} and \cite{GLC2}. 

\ssec{Acknowledgements}

The third author would like to thank Sasha Beilinson for generously sharing his ideas, in particular about the factorization
perspective on local representation theory and \eqref{e:integral}. 

\medskip

We would like to thank Gurbir Dhillon
for helping us with an observation that proved crucial for \thmref{t:int} (see Remark \ref{r:vacuum in heart}).

\medskip

We would also like to thank Sam Raskin for some very helpful discussions on topics related to this paper: 
much of the mathematics that this paper relies on was developed in collaboration with him. 

\section{Background and preliminaries} \label{s:background}

For the readers convenience, in this section, we will (re)collect the definitions of the main 
players appearing in the paper.  The discussions will mostly repeat \cite[Sect. B and C]{GLC2}. 

\ssec{The arc and loop groups}

The material here follows \cite[Sects. B.3 and B.4]{GLC2}.

\sssec{}

The Ran space of $X$, denoted $\Ran$ is a prestack that attaches to a test affine scheme $S$ the
set of finite non-empty subsets of $\Hom(S,X)$.

\medskip

We let $\Ran_{x_0}$ be a variant of $\Ran$, where we consider finite subsets with a distinguished element
corresponding to
$$S\to \on{pt}\overset{x_0}\to X.$$

\sssec{}

For an $S$-point $\ul{x}$ of $\Ran$, let $\wh\cD_{\ul{x}}$ be the \emph{complete} formal disc around $\ul{x}$, 
i.e. the completion of $S\times X$ at the union of graphs of the maps comprising $\ul{x}$. 
This is ind-affine ind-scheme. We let $\cD_{\ul{x}}$ be the colimit of $\wh\cD_{\ul{x}}$ taken in the category of
affine schemes.

\medskip

For example, for $\ul{x}=\{x_0\}$ and $S=\Spec(R)$, if $t$ is a local coordinate at $x_0$, then
$$\wh\cD_{\ul{x}}=\on{Spf}(R\qqart):=\underset{n}{``\on{colim}"}\, \Spec(R[t]/t^n) \text{ and }
\cD_{\ul{x}}=\Spec(R\qqart).$$

\medskip

The above union of graphs is naturally a closed subscheme of $\cD_{\ul{x}}$. We denote by
$\ocD_{\ul{x}}$ the open complement. This is also an affine scheme.

\medskip

In the above example, 
$$\ocD_{\ul{x}}=\Spec(R\ppart).$$

\sssec{}

We let $\fL^+(G)_\Ran$ be the group-scheme over $\Ran$, whose points are pairs $(\ul{x},g)$, where
$\ul{x}\in \Hom(S,\Ran)$ and $g$ is a point of $\Hom(\cD_{\ul{x}},G)$. Note that the latter is the same as
$\Hom(\wh\cD_{\ul{x}},G)$, since $G$ is affine. 

\medskip

We let $\fL(G)_\Ran$ be the group-scheme over $\Ran$, whose points are pairs $(\ul{x},g)$, where
$\ul{x}\in \Hom(S,\Ran)$ and $g$ is a point of $\Hom(\ocD_{\ul{x}},G)$. 

\medskip

We will denote by $(-)_{\Ran_{x_0}}$ and $(-)_{x_0}$ the base change of the above objects along
$\Ran_{x_0}\to \Ran$ and $\on{pt}\overset{\{x_0\}}\to \Ran$, respectively. 

\medskip

Explicitly, the group of $S$-points of $\fL^+(G)_{x_0}$ is
$$\Maps(\on{Spf}(R\qqart),G)\simeq \Maps(\Spec(R\qqart),G)=G\left(R\qqart\right).$$

The group pf $S$-points of $\fL(G)_{x_0}$ is
$$\Maps(\Spec(R\ppart),G)=G\left(R\ppart\right).$$

\sssec{}

The object of study of this paper is categories equipped with an action of $\fL(G)_{x_0}$. We refer the reader
to \secref{ss:cat rep loop group}, where this notion is reviewed.  

\ssec{Factorization spaces} \label{ss:fact sp}

The material here follows \cite[Sects. B.1 and B.2]{GLC2}.

\sssec{}

A factorization space $\CT$ is a prestack $\CT_\Ran$ over $\Ran$ equipped with the datum of isomorphisms:
\begin{equation} \label{e:fact space}
\CT_\Ran\underset{\Ran,\on{union}}\times (\Ran\times \Ran)_\disj\simeq 
(\CT_\Ran\times \CT_\Ran) \underset{\Ran\times \Ran}\times (\Ran\times \Ran)_\disj,
\end{equation} 
where:

\begin{itemize} 

\item $(\Ran\times \Ran)_\disj\subset (\Ran\times \Ran)$ is the disjoint locus, i.e., the open subfunctor
consisting of pairs $(\ul{x}_1,\ul{x}_2)\in (\Ran\times \Ran)$ such that $\on{Graph}_{\ul{x}_1}\cap \on{Graph}_{\ul{x}_2}=\emptyset$;

\medskip

\item $\on{union}$ is the union map $\Ran\times \Ran\to \Ran$.

\end{itemize}

The isomorphisms \eqref{e:fact space} must be equipped with a homotopy-coherent data of commutativity and associativity.

\begin{rem}

We distinguish notationally $\CY$ and $\CY_\Ran$: the latter is a just prestack over $\Ran$, and the former 
taken into account the factorization structure.

\end{rem} 

\sssec{}

Note that for $(\ul{x}_1,\ul{x}_2)\in (\Ran\times \Ran)_\disj$ and $\ul{x}=\ul{x}_1\cup \ul{x}_2$, we have
$$\cD_{\ul{x}}\simeq \cD_{\ul{x}_1}\sqcup \cD_{\ul{x}_2} \text{ and } \ocD_{\ul{x}}\simeq \ocD_{\ul{x}_1}\sqcup \ocD_{\ul{x}_2}.$$ 
 
These isomorphisms endow $\fL^+(G)_{\Ran}$ and $\fL(G)_{\Ran}$ with a factorization structure. We denote the 
resulting factorization spaces by $\fL^+(G)$ and $\fL(G)$, respectively.

\sssec{}
\label{sss GrG factspc}

A key geometric player is the factorization space
$$\Gr_G:=\fL(G)/\fL^+(G).$$

Explicitly, for a point $\ul{x}$ of $\Ran$, the fiber $\Gr_{G,\ul{x}}$ is the set of pairs $(\CP_G,\alpha)$, where:

\begin{itemize}

\item $\CP_G$ is a $G$-bundle on $\cD_{\ul{x}}$;

\item $\alpha$ is a trivialization of $\CP_G|_{\ocD_{\ul{x}}}$.

\end{itemize}
Or equivalently, it is the set of pairs $(\CP_G^{\on{glob}},\beta)$
\begin{itemize}

\item $\CP_G^{\on{glob}}$ is a $G$-bundle on $X$;

\item $\beta$ is a trivialization of $\CP_G^{\on{glob}}|_{X\setminus \ul{x}}$.

\end{itemize}

\sssec{} \label{sss:fact mod sp}

Let $\CT$ be a factorization space. A factorization module space $\CT_m$ at $x_0$ with respect to $\CT$ is a prestack
$(\CT_m)_{\Ran_{x_0}}$ over $\Ran_{x_0}$, equipped with the datum of isomorphisms:
\begin{equation} \label{e:fact space module}
(\CT_m)_{\Ran_{x_0}}\underset{\Ran_{x_0},\on{union}}\times (\Ran\times \Ran_{x_0})_\disj\simeq 
(\CT_\Ran\times (\CT_m)_{\Ran_{x_0}}) \underset{\Ran\times \Ran_{x_0}}\times (\Ran\times \Ran_{x_0})_\disj,
\end{equation} 
equipped with a homotopy-coherent datum of associativity against \eqref{e:fact space}.

\medskip

We will often think of $\CT_m$ as a prestack $(\CT_m)_{x_0}$, equipped with an additional datum of extension to a prestack
over $\Ran_{x_0}$, all of whose fibers are specified by \eqref{e:fact space module}.

\sssec{} \label{sss:vac mod space}

For a factorization space $\CT$, the pullback
$$\CT_{\Ran_{x_0}}:=\CT_\Ran\underset{\Ran}\times \Ran_{x_0}$$
has a natural factorization structure against $\CT$.

\medskip

We denote the resulting factorization module space by $\CT^{\on{fact}_{x_0}}$. We refer to it as the 
\emph{vacuum} factorization module space at $x_0$. 

\sssec{}

For a factorization space, one can talk about it having a \emph{unital}, \emph{counital} and \emph{corr-unital}
structure. See \secref{ss unital factspc} for their definitions.

\medskip

For example, $\Gr_G$ is unital, $\fL^+(G)$ is counital, and $\fL(G)$ is corr-unital. 

\medskip

If $\CT$ is factorization space that is \emph{unital} (resp., \emph{counital}, \emph{corr-unital}) and $\CT_m$
is a factorization module space at $x_0$ with respect to $\CT$, one can talk about $\CT_m$ being 
\emph{unital} (resp., \emph{counital}, \emph{corr-unital}) against the corresponding structure on $\CT$. 

\ssec{Factorization categories}

The material here follows \cite[Sect. B.11]{GLC2}. The reader is referred to {\it loc. cit.} for more details.

\sssec{}

Let $\CY$ be a prestack (assumed locally of finite type). A crystal of categories $\ul\bC$ on $\CY$ is an assignment:

\begin{itemize}

\item $(S\overset{y}\to \CY) \mapsto \bC_{S,y}\in \Dmod(S)\mmod$, where $S$ is an affine scheme of finite type,
and $\Dmod(S)$ is viewed as a symmetric monoidal category via the $\sotimes$ operation;

\medskip

\item $(S_1\overset{f}\to S_2) \mapsto \bC_{S_1,y_1}\simeq \Dmod(S_1)\underset{\Dmod(S_2)}\otimes \bC_{S_2,y_2}$,
where $y_1=y_2\circ f$, and the symmetric monoidal functor $\Dmod(S_2)\to \Dmod(S_1)$ is $f^!$.

\item A homotopy-coherent system of compatibilities for higher order compositions.

\end{itemize}

\sssec{}

The most basic example of a crystal of categories is $\ul{\Dmod}(\CY)$, whose value on every $S\overset{y}\to \CY$
is $\Dmod(S)$. 

\sssec{}

We let $\mathbf{CrysCat}(\CY)$ denote the 2-category of crystals of categories over $\CY$. 

\medskip

We have a naturally defined functor
\begin{equation} \label{e:bGamma}
\bGamma(\CY,-):\mathbf{CrysCat}(\CY)\to \Dmod(\CY)\mmod, \quad
\ul\bC\mapsto \underset{S\overset{y}\to \CY}{\on{lim}}\, \bC_{S,y}.
\end{equation}

Recall (see \cite[Definition 1.3.7]{Ga3}) that $\CY_{\dr}$ is said to be \emph{1-affine} if the functor \eqref{e:bGamma}
is an equivalence. 

\medskip

For example $\Ran_\dr$ is 1-affine (see \cite[Lemma B.8.15]{GLC2}).

\sssec{}

For a map $g:\CY_1\to \CY_2$, there is a tautologically defined pullback functor 
$$g^*:\mathbf{CrysCat}(\CY_2)\to \mathbf{CrysCat}(\CY_1).$$

When there is no ambiguity for $g$, for $\ul\bC\in \mathbf{CrysCat}(\CY_2)$, we will sometimes write
$$\ul\bC|_{\CY_1}:=g^*(\ul\bC).$$

\medskip

We have a naturally defined functor
$$g^!:\bGamma(\CY_2,\ul\bC)\to \bGamma(\CY_1,g^*(\ul\bC)).$$

\medskip

For $\ul\bC_i\in \mathbf{CrysCat}(\CY_i)$, $i=1,2$, we define 
$$\ul\bC_1\boxt \ul\bC_2\in \mathbf{CrysCat}(\CY_1\times \CY_2)$$
naturally: for $y_i\in \Hom(S,\CY_i)$, 
$$(\ul\bC_1\boxt \ul\bC_2)_{S,(y_1,y_2)}:=(\ul\bC_1)_{S,y_1}\underset{\Dmod(S)}\otimes (\ul\bC_2)_{S,y_2}.$$

We have a naturally defined functor, 
$$\bGamma(\CY_1,\ul\bC_1)\otimes \bGamma(\CY_2,\ul\bC_2)\to \bGamma(\CY_1\times \CY_2,\ul\bC_1\boxt \ul\bC_2),$$
to be denoted 
$$\bc_1,\bc_2\mapsto \bc_1\boxt \bc_2.$$

\sssec{}
\label{sss nonunital factcat}

A factorization category $\bA$ is a crystal of categories $\ul\bA$ over $\Ran$, equipped with an equivalence
\begin{equation} \label{e:fact cat}
\on{union}^*(\ul\bA)|_{(\Ran\times \Ran)_\disj}\simeq (\ul\bA\boxt \ul\bA)|_{(\Ran\times \Ran)_\disj}
\end{equation}
equipped additionally with a homotopy-coherent datum of commutativity and associativity. 

\begin{rem}

Even though $\Ran_{\dr}$ is 1-affine, we distinguish notationally between $\bA$, $\ul\bA$ and
$$\bA_\Ran:=\bGamma(\Ran,\ul\bA).$$

That said, the data of \eqref{e:fact cat}, can be equivalently spelled out in terms of $\bA_\Ran$. 

\end{rem}

\sssec{}

The most basic example of a factorization category, denoted $\Vect$, is when the corresponding crystal
of categories over $\Ran$ is $\ul{\Dmod}(\Ran)$ itself. 

\sssec{} \label{sss:Dmods on fact}

Let $\CT$ be a factorization space. Assume that $\CT_{\Ran}$ is locally of finite type. Then the crystal of categories
$$(\ul{x}\in \Hom(S,\Ran)) \mapsto \Dmod(\CT_{\ul{x}})$$ 
has a natural factorization structure.

\medskip 

We denote the resulting factorization category by $\Dmod(\CT)$.

\sssec{} \label{sss:Gr fact}

A prime example of this is when $\CT=\Gr_G$. This way, we obtain a factorization category $\Dmod(\Gr_G)$,
which is the second main player in this paper.

\begin{rem}

Let us weaken the hypothesis that $\CT_{\Ran}$ is locally of finite type. Instead, let us require that for every $\ul{x}\in \Hom(S,\Ran)$
(with $S$ an affine scheme of finite type), the prestack $\CT_{\ul{x}}$ is an \emph{ind-placid ind-scheme}, see \secref{sss:ind-placid}
for what this means.

\medskip

In this case, the assignments
$$(\ul{x}\in \Hom(S,\Ran)) \mapsto \Dmod^!(\CT_{\ul{x}}) \text{ and } (\ul{x}\in \Hom(S,\Ran)) \mapsto \Dmod_*(\CT_{\ul{x}})$$
are crystals of categories equipped with natural factorization structures.

\medskip

We denote the resulting factorization categories by
$$\Dmod^!(\CT) \text{ and } \Dmod_*(\CT),$$
respectively.

\end{rem}

\sssec{}

Given a factorization category, one can talk about a \emph{unital} structure on it. We refer the reader to \cite[Sect. C.11]{GLC2}.
Some of this material will be reviewed in \secref{s:fact cat} of this paper.

\ssec{Factorization \emph{module} categories} \label{ss:fact mod cat}

The material here follows \cite[Sect. B.12]{GLC2}. The reader is referred to {\it loc. cit.} for more details.

\sssec{}

Let $\bA$ be a factorization category. A factorization module category $\bC$ at $x_0$ with respect to $\bA$ is a crystal 
of categories $\ul\bC$ over $\Ran_{x_0}$ equipped with an equivalence
\begin{equation} \label{e:fact mod cat}
\on{union}^*(\ul\bC)|_{(\Ran\times \Ran_{x_0})_\disj}\simeq (\ul\bA\boxt \ul\bC)|_{(\Ran\times \Ran_{x_0})_\disj},
\end{equation}
equipped with a homotopy-coherent data of associativity against \eqref{e:fact cat}.

\medskip

The totality of factorization module categories at $x_0$ with respect to $\bA$ naturally forms a 2-category, to be denoted
$$\bA\mmod^\onfact_{x_0}.$$

\sssec{}

We have a tautological forgetful functor
\begin{equation} \label{e:oblv fact mod cat}
\oblv_\bA:\bA\mmod^\onfact_{x_0}\to \DGCat, \quad \bC\mapsto \bC_{x_0}.
\end{equation} 

Note, however, that the functor \eqref{e:oblv fact mod cat} is \emph{not} conservative. Rather, it induces conservative functors between $\on{Funct}$ categories. As a formal corollary, we obtain the following:

\begin{lem}
Let $F:\bC_1\to \bC_2$ be a 1-morphism in $\bA\mmod^\onfact_{x_0}$ that admits an adjoint (on either side).
Then if 
$$\oblv_\bA(F):\oblv_\bA(\bC_1)\to \oblv_\bA(\bC_2)$$
is an equivalence, then so is $F$.
\end{lem}

\sssec{} \label{sss:colimits in fact}

The category $\bA\mmod^\onfact_{x_0}$ has colimits that commute with the forgeftul functors
\begin{equation} \label{e:eval x}
\bA\mmod^\onfact_{x_0}\to \Dmod(S)\mmod
\end{equation} 
for any $\ul{x}:S\to \Ran$.

\medskip

Assume now that $\bA$ is such that for any $\ul{x}:S\to \Ran$, the category $\bA_{S,\ul{x}}$ is dualizable. 
Then $\bA\mmod^\onfact_{x_0}$ also contains all limits that commute with the forgetful functors \eqref{e:eval x}.

\sssec{}
\label{sss:fact mod cat via fact mod space}

Let $\CT$ be as in \secref{sss:Dmods on fact}, and let $\CT_m$ be a factorization module space at $x_0$ with respect to 
$\CT$, also assumed locally of finite type.

\medskip

Then the assignment
$$(\ul{x}\in \Hom(S,\Ran_{x_0})) \mapsto \Dmod((\CT_m)_{\ul{x}})$$ 
is a crystal of categories that carries a natural factorization structure against $\Dmod(\CT)$.

\medskip

We denote the resulting object of $\Dmod(\CT)\mmod^\onfact_{x_0}$ by
$$\Dmod(\CT_m)\in \Dmod(\CT)\mmod^\onfact_{x_0}.$$ 

\sssec{} \label{sss:vac fact mod cat}

Let $\bA$ be a factorization category. The pullback of $\ul\bA$ along $\Ran_{x_0}\to \Ran$ has a natural factorization
structure against $\bA$. 

\medskip

We denote the resulting object by
$$\bA^{\on{fact}_{x_0}}\in \bA\mod^\onfact_{x_0},$$
and refer to it as the \emph{vacuum factorization module category} at $x_0$. 

\medskip

Note that in the example of \secref{sss:Dmods on fact}, 
we have
$$\Dmod(\CT)^{\on{fact}_{x_0}}\simeq \Dmod(\CT^{\on{fact}_{x_0}}),$$
where $\CT^{\on{fact}_{x_0}}$ is as in \secref{sss:vac mod space}.

\sssec{} \label{sss:fact Rres}

Let $\Phi:\bA_1\to \bA_2$ be a homomorphism of factorization categories. For $\bC_2\in \bA_2\mod^\onfact_{x_0}$
one attaches an object
$$\Rres_\Phi(\bC_2)\in \bA_1\mod^\onfact_{x_0},$$
which is the universal object such that there is a morphism $\Rres_\Phi(\bC_2) \to \bC_2$ compatible with $\Phi$, see \cite[Sect. B.12.11]{GLC2}.

\medskip

This construction will be reviewed in \secref{s:fact cat} of the present paper. 

\sssec{}

Given a unital factorization category $\bA$, one can talk about \emph{unital factorization module categories} at $x_0$ with respect to
$\bA$, see \cite[Sect. C.14]{GLC2}. We denote the resulting 2-category by
$$\bA\mmod^\onfact_{x_0}.$$

When we want to ignore the unital structure on $\bA$, we will denote the resulting 2-category of \emph{non-unital} factorization
module categories by 
$$\bA\mmod^{\on{fact-n.u.}}_{x_0}.$$

There exists an obvious forgetful functor
\begin{equation} \label{e:forget unital cat}
\bA\mmod^\onfact_{x_0}\to \bA\mmod^{\on{fact-n.u.}}_{x_0}.
\end{equation}

Note, however, that \eqref{e:forget unital cat} is \emph{not} fully faithful. Yet, it is 1-fully faithful, i.e., it is fully
faithful on the $\on{Funct}$-categories. 

\ssec{Factorization algebras and modules}

\sssec{} \label{sss:fact alg}

Let $\bA$ be a factorization category. A factorization algebra $\CA$ in $\bA$ is an object $\CA_\Ran\in \bA_\Ran$, equipped
with an identification between
$$\on{union}^!(\CA)\in \bGamma\left((\Ran\times \Ran)_\disj,\on{union}^*(\ul\bA)\right) \text{ and }
\CA\boxt \CA\in \bGamma\left((\Ran\times \Ran)_\disj,\bA\boxt \bA\right)$$
with respect to the equivalence \eqref{e:fact cat}, and further equipped with a homotopy-coherent
system of compatibilities. 

\medskip

Taking $\bA=\Vect$, we recover the usual notion of factorization algebra. 

\sssec{} \label{sss:ind fact alg}

Let $\Phi:\bA_1\to \bA_2$ be a  homomorphism of factorization categories. Then $\Phi$
sends factorization algebras in $\bA_1$ to factorization algebras in $\bA_2$.  

\sssec{} \label{sss:fact mod}

Let $\bC$ be a factorization module category at $x_0$ with respect to $\bA$. Given a factorization algebra $\CA$ in 
$\bA$, we define a factorization $\CA$-module at $x_0$ in $\bC$ to be an object $\CM\in \bC_{\Ran_{x_0}}$, equipped 
with an identification
$$\on{union}^!(\CM)\in \bGamma\left((\Ran\times \Ran_{x_0})_\disj,\on{union}^*(\ul\bC)\right) \text{ and }
\CA\boxt \CM\in \bGamma\left((\Ran\times \Ran_{x_0})_\disj,\bA\boxt \bA\right)$$
with respect to the equivalence \eqref{e:fact mod cat}, and further equipped with a homotopy-coherent
system of compatibilities. 

\sssec{}
\label{sss:notation for intermal factmod}

We let
$$\CA\mod^\onfact(\bC)_{x_0}$$
denote category of factorization $\CA$-modules at $x_0$ in $\bC$. 

\medskip

We have a naturally defined \emph{conservative} forgetful functor
$$\oblv_\CA:\CA\mod^\onfact(\bC)_{x_0}\to \bC_{x_0}.$$

\sssec{} \label{sss:vac}

In the particular case when $\bC=\bA^{\on{fact}_{x_0}}$, we will use a short-hand notation
$$\CA\mod^\onfact_{x_0}:= \CA\mod^\onfact(\bA^{\on{fact}_{x_0}})_{x_0}.$$

In the particular case when $\bA=\Vect$, and thus $\CA$ is a usual factorization algebra we thus recover the usual category 
$$\CA\mod^\onfact_{x_0}$$
of factorization $\CA$-modules at $x_0$. 

\sssec{}

Consider the pullback of $\CA_{\Ran}$ with respect to $\Ran_{x_0}\to \Ran$. It has a natural factorization module structure
against $\CA$. We denote the resulting object by
$$\CA^{\on{fact}_{x_0}}\in \CA\mod^\onfact_{x_0}.$$

We refer to it as the \emph{vacuum} factorization $\CA$-module at $x_0$. 

\sssec{}

Let $\phi:\CA_1\to \CA_2$ be a homomorphism between factorization algebras in $\bA$. Then to $\phi$ one attaches a restriction
functor
$$\Res_\phi:\CA_2\mod^\onfact(\bC)_{x_0}\to \CA_1\mod^\onfact(\bC)_{x_0},$$
see \cite[Sect. B.9.25]{GLC2}.

\medskip

This construction will be reviewed also in \secref{s:fact cat} of the present paper. 

\sssec{}

Let now $\Phi:\bA_1\to \bA_2$ be a homomorphism between factorization categories. Let $\CA_1$ be a factorization algebra in $\bA_1$,
and set $\CA_2:=\Phi(\CA_1)$, which we view as a factorization algebra in $\bA_2$ (see \secref{sss:ind fact alg} above).

\medskip

Let $\bC_2$ be an object of $\bA_2\mmod^\onfact_{x_0}$ and set
$$\bC_1:=\Rres_\Phi(\bC_2),$$
see \secref{sss:fact Rres}. 

\medskip

In this case, we have a canonical equivalence
\begin{equation} \label{e:equiv fact mod}
\CA_1\mod^\onfact(\bC_1)_{x_0}\simeq \CA_2\mod^\onfact(\bC_2)_{x_0},
\end{equation}
see \cite[Lemma B.12.14]{GLC2}, to be reviewed in \secref{s:fact cat}. 

\sssec{}
\label{sss:notation for intermal unital factmod}

Assume that $\bA$ is unital. In this case, one can talk about a factorization algebra being unital,
see \cite[Sect. 7]{GLC2}.

\medskip

Let $\bC$ be a unital factorization module category at $x_0$ with respect to $\bA$, and let $\CA$
be a unital factorization algebra in $\bA$. In this case, the category of factorization $\CA$-modules 
at $x_0$ in $\bC$ contains a full subcategory of \emph{unital} factorization modules, see \cite[C.11.19]{GLC2},
to be reviewed in \secref{s:fact cat} of the present paper.

\medskip

In this case, we denote this subcategory by $\CA\mod^\onfact(\bC)_{x_0}$, and the category of
not necessarily unital $\CA$-modules by $\CA\mod^{\on{fact-n.u.}}(\bC)_{x_0}$. 

\section{Statement of the result} \label{s:functor}

In this section we state the main result of this paper, \thmref{t:initial}, and then reformulate it
in terms of the computation of a category of factorization modules with respect to
a certain factorization algebra. 

\ssec{From loop group modules to factorization modules over \texorpdfstring{$\Gr_G$}{GrG}} \label{ss:constr of funct}

In this subsection we will construct a functor
\begin{equation} \label{e:functor in one direction}
\fL(G)_{x_0}\mmod \to \Dmod(\Gr_G)\mmod_{x_0}^\onfact
\end{equation} 
with the basic property that it preserves the forgetful functors from both sides to $\DGCat$. 

\sssec{}

Recall that $\Gr_G$ is a unital factorization space and $\Dmod(\Gr_G)$ is a unital factorization category, see \secref{sss:Gr fact} and \secref{sss unital factspc example}.

\medskip

The functor \eqref{e:functor in one direction} will be constructed by exhibiting a suitable bimodule
category, to be denoted $\Dmod(\fL(G)_{x_0})^{\on{fact}_{x_0},\Dmod(\Gr_G)}$, which is an object of
$\Dmod(\Gr_G)\mmod^\onfact_{x_0}$, and as such carries an action of $\fL(G)_{x_0}$.

\medskip

Moreover, the underlying DG category of $\Dmod(\fL(G)_{x_0})^{\on{fact}_{x_0},\Dmod(\Gr_G)}$ (i.e., its
fiber over $x_0\in \Ran_{x_0}$) will
identify with $\Dmod(\fL(G)_{x_0})$ itself, with the $\fL(G)_{x_0}$-action given by right translations. 

\sssec{}

The category $\Dmod(\fL(G)_{x_0})^{\on{fact}_{x_0},\Dmod(\Gr_G)}$ will be defined as the category of D-modules
on a certain prestack, to be denoted $\Gr^{\on{level}^\infty_{x_0}}_{G,\Ran_{x_0}}$, which is a factorization module space
\emph{at $x_0$} with respect to the factorization algebra space $\Gr_G$ (see \secref{sss:fact mod sp}), and as such equipped with an action of $\fL(G)_{x_0}$.  

\medskip

Moreover, the fiber of 
$\Gr^{\on{level}^\infty_{x_0}}_{G,\Ran_{x_0}}$ at $x_0\in \Ran_{x_0}$ will identify with $\fL(G)_{x_0}$ itself, with the $\fL(G)_{x_0}$-action given by right translations. 

\sssec{} \label{sss:construct space}

The space $\Gr^{\on{level}^\infty_{x_0}}_{G,\Ran_{x_0}}$ is constructed as follows:

\medskip

For an affine test scheme $S$ and a given $S$-point $\ul{x}$ of $\Ran_{x_0}$, the fiber 
$$\Gr^{\on{level}^\infty_{x_0}}_{G,\ul{x}}$$
of $\Gr^{\on{level}^\infty_{x_0}}_{G,\Ran_{x_0}}$ over it consists of the data of 
$$(\CP_G,\alpha,\epsilon),$$
where:

\begin{itemize}

\item $\CP_G$ is a $G$-bundle on $S\times X$;

\item $\alpha$ is a trivialization of $\CP_G$ over $S\times X-\on{Graph}_{\ul{x}}$;

\item $\epsilon$ is a trivialization of $\CP_G$ over the formal completion of $S\times X$ along $S\times x_0$.

\end{itemize}

\medskip 

The (unital) factorization module space structure on $\Gr^{\on{level}^\infty_{x_0}}_{G,\Ran_{x_0}}$ is contructed using Beauville-Laszlo theorem.

\medskip

The action of $\fL(G)_{x_0}$ on $\Gr^{\on{level}^\infty_{x_0}}_{G,\Ran_{x_0}}$ is given by the standard regluing procedure. It is easy to see this action preserves the factorization module space structure in above.

\sssec{}

As in Sect. \ref{ss:arcs and loops}, one can show $\Gr^{\on{level}^\infty_{x_0}}_{G,\ul{x}}$ is ind-placid, and there is a canonical equivalence $\Dmod^!(\Gr^{\on{level}^\infty_{x_0}}_{G,\ul{x}})\simeq \Dmod_*(\Gr^{\on{level}^\infty_{x_0}}_{G,\ul{x}})$. We will simply write $\Dmod(\Gr^{\on{level}^\infty_{x_0}}_{G,\ul{x}})$ for these categories.

\medskip
Now the assignment
\[
    \ul{x} \mapsto \Dmod(\Gr^{\on{level}^\infty_{x_0}}_{G,\ul{x}})
\]
is a crystal of categories over $\Ran_{x_0}$ that carries a natural (unital) factorization structure against $\Dmod(\Gr_G)$. See \secref{sss:fact mod cat via fact mod space}. We define
\[
    \Dmod(\fL(G)_{x_0})^{\on{fact}_{x_0},\Dmod(\Gr_G)} \in \Dmod(\Gr_G)\mmod_{x_0}^\onfact
\]
to be this object\footnote{Using Proposition \ref{sss:adj test for factres}, one can show 
    \[
        \Dmod(\fL(G)_{x_0})^{\on{fact}_{x_0},\Dmod(\Gr_G)} \simeq \Rres_{p^!} \big( \Dmod(\fL(G)_{x_0})^{\on{fact}_{x_0},\Dmod(\fL(G))}\big),
    \]
    where $p^!:\Dmod(\Gr_G) \to \Dmod^!(\fL(G))$ is the unital factorization functor given by $!$-pullbacks and $\Rres_{p^!}$ is the restriction functor along it (see \secref{sss:fact Rres}).
}.

\medskip

By construction, $\Dmod(\fL(G)_{x_0})$ acts on $\Dmod(\fL(G)_{x_0})^{\on{fact}_{x_0},\Dmod(\Gr_G)}$. We define the functor in \eqref{e:functor in one direction} by
\begin{equation} \label{e:functor in one direction formula}
\bC\mapsto \Dmod(\fL(G)_{x_0})^{\on{fact}_{x_0},\Dmod(\Gr_G)}\underset{\Dmod(\fL(G)_{x_0})}\otimes \bC.
\end{equation} 

By \corref{c:tensor over LG} and \secref{sss:colimits in fact}, the functor \eqref{e:functor in one direction formula} commutes with limits and colimits. 

\sssec{}

We can now state the main result of this paper:

\begin{thm} \label{t:initial}
The functor \eqref{e:functor in one direction} is fully faithful.
\end{thm} 

\sssec{}

Before we proceed any further, let us remark that the functor \eqref{e:functor in one direction} is \emph{not}
an equivalence. Indeed, it is easy to exhibit an object in $\Dmod(\Gr_G)\mmod^\onfact_{x_0}$ that 
which is not in the essential image of \eqref{e:functor in one direction}. 

\medskip

Namely, restriction to the unit section defines a factorization functor
$$\iota^!:\Dmod(\Gr_G)\to \Vect.$$

Hence, we obtain a restriction functor
$$\Rres_{\iota^!}: \Vect\mmod^\onfact_{x_0}\to  \Dmod(\Gr_G)\mmod^\onfact_{x_0},$$
(see \secref{sss:fact Rres}).

\medskip

Nonzero objects in the essential image of this functor do not lie in the essential image of the functor
\eqref{e:functor in one direction} unless $G$ is trivial \footnote{Sketch of proof: suppose $\Rres_{\iota^!}(\bC)$ is contained in the essential image of \eqref{e:functor in one direction}. By Theorem \ref{t:main} below, we have
\[
    \bC_{x_0}^{\fL(G)_{x_0}} \simeq \omega_{\Gr_G}\mod^\onfact(\Rres_{\iota^!}(\bC))_{x_0} \simeq \iota^!(\omega_{\Gr_G})\mod^\onfact(\bC)_{x_0} \simeq k\mod^\onfact(\bC)_{x_0} \simeq \bC_{x_0},
\]
which is impossible unless $G$ is trivial or $\bC\simeq 0$.}. 

\begin{rem}
    For the trival group, it is easy to see the corresponding functor
    \begin{equation}
        \label{e:functor in one direction trival group}
        \DGCat \to \Vect\mmod^\onfact_{x_0}
    \end{equation}
    is fully faithful. However, it is \emph{not} an equivalence. To see this, consider the Ran space $\Ran_\circ$ for the punctured curve $X-x_0$. Write $\Ran_\circ:=\Ran_\circ\sqcup \{\emptyset\}$, and consider the map
    \[
        j:\Ran_\circ \to \Ran_{x_0},\; \ul{y} \mapsto \ul{y}\sqcup\{x_0\}.
    \]
    One can show $\Ran_\circ$, viewed as a prestack over $\Ran_{x_0}$, has a natural factorization structure against the factorization space $\Ran$ (see \secref{sss:fact mod sp}). By \secref{sss:fact mod cat via fact mod space}, we obtain a factorization module category with respect to $\Vect$, such that $\bGamma(\Ran_{x_0},-)$ sends it to $\Dmod(\Ran_\circ)$. Moreover, using the \emph{unital} Ran spaces (see \secref{sss:unital Ran space}), we can upgrade it to a \emph{unital} factorization module category with respect to $\Vect$. In other words, we obtain a bizarre object, denote by
    \[
        \Vect^{\on{fact}_{x_0},\on{disj}} \in \Vect\mmod^\onfact_{x_0},
    \]
    which is not isomorphic to the vaccum factorization module category $\Vect^{\on{fact}_{x_0}}$, but has the same fiber at $x_0$ (i.e. $\Vect$). It is clear this obect is not in the essential image of \eqref{e:functor in one direction trival group}.

\end{rem}

\begin{rem}
At the moment we do not know how to characterize (even conjecturally) the full subcategory of $\Dmod(\Gr_G)\mmod^\onfact_{x_0}$
equal to the essential image of \eqref{e:functor in one direction}.
\end{rem} 

\ssec{Functoriality} \label{ss:functoriality}

In this subsection we establish a functoriality property of the construction in \secref{ss:constr of funct} with respect to homomorphisms
of reductive groups. 

\sssec{}

Let $\phi:G'\to G$ be a homomorphism between connected reductive groups. By a slight abuse of notation we will
denote by the same symbol $\phi$ the resulting homomorphism
$$\fL(G')_{x_0}\to \fL(G)_{x_0}.$$

Let 
$$\Gr_\phi:\Gr_{G'}\to \Gr_{G}$$
between (unital) factorization spaces.

\medskip

Direct image with respect to $\Gr_\phi$ has a natural structure of (unital) factorization functor
$$(\Gr_\phi)_*:\Dmod(\Gr_{G'})\to \Dmod(\Gr_G).$$

\sssec{} \label{sss:Gr level morphism}

Note that we also have a morphism between (unital) factorization module spaces
\begin{equation} \label{e:Gr level morphism}
\Gr^{\on{level}^\infty_{x_0}}_{G',\Ran_{x_0}} \to \Gr^{\on{level}^\infty_{x_0}}_{G,\Ran_{x_0}},
\end{equation}
compatible with $\Gr_\phi$ and the right actions of $\fL(G')_{x_0}$ and $\fL(G)_{x_0}$, respectively.

\medskip

In particular, \eqref{e:Gr level morphism} induces a morphism
\begin{equation} \label{e:Gr level morphism bis}
\underset{\fL(G')_{x_0}}{\underline{\Gr^{\on{level}^\infty_{x_0}}_{G',\Ran_{x_0}} \times \fL(G)_{x_0}}} \to \Gr^{\on{level}^\infty_{x_0}}_{G,\Ran_{x_0}},
\end{equation}
compatible with the actions of $\fL(G)_{x_0}$, where $\underset{H}{\underline{(-)}}$ means ``divide by the diagonal action of $H$."

\medskip

Moreover, the diagram
$$
\CD
\underset{\fL(G')_{x_0}}{\underline{\Gr^{\on{level}^\infty_{x_0}}_{G',\Ran_{x_0}} \times \fL(G)_{x_0}}} @>>> \Gr^{\on{level}^\infty_{x_0}}_{G,\Ran_{x_0}} \\
@VVV @VVV \\
\Gr_{G',\Ran_{x_0}} @>>> \Gr_{G,\Ran_{x_0}}
\endCD
$$
in Cartesian.

\medskip

In particular, we obtain that the morphism \eqref{e:Gr level morphism bis} is ind-proper. 

\sssec{}

Let us denote by $\Rres_\phi$ the restriction functor
$$\Dmod(\Gr_G)\mmod^\onfact_{x_0}\to \Dmod(\Gr_{G'})\mmod^\onfact_{x_0},$$
corresponding to the factorization functor $(\Gr_\phi)_*$, see \secref{sss:fact Rres}.

\medskip

From \secref{sss:Gr level morphism}, for any $\fL(G)_{x_0}$-module category $\bC$ and the $\fL(G')_{x_0}$-module category $\bC'$ given by restriction along $\fL(G')_{x_0}\to \fL(G)_{x_0}$, we obtain a naturally defined functor
\begin{equation} \label{e:res functor grps}
(\bC')^{\on{fact}_{x_0},\Dmod(\Gr_{G'})}\to  \bC^{\on{fact}_{x_0},\Dmod(\Gr_G)},
\end{equation}
compatible with $\Gr_\phi$.

\medskip

In particular, we obtain a 1-morphism
\begin{equation} \label{e:res functor grps bis}
(\bC')^{\on{fact}_{x_0},\Dmod(\Gr_{G'})}\to \Rres_{\Gr_\phi}(\bC^{\on{fact}_{x_0},\Dmod(\Gr_G)}).
\end{equation} 
in $\Dmod(\Gr_{G'})\mmod^\onfact_{x_0}$. 

\medskip

We claim:

\begin{lem} \label{l:res functor grps}
The 1-morphism \eqref{e:res functor grps bis}
is an isomorphism.
\end{lem}

\begin{proof}

Follows from Proposition \ref{sss:adj test for factres}.

\end{proof} 

\ssec{The first reduction step} \label{ss:first reduction}

From now on, until \secref{s:KM}, we will be occupied with the proof of \thmref{t:initial}.  In this subsection we perform
the first reduction case: we show that we can assume that the source object of $\fL(G)_{x_0}\mmod$ is $\Vect$,
equipped with the trivial action of $\fL(G)_{x_0}$.

\sssec{}

Let $\bC$ be a category acted on by $\fL(G)$. Let 
\begin{equation} \label{e:factor C}
\bC^{\on{fact}_{x_0},\Dmod(\Gr_G)}\in \Dmod(\Gr_G)\mmod^\onfact_{x_0}
\end{equation}
denote its image under the functor \eqref{e:functor in one direction formula}.

\medskip

The assertion of \thmref{t:initial} is that for $\bC_1,\bC_2$ as above, the functor
\begin{equation} \label{e:ff restated}
\on{Funct}_{\fL(G)_{x_0}\mmod}(\bC_1,\bC_2) \to 
\on{Funct}_{\Dmod(\Gr_{G,\Ran})\mmod^\onfact_{x_0}}(\bC_1^{\on{fact}_{x_0},\Dmod(\Gr_G)},\bC_2^{\on{fact}_{x_0},\Dmod(\Gr_G)})
\end{equation} 
is an equivalence.

\medskip

In this subsection we will perform the first reduction step in the proof of \thmref{t:initial}. Namely,
we will show that we can assume that $\bC_1=\Vect$, equipped with the trivial action of $\fL(G)_{x_0}$.

%

\sssec{}

First, since the functor \eqref{e:functor in one direction formula} is compatible with \emph{colimits}, both sides
in \eqref{e:ff restated} send colimits in $\bC_1$ to limits in $\DGCat$. 

\medskip

Hence, it is enough to show that \eqref{e:ff restated} is an isomorphism on objects of $\fL(G)\mmod$
that generate this category under colimits.

\sssec{}

Note that a collection of generating objects for $\fL(G)_{x_0}\mmod$ provided by
\begin{equation} \label{e:free generators}
\Dmod(\fL(G)_{x_0})\otimes \bC_0,
\end{equation} 
where $\bC_0$ is a plain DG category (i.e., the action of $\fL(G)_{x_0}$ on \eqref{e:free generators} comes from the action
by left translations on the first factor). 

\medskip

This reduces the assertion that \eqref{e:ff restated} is an equivalence to the particular case when $\bC_1$ is of the form $\Dmod(\fL(G)_{x_0})\otimes \bC_0$.

\sssec{}

It is easy to see that for $\bC_1:=\bC'_1\otimes \bC_0$, where the action comes from the first factor, we have
$$\on{Funct}_{\fL(G)_{x_0}\mmod}(\bC_1,\bC_2)  \simeq \on{Funct}(\bC_0,\on{Funct}_{\fL(G)_{x_0}\mmod}(\bC'_1,\bC_2))$$
and
\begin{multline*}
\on{Funct}_{\Dmod(\Gr_G)\mmod_{x_0}^\onfact}(\bC_1^{\on{fact}_{x_0},\Dmod(\Gr_G)},\bC_2^{\on{fact}_{x_0},\Dmod(\Gr_G)})  \simeq \\
\simeq 
\on{Funct}(\bC_0,\on{Funct}_{\Dmod(\Gr_G)\mmod_{x_0}^\onfact}(\bC'_1{}^{\on{fact}_{x_0},\Dmod(\Gr_G)},\bC_2^{\on{fact}_{x_0},\Dmod(\Gr_G)})). 
\end{multline*} 


Hence, if \eqref{e:ff restated} is an equivalence for $\bC'_1$, then it is an equivalence for $\bC_1$.

\sssec{}

Thus we obtain a further reduction of the assertion that \eqref{e:ff restated} is an equivalence to the case when $\bC_1=\Dmod(\fL(G)_{x_0})$.

\sssec{}

Note that both sides in \eqref{e:functor in one direction} have a natural symmetric monoidal structure: 

\medskip

In the left-hand side, if $\bC_1,\bC_2$ are DG categories equipped with an action of $\fL(G)_{x_0}$, then the tensor product
$\bC_1\otimes \bC_2$ acquires a $\fL(G)_{x_0}$-action via the diagonal action map $\fL(G)_{x_0}\to \fL(G)_{x_0}\times \fL(G)_{x_0}$.
Furthermore, an object in $\fL(G)_{x_0}\mmod$ is dualizable if and only if the underlying DG category is dualizable. 

\medskip

In the right-hand side, if $\wt\bC_1$ and $\wt\bC_2$ are factorization module categories at $x_0$ with respect to $\Dmod(\Gr_G)$,
then $\wt\bC_1\otimes \wt\bC_2$ (i.e., the tensor product of the corresponding crystals of categories over $\Ran_{x_0}$) is
naturally a  factorization module category at $x_0$ with respect to $\Dmod(\Gr_G)\otimes \Dmod(\Gr_G)$. We produce the
sought-for factorization module category at $x_0$ with respect to $\Dmod(\Gr_G)$ by applying the restriction functor
(see \secref{sss:fact Rres}) along the (unital) factorization functor
$$(\Delta_{\Gr_G})_*:\Dmod(\Gr_G)\to \Dmod(\Gr_G\times \Gr_G)\simeq \Dmod(\Gr_G)\otimes \Dmod(\Gr_G).$$

\medskip

The following assertion follows from \lemref{l:res functor grps}: 

\begin{lem}
The functor \eqref{e:functor in one direction} has a canonical symmetric monoidal structure.
\end{lem}

Thus, we obtain that if $\bC_1\in \fL(G)\mmod$ is dualizable, then the functor \eqref{e:ff restated} is an equivalence
if and only if it is an equivalence for $\bC_1$ replaced by $\Vect$ and $\bC_2$ replaced by $\bC_1^\vee\otimes \bC_2$.

\sssec{} \label{sss:concl red step 1}

Since $\Dmod(\fL(G))$ is dualizable as a DG category (and, hence, as an object of $\fL(G)_{x_0}\mmod$), we obtain that in order 
to prove that \eqref{e:ff restated} is an equivalence, it is enough to do so in the case when $\bC^1=\Vect$. 

\medskip

In other words, it suffices to show that for $\bC\in \fL(G)_{x_0}\mmod$, the functor
\begin{equation} \label{e:ff from Vect}
\bC^{\fL(G)_{x_0}}\simeq \on{Funct}_{\fL(G)_{x_0}\mmod}(\Vect,\bC)\to 
\on{Funct}_{\Dmod(\Gr_G)\mmod_{x_0}^\onfact}(\Vect^{\on{fact}_{x_0},\Dmod(\Gr_G)},\bC^{\on{fact}_{x_0},\Dmod(\Gr_G)})
\end{equation}
is an equivalence. 

%
%
%

\section{A reformulation: factorization modules for the dualizing sheaf} \label{s:dualizing}

In the previous section we reduced the assertion of \thmref{t:initial} to its particular case, when the source $\fL(G)_{x_0}$-module
is $\Vect$.

\medskip 

From now on we will focus on this particular case and we will formulate a statement, \thmref{t:main}, which talks about calculating
the category of factorization modules for a particular factorization algebra. We will show that \thmref{t:main} is equivalent to 
the above particular case of \thmref{t:initial}.  

\ssec{Factorization modules for the dualizing sheaf}

In this subsection we state \thmref{t:main}. In a sense, this theorem on its own is no less interesting than \thmref{t:initial}: it gives a
recipe of how to calculate $\fL(G)_{x_0}$-invariants using factorization algebras. 

\sssec{}

Denote by $\pi_{\Ran}$ the projection 
$$\Gr_{G,\Ran}\to \Ran.$$

It has a natural structure of map between factorization spaces, which we denote by 
$$\pi:\Gr_G\to \on{pt}.$$

\medskip

We consider 
\begin{equation} \label{e:projection on Ran}
\pi^!: \Vect\simeq \Dmod(\on{pt})\to \Dmod(\Gr_G)
\end{equation}
as a \emph{lax-unital} factorization functor between unital factorization categories (see \secref{sss:lax unital factfun}).

\medskip

The image of the unit factorization algebra $k\in \Vect$ under $\pi^!$ is a \emph{unital} factorization 
algebra, denoted $\omega_{\Gr_G}$ in $\Dmod(\Gr_G)$. The corresponding object
$$(\omega_{\Gr_G})_\Ran\in (\Dmod(\Gr_G))_\Ran:=\Dmod(\Gr_{G,\Ran})$$
is $\omega_{\Gr_{G,\Ran}}$, equipped with its natural factorization structure. 

\sssec{}

Denote by $\pi^{\on{level}^\infty_{x_0}}_{\Ran_{x_0}}$ the projection $\Gr^{\on{level}^\infty_{x_0}}_{G,\Ran_{x_0}}\to \Ran_{x_0}$. 
We can regard it as a map between factorization module spaces for $\Gr_G$ and $\on{pt}$, respectively, compatible with $\pi$. 
Denote by
$$(\pi^{\on{level}^\infty_{x_0}})^!:\Vect\simeq \Dmod(\on{pt})\to \Dmod(\fL(G)_{x_0})^{\on{fact}_{x_0},\Dmod(\Gr_G)}$$
the resulting (lax-unital) functor between the factorization module categories with respect to $\Vect\simeq \Dmod(\on{pt})$ and $\Dmod(\Gr_G)$, respectively,
compatible with $\pi^!$ (see \secref{sss:lax unital factfun between modules}). 

\medskip

According to \ref{sss:lax unital sends factmod to factmod}, the functor $(\pi^{\on{level}^\infty_{x_0}})^!$ induces a functor
\begin{equation} \label{e:main functor LG}
\Vect\simeq k\mod^\onfact(\Vect)_{x_0}\to \omega_{\Gr_G}\mod^\onfact(\Dmod(\fL(G)_{x_0})^{\on{fact}_{x_0},\Dmod(\Gr_G)})_{x_0},
\end{equation} 
where
\begin{equation} \label{e:omega mod cat}
\omega_{\Gr_G}\mod^\onfact(\Dmod(\fL(G)_{x_0})^{\on{fact}_{x_0},\Dmod(\Gr_G)})_{x_0},
\end{equation} 
is the category of (unital) factorization modules at $x_0$ with respect to 
$$\omega_{\Gr_G}\in \on{FactAlg}(\Dmod(\Gr_G))$$
in the (unital) factorization module category $\Dmod(\fL(G)_{x_0})^{\on{fact}_{x_0},\Dmod(\Gr_G)}$ at $x_0$ with respect to
$\Dmod(\Gr_G)$, see \secref{sss:notation for intermal factmod} and \secref{sss:notation for intermal unital factmod}.

\sssec{}

The functor \eqref{e:main functor LG} sends the generator $\sfe\in \Vect$ to an object 
\begin{equation} \label{e:omega x0}
    (\omega_{\fL(G)_{x_0}})^{\on{fact}_{x_0},\omega_{\Gr_G}} \in \omega_{\Gr_G}\mod^\onfact(\Dmod(\fL(G)_{x_0})^{\on{fact}_{x_0},\Dmod(\Gr_G)})_{x_0}
\end{equation}
in \eqref{e:omega mod cat}. Explicitly, this object is given by the dualizing sheaf
\[
\omega_{\Gr^{\on{level}^\infty_{x_0}}_{G,\Ran_{x_0}}}\in \Dmod(\Gr^{\on{level}^\infty_{x_0}}_{G,\Ran_{x_0}}),
\]
equipped with its natural factorization structure with respect to $\omega_{\Gr_{G}}$.

\sssec{}

Note that the object \eqref{e:omega x0} is naturally $\fL(G)_{x_0}$-equivariant,
with respect to the action of $\fL(G)_{x_0}$ on \eqref{e:omega mod cat} induced from its action on $\Dmod(\fL(G)_{x_0})^{\on{fact}_{x_0},\Dmod(\Gr_G)}$
by \emph{right translations} (see \secref{sss:construct space}).

\medskip

Since 
$$\left(\Dmod(\fL(G)_{x_0})^{\on{fact}_{x_0},\Dmod(\Gr_G)}\right)^{\fL(G)_{x_0}}\simeq
\Vect^{\on{fact}_{x_0},\Dmod(\Gr_G)},$$
the object \eqref{e:omega x0} corresponds to an object 
\begin{equation} \label{e:omega x0 with inv}
k^{\on{fact}_{x_0},\omega_{\Gr_G}} \in \omega_{\Gr_G}\mod^\onfact(\Vect^{\on{fact}_{x_0},\Dmod(\Gr_G)})_{x_0}.
\end{equation}

\sssec{}

By functoriality, the object \eqref{e:omega x0 with inv} gives rise to a functor
\begin{equation} \label{e:main functor C}
\bC^{\fL(G)_{x_0}}\simeq \on{Funct}_{\fL(G)_{x_0}\mmod}(\Vect,\bC) \to  \omega_{\Gr_G}\mod^\onfact(\bC^{\on{fact}_{x_0},\Dmod(\Gr_G)})_{x_0}
\end{equation}
for $\bC\in \fL(G)_{x_0}\mmod$.

\sssec{}

Over the course Sects. \ref{s:almost trivial}-\ref{s:proof of main-end} we will prove:

\begin{thm} \label{t:main}
The functor \eqref{e:main functor C} is an equivalence.
\end{thm} 

A particular case of \thmref{t:main} is:

\begin{cor} \label{c:main}
The functor \eqref{e:main functor LG} is an equivalence.
\end{cor}

\sssec{}

In \secref{ss:deduce initial} below we will show that the assertion of \thmref{t:main} is logically equivalent to the statement that 
the functor \eqref{e:ff from Vect} is an equivalence. As we have concluded in \secref{sss:concl red step 1}, the 
latter is equivalent to the statement of \thmref{t:initial}.

\ssec{The implication \thmref{t:main} \texorpdfstring{$\Rightarrow$}{=>} \thmref{t:initial}} \label{ss:deduce initial}

In this subsection we will show that \thmref{t:main} implies (the particular case of $\bC_1=\Vect$ of) 
\thmref{t:initial}. 
 
\medskip
 
The basic tool here is the adjunction of Proposition \ref{prop:basic adjunction}. 

\sssec{} \label{sss:pi}

Note that since the map $\pi_\Ran:\Gr_{G,\Ran}\to \Ran$ is proper, the functor $\pi_\Ran^!$ admits a 
left adjoint, given by $(\pi_\Ran)_!$, compatible with the factorization structure.

\medskip

We will denote by $\pi_!$ the resulting (unital) factorization functor $\Dmod(\Gr_G)\to \Dmod(\on{pt})=\Vect$. 
The functors $(\pi_!,\pi^!)$ form an adjoint pair as factorization functors, where $\pi^!$ is \emph{lax-unital}. 

\medskip

Denote by $\Rres_{\pi_!}$ 
the resulting restriction operation on factorization categories, see \secref{sss:fact Rres}. 

\sssec{}

Recall that for a factorization category $\bA$, we denote by $\bA^{\on{fact}_{x_0}}$ the tautological
(i.e., vacuum) factorization module at $x_0$ with respect to $\bA$, see \secref{sss:vac fact mod cat}.

\medskip

In particular, for $\bA=\Vect$, we can consider the (unital) factorization module category $\Vect^{\on{fact}_{x_0}}$ at $x_0$,
which under the embedding
$$\DGCat\hookrightarrow \Vect\mmod^\onfact_{x_0}$$
corresponds to $\Vect\in \DGCat$. 

\medskip

We claim: 

\begin{prop} \label{p:Vect over Gr}
We have a canonical identification
$$\Vect^{\on{fact}_{x_0},\Dmod(\Gr_G)} \simeq \Rres_{\pi_!}(\Vect^{\on{fact}_{x_0}}).$$
\end{prop}

The proof will be given in \secref{ss:proof of Vect over Gr}. 

\sssec{}

Recall now that for the \emph{lax-unital} factorization functor $\pi^!$ there is a \emph{unital} restriction operation
$$\Rres^\untl_{\pi^!}:\Dmod(\Gr_G)\mmod^\onfact_{x_0}\to \Vect\mmod^\onfact_{x_0},$$
see \secref{sss:unital restriction}.

\medskip

Moreover, we have the following basic facts:

\medskip

\noindent For
$\wt\bC\in \Dmod(\Gr_G)\mmod^\onfact_{x_0}$, there is a canonical identification: 
\begin{equation} \label{e:basic adj}
\on{Funct}_{\Dmod(\Gr_G)\mmod^\onfact_{x_0}}(\Rres_{\pi_!}(\Vect^{\on{fact}_{x_0}}),\wt\bC)
\simeq \on{Funct}_{\Vect\mmod^\onfact_{x_0}}(\Vect^{\on{fact}_{x_0}},
\Rres^\untl_{\pi^!}(\wt\bC)),
\end{equation} 
see Proposition \ref{prop:basic adjunction}. 

\medskip

\noindent For $\wt\bC\in \Dmod(\Gr_G)\mmod^\onfact_{x_0}$, we have
\begin{equation} \label{e:fiber untl rest}
\Rres^\untl_{\pi^!}(\wt\bC)_{x_0}\simeq \pi^!(\on{unit}_{\Vect})\mod^\onfact(\wt\bC)_{x_0},
\end{equation}
see Proposition \ref{prop:fiber of unital restriction}. 

\sssec{}

Combining \propref{p:Vect over Gr} with \eqref{e:basic adj} and \eqref{e:fiber untl rest}, we obtain an equivalence
\begin{multline} \label{e:comparison}
\on{Funct}_{\Dmod(\Gr_G)\mmod_{x_0}^\onfact}(\Vect^{\on{fact}_{x_0},\Dmod(\Gr_G)},\wt\bC) \simeq
\on{Funct}_{\Dmod(\Gr_G)\mmod_{x_0}^\onfact}(\Rres_{\pi_!}(\Vect^{\on{fact}_{x_0}}),\wt\bC) \simeq \\
\simeq \on{Funct}_{\Vect\mmod^\onfact_{x_0}}(\Vect^{\on{fact}_{x_0}},\Rres^\untl_{\pi^!}(\wt\bC)) \simeq
\Rres^\untl_{\pi^!}(\wt\bC)_{x_0} \simeq \pi^!(\on{unit}_{\Vect})\mod^\onfact(\wt\bC)_{x_0}= \\
=\omega_{\Gr_G}\mod^\onfact(\wt\bC)_{x_0}.
\end{multline}

\sssec{}

Let now $\bC$ be an object of $\fL(G)_{x_0}\mmod$. Unwinding the constructions, we obtain:

\begin{lem} \label{l:deduce initial} 
The functor
\begin{multline*}
\bC^{\fL(G)_{x_0}} \overset{\text{\eqref{e:ff from Vect}}}\longrightarrow 
\on{Funct}_{\Dmod(\Gr_G)\mmod_{x_0}^\onfact}(\Vect^{\on{fact}_{x_0},\Dmod(\Gr_G)},\bC^{\on{fact}_{x_0},\Dmod(\Gr_G)}) \simeq \\
\overset{\text{\eqref{e:comparison}}}\simeq \omega_{\Gr_G}\mod^\onfact(\bC^{\on{fact}_{x_0},\Dmod(\Gr_G)})_{x_0}
\end{multline*}
identifies canonically with the functor \eqref{e:main functor C}. 
\end{lem}

\sssec{}

From \lemref{e:main functor C}, we obtain that the assertion of \thmref{t:main} is equivalent to the statement that 
the functor \eqref{e:ff from Vect} is an equivalence, which in turn is equivalent to the statement of \thmref{t:initial}.

\ssec{Proof of \propref{p:Vect over Gr}} \label{ss:proof of Vect over Gr}

\sssec{}

By construction, we can identify $\Vect^{\on{fact}_{x_0},\Dmod(\Gr_G)}$ with 
\begin{equation} \label{e:Vect fact Gr as inv}
(\Dmod(\fL(G)_{x_0})^{\on{fact}_{x_0},\Dmod(\Gr_G)})^{\fL(G)_{x_0}},
\end{equation}
where we take invariants with respect to the (right) action of $\fL(G)_{x_0}$ on 
$\Dmod(\fL(G)_{x_0})^{\on{fact}_{x_0},\Dmod(\Gr_G)}$, viewed as a factorization module category 
at $x_0$ with respect to $\Dmod(\Gr_G)$.

\medskip

We will now rewrite \eqref{e:Vect fact Gr as inv} slightly differently.

\sssec{}

Consider the vacuum factorization module space over $\Gr_G$ at $x_0$; denote it
$(\Gr_G)^{\on{fact}_{x_0}}$, see \secref{sss:vac mod space}. The resulting factorization module
category identifies with $\Dmod(\Gr_G)^{\on{fact}_{x_0}}$. 

\medskip

The $\fL(G)_{x_0}$-action on $\Gr^{\on{level}^\infty_{x_0}}_{G,\Ran_{x_0}}$ gives rise to an action of the
Hecke groupoid $\on{Hecke}_{x_0}$ at $x_0$ on $(\Gr_G)^{\on{fact}_{x_0}}$. 

\medskip

Tautologically, we can rewrite \eqref{e:Vect fact Gr as inv} as
\begin{equation} \label{e:Vect fact Gr as Hecke inv}
(\Dmod(\Gr_G)^{\on{fact}_{x_0}})^{\on{Hecke}_{x_0}}.
\end{equation}

Since $\on{Hecke}_{x_0}$ is proper, we can rewrite \eqref{e:Vect fact Gr as Hecke inv}
also as 
\begin{equation} \label{e:Vect fact Gr as Hecke coinv}
(\Dmod(\Gr_G)^{\on{fact}_{x_0}})_{\on{Hecke}_{x_0}}.
\end{equation}

\sssec{}

The functor of !-pushforward along 
$$\pi_{\Ran_{x_0}}:(\Gr_G)^{\on{fact}_{x_0}}\to \Ran_{x_0}$$
gives rise to a functor
$$(\pi_{x_0})_!:\Dmod(\Gr_G)^{\on{fact}_{x_0}}\to \Vect^{\on{fact}_{x_0}}$$
as factorization module categories with respect to $\Dmod(\Gr)$ and $\Vect$,
compatible with the factorization functor $\pi_!:\Dmod(\Gr)\to \Vect$. 

\medskip

Moreover, the functor $(\pi_{x_0})_!$ canonically factors via a functor
\begin{equation} \label{e:Vect over Gr 0}
(\Dmod(\Gr_G)^{\on{fact}_{x_0}})_{\on{Hecke}_{x_0}}\to \Vect^{\on{fact}_{x_0}}.
\end{equation}
i.e., a functor
\begin{equation} \label{e:Vect over Gr 1}
\Vect^{\on{fact}_{x_0},\Dmod(\Gr_G)}\to \Vect^{\on{fact}_{x_0}}, 
\end{equation}
as factorization module categories with respect to $\Dmod(\Gr)$ and $\Vect$,
compatible with the factorization functor $\pi_!:\Dmod(\Gr)\to \Vect$.

\medskip

By the definition of the restriction operation $\Rres_{\pi_!}$, the functor \eqref{e:Vect over Gr 1}
gives rise to a (unital) functor
\begin{equation} \label{e:Vect over Gr 2}
\Vect^{\on{fact}_{x_0},\Dmod(\Gr_G)}\to \Rres_{\pi_!}(\Vect^{\on{fact}_{x_0}})
\end{equation}
as factorization module categories over $\Dmod(\Gr_G)$.

\medskip

The functor \eqref{e:Vect over Gr 2} is the sought-for functor in \propref{p:Vect over Gr}. 

\sssec{}

We will now show that \eqref{e:Vect over Gr 2} is an equivalence. In order to do so, we will apply
Proposition \ref{sss:adj test for factres}.

\medskip

Condition (i) in this lemma is satisfied because the morphism $\pi_\Ran$ is proper. Condition (iii) 
is satisfied, since at the level of fibers at $x_0$, the functor \eqref{e:Vect over Gr 2} induces an identity
endofunctor of $\Vect$. 

\medskip

Hence, it remains to show that the functor \eqref{e:Vect over Gr 2} admits a right adjoint, viewed
as a functor between sheaves of categories over $\Ran_{x_0}$. 

\sssec{}

Pullback along $\pi_{\Ran_{x_0}}$ is a functor
$$\Vect^{\on{fact}_{x_0}}\to \Dmod((\Gr_G)^{\on{fact}_{x_0}}),$$
which naturally factors via a functor
\begin{equation} \label{e:Vect over Gr 3}
\Vect^{\on{fact}_{x_0}}\to (\Dmod((\Gr_G)^{\on{fact}_{x_0}}))^{\on{Hecke}_{x_0}}.
\end{equation}

Interpreting $\Vect^{\on{fact}_{x_0},\Dmod(\Gr_G)}$ as \eqref{e:Vect fact Gr as Hecke inv}, the functor
\eqref{e:Vect over Gr 3} provides a right adjoint to \eqref{e:Vect over Gr 2}.

\qed[\propref{p:Vect over Gr}]

\begin{rem}

Unwinding the construction, one can describe the functor right adjoint (which is also the inverse) of \eqref{e:Vect over Gr 2}
as follows:

\medskip

By \eqref{e:basic adj}, a datum of a functor
\begin{equation} \label{e:Vect over Gr 4}
\Rres_{\pi_!}(\Vect^{\on{fact}_{x_0}}) \to \Vect^{\on{fact}_{x_0},\Dmod(\Gr_G)}
\end{equation}
is equivalent to that of a functor
$$\Vect^{\on{fact}_{x_0}} \to \Rres^\untl_{\pi^!}(\Vect^{\on{fact}_{x_0},\Dmod(\Gr_G)}),$$
while the latter is equivalent to that of a functor
\begin{equation} \label{e:Vect over Gr 5}
\Vect^{\on{fact}_{x_0}} \to \Vect^{\on{fact}_{x_0},\Dmod(\Gr_G)}
\end{equation}
of factorization module categories over $\Vect$ and $\Dmod(\Gr_G)$, respectively, compatible with 
$\pi^!$. 

\medskip

The functor \eqref{e:Vect over Gr 5} is given by the natural factorization of 
$$\pi_{x_0}^!:\Vect^{\on{fact}_{x_0}} \to \Dmod((\Gr_G)^{\on{fact}_{x_0}})$$
as
$$\Vect^{\on{fact}_{x_0}} \to  (\Dmod(\Gr_G)^{\on{fact}_{x_0}})^{\on{Hecke}_{x_0}}\simeq \Vect^{\on{fact}_{x_0},\Dmod(\Gr_G)}.$$

\end{rem} 

\ssec{Example: the case of \texorpdfstring{$\Vect$}{Vect}}  

In this section we run a plausibility check for \thmref{t:main} when $\bC=\Vect$ and the group $G$ is semi-simple and simply-connected. 
We calculate explicitly both sides and show that they are abstractly isomorphic. 

\sssec{}

Let us apply \thmref{t:main} to $\bC:=\Vect$, viewed as an object of $\fL(G)_{x_0}\mmod$, equipped with the trivial 
action. We obtain:

\begin{cor} \label{c:main Vect}
The functor
\begin{equation} \label{e:main functor Vect}
\Vect^{\fL(G)_{x_0}}\to \omega_{\Gr_G}\mod^\onfact(\Vect^{\on{fact}_{x_0},\Dmod(\Gr_G)})_{x_0}
\end{equation}
of \eqref{e:main functor C}
is an equivalence.
\end{cor} 

\sssec{}

In \secref{sss:vect implies} we will show that the assertion of \corref{c:main Vect} is equivalent to a key
calculation involved in the proof of \thmref{t:main}. 

\medskip

In the rest of this subsection, we will explain that the existence of \emph{an} equivalence
\begin{equation} \label{e:case of Vect abs}
\Vect^{\fL(G)_{x_0}}\simeq \omega_{\Gr_G}\mod^\onfact(\Vect^{\on{fact}_{x_0},\Dmod(\Gr_G)})_{x_0}
\end{equation}
is a priori known, at least when $G$ is semi-simple and simply-connected. 

\sssec{}

Recall that according to \propref{p:Vect over Gr}, we have 
$$\Vect^{\on{fact}_{x_0},\Dmod(\Gr_G)} \simeq \Rres_{\pi_!}(\Vect^{\on{fact}_{x_0}}).$$

\medskip

Combining with \eqref{e:equiv fact mod}, we obtain 
$$\omega_{\Gr_G}\mod^\onfact(\Vect^{\on{fact}_{x_0},\Dmod(\Gr_G)})_{x_0} \simeq \pi_!(\omega_{\Gr_G})\mod^\onfact(\Vect^{\on{fact}_{x_0}})_{x_0}
=:\pi_!(\omega_{\Gr_G})\mod^\onfact_{x_0},$$
where in the right-hand side, $\pi_!(\omega_{\Gr_G})$ is viewed as a plain unital factorization algebra, and $\mod^\onfact_{x_0}$ refers to the plain
category of unital factorization modules at $x_0$.

\sssec{}

Assume now that $G$ is semi-simple and simply-connected. Let $\fa_{BG}$ be the Lie algebra\footnote{The Lie algebra $\fa_{BG}$
is actually abelian. Indeed, $\on{C}^\cdot(\fa_{BG})\simeq \on{C}^\cdot(BG)$, which is isomorphic to a polynomial algebra.} that controls the 
rational homotopy type of the classifying stack $BG$ of $G$, which is characterized by a canonical isomorphism between cocommutative coalgebras
\[
    \on{C}_\cdot(\fa_{BG})\simeq \on{C}_\cdot(BG),
\]
where the first $\on{C}_\cdot(-)$ denoted the homological Chevalley complex of a Lie algebra. 

\medskip
Let $\fa_{BG,X}:=\fa_{BG}\otimes k_X$ be the corresponding constant Lie$^\star$-algebra and $U^{\on{ch}}( \fa_{BG,X} )$ be its chiral universal enveloping algebra. We have the following  result:

\begin{lem}
    Let
    \[
        \pi_!(\omega_{\Gr_G})^{\on{ch}}:= \pi_!(\omega_{\Gr_G})|_X[-1]
    \]
    be the unital chiral algebra corresponding to the unital factorization algebra $\pi_!(\omega_{\Gr_G})$ (see \cite[Sect. D.1]{GLC2}). We have a canonical isomorphism
    \[
        \pi_!(\omega_{\Gr_G})^{\on{ch}}\simeq U^{\on{ch}}( \fa_{BG,X} ).
    \]
\end{lem}

\proof[Sketch]
    By \cite[Theorem 15.3.3]{Ga1}, the \emph{augmented cocommutative factorization algebra}\footnote{A unital cocommutative factorization algebra is a unital factorization algebra $\CA$ whose structural isomorphism
$$\on{union}^!(\CA)|_{(\Ran\times \Ran)_\disj} \xrightarrow{\simeq} (\CA\boxt \CA)|_{(\Ran\times \Ran)_\disj} $$
is extended to a not necessarily invertible map $\on{union}^!(\CA)\to \CA\boxt \CA$ (and is equipped with a homotopy-coherent data of commutativity and associativity). Being augmented means it is equipped with a homomorphism $\CA \to \on{unit}_{\on{Vect}}$ compatible with the cocommutative factorization structures. Informally speaking, this means the cocommutative coalgebra structure on the cochain complex $\on{C}_\cdot(\Gr_G)$ is naturally compatible with its factorization structure induced from the factorization space $\Gr_G$.}  $\CA:=\pi_!(\omega_{\Gr_G})$ and the \emph{augmented commutative factorization algebra}\footnote{A unital commutative factorization algebra is a unital factorization algebra $\CB$ whose structural isomorphism is extended to a not necessarily invertible map $\CB\boxt \CB \to \on{union}^!(\CB)$ (and is equipped with a homotopy-coherent data). By \cite[Sect. 3.4.20-3.4.22]{BD1}, knowing a unital commutative factorization algebra is equivalent to knowing a unital commutative algebra in the symmetric monoidal category $\Dmod(X)$. Via this correspondence, $\CB$ is given by $\on{C}^\cdot(BG)\otimes \omega_X$.} $\CB:= \on{C}^\cdot(BG)$ are ``Verdier dual'' to each other after removing the augmented units. Here the notion of Verdier duality is developed in \cite{Ga1} and is subtler than the usual one because the D-modules $\CA_\Ran$ and $\CB_\Ran$ are not compact. Nevertheless, \cite[Theorem 1.5.9]{Ho} shows that after passing to the \emph{chiral Koszul duals} (see \cite{FG} for what this means), this Verdier duality becomes the usual one. In other words, 
\[
    \CA^{\vee,\on{KD}} \simeq \mathbb{D}^{\on{Ver}}_X( \CB^{\vee,\on{KD}}),
\]
where $\CA^{\vee,\on{KD}} $ is the Lie$^\star$-algebra that is Koszul dual to $\CA$, while $\CB^{\vee,\on{KD}}$ is the Lie$^\star$-coalgebra that is Koszul dual to $\CB$. Since $\CB|_X$ is constant with $!$-fibers equal to $\on{C}^\cdot(BG)$, we see $\CB^{\vee,\on{KD}}$ is constant with $!$-fibers equal to the Koszul dual of $\on{C}^\cdot(BG)$, which is just the Lie coalgebra $(\fa_{BG})^*$. It follows that $\CA^{\vee,\on{KD}}$ is constant with $*$-fibers equal to the Lie algebra $\fa_{BG}$. In other words,
\[
    \CA^{\vee,\on{KD}} \simeq \fa_{BG,X}.
\]
By \cite[Proposition 6.1.2]{FG}, this implies
\[
    \CA \simeq U^{\on{ch}}( \fa_{BG,X} ).
\]

\qed

\sssec{}

As a corollary, we have
\[
    \pi_!(\omega_{\Gr_G})\mod^\onfact_{x_0} \simeq \pi_!(\omega_{\Gr_G})^{\on{ch}}\mod^{\on{ch}}_{x_0} \simeq \fa_{BG,X}\mod^{\on{Lie}^\star}_{x_0}.
\]
Since $\fa_{BG,X}$ is a constant Lie$^\star$-algebra, the above category is equivalent to the category of modules over the Lie algebra
\begin{equation} \label{e:Lie alg LG}
\fa_{BG}\otimes \on{C}^\cdot(\cD^\times_{x_0}),
\end{equation} 
or which is the same, to the category of modules over
\begin{equation} \label{e:ass alg LG}
U(\fa_{BG}\otimes \on{C}^\cdot(\cD^\times_{x_0})).
\end{equation}

\sssec{}

The category $\Vect^{\fL(G)_{x_0}}$ is equivalent to the category of modules over 
\begin{equation} \label{e:homology LG}
\on{C}_\cdot(\fL(G)_{x_0}),
\end{equation}
viewed as an associative algebra via the product operation on $\fL(G)_{x_0}$. 

\medskip

Let $\fa_G$ be the group-object in the category of Lie algebras that controls the homotopy of $G$.
The assumption that $G$ is semi-simple and simply-connected implies that $\fL(G)_{x_0}$ is connected
and simply-connected. According to \cite[Theorem 1.4.4]{GL}, the group-object in the category of Lie algebras 
that controls the rational homotopy of $\fL(G)_{x_0}$ is canonically isomorphic to
$$\fa_G\otimes \on{C}^\cdot(\cD^\times_{x_0}).$$

Hence,
\begin{equation} \label{e:cohomology of loop group}
\on{C}_\cdot(\fL(G)_{x_0}) \simeq \on{C}_\cdot(\fa_G\otimes \on{C}^\cdot(\cD^\times_{x_0})).
\end{equation}

\medskip

The structure of an associative algebra on $\on{C}_\cdot(\fa_G\otimes \on{C}^\cdot(\cD^\times_{x_0}))$ is induced
by the group structure on $\fa_G\otimes \on{C}^\cdot(\cD^\times_{x_0})$ as a Lie algebra, where the latter results from the
group structure on $G$. 

\sssec{} \label{sss:rat hom 3}

Note that we have
$$\fa_G=\Omega(\fa_{BG}),$$
as group-objects in the category of Lie algebras, 
where $\Omega(-)$ is the loop functor on the category of Lie algebras.

\medskip

Hence, we also have 
$$\fa_G\otimes \on{C}^\cdot(\cD^\times_{x_0}) \simeq \Omega(\fa_{BG}\otimes \on{C}^\cdot(\cD^\times_{x_0})).$$

Finally, according to \cite[Chapter 5, Theorem 6.1.2]{GR2}, we have
$$\on{C}_\cdot(\Omega(-))\simeq U(-),$$
and hence the associative algebras \eqref{e:ass alg LG} and \eqref{e:homology LG} are canonically 
isomorphic. 

\sssec{}

Summarizing, we obtain
\begin{multline*}
\Vect^{\fL(G)_{x_0}}\simeq \on{C}_\cdot(\fL(G)_{x_0})\mod\simeq 
\on{C}_\cdot(\Omega(\fa_{BG}\otimes \on{C}^\cdot(\cD^\times_{x_0})))\mod \simeq \\
\simeq U(\fa_{BG}\otimes \on{C}^\cdot(\cD^\times_{x_0}))\mod \simeq \pi_!(\omega_{\Gr_G})\mod^\onfact_{x_0}\simeq
\omega_{\Gr_G}\mod^\onfact(\Vect^{\on{fact}_{x_0},\Dmod(\Gr_G)})_{x_0}.
\end{multline*}

\section{Proof of \thmref{t:main}: reduction to the case of a trivial action} \label{s:almost trivial}

In this subsection we will reduce the assertion of \thmref{t:main} to the case when $\bC=\Vect$. The main
tool will be the notion of \emph{almost trivial} action of a group on a category. 

\ssec{Almost constant sheaves}

In this subsection we discuss the notion of \emph{almost constant} sheaf on a scheme of finite type. 
This notion will be relevant for defining the notion of \emph{almost trivial} action in the subsequent subsections. 

\sssec{}

Let $Y$ be a scheme of finite type. Let
\begin{equation} \label{e:alm const Y}
\Dmod(Y)^{\on{alm-const}}\subset \Dmod(Y)
\end{equation} 
be the full subcategory generated by the constant sheaf $\ul{k}_Y$. 

\medskip

We have a canonical identification
\begin{equation} \label{e:almost const as cochains}
\Dmod(Y)^{\on{alm-const}}\simeq \on{C}^\cdot(Y)\mod,
\end{equation} 
given by the action of $\on{C}^\cdot(Y)$ on $\ul{k}_Y$. 

\begin{rem}
We warn the reader that the assignment $Y\mapsto \Dmod(Y)^{\on{alm-const}}$ (unlike its close relative
$\Dmod(Y)^{\on{q-const}}$, see \secref{sss:q-const}) is \emph{not} a sheaf even 
for the Zariski topology. See, however, \secref{l:descent alm} for a descent-type statement. 
\end{rem} 

\sssec{}

The embedding \eqref{e:alm const Y} admits a right adjoint. In terms of the identification \eqref{e:almost const as cochains},
this right adjoint is given by
$$\CF\in \Dmod(Y) \mapsto \on{C}^\cdot(Y,\CF)\in \on{C}^\cdot(Y)\mod.$$

\medskip

Thus, we can view $\Dmod(Y)^{\on{alm-const}}$ as a quotient of $\Dmod(Y)$ by the full subcategory 
$$\{\CF\in \Dmod(Y)\,|\, \on{C}^\cdot(Y,\CF)=0\}.$$

\sssec{} \label{sss:q-const}

The category $\Dmod(Y)^{\on{alm-const}}$ carries a natural t-structure, 
in which $\ul{k}_Y$ is in the heart.

\medskip

Let 
$$\Dmod(Y)^{\on{q-const}}$$
be the left-completion of $\Dmod(Y)^{\on{alm-const}}$ with respect to this t-structure. 
It is easy to see that the embedding \eqref{e:alm const Y} extends to a fully faithful embedding
\begin{equation} \label{e:q const Y}
\Dmod(Y)^{\on{q-const}}\subset \Dmod(Y).
\end{equation} 

The essential image of \eqref{e:q const Y} is the full subcategory of $\Dmod^{\on{hol}}(Y)$ consisting
of objects, whose cohomologies with respect to either perverse or the \emph{usual} t-structure\footnote{One can mimic the definition
of the usual t-structure on constructible sheaves and define its counterpart on $\Dmod^{\on{hol}}(Y)$.}
admit a filtration with
subquotients isomorphic to $\ul{k}_Y$. 

\begin{rem} 

The (fully faithful) embedding
\begin{equation} \label{e:ind to alm const Y}
\Dmod(Y)^{\on{alm-const}}\hookrightarrow \Dmod(Y)^{\on{q-const}}
\end{equation} 
is not always an equivalence. E.g., it fails to be an equivalence for $Y=\BP^1$ (and, which is more, relevant for us,
for $Y$ being a semi-simple group).

\end{rem} 

\sssec{} \label{sss:rat hom type}

Assume that $Y$ is connected and simply-connected. Choose a base point $y\in Y$, and let
$L_Y$ be the Lie algebra that controls the homotopy type of $Y$ (since $Y$ is simply-connected,
$L_Y$ sits in cohomological degrees $\leq -1$). In particular, we have
$$\on{C}^\cdot(Y) \simeq \on{C}^\cdot(L_Y) \text{ and } \on{C}_\cdot(Y) \simeq \on{C}_\cdot(L_Y).$$

\medskip

The functor of *-fiber at $y$ defines an equivalence
$$\Dmod(Y)^{\on{q-const}}\simeq L_Y\mod.$$

\medskip

In particular, we obtain that \eqref{e:ind to alm const Y} is an equivalence if and only if 
the universal enveloping algebra of $L_Y$ is eventually coconnetive (equivalently, if $L_Y$
sits only in \emph{odd} degrees). 

\medskip

Note also that if $U(L_Y)$ is \emph{not} eventually coconnetive, the embedding \eqref{e:q const Y}
does not preserve compactness; in particular, its right adjoint is discontinuous. 

\sssec{} \label{sss:almost const omega}

Let 
\begin{equation} \label{e:almost const omega}
\Dmod(Y)^{\omega,\on{alm-const}}\hookrightarrow \Dmod(Y)
\end{equation}
be the full subcategory of $\Dmod(Y)$ obtained from ones in \eqref{e:ind to alm const Y} 
by $\overset{*}\otimes$-tensoring with $\omega_Y$.

\medskip

We still have an equivalence
\begin{equation} \label{e:almost const as cochains omega}
\on{C}^\cdot(Y)\mod \simeq \Dmod(Y)^{\omega,\on{alm-const}},
\end{equation} 
given by sending 
$$\on{C}^\cdot(Y)\in \on{C}^\cdot(Y)\mod \,\mapsto \omega_Y\in \Dmod(Y)^{\omega,\on{alm-const}}.$$

\sssec{}

Note that Verdier duality identifies 
$$\Dmod(Y)^{\omega,\on{alm-const}}\simeq (\Dmod(Y)^{\on{alm-const}})^\vee.$$

Under this identification, the functor \eqref{e:almost const omega} is the dual of the 
right adjoint of \eqref{e:alm const Y}. 

\sssec{} \label{sss:almost constant prestacks omega}

Let now $\CY$ be an ind-scheme of ind-finite type:
$$\CY\simeq \underset{i}{``\on{colim}"}\, Y_i,$$
where $Y_i$'s are schemes of finite type, and the transition maps $Y_{i_1}\to Y_{i_2}$ are closed embeddings.

\medskip

Recall that the category $\Dmod(\CY)$ is defined as 
$$\underset{i}{\on{lim}}\, \Dmod(Y_i),$$
where the limit is taken with respect to the !-pullbacks. 

\medskip

Let 
$$\Dmod(\CY)^{\omega,\on{alm-const}} \subset \Dmod(\CY)$$
be the full subcategory equal to 
$$\underset{i}{\on{lim}}\, \Dmod(Y_i)^{\omega,\on{alm-const}} \subset \underset{i}{\on{lim}}\, \Dmod(Y_i)$$

\sssec{}

Recall that 
$$\on{C}^\cdot(\CY)\simeq \underset{i}{\on{lim}}\, \on{C}^\cdot(Y_i).$$

Hence, $\on{C}^\cdot(\CY)$ acts on $\omega_\CY\in \Dmod(\CY)^{\omega,\on{alm-const}}$, and we obtain a functor
\begin{equation} \label{e:almost const as cochains ind sch omega}
\on{C}^\cdot(\CY)\mod \to \Dmod(\CY)^{\omega,\on{alm-const}}.
\end{equation} 

However, the functor \eqref{e:almost const as cochains ind sch omega} is in general not even fully faithful
(e.g., it fails to be such for $\CY=\BP^\infty$). 

\sssec{}

Note that we can also write
\begin{equation} \label{e:ind sch presentation}
\Dmod(\CY) \simeq  \underset{i}{\on{colim}}\, \Dmod(Y_i),
\end{equation}
with transition functors being given by *-pushforward. 

\medskip

We let $\Dmod(\CY)^{\on{alm-const}}$ be the \emph{quotient} category of $\Dmod(\CY)$,
defined in terms of \eqref{e:ind sch presentation} as 
$$\underset{i}{\on{colim}}\, \Dmod(Y_i)^{\on{alm-const}}.$$

We do not know whether the category $\Dmod(\CY)^{\on{alm-const}}$ is dualizable in general.
However, we have:
\begin{equation} \label{e:dual of alm-const ind sch}
\on{Funct}_{\DGCat}(\Dmod(\CY)^{\on{alm-const}},\Vect)\simeq \Dmod(\CY)^{\omega,\on{alm-const}}.
\end{equation} 

\sssec{} \label{sss:almost constant prestacks}

In terms of the equivalence \eqref{e:almost const as cochains}, we have:
\begin{equation} \label{e:almost const as cochains ind-sch}
\Dmod(\CY)^{\on{alm-const}} \simeq \underset{i}{\on{colim}}\, \on{C}^\cdot(Y_i)\mod,
\end{equation} 
where the transition functors are given by restriction along the maps
$$\on{C}^\cdot(Y_{i_2})\to \on{C}^\cdot(Y_{i_1}).$$

\medskip

We have a naturally defined functor
\begin{equation} \label{e:almost const as cochains ind sch}
\underset{i}{\on{colim}}\, \on{C}^\cdot(Y_i)\mod\to \on{C}^\cdot(\CY)\mod,
\end{equation}
given by restriction.

\medskip

In terms of \eqref{e:dual of alm-const ind sch}, the dual of the functor \eqref{e:almost const as cochains ind sch} 
is the functor \eqref{e:almost const as cochains ind sch omega}.

\ssec{Almost trivial actions: the case of algebraic groups} \label{ss:alm triv pro}

In this subsection we develop the notion of \emph{almost trivial} action for groups of finite type. 

\sssec{}

Let $H$ be an algebraic group of finite type. 

\medskip

Consider the embedding 
\begin{equation} \label{e:almost const H}
\Dmod(H)^{\on{alm-const}}\hookrightarrow\Dmod(H).
\end{equation} 

By \eqref{e:almost const as cochains}, we can identify 
\begin{equation} \label{e:almost const as cochains H}
\Dmod(H)^{\on{alm-const}}\simeq \on{C}^\cdot(H)\mod.
\end{equation}

\sssec{}

The subcategory \eqref{e:almost const H} is preserved by the monoidal operation, and
hence inherits a monoidal structure. The \emph{right} adjoint to \eqref{e:almost const H}
is (strictly) compatible with monoidal structures. 

\medskip

In terms of the identification \eqref{e:almost const as cochains H}, the monoidal structure
on $\Dmod(H)^{\on{alm-const}}$ corresponds to the Hopf algebra structure on $\on{C}^\cdot(H)$,
induced by the group-structure on $H$. 

\medskip

The latter description implies that the monoidal category $\Dmod(H)^{\on{alm-const}}$ 
is \emph{semi-rigid} (see \cite[Appendix C]{AGKRRV} for what this means). 

\sssec{}

The monoidal functor 
$$\on{C}^\cdot(H,-):\Dmod(H)\to \Vect$$
induces a monoidal functor 
\begin{equation} \label{e:dR enh}
\Dmod(H)^{\on{alm-const}}\to \Vect,
\end{equation}
which admits a left adjoint. The functor \eqref{e:dR enh} is conservative. 

\medskip

In terms of the identification \eqref{e:almost const as cochains H}, the functor \eqref{e:dR enh}
corresponds to the forgetful functor
$$\on{C}^\cdot(H)\mod\to \Vect.$$

\sssec{}

Let $\bC$ be a category acted on by $H$. Set
$$\on{alm-inv}_H(\bC):=\Dmod(H)^{\on{alm-const}}\underset{\Dmod(H)}\otimes \bC.$$

\medskip

The embedding \eqref{e:almost const H} and its right adjoint give rise to a pair of adjoint functors
\begin{equation} \label{e:almost inv}
\on{alm-inv}_H(\bC) \rightleftarrows \bC,
\end{equation} 
with the left adjoint being fully faithful.

\sssec{}

We shall say that action of $\bC$ is \emph{almost trivial} if the functors \eqref{e:almost inv}
are mutually inverse equivalences.

\sssec{Example} \label{sss:bad example}

An example to keep in mind of an action that is \emph{not} almost trivial is
$$\bC:=\Dmod(G)^{\on{q-const}}.$$

In fact, for 
$$\bC':=\Dmod(G)^{\on{q-const}}/\Dmod(G)^{\on{alm-const}},$$
we have $(\bC')^G=0$. 

\sssec{}

Let $(H\mmod)_{\on{alm-trivial}}\subset H\mmod$ be the full subcategory that consists of $H$-module
categories equipped with an almost trivial action.

\medskip

The embedding 
\begin{equation} \label{e:embed almost inv}
(H\mmod)_{\on{alm-trivial}}\hookrightarrow H\mmod
\end{equation} 
admits 
a right adjoint, given by 
\begin{equation} \label{e:pass almost inv}
\bC\mapsto \on{alm-inv}_H(\bC).
\end{equation} 

\medskip

The counit of this adjunction is the left adjoint in \eqref{e:almost inv}. Since \eqref{e:almost inv}
admits a right adjoint, we can identify \eqref{e:pass almost inv} also with the \emph{left} adjoint
of \eqref{e:embed almost inv}.

\sssec{}

Consider the category 
$$\on{inv}_H(\bC):=\bC^H\simeq \Vect\underset{\Dmod(H)}\otimes \bC.$$

Consider the corresponding pair of
adjoint functors
$$\oblv_H:\bC^H\rightleftarrows \bC:\on{Av}^H_*.$$

It is clear that the functor $\oblv_H$ has essential image contains in $\on{alm-inv}_H(\bC)$. 
Hrnce, the functor $\on{Av}^H_*$ factors as
$$\bC \to \on{alm-inv}_H(\bC) \to \bC^H,$$
where the first arrow is the right adjoint in \eqref{e:almost inv}, and the second arrow is the right adjoint to
\begin{equation} \label{e:oblv to almost}
\oblv_H:\bC^H\to \on{alm-inv}_H(\bC).
\end{equation} 

\medskip

One can view the adjunction
$$\bC^H\rightleftarrows \on{alm-inv}_H(\bC)$$
as obtained by tensoring $-\underset{\Dmod(H)}\otimes \bC$ 
from the adjunction
\begin{equation} \label{e:Vect alm adj}
\Vect \rightleftarrows \Dmod(H)^{\on{alm-const}}.
\end{equation} 

\medskip

From the above it follows that the essential image of $\oblv_H$ generates $\on{alm-inv}_H(\bC)$ under colimits.
Moreover, we have:

\begin{lem} \label{l:ker Av}
The kernel of the right adjoint in \eqref{e:almost inv} equals $\on{ker}(\on{Av}^H_*)$.
\end{lem} 

\sssec{} \label{sss:Av !}

Note that the left adjoint in \eqref{e:Vect alm adj} itself admits a left adjoint\footnote{This follows, e.g., from the fact that
$\on{C}^\cdot(H)$ is finite-dimensional.}. Since the monoidal category 
$\Dmod(H)^{\on{alm-const}}$ is semi-rigid, this left adjoint is automatically $\Dmod(H)^{\on{alm-const}}$-linear.

\medskip

This implies that the forgetful functor $\bC^H\to \on{alm-inv}_H(\bC)$ also admits a left adjoint,
functorially in $\bC$. We denote this left adjoint by $\on{Av}^H_!$. 

\sssec{}

Consider the functor
$$\on{inv}_H:H\mmod \to \DGCat.$$
It naturally enhances to a functor
$$\on{inv}^{\on{enh}}_H:H\mmod \to \Vect^H\mmod,$$
where we identify
$$\Vect^H\simeq \on{Funct}_{H\mmod}(\Vect,\Vect)$$
as a monoidal category.

\medskip

The functor $\on{inv}^{\on{enh}}_H$ admits a left adjoint, given by
$$\wt\bC\mapsto \Vect\underset{\Vect^H}\otimes \wt\bC.$$

It is clear, however, that the adjunction $((\on{inv}^{\on{enh}}_H)^L,\on{inv}^{\on{enh}}_H)$ factors
as 
$$\Vect^H\mmod\rightleftarrows (H\mmod)_{\on{alm-triv}} \rightleftarrows H\mmod.$$

We claim: 

\begin{prop}  \label{p:oblv enh}
The adjoint functors
\begin{equation} \label{e:oblv enh}
\Vect^H\mmod\rightleftarrows (H\mmod)_{\on{alm-triv}}
\end{equation} 
are mutually inverse equivalences. 
\end{prop}

The proposition will be proved in \secref{ss:proof of oblv enh}.

\begin{cor}
For $\bC\in H\mmod$, the counit of the adjunction
$$\Vect\underset{\Vect^H}\otimes \bC^H\to \bC$$
is fully faithful with essential image $\on{alm-inv}_H(\bC)$.
\end{cor} 

\ssec{Almost constant sheaves on the affine Grassmannian} 

In this subsection we give a Koszul-dual description of the category of almost constant sheaves on the affine Grassmannian
and related geometries. 

\sssec{Notational change} From now and until \secref{ss:red to triv} we will adopt the following 
notational change\footnote{We do it since the geometry of the curve will not be involved, unlike other
places in this paper.}
$$\fL(G)_{x_0}\rightsquigarrow \fL(G),\,\, \fL^+(G)_{x_0}\rightsquigarrow \fL^+(G),\,\, \Gr_{G,x_0} \rightsquigarrow \Gr_G.$$

\sssec{} \label{sss:alm const Gr intro}

We return to the setting of \secref{sss:almost constant prestacks}. We take
$\CY$ to be the \emph{neutral connected component} of $\fL(G)/K$, where 
$K=K_i$ for $i\geq 0$. 

\sssec{}

Let $\CY_0:=\fL^+(G)/K$. Consider the restriction map
$$\on{C}^\cdot(\CY)\to \on{C}^\cdot(\CY_0).$$

From \eqref{e:cohomology of loop group}\footnote{Indeed, the cochain algebras involved are the same as in the
semi-simple simply-connected case.}, we obtain that the restriction map
\begin{equation} \label{e:generate mods over cohomology}
\on{C}^\cdot(\CY_0)\mod\to \on{C}^\cdot(\CY)\mod
\end{equation}
preserves compactness.

\medskip

We let $\on{C}^\cdot(\CY)\mod_0$ be the full subcategory of $\on{C}^\cdot(\CY)\mod$,
generated by objects in the essential image of \eqref{e:generate mods over cohomology}.
The embedding
$$\on{C}^\cdot(\CY)\mod_0\hookrightarrow \on{C}^\cdot(\CY)\mod$$
admits a continuous right adjoint.

\sssec{Example} \label{sss:cohom Gr}

Let $i=0$, so $\CY$ is the neutral component of $\Gr_{G,{x_0}}$. We have
$$\on{C}^\cdot(\CY)\simeq \Sym(V),$$
where $V$ is a finite-dimensional, cohomologically graded vector space,
concentrated in positive even degrees. 

\medskip

In this case
$$\on{C}^\cdot(\CY)\mod_0\simeq \Sym(V)\mod_0,$$
where the category in the right-hand side is the full subcategory of $\Sym(V)\mod$,
generated by the augmentation module.

\sssec{}

Recall the functor \eqref{e:almost const as cochains ind sch}. We claim:

\begin{prop} \label{p:almost const Gr}
The functor \eqref{e:almost const as cochains ind sch}, i.e.,
$$\Dmod(\CY)^{\on{alm-const}}\to \on{C}^\cdot(\CY)\mod,$$
is an equivalence onto $\on{C}^\cdot(\CY)\mod_0\subset \on{C}^\cdot(\CY)\mod$.
\end{prop} 

The proposition will be proved in \secref{ss:proof of almost const Gr}.

\begin{cor}
The category $\Dmod(\CY)^{\on{alm-const}}$ is dualizable.
\end{cor}

\begin{cor} \label{c:alm const Gr}
The category $\Dmod(\fL(G)/K)^{\on{alm-const}}$ is dualizable.
\end{cor}

\begin{cor}  \label{c:ten omega}
For a category $\bC$, the functor
$$\Dmod(\fL(G)/K)^{\omega,\on{alm-const}}\otimes \bC\to \on{Funct}_{\DGCat}(\Dmod(\CY)^{\on{alm-const}},\bC)$$
is an equivalence.
\end{cor}

\begin{cor} \label{c:omega ff}
For a category $\bC$, the functor
$$\Dmod(\fL(G)/K)^{\omega,\on{alm-const}}\otimes \bC\to  \Dmod(\fL(G)/K)\otimes \bC$$
is fully faithful.
\end{cor}

\begin{proof}

Follows from the commutative diagram
$$
\CD
\Dmod(\fL(G)/K)^{\omega,\on{alm-const}}\otimes \bC @>>> \Dmod(\fL(G)/K)\otimes \bC \\
@V{\sim}VV @VV{\sim}V \\
\on{Funct}_{\DGCat}(\Dmod(\fL(G)/K)^{\on{alm-const}},\bC) @>>> \on{Funct}_{\DGCat}(\Dmod(\fL(G)/K),\bC), 
\endCD
$$
in which the right vertical arrow is given by Verdier duality on $\Dmod(\fL(G)/K)$, and the bottom horizontal arrow
is fully faithful. 

\end{proof} 

\begin{rem} \label{r:almost const Gr}
As another consequence of \propref{p:almost const Gr}, we obtain that
$$\Dmod(\fL(G)/K)^{\omega,\on{alm-const}} \subset \Dmod(\fL(G)/K)$$
is the full subcategory generated under colimits by the dualizing sheaf.
\end{rem} 

\ssec{Almost constant sheaves on the loop group}

In this subsection we define what we mean by the category of almost constant sheaves on $\fL(G)$. As usual,
some extra care is needed here, since $\fL(G)$ is of infinite type. 

\sssec{}

Let us return to the setting of \secref{sss:almost const omega}. Let $f:Y'\to Y$ be a map of schemes 
of finite type that is a universal homological equivalence. 

\medskip

In this case the functor
$$f^!:\Dmod(Y)\to \Dmod(Y')$$
gives rise to an \emph{equivalence}
$$\Dmod(Y)^{\omega,\on{alm-const}} \overset{\sim}\to \Dmod(Y')^{\omega,\on{alm-const}}.$$

A similar observation applies to a map between indschemes $f:\CY\to \CY'$.  

\sssec{}

Let $K'\subset K$ be subgroups as in \secref{sss:alm const Gr intro}, but we assume that $K$ is pro-unipotent.
We obtain that the pullback functor
$$\Dmod(\fL(G)/K)\to \Dmod(\fL(G)/K')$$
induces an equivalence 
$$\Dmod(\fL(G)/K)^{\omega,\on{alm-const}} \overset{\sim}\to \Dmod(\fL(G)/K')^{\omega,\on{alm-const}}.$$

\medskip

Denote by
$$\Dmod(\fL(G))^{\omega,\on{alm-const},\on{right}}\subset \Dmod(\fL(G))$$
the full subcategory equal to the essential image of
$$\Dmod(\fL(G)/K)^{\omega,\on{alm-const}} \hookrightarrow \Dmod(\fL(G)/K) \overset{!\on{-pullback}}\hookrightarrow 
\Dmod(\fL(G))$$
for some/any $K$ as above.

\medskip

Define
$$\Dmod(\fL(G))^{\omega,\on{alm-const},\on{left}}\subset \Dmod(\fL(G))$$
similarly, by swapping the roles of left and right.

\sssec{}

We claim:

\begin{prop} \label{p:left vs right}
The subcategories 
$$\Dmod(\fL(G))^{\omega,\on{alm-const},\on{right}}\subset \Dmod(\fL(G)) \supset \Dmod(\fL(G))^{\omega,\on{alm-const},\on{left}}$$
coincide. 
\end{prop} 

\begin{proof}

Let $Y$ be a $\fL^+(G)\times \fL^+(G)$-invariant subscheme of $\fL(G)$. Let $K$ be as above.
Note we can find $K'$ sufficiently small so that the projection
$$Y\to Y/K$$
factors $K'\backslash Y\to Y/K$; moreover, the latter map is smooth with contractible fibers, and hence is a universal homological
equivalence.

\medskip

This implies that the subcategory
$$\Dmod(Y)^{\omega,\on{alm-const},\on{right}}\subset \Dmod(Y)$$
defined to be the essential image of
\begin{equation} \label{e:Y right}
\Dmod(Y/K_1)^{\omega,\on{alm-const}}\hookrightarrow \Dmod(Y/K_1) \overset{!\on{-pullback}}\hookrightarrow \Dmod(Y)
\end{equation} 
for some/any $K_1$, coincides with the subcategory
$$\Dmod(Y)^{\omega,\on{alm-const},\on{left}}\subset \Dmod(Y)$$
defined to be the essential image of
\begin{equation} \label{e:Y left}
\Dmod(K_2\backslash Y)^{\omega,\on{alm-const}}\hookrightarrow \Dmod(K_2\backslash Y) \overset{!\on{-pullback}}\hookrightarrow \Dmod(Y)
\end{equation} 
for some/any $K_2$. 

\medskip

This implies the statement of the proposition, since 
$$\Dmod(\fL(G))^{\omega,\on{alm-const},\on{right}}\subset \Dmod(\fL(G))$$
equals
$$\underset{Y}{\on{lim}}\, \Dmod(Y)^{\omega,\on{alm-const},\on{right}} \subset \underset{Y}{\on{lim}}\, \Dmod(Y)$$
and 
$$\Dmod(\fL(G))^{\omega,\on{alm-const},\on{left}}\subset \Dmod(\fL(G))$$
equals
$$\underset{Y}{\on{lim}}\, \Dmod(Y)^{\omega,\on{alm-const},\on{left}} \subset \underset{Y}{\on{lim}}\, \Dmod(Y).$$

%

\end{proof}

\sssec{}

Denote the subcategory
$$\Dmod(\fL(G))^{\omega,\on{alm-const},\on{right}}=\Dmod(\fL(G))^{\omega,\on{alm-const},\on{left}}$$
of $\Dmod(\fL(G))$ by $\Dmod(\fL(G))^{\omega,\on{alm-const}}$.

\sssec{}

From \propref{p:left vs right} and \corref{c:ten omega}, we obtain:

\begin{cor} \label{c:coprod loc const}
The coproduct functor
$$\Dmod(\fL(G))\to \Dmod(\fL(G))\otimes \Dmod(\fL(G))$$
sends $\Dmod(\fL(G))^{\omega,\on{alm-const}}\subset \Dmod(\fL(G))$ to the full subcategory
$$\Dmod(\fL(G))^{\omega,\on{alm-const}}\otimes \Dmod(\fL(G))^{\omega,\on{alm-const}} \subset
\Dmod(\fL(G))\otimes \Dmod(\fL(G)).$$
\end{cor} 

\sssec{}

It follows from \corref{c:alm const Gr} that the category $\Dmod(\fL(G))^{\omega,\on{alm-const}}$ is dualizable.
Denote its dual by $\Dmod(\fL(G))^{\on{alm-const}}$; we can view it as a quotient category of $\Dmod(\fL(G))$.

\medskip

From \corref{c:coprod loc const} we obtain that the \emph{monoidal} structure on $\Dmod(\fL(G))$ gives 
rise to a monoidal structure on $\Dmod(\fL(G))^{\on{alm-const}}$. 

\begin{rem}

It follows from Remark \ref{r:almost const Gr} that 
$$\Dmod(\fL(G))^{\omega,\on{alm-const}}\subset \Dmod(\fL(G))$$
is the full subcategory, generated under colimits by the dualizing sheaf.
\end{rem} 

\ssec{Almost invariants for the loop group} \label{ss:alm inv loop}


\sssec{}

Let $\bC$ be a category equipped with an action of $\fL^+(G)$. We let
$\on{alm-inv}_{\fL^+(G)}(\bC)$ be the full subcategory of $\bC$ equal to
$$\on{alm-inv}_{G}(\bC^{K_1}).$$

The contents of \secref{ss:alm triv pro} apply equally well to this situation.

\sssec{}

Let now $\bC$ be equipped with an action of $\fL(G)$. We let 
$\on{alm-inv}_{\fL(G)}(\bC)$ to be the full subcategory of $\bC$
consisting of objects that are sent by the co-action functor
$$\bC\to \Dmod(\fL(G))\otimes \bC$$
to the full subcategory 
$$\Dmod(\fL(G))^{\omega,\on{alm-const}}\otimes \bC\subset \Dmod(\fL(G))\otimes \bC.$$

\medskip

We can identify
$$\on{alm-inv}_{\fL(G)}(\bC)=\Dmod(\fL(G))^{\omega,\on{alm-const}}\underset{\Dmod(\fL(G))}\otimes \bC\simeq
\on{Funct}_{\Dmod(\fL(G))}(\Dmod(\fL(G))^{\on{alm-const}},\bC).$$

\sssec{}

It is easy to see that $\on{alm-inv}_{\fL(G)}(\bC)$, viewed as a full subcategory of $\bC$,
is preserved by the action of $\fL(G)$, i.e., it can itself be viewed as a category acted on by
$\fL(G)$. 

\sssec{Example} \label{e:alm inv torus}

Let $G=T$ be a torus. Then it is easy to see that the inclusion
$$\on{alm-inv}_{\fL(T)}(\bC)\subset \on{alm-inv}_{\fL^+(T)}(\bC)$$
is an equality. 

\sssec{}

Unwinding the definitions, we obtain:

\begin{lem} \label{l:alm const vs alm triv} \hfill

\smallskip

\noindent{\em(i)} The full subcategories
$$\on{alm-inv}_{\fL(G)\on{-right}}(\Dmod(\fL(G))),\,\, 
\on{alm-inv}_{\fL(G)\on{-left}}(\Dmod(\fL(G))),\,\, \on{alm-inv}_{\fL(G)\times \fL(G)}(\Dmod(\fL(G)))$$ and 
$$\Dmod(\fL(G))^{\omega,\on{alm-const}}$$
of $\Dmod(\fL(G))$, coincide.

\smallskip

\noindent{\em(ii)} For $K=K_i$, $i\geq 1$, the full subcategories
$$\on{alm-inv}_{\fL(G)}(\Dmod(\fL(G)/K)) \text{ and } \Dmod(\fL(G)/K)^{\omega,\on{alm-const}}$$
of $\Dmod(\fL(G)/K)$, coincide.

\end{lem}

\sssec{}

Consider
$$\Dmod(\Gr_G)\in \fL(G)\mmod.$$

It follows from the definitions that we have an inclusion of subcategories
\begin{equation} \label{e:alm const aff Gr}
\Dmod(\Gr_G)^{\omega,\on{alm-const}}\subset \on{alm-inv}_{\fL(G)}(\Dmod(\Gr_G)).
\end{equation}

We claim:

\begin{prop} \label{p:alm const aff Gr}
The inclusion \eqref{e:alm const aff Gr} is an equality.
\end{prop}

The proof will be given in \secref{ss:proof of alm const aff Gr}.

\sssec{} \label{sss:sph gen}

Let $\bC$ be an object of $\fL(G)\mmod$. Recall that $\bC$ is said to be \emph{spherically generated} if the
(a priori fully faithful) functor
\begin{equation} \label{e:Sph gen}
\Dmod(\Gr_G)\underset{\Sph_G}\otimes \bC^{\fL^+(G)}\to \bC
\end{equation} 
is an equivalence, where
$$\Sph_G:=\Dmod(\fL(G))^{\fL^+(G)\times \fL^+(G)}.$$

We claim:

\begin{prop} \label{p:left adj into Sph gen}
Suppose that $\bC$ is spherically generated. Then the embedding
$$\on{alm-inv}_{\fL(G)}(\bC)\hookrightarrow \bC$$
admits a left adjoint. 
\end{prop}

\begin{proof}

It suffices to consider the universal case, i.e., $\bC=\Dmod(\Gr_G)$. I.e., we need to show that the 
embedding
$$\on{alm-inv}_{\fL(G)}(\Dmod(\Gr_G))\hookrightarrow \Dmod(\Gr_G)$$
admits a left adjoint. 

\medskip

By \propref{p:alm const aff Gr}, this is equivalent to showing that the embedding 
$$\Dmod(\Gr_G)^{\omega,\on{alm-const}}\hookrightarrow \Dmod(\Gr_G)$$
admits a left adjoint. 

\medskip

Dually, we need to show that the projection functor 
$$\Dmod(\Gr_G)\to \Dmod(\Gr_G)^{\on{alm-const}}$$
admits a continuous right adjoint, i.e., that it preserves compactness.  

\medskip

We can work at one connected component of $\Gr_G$ at a time, and it is enough
to consider the neutral connected component; denote it by $\CY$. Thus, by \propref{p:almost const Gr},
we have to show that the functor
$$\CF\mapsto \on{C}^\cdot(\CY,\CF), \quad \Dmod(\Gr_G)\to \on{C}^\cdot(\CY)\mod_0$$
preserves compactness. 

\medskip

Note, however, that since $\on{C}^\cdot(\CY)$ is isomorphic to a polynomial algebra on generators in \emph{even} degrees,
an object of $\on{C}^\cdot(\CY)\mod_0$ is compact if and if the underlying vector space is finite-dimensional.

\medskip

The required assertion follows now from the fact that $\CY$ is ind-proper, and so the functor 
$$\on{C}^\cdot(\CY,-):\Dmod(\Gr_G)\to \Vect$$
preserves compactness. 

\end{proof} 

\begin{cor}  \label{c:alm loop left adj}
The embedding
$$\on{alm-inv}_{\fL(G)}(\bC)\hookrightarrow \on{alm-inv}_{\fL^+(G)}(\bC)$$
admits a left adjoint.
\end{cor}

\begin{proof}

Note that $\on{alm-inv}_{\fL^+(G)}(\bC)$ is contained in the spherically generated subcategory of
$\bC$, i.e., the essential image of \eqref{e:Sph gen}. Hence, we can assume that $\bC$ is
spherically generated. 

\medskip

Now the assertion follows from \propref{p:left adj into Sph gen}. 

\end{proof} 

\begin{rem}

It follows from the proof of \propref{p:left adj into Sph gen} that the formation of the left adjoint in
\propref{p:left adj into Sph gen} (resp., \corref{c:alm loop left adj}) is functorial in $\bC$, i.e.,
the corresponding Beck-Chevalley natural transformation is an isomorophism.

\end{rem} 

\ssec{Almost trivial actions of the loop group} \label{ss:alm triv loop}

In this subsection we define what it means for a $\fL(G)_{x_0}$-action on a group to be almost trivial,
and we give a Koszul-dual description of the totality of such categories. 

\sssec{}

We shall say that an action of $\fL(G)$ on $\bC$ is \emph{almost trivial} if the embedding 
\begin{equation} \label{e:alm inv loop}
\on{alm-inv}_{\fL(G)}(\bC)\hookrightarrow \bC
\end{equation} 
is an equivalence. 

\sssec{}

Note that we can also characterize almost trivial actions as follows: an action of $\fL(G)$ on $\bC$ is almost trivial if and only if
the monoidal action of $\Dmod(\fL(G))$ on $\bC$ factors through the quotient
$$\Dmod(\fL(G))\twoheadrightarrow \Dmod(\fL(G))^{\on{alm-const}}.$$

\sssec{}

For future use, we notice:

\begin{lem} \label{l:alm triv cons}
Let $F:\bC_1\to \bC_2$ be a 1-morphism in $\fL(G)\mmod$, which is 
\emph{conservative} as a functor on the underlying categories. Then if
the action on $\bC_2$ is almost trivial, then so it is on $\bC_1$.
\end{lem} 

\begin{proof}

Since the action on $\bC_2$ is almost trivial, both arrows in
$$\on{alm-inv}_{\fL(G)}(\bC_2)\hookrightarrow \on{alm-inv}_{\fL^+(G)}(\bC_2)\hookrightarrow \bC_2$$
are equivalences. We need to show that the same is true for 
$$\on{alm-inv}_{\fL(G)}(\bC_1)\hookrightarrow \on{alm-inv}_{\fL^+(G)}(\bC_1)\hookrightarrow \bC_1.$$

\medskip

We have a commutative diagram\footnote{The horizontal arrows are the adjoint functors from \eqref{e:almost inv}, both of which
are functorial in $\bC$.}
$$
\CD
\bC_1 @>>> \on{alm-inv}_{\fL^+(G)}(\bC_1) @>>> \bC_1 \\
@V{F}VV @VVV @VV{F}V \\
\bC_2 @>>> \on{alm-inv}_{\fL^+(G)}(\bC_2) @>>>  \bC_2.
\endCD
$$

\medskip

Hence, the fact that the counit of the adjunction
$$\on{alm-inv}_{\fL^+(G)}(\bC_2) \rightleftarrows \bC_2$$
is an isomorphism and the conservativity of $F$ imply that 
the unit of the adjunction
$$\on{alm-inv}_{\fL^+(G)}(\bC_1) \rightleftarrows \bC_1$$
is an isomorphism. 

\medskip

The proof for 
$$\on{alm-inv}_{\fL(G)}(\bC_i)\hookrightarrow \on{alm-inv}_{\fL^+(G)}(\bC_i)$$
is similar using left adjoints and \corref{c:alm loop left adj}. 

\end{proof}

\sssec{}

Let
\begin{equation} \label{e:embed alm inv loop}
(\fL(G)\mmod)_{\on{alm-triv}}\subset \fL(G)\mmod
\end{equation} 
be the full subcategory, consisting of $\fL(G)$-module categories, on which the action is almost trivial.

\medskip

Note that we can identify
$$(\fL(G)\mmod)_{\on{alm-triv}}\simeq \Dmod(\fL(G))^{\on{alm-const}}\mmod,$$
viewed as a full subcategory of 
$$\fL(G)\mmod\simeq \Dmod(\fL(G))\mmod.$$

\medskip

The assignment 
$$\bC\mapsto \on{alm-inv}_{\fL(G)}(\bC)$$
is a right adjoint to the embedding \eqref{e:embed alm inv loop}. 

\sssec{}

Consider the (symmetric) monoidal category
$\Vect^{\fL(G)}$.

\medskip

The functor
$$\on{inv}_{\fL(G)}:\fL(G)\mmod\to \DGCat$$
naturally upgrades to a functor
$$\on{inv}^{\on{enh}}_{\fL(G)}:\fL(G)\mmod\to \Vect^{\fL(G)}\mmod,$$
which admits a left adjoint, given by
\begin{equation} \label{e:LG in tensor up}
\wt\bC\mapsto \Vect\underset{\Vect^{\fL(G)}}\otimes \wt\bC.
\end{equation}

\medskip

It is clear that the above adjoint pair factors as
$$\Vect^{\fL(G)}\mmod \rightleftarrows (\fL(G)\mmod)_{\on{alm-triv}} \rightleftarrows \fL(G)\mmod .$$

\sssec{}

We will prove:

\begin{thm}  \label{t:Kosz for loop}
The adjoint functors 
$$\Vect^{\fL(G)}\mmod \rightleftarrows (\fL(G)\mmod)_{\on{alm-triv}}$$
are mutually inverse equivalences.
\end{thm}

The theorem will be proved in \secref{ss:proof of Kosz for loop}. 

\sssec{}

From \thmref{t:Kosz for loop} we obtain:

\begin{cor}
The naturally defined functor 
$$\Vect\underset{\Vect^{\fL(G)}}\otimes \Vect\to \Dmod(\fL(G))^{\omega,\on{alm-const}}$$
is an equivalence.
\end{cor} 

\sssec{} \label{sss:Av Gr}

For future use, we record the following consequence of \thmref{t:Kosz for loop}.

\medskip

Let $\bC$ be an object of $\fL(G)\mmod$. Note that since $\Gr_G$ is proper, the forgetful functor
$$\oblv_{\fL(G)\to \fL^+(G)}:\bC^{\fL(G)}\to \bC^{\fL^+(G)}$$
admits a \emph{left} adjoint, to be denoted 
$$\on{Av}^{\fL^+(G)\to \fL(G)}_!,$$
see \secref{sss:Av Gr App}. 

\medskip

We claim:

\begin{prop} \label{p:when al triv}
Let $\bC$ be spherically generated. Then the action of $\fL(G)$ is almost trivial if 
the functor $\on{Av}^{\fL^+(G)\to \fL(G)}_!$ is conservative. 
\end{prop} 

\begin{proof}

Since both $\bC$ and $\on{alm-inv}_{\fL(G)}(\bC)$ are spherically generated, the inclusion 
$$\on{alm-inv}_{\fL(G)}(\bC)\hookrightarrow \bC$$
is an equivalence if and only if the functor
$$((\on{alm-inv}_{\fL(G)}(\bC))^{\fL^+(G)}\to \bC^{\fL^+(G)},$$
which is a priori also a fully faithful inclusion, is an equivalence. 

\medskip

By \thmref{t:Kosz for loop}, we can identify the above functor with
\begin{equation} \label{e:Sph triv}
\Vect^{\fL^+(G)}\underset{\Vect^{\fL(G)}}\otimes \bC^{\fL(G)}\to \bC^{\fL^+(G)}.
\end{equation}

\medskip

Thus, the action of $\fL(G)$ is almost trivial if and only if the left adjoint to \eqref{e:Sph triv} is conservative. 

\medskip

The precomposition of \eqref{e:Sph triv} with
\begin{equation} \label{e:Sph triv one}
\bC^{\fL(G)}\simeq \Vect^{\fL(G)}\underset{\Vect^{\fL(G)}}\otimes \bC^{\fL(G)}\to \Vect^{\fL^+(G)}\underset{\Vect^{\fL(G)}}\otimes \bC^{\fL(G)}
\end{equation}
is the forgetful functor $\oblv_{\fL(G)\to \fL^+(G)}$.

\medskip

Thus, if the left adjoint to  
$\oblv_{\fL(G)\to \fL^+(G)}$, i.e., $\on{Av}^{\fL^+(G)\to \fL(G)}_!$, is conservative, then so is the left adjoint to \eqref{e:Sph triv}. 
%
%

\end{proof} 

\begin{rem}

In fact, one can show that the assertion of \propref{p:when al triv} is ``if and only if", but we will not need this.

\end{rem} 

\sssec{} \label{sss:Frob}

Note that we an identify
$$\Vect^{\fL(G)}\simeq \on{C}_\cdot(\fL(G))\mod,$$
as monoidal categories. 

\medskip

Hence, being equivalent to the category of modules over a Hopf algebra, $\Vect^{\fL(G)}$ is a Frobenius
algebra in $\DGCat$. In particular, we have a canonical equivalence
$$(\Vect^{\fL(G)})^\vee\simeq \Vect^{\fL(G)}$$
as $\Vect^{\fL(G)}$-module categories. 

\ssec{The reduction step} \label{ss:red to triv}

After all the preparations, in this subsection we will finally finally formulate a reduction step in the proof
of \thmref{t:main}: the claim is that it is sufficient to prove it for $\bC=\Vect$. 

\sssec{} \label{sss:pass to alm inv}

It is clear that the embedding \eqref{e:alm inv loop}
induces an equivalence
$$(\on{alm-inv}_{\fL(G)}(\bC))^{\fL(G)}\overset{\sim}\to \bC^{\fL(G)}.$$

\sssec{}

We will prove:

\begin{thm} \label{t:red to alm const case}
The embedding \eqref{e:alm inv loop} induces an equivalence
$$\omega_{\Gr_G}\mod^\onfact((\on{alm-inv}_{\fL(G)}(\bC))^{\on{fact}_{x_0},\Dmod(\Gr_G)})_{x_0} \overset{\sim}\to 
\omega_{\Gr_G}\mod^\onfact(\bC^{\on{fact}_{x_0},\Dmod(\Gr_G)})_{x_0}.$$
\end{thm}  

\thmref{t:red to alm const case} will be proved in Sects. \ref{s:torus}-\ref{s:proof of main-end}. In the remainder of this section and 
\secref{s:comp}, we will show how \thmref{t:red to alm const case} implies \thmref{t:main}. 

\sssec{}

Note that by \thmref{t:red to alm const case} and \secref{sss:pass to alm inv} we obtain that it suffices to prove
\thmref{t:main} for $\bC\in (\fL(G)\mmod)_{\on{alm-triv}}$. 

\medskip

By \thmref{t:Kosz for loop}, we can assume that $\bC$ is of the form
$$\Vect\underset{\Vect^{\fL(G)_{x_0}}}\otimes \wt\bC, \quad \wt\bC\in \Vect^{\fL(G)_{x_0}}\mmod.$$

\sssec{}

It follows from \thmref{t:Kosz for loop} that the operation
$$\wt\bC\mapsto \Vect\underset{\Vect^{\fL(G)_{x_0}}}\otimes \wt\bC$$
preserves limits.

\medskip

Moreover, by \secref{sss:Frob}, any object of $\Vect^{\fL(G)_{x_0}}\mmod$ can be written as a totalization of a cosimplicial object with terms 
of the form 
$$\Vect^{\fL(G)_{x_0}}\otimes \wt\bC_0, \quad \wt\bC_0\in \DGCat.$$

\sssec{}

This reduces the assertion of \thmref{t:main} to the case when the action 
$$\bC=\Vect\otimes \wt\bC_0\simeq \wt\bC_0, \quad \wt\bC_0\in \DGCat,$$
equipped with the trivial action.

\medskip

In \secref{s:comp}, we will prove the assertion of \thmref{t:main} in this case by
an explicit calculation. 

\sssec{} \label{sss:vect implies}

To simplify the exposition we will consider the case when $\wt\bC_0=\Vect$. 
The case of a general category is completely analogous.

\medskip

Note that the assertion of \thmref{t:main} for $\bC=\Vect$ coincides with that of
\corref{c:main Vect}. 

\section{Proofs of Propositions \ref{p:oblv enh}, \ref{p:almost const Gr}, \ref{p:alm const aff Gr} and \thmref{t:Kosz for loop}} \label{s:proofs Kosz}

The goal of this subsection is to supply proofs of \propref{p:oblv enh}, \ref{p:almost const Gr}, \ref{p:alm const aff Gr} and \thmref{t:Kosz for loop}.
These all are Koszul duality-type statements, and we essentially need to take care of convergence issues. 

\medskip

For the duration of this section we keep the notational change
$$\fL(G)_{x_0}\rightsquigarrow \fL(G),\,\, \fL^+(G)_{x_0}\rightsquigarrow \fL^+(G),\,\, \Gr_{G,x_0} \rightsquigarrow \Gr_G.$$

\ssec{Proof of \propref{p:oblv enh}} \label{ss:proof of oblv enh}

\sssec{}

It is easy to see that the statement of \propref{p:oblv enh} for $H$ follows from the corresponding statement
for its neutral connected component. Hence, for the duration of this subsection, we will assume that $H$
is connected. 

\sssec{} \label{sss:Kosz gen}

Let $\bA$ be a monoidal category equipped with a monoidal $\phi$ functor to $\Vect$. Denote
$$\bB:=\on{Funct}_{\bA\mmod}(\Vect,\Vect)^{\on{rev}},$$
where $\bA$ acts on $\Vect$ via $\phi$. 

\medskip

The functor 
$$\bA\mmod\to \Vect, \quad \bC\mapsto \on{inv}_\bA=\on{Funct}_{\bA\mmod}(\Vect,\bC)$$
upgrades to a functor 
$$\on{inv}^{\on{enh}}_\bA:\bA\mmod\to \bB\mmod,$$
and the latter admits a left adjoint given by
$$\wt\bC\mapsto \Vect\underset{\bB}\otimes \wt\bC,$$
where the augmentation on $\bB$ is given by the forgetful functor
$$\on{Funct}_{\bA\mmod}(\Vect,\Vect)\to \on{Funct}_{\DGCat}(\Vect,\Vect)\simeq \Vect.$$

\medskip

We shall say that the pair
$(\bA,\phi)$ satisfies Koszul duality if the adjoint functors $((\on{inv}^{\on{enh}}_\bA)^L,\on{inv}^{\on{enh}}_\bA)$
are mutually inverse equivalences.

\sssec{} \label{sss:fin dim Hopf}

Let $A$ be a finite-dimensional Hopf algebra, and let $\bA:=A\mod$. Let $\phi$ be the tautological forgetful
functor $A\mod\to \Vect$.

\medskip

It is easy to see that in this case 
$$\bB\simeq B\mod,$$
where $B$ is the linear dual of $A$. 

\sssec{Example} \label{sss:simple Koszul}

Let $W$ be a compact object of $\Vect$, concentrated in odd degrees. Set
$$\bA:=\Sym(W)\mod,$$
which we regard as a monoidal category with respect to \emph{convolution}.
Let $\phi$ be the tautological forgetful functor $\Sym(W)\mod\to \Vect$. 

\medskip

According to \secref{sss:fin dim Hopf}, 
$\bB\simeq \Sym(W^*)\mod$, viewed as a monoidal category
also with respect to \emph{convolution}.

\medskip

Then it is easy to see that this pair $(\bA,\phi)$ satisfies Koszul duality. 

\sssec{} \label{sss:BH}

The statement of \propref{p:oblv enh} is equivalent to the fact that the monoidal category $\on{C}^\cdot(H)\mod$ equipped
with the tautological forgetful functor to $\Vect$ satisfies Koszul duality. 

\medskip

By construction, this example fits the pattern of \secref{sss:fin dim Hopf}. Hence, it suffices to show that it fits 
in fact the pattern of \secref{sss:simple Koszul}.

\medskip

Let $\fa$ be the Lie algebra that controls the rational homotopy type of $BH$. Note that $\on{C}^\cdot(H)\simeq \on{C}^\cdot(\Omega(\fa))$. 
Recall also that $\fa$ is abelian and is concentrated in odd degrees. Hence, 
$$\on{C}^\cdot(\Omega(\fa))\simeq \Sym(\ul\fa^*),$$
as Hopf algebras, where $\ul\fa$ is the vector space underlying $\fa$.

\qed[\propref{p:oblv enh}]

\begin{rem}

Note that by \secref{sss:fin dim Hopf}, we obtain that
$$\Vect^H\simeq \on{C}_\cdot(H)\mod,$$
as is supposed to be the case. 

\end{rem} 

\ssec{A descent result for almost constant sheaves} \label{ss:descent alm-const}

\sssec{}

Let $H$ be a connected algebraic group, and let 
$$f:\wt{Y}\to Y$$
be an $H$-torsor and $Y$ is a scheme of finite type. 

\medskip

We consider the category $\Dmod(\wt{Y})$ as equipped with an action of $H$. 
Its full subcategory $\Dmod(\wt{Y})^{\on{alm-const}}$ is stable under this action; moreover
the $H$-action on $\Dmod(\wt{Y})^{\on{alm-const}}$ is almost trivial. 

\sssec{}

The functor of *-pullback identifies
$$\Dmod(Y) \overset{\sim}\to (\Dmod(\wt{Y}))^H.$$

Since $f^*$ sends
$$\Dmod(Y)^{\on{alm-const}}\to \Dmod(\wt{Y})^{\on{alm-const}}$$
and 
$$(\Dmod(\wt{Y})^{\on{alm-const}})^H=\Dmod(\wt{Y})^{\on{alm-const}}\underset{\Dmod(\wt{Y})}\times (\Dmod(\wt{Y}))^H,$$
we obtain a commutative diagram
\begin{equation} \label{e:descent alm}
\CD
(\Dmod(\wt{Y})^{\on{alm-const}})^H @>>> \Dmod(\wt{Y})^H \\
@AAA @AA{\sim}A \\
\Dmod(Y)^{\on{alm-const}} @>>> \Dmod(Y),
\endCD
\end{equation} 
where the horizontal arrows are fully faithful.

\medskip

We claim:

\begin{lem} \label{l:descent alm}
The above functor 
$$\Dmod(Y)^{\on{alm-const}} \to (\Dmod(\wt{Y})^{\on{alm-const}})^H$$
is an equivalence. 
\end{lem}

\begin{proof}

The functor in question is fully faithful and preserves compactness (since the other three arrows in 
\eqref{e:descent alm} have this property).

\medskip

Hence, it is enough to show that it sends compact generators of $\Dmod(Y)^{\on{alm-const}}$ to
generators of $(\Dmod(\wt{Y})^{\on{alm-const}})^H$. 

\medskip

This functor sends 
$$\ul{k}_Y\mapsto \ul{k}_{\wt{Y}}.$$

Hence, it remains to show that the latter is a generator of $(\Dmod(\wt{Y})^{\on{alm-const}})^H$. (It is here that
the assumption that $H$ is connected will be used.) 

\medskip

Indeed, for a category $\bC$ with an action of $H$ and $\bc,\bc'\in \bC^H$, the object
$$\CHom_\bC(\bc,\bc')\in \Vect$$ naturally upgrades to an object of $\Vect^H$, while
$$\CHom_{\bC^H}(\bc,\bc')\simeq \on{inv}_H(\CHom_\bC(\bc,\bc')).$$

Now, if $\bc$ is a generator of $\bC$, we have
$$\bc'\neq 0 \, \Rightarrow \CHom_\bC(\bc,\bc')\neq 0.$$ 

We now use the fact that for a connected $H$, the functor
$$\on{inv}_H:\Vect^H\to \Vect$$
is conservative.\footnote{Indeed, in the notations of \secref{sss:BH}, we have $\Vect^H\simeq \Sym(\ul\fa^*)\mod$,
and since $\ul\fa^*$ is concentrated in \emph{odd} degrees, this category is generated by the augmentation module.}


\end{proof} 

\begin{cor} \label{c:descent alm}
The functor $f^!$ induces an equivalence
$$\Dmod(Y)^{\omega,\on{alm-const}} \to (\Dmod(\wt{Y})^{\omega,\on{alm-const}})^H.$$
\end{cor} 

\ssec{Proof of \propref{p:almost const Gr}} \label{ss:proof of almost const Gr}

\sssec{Reduction to the case of the affine Grassmannian}

We change the notations slightly and denote by $\CY$ the neutral connected component of $\Gr_G$, and by
$\wt\CY$ its preimage in $\fL(G)/K$. 

\medskip

Write 
$$\CY_i=\underset{i}{``\on{colim}"}\, Y_i,$$
and set
$$\wt{Y}_i:=Y_i\underset{\CY}\times \wt\CY,$$
so that
$$\wt\CY\simeq \underset{i}{``\on{colim}"}\, \wt{Y}_i.$$

\medskip 

We view both sides of 
$$\underset{i}{\on{colim}}\, \on{C}^\cdot(\wt{Y}_i)\mod \to \on{C}^\cdot(\wt{\CY})\mod,$$
as acted on (almost trivially) by $\fL^+(G)/K$.  Since operation $\on{inv}_{\fL^+(G)/K}(-)$ is conservative 
on the subcategory of almost trivial modules, it suffices to show that the functor
\begin{equation} \label{e:almost const as cochains ind sch inv}
\underset{i}{\on{colim}}\, (\on{C}^\cdot(\wt{Y}_i)\mod)^{\fL^+(G)/K}\overset{\sim}\to 
(\underset{i}{\on{colim}}\, \on{C}^\cdot(\wt{Y}_i)\mod)^{\fL^+(G)/K}\to (\on{C}^\cdot(\wt{\CY})\mod)^{\fL^+(G)/K}
\end{equation} 
is an equivalence onto
$$(\on{C}^\cdot(\wt{\CY})\mod_0)^{\fL^+(G)/K}\subset (\on{C}^\cdot(\wt{\CY})\mod)^{\fL^+(G)/K}.$$

\medskip

By \lemref{l:descent alm}, we can identify the terms 
$$(\on{C}^\cdot(\wt{Y}_i)\mod)^{\fL^+(G)/K}\simeq (\Dmod(\wt{Y}_i)^{\on{alm-const}})^{\fL^+(G)/K}$$
with 
$$\Dmod(Y_i)^{\on{alm-const}}\simeq \on{C}^\cdot(Y_i)\mod.$$

\medskip

Similarly, it is easy to see (e.g., using \eqref{e:cohomology of loop group}) that the right-hand side in \eqref{e:almost const as cochains ind sch inv}
identifies with $\on{C}^\cdot(\CY)\mod$, which contains $\on{C}^\cdot(\CY)\mod_0$ as a full subcategory. 

\medskip

Thus, we obtain that it suffices to show that the resulting map
$$\underset{i}{\on{colim}}\, \on{C}^\cdot(Y_i)\mod \to 
\on{C}^\cdot(\CY)\mod_0.$$
is an equivalence. 
 
\sssec{}

Note that for a \emph{coconnective} algebra $A$ and $n\geq 0$, the truncation $A^{\leq n}$ has a natural structure 
of algebra, equipped with a map from $A$.

\medskip

Consider the commutative diagram
\begin{equation} \label{e:trunc diag}
\CD
\underset{i,n}{\on{colim}}\, (\on{C}^\cdot(Y_i))^{\leq n}\mod  @>>> \underset{i}{\on{colim}}\, \on{C}^\cdot(Y_i)\mod \\
@VVV @VVV \\
\underset{n}{\on{colim}}\, (\on{C}^\cdot(\CY))^{\leq n}\mod  @>>> \on{C}^\cdot(\CY)\mod_0.
\endCD
\end{equation} 

We need to show that the right vertical arrow is an equivalence. We will achieve this by showing that
the other three arrows in \eqref{e:trunc diag} are equivalences.

\sssec{}

The equivalence is immediate for the top horizontal arrow: indeed for a fixed $i$, the family
$$n\rightsquigarrow (\on{C}^\cdot(Y_i))^{\leq n}$$
stabilizes to $\on{C}^\cdot(Y_i)$, since $Y_i$ is finite-dimensional.

\sssec{}

We claim that for a fixed $n$, the family 
$$i\rightsquigarrow (\on{C}^\cdot(Y_i))^{\leq n}$$
also stabilizes to $(\on{C}^\cdot(\CY))^{\leq n}$. Indeed, this follows from the cellular decomposition of the affine Grassmannian.

\medskip

Hence, the left vertical arrow in \eqref{e:trunc diag} is an equivalence. 

\sssec{}

Thus, it remains to show that the functor
$$\underset{n}{\on{colim}}\, (\on{C}^\cdot(\CY))^{\leq n}\mod \to \on{C}^\cdot(\CY)\mod_0$$
is an equivalence. 

\medskip

We identify $\on{C}^\cdot(\CY)$ with $\Sym(V)$, see \secref{sss:cohom Gr}. So we need to show that the functor
$$\underset{n}{\on{colim}}\, (\Sym(V)/\Sym^{>n}(V))\mod \to \Sym(V)\mod_0$$
is an equivalence.

\sssec{}

It is easy to reduce the assertion to the case when $V$ is one dimensional. In this case, we can use 
the grading-shearing trick (see \cite[Sect. A.2]{AG}), 
and assume that $V$ is a finite-dimensional vector space in cohomological degree $0$. 

\medskip

Hence, the assertion becomes that
$$\underset{n}{\on{colim}}\, (\Sym(V)/\Sym^{>n}(V))\mod \to \Sym(V)\mod_0$$
is an equivalence. I.e., we have to show that 
\begin{equation} \label{e:QCoh trunc}
\underset{n}{\on{colim}}\, \QCoh(S_n)\to \QCoh(\BA^1)_0
\end{equation} 
is an equivalence, where $S_n=\Spec(k[t]/t^n)$. 

\medskip

Note, however, that the composition
\begin{equation} \label{e:QCoh trunc bis}
\underset{n}{\on{colim}}\, \IndCoh(S_n)\to \underset{n}{\on{colim}}\, \QCoh(S_n)\overset{\text{\eqref{e:QCoh trunc}}}\longrightarrow  \QCoh(\BA^1)_0
\end{equation} 
is an equivalence (e.g., by \cite[Proposition 7.4.5]{GR0}). This implies that \eqref{e:QCoh trunc} is an equivalence, since the first arrow
in  \eqref{e:QCoh trunc bis} is a Verdier quotient. 

\qed[\propref{p:almost const Gr}]

\ssec{Proof of \propref{p:alm const aff Gr}} \label{ss:proof of alm const aff Gr}

\sssec{}

Denote $\wt\Gr_G:=\fL(G)/K_1$, so that
$$\Gr_G\simeq \wt\Gr_G/G.$$

\medskip

Unwinding the definitions and using \lemref{l:alm const vs alm triv}, we obtain that the category $\on{alm-inv}_{\fL(G)}(\Dmod(\Gr_G))$
identifies with 
$$(\Dmod(\wt\Gr_G)^{\omega,\on{alm-const}})^G.$$

\medskip

Thus, we need to show that the inclusion
$$\Dmod(\Gr_G)^{\omega,\on{alm-const}}\hookrightarrow (\Dmod(\wt\Gr_G)^{\omega,\on{alm-const}})^G$$
is an equality.

\sssec{} \label{sss:proof of alm const aff Gr}

Write 
$$\Gr_G=\underset{i}{``\on{colim}"}\, Y_i,$$
and set
$$\wt{Y}_i:=Y_i\underset{\Gr_G}\times \wt\Gr_G,$$
so that
$$\wt\Gr_G\simeq \underset{i}{``\on{colim}"}\, \wt{Y}_i.$$

We have:
$$\Dmod(\Gr_G)^{\omega,\on{alm-const}}\simeq \underset{i}{\on{lim}}\, \Dmod(Y_i)^{\omega,\on{alm-const}}$$
and 
$$\Dmod(\wt\Gr_G)^{\omega,\on{alm-const}}\simeq \underset{i}{\on{lim}}\, \Dmod(\wt{Y}_i)^{\omega,\on{alm-const}},$$
and hence
$$(\Dmod(\wt\Gr_G)^{\omega,\on{alm-const}})^G\simeq \underset{i}{\on{lim}}\, (\Dmod(\wt{Y}_i)^{\omega,\on{alm-const}})^G.$$

\medskip

Now the required isomorphism follows from \corref{c:descent alm}.

\qed[\propref{p:alm const aff Gr}]

\ssec{Proof of \thmref{t:Kosz for loop}} \label{ss:proof of Kosz for loop}

\sssec{}

We consider the adjunction
\begin{equation} \label{e:inf LG adj}
\Vect^{\fL(G)}\mmod \rightleftarrows (\fL(G)\mmod)_{\on{alm-triv}}.
\end{equation} 

Since the functor $\on{inv}_{\fL(G)}$ (and hence $\on{inv}^{\on{enh}}_{\fL(G)}$)
commutes with tensor products, the left adjoint in \eqref{e:inf LG adj} is fully faithful. 

\medskip

Hence, the two functors are mutually inverse equivalences if and only if $\on{inv}^{\on{enh}}_{\fL(G)}$
is conservative. 

\sssec{} 

We first consider the case when $G$ is semi-simple and simply-connected. 

\medskip

We view 
$\bA:=\Dmod(\fL(G))^{\on{alm-const}}$ as a monoidal
category under convolution and a natural monoidal functor $\phi$ to $\Vect$. 
We claim that this pair $(\bA,\phi)$ satisfies Koszul duality (see \secref{sss:Kosz gen}). 

\medskip

By \propref{p:almost const Gr}, we can identify
$$\bA\simeq \on{C}^\cdot(\fL(G))\mod_0,$$
equipped with the tautological forgetful functor to $\Vect$.

\medskip

By \eqref{e:cohomology of loop group} and \secref{sss:rat hom 3}, we can identify $\on{C}^\cdot(\fL(G))$ as a Hopf algebra with 
$$\Sym(V)\otimes \Sym(W),$$
where $V$ is a cohomologically graded vector space concentrated in positive even degrees,
and $W$ is a cohomologically graded vector space concentrated in positive odd degrees. 

\medskip

Under this identification $\on{C}^\cdot(\fL(G))\mod_0$ corresponds to
$$\Sym(V)\mod_0\otimes \Sym(W)\mod.$$

\sssec{}

It is enough to show that both monoidal categories $$\bA_1:=\Sym(V)\mod_0 \text{ and } \bA_2:=\Sym(W)\mod,$$
equipped with the forgetful functors to $\Vect$, satisfy Koszul duality.

\medskip

The case of $\bA_2$ is immediate, see \secref{sss:simple Koszul}. 

\sssec{}

In the case of $\bA_1$, Using the grading-shearing trick (see \cite[Sect. A.2]{AG}), we can assume that $V$ is a 
finite-dimensional vector space in cohomological degree $0$. In this case, we identify
$$\Sym(V)\mod_0\simeq \Rep(V^*),$$
where $V^*$ is regarded as an (additive) algebraic group. In this case, the Koszul duality
statement is well-known. 

\sssec{} \label{sss:Koszul torus}

Next we consider the case when $G=T$ is a torus. We will show
directly that the counit of the adjunction in \eqref{e:inf LG adj} is an equivalence. Let $\bC$ be an object of $\fL(T)\mmod_{\on{alm-triv}}$.
The counit is the functor
\begin{equation} \label{e:counit torus}
\Vect\underset{\Vect^{\fL(T)}}\otimes \bC^{\fL(T)}\to \bC.
\end{equation} 

Consider the short exact sequence
$$1\to \fL^+(T)\to \fL(T)\to \Lambda\to 1.$$

By \propref{p:oblv enh}, in order to show that \eqref{e:counit torus}
is an equivalence, it suffices to show that
\begin{equation} \label{e:counit torus 1}
\Vect^{\fL^+(T)}\underset{\Vect^{\fL(T)}}\otimes \bC^{\fL(T)}\overset{\sim}\to
(\Vect\underset{\Vect^{\fL(T)}}\otimes \bC^{\fL(T)})^{\fL^+(T)}\to \bC^{\fL^+(T)}
\end{equation}
is an equivalence.

\medskip

We regard both sides of \eqref{e:counit torus 1} as acted on by $\fL(T)/\fL^+(T)\simeq \Lambda$. 
In order to show that \eqref{e:counit torus 1} is an equivalences, it is sufficient that it becomes
so after taking $\Lambda$-invariants:
\begin{equation} \label{e:counit torus 2}
(\Vect^{\fL^+(T)})^\Lambda\underset{\Vect^{\fL(T)}}\otimes \bC^{\fL(T)}\overset{\sim}\to
\left(\Vect^{\fL^+(T)}\underset{\Vect^{\fL(T)}}\otimes \bC^{\fL(T)}\right)^\Lambda \to  (\bC^{\fL^+(T)})^\Lambda \simeq \bC^{\fL(T)}.
\end{equation}

However, the latter composition is the identity functor
$$\bC^{\fL(T)}\simeq \Vect^{\fL(T)}\underset{\Vect^{\fL(T)}}\otimes \bC^{\fL(T)}\simeq
(\Vect^{\fL^+(T)})^\Lambda\underset{\Vect^{\fL(T)}}\otimes \bC^{\fL(T)}\to \bC^{\fL(T)}.$$

\sssec{}

We now consider the case when the derived group $G'$ of $G$ is simply connected. We have a short
exact sequence
$$1\to G'\to G\to T_0\to 1,$$
where $T_0$ is a torus. 

\medskip

In this case, the fact that the counit of the adjunction in \eqref{e:inf LG adj} is an equivalence follows from the validity of 
\thmref{t:Kosz for loop} for $G'$ and $T_0$ by the argument in \secref{sss:Koszul torus} above. 

\sssec{}

Finally, let $G$ be arbitrary. We wish to show that the functor $\on{inv}_{\fL(G)}(-)$ is conservative. 

\medskip

Choose a short exact sequence
$$1\to T_0\to \wt{G}\to G\to 1,$$
where $T_0$ is a torus and $\wt{G}$ is such that its derived group is simply-connected.  By what we proved
above, the the operation $\on{inv}_{\fL(\wt{G})_{x_0}}(-)$ is conservative. Hence, it suffices to show that for
a functor $\bC_1\to \bC_2$ if
\begin{equation} \label{e:C12 functor inv}
\on{inv}_{\fL(G)}(\bC_1)\to \on{inv}_{\fL(G)}(\bC_2)
\end{equation}
is an equivalence, then so is 
\begin{equation} \label{e:C12 functor inv tilde}
\on{inv}_{\fL(\wt{G})}(\bC_1)\to \on{inv}_{\fL(\wt{G})}(\bC_2).
\end{equation}

\medskip

For a category $\bC$ with an action of $\fL(G)$, we have
$$\on{inv}_{\fL(\wt{G})}(\bC)\simeq \on{inv}_{\fL(G)}(\bC\otimes \Vect^{\fL(T_0)}).$$

We have a monadic adjunction
$$\Vect^{\fL(T_0)}\rightleftarrows \Vect.$$

From here we obtain a monadic adjunction
$$\on{inv}_{\fL(\wt{G})}(\bC) \rightleftarrows \on{inv}_{\fL(G)}(\bC).$$

Moreover, for a functor $\bC_1\to \bC_2$, the functor \eqref{e:C12 functor inv} intertwines
the two monads. Hence, if \eqref{e:C12 functor inv} is an equivalence, so is
\eqref{e:C12 functor inv tilde}.

\qed[\thmref{t:Kosz for loop}]

\section{Proof of \thmref{t:main}: the case of a trivial action} \label{s:comp}

In this section we will prove \thmref{t:main} for $\bC=\Vect$. The proof will amount
to calculating a certain monad, and this calculation will turn out to be equivalent to the
\emph{contractibility}\footnote{A.k.a., \emph{non-abelian Poincaré duality}.} statement
from \cite{Ga1}.

\medskip

This calculation is the crux of the proof, and expresses the intuitive idea (alluded to in the
Introduction) that
$$\underset{\text{punctured disc}}\int\, \Gr_G\simeq \fL(G).$$

\ssec{Setting up the monad}

In this subsection we will reduce the assertion of \thmref{t:main} for $\bC=\Vect$ to a calculation that says that
some particular map (in $\Vect$) is an isomorphism. 

\sssec{}

We consider the functor
\begin{equation} \label{e:main functor Vect again}
\Phi:\Vect^{\fL(G)_{x_0}}\to \omega_{\Gr_G}\mod^\onfact(\Vect^{\on{fact}_{x_0},\Dmod(Gr_G)})_{x_0}
\end{equation}
of \eqref{e:main functor Vect} (which is a particular case of \eqref{e:main functor C}).

\medskip

It makes the diagram
\begin{equation} \label{e:diagram for the monad}
\CD
\Vect @>{\on{Id}}>> \Vect \\
@A{\oblv_{\fL(G)_{x_0}}}AA @AA{\oblv_{\omega_{\Gr_G}}}A  \\
\Vect^{\fL(G)_{x_0}} @>{\Phi}>>  \omega_{\Gr_G}\mod^\onfact(\Vect^{\on{fact}_{x_0},\Dmod(\Gr_G)})_{x_0}
\endCD
\end{equation} 
commute, where the vertical arrows are the tautological forgetful functors. 

\sssec{} \label{sss:need monad}

Note that the left vertical arrow in \eqref{e:diagram for the monad} is conservative and admits a left
adjoint, to be denoted $\on{Av}^{\fL(G)_{x_0}}_!$.

\medskip

By the Barr-Beck-Lurie theorem, we can identify $\Vect^{\fL(G)_{x_0}}$ with the category of modules in $\Vect$
over the resulting monad. 

\medskip

The right vertical arrow in \eqref{e:diagram for the monad} is also conservative. We will show (shortly)
that it also admits a left adjoint.

\medskip

Thus, in order to prove that $\Phi$ is an equivalence, it suffices to show that $\Phi$ 
induces an isomorphism between the two monads. 

\sssec{}

Denote $\wt\Gr_{G,x_0}:=\fL(G)_{x_0}/K$, where $K=K_i$ for some/any $i\geq 1$. We consider it as an ind-scheme, equipped with an action of
$\fL(G)_{x_0}$. Consider the corresponding factorization module category at $x_0$ with respect to $\Dmod(\Gr_G)$: 
\begin{equation} \label{e:fact Gr tilde}
\Dmod(\wt\Gr_{G,x_0})^{\on{fact}_{x_0},\Dmod(\Gr_G)}.
\end{equation} 

\medskip

Denote:
$$\wt\Gr_{G,\Ran_{x_0}}:=\Gr^{\on{level}^\infty_{x_0}}_{G,\Ran_{x_0}}/K_i.$$

We regard $\wt\Gr_{G,\Ran_{x_0}}$ as a factorization module
space at $x_0$ with respect to the factorization space $\Gr_G$. The factorization module
category \eqref{e:fact Gr tilde} is given by considering D-modules on $\wt\Gr_{G,\Ran_{x_0}}$,
viewed as a crystal of categories over $\Ran_{x_0}$ equipped with a natural factorization
structure against $\Dmod(\Gr_G)$.

\sssec{}

Let $\wt\pi_{x_0}$ denote the projection $\wt\Gr_{G,x_0}\to \on{pt}$, and let $\wt\pi^{\on{fact}_{x_0},\Gr_G}$ denote the projection
$$\wt\Gr_{G,\Ran_{x_0}}\to \Gr^{\on{level}^\infty_{x_0}}_{G,\Ran_{x_0}}/\fL(G)_{x_0}\simeq \Gr_{G,\Ran_{x_0}}/\on{Hecke}_{x_0},$$
viewed as a map between factorization module spaces at $x_0$ over $\Gr_G$.

\medskip

Pullback with respect to $\wt\pi^{\on{fact}_{x_0},\Gr_G}$ can be viewed as a functor 
$$\Vect^{\on{fact}_{x_0},\Dmod(\Gr_G)}\to \Dmod(\wt\Gr_{G,x_0})^{\on{fact}_{x_0},\Dmod(\Gr_G)}$$
as factorization module categories at $x_0$ with respect to $\Dmod(\Gr_G)$. 

\sssec{}

Consider the following diagram
\begin{equation} \label{e:big monad diagram}
\CD
\Vect @>{\on{Id}}>> \Vect \\ 
& & @AA{\iota_1^!}A \\
@AA{\on{Id}}A  \Dmod(\wt\Gr_{G,x_0}) \\
& &  @AA{\oblv_{\omega_{\Gr_G}}}A \\
\Vect @>{\wt\Phi}>> \omega_{\Gr_G}\mod^\onfact(\Dmod(\wt\Gr_{G,x_0})^{\on{fact}_{x_0},\Dmod(\Gr_G)})_{x_0} \\
@A{\oblv_{\fL(G)_{x_0}}}AA @AA{(\wt\pi^{\on{fact}_{x_0},\Gr_G})^!}A \\
\Vect^{\fL(G)_{x_0}} @>{\Phi}>>  \omega_{\Gr_G}\mod^\onfact(\Vect^{\on{fact}_{x_0},\Dmod(\Gr_G)})_{x_0},
\endCD
\end{equation} 

where:
\begin{itemize}

\item $\iota_1$ denotes the embedding of the unit point into $\wt\Gr_{G,x_0}$;

\item $\wt\Phi$ denotes the functor that sends the generator $k\in \Vect$ to $\omega_{\wt\Gr_{G,\Ran_{x_0}}}$,
equipped with its natural factorization structure against $\omega_{\Gr_G}$. 

\end{itemize}

It is easy to see that the outer diagram in \eqref{e:big monad diagram} identifies with \eqref{e:diagram for the monad}. 

\sssec{} \label{sss:properties}

We will show that:

\bigskip

\begin{enumerate}

\item The functor $\iota_1^!\circ \oblv_{\omega_{\Gr_G}}$ admits a left adjoint;

\medskip

\item The partially defined functor $(\wt\pi^{\on{fact}_{x_0},\Gr_G})_!$, left adjoint to the lower-right vertical functor
in \eqref{e:big monad diagram}, is defined on the essential images of $\wt\Phi$ and $(\iota_1^!\circ \oblv_{\omega_{\Gr_G}})^L$;

\medskip

\item The Beck-Chevalley natural transformation
$$(\wt\pi^{\on{fact}_{x_0},\Gr_G})_!\circ \wt\Phi \to \Phi\circ \on{Av}^{\fL(G)_{x_0}}_!$$
(arising from the lower portion of \eqref{e:big monad diagram}) becomes an isomorphism after applying the functor
$$\omega_{\Gr_G}\mod^\onfact(\Vect^{\on{fact}_{x_0},\Dmod(\Gr_G)})_{x_0}
\overset{\oblv_{\omega_{\Gr_G}}}\longrightarrow \Vect.$$

\medskip

\item The Beck-Chevalley natural transformation 
$$(\iota_1^!\circ \oblv_{\omega_{\Gr_G}})^L\to \wt\Phi$$
(arising from the upper portion of \eqref{e:big monad diagram}) becomes an isomorphism after applying the functor
\begin{multline} \label{e:complicated forgetful}
\omega_{\Gr_G}\mod^\onfact(\Dmod(\wt\Gr_{G,x_0})^{\on{fact}_{x_0},\Dmod(\Gr_G)})_{x_0} 
\overset{(\wt\pi^{\on{fact}_{x_0},\Gr_G})_!}\longrightarrow \\
\to \omega_{\Gr_G}\mod^\onfact(\Vect^{\on{fact}_{x_0},\Dmod(\Gr_G)})_{x_0}
\overset{\oblv_{\omega_{\Gr_G}}}\longrightarrow \Vect.
\end{multline}

\end{enumerate} 

\medskip

It is clear that the above properties (1)-(4) imply the required property of the monad from \secref{sss:need monad}.

\bigskip

The rest of this section is devoted to the verification of properties (1)-(4). 

\begin{rem}

We will give a purely geometric proof of Properties (2) and (3). However, if we allow ourselves to use
\thmref{t:red to alm const case} (which will be proved independently), the proof of both properties
can be significantly simplified.

\medskip

Indeed, by \thmref{t:red to alm const case}, we can replace 
$$\omega_{\Gr_G}\mod^\onfact(\Dmod(\wt\Gr_{G,x_0})^{\on{fact}_{x_0},\Dmod(\Gr_G)})_{x_0} \rightsquigarrow
\omega_{\Gr_G}\mod^\onfact(\on{alm-inv}_{\fL(G)}(\Dmod(\wt\Gr_{G,x_0}))^{\on{fact}_{x_0},\Dmod(\Gr_G)})_{x_0}.$$

Then, the left adjoint in point (2) is induced by the functor
\begin{equation} \label{e:Ccs loop}
\Dmod(\wt\Gr_{G,x_0})^{\omega,\on{alm-const}}\simeq \on{alm-inv}_{\fL(G)_{x_0}}(\Dmod(\wt\Gr_{G,x_0}))\to \Vect
\end{equation}
as objects of $\fL(G)_{x_0}\mmod$, where \eqref{e:Ccs loop} is the left adjoint to 
$$\Vect\to \Dmod(\wt\Gr_{G,x_0})^{\omega,\on{alm-const}}, \quad k\mapsto \omega_{\wt\Gr_{G,x_0}},$$
which exists, e.g., since the objects of $\Dmod(\wt\Gr_{G,x_0})^{\omega,\on{alm-const}}$ are ind-holonomic\footnote{Alternatively
for any $\bC\in \fL(G)\mmod$, the functor $\bC^{\fL(G)}\to \on{alm-inv}_{\fL(G)}(\bC)$ admits a left adjoint, which is the composition
$\on{alm-inv}_{\fL(G)}(\bC)\hookrightarrow \on{alm-inv}_{\fL^+(G)}(\bC)\overset{\on{Av}^{\fL^+(G)}_!}\longrightarrow
\bC^{\fL^+(G)}\overset{\on{Av}^{\fL^+(G)\to \fL(G)}_!}\longrightarrow \bC^{\fL(G)}$.}.

\medskip

Property (3) follows from the fact that the Beck-Chevalley natural transformation arising from the commutative diagram
$$
\CD
\Vect @>{k\mapsto \omega_{\wt\Gr_{G,x_0}}}>> \Dmod(\wt\Gr_{G,x_0})^{\omega,\on{alm-const}} \\
@A{\oblv_{\fL(G)_{x_0}}}AA @AA{\wt\pi_{x_0}^!}A  \\
\Vect^{\fL(G)_{x_0}} @>{\oblv_{\fL(G)_{x_0}}}>> \Vect
\endCD
$$
is an isomorphism. This follows by identifying the above diagram with
$$
\CD
\bC_1^{\fL(G)_{x_0}}  @>>> \bC_1 \\
@AAA @AAA \\ 
\bC_2^{\fL(G)_{x_0}}  @>>> \bC_2
\endCD
$$
with
$\bC_1:=\Dmod(\wt\Gr_{G,x_0})^{\omega,\on{alm-const}}$, $\bC_2=\Vect$ and the functor being $\wt\pi_{x_0}^!$, and the above fact that 
$\wt\pi_{x_0}^!$ admits a left adjoint in $\fL(G)_{x_0}\mmod$. 

\end{rem} 

\ssec{Left adjoint for factorization modules}

In this subsection we establish Property (1) in \secref{sss:properties}. 

\sssec{}

We will apply Proposition \ref{prop fact ind} to 
$$\bA:=\Dmod(\Gr_G),\,\, \bC=\Dmod(\wt\Gr_{G,x_0})^{\on{fact}_{x_0},\Dmod(\Gr_G)},\,\, \CA:=\omega_{\Gr_G}.$$

\begin{lem}
  \label{lem hol is adapt}
    In the above setting, any object $\CF\in \Dmod_{\on{hol}}(\wt\Gr_{G,x_0})$ is adapted to $\omega_{\Gr_G}$-induction (see \secref{sss adapt to induction}).
\end{lem}

\sssec{}
    To prove the lemma, we need some notations.

    \medskip 
    Unwinding the definitions, the functor \eqref{eqn-j!} is given by
    \[
       \jmath_I^!: \Dmod^\lax( \CZ_I  ) \to \Dmod^\lax( \CY_I  ),
    \]
    where
    \begin{itemize}
      \item $\CZ_I$ is the categorical prestack (see \secref{sss categorical prestack})
      \[
          \big( \big(\prod_{i\in I^\circ} \Gr_{G,\Ran^\untl}\big) \times \wt\Gr_{G,\Ran_{x_0}^\untl}\big)_\disj,
      \]
      where
      \begin{itemize}
        \item $\Gr_{G,\Ran^\untl}$ is the space encoding the unital factorization structure of $\Gr_G$ (see \secref{sss GrG untl});

        \item  $\wt\Gr_{G,\Ran_{x_0}^\untl}$ is the space encoding the unital factorization $\Gr_G$-module structure of $\wt\Gr_{G,x_0}$

        \item $(-)_\disj$ means we apply base-change along the open subspace
        \[
            \big( \big(\prod_{i\in I^\circ} \Ran^\untl\big) \times \Ran_{x_0}^\untl \big)_\disj \subseteq \big(\prod_{i\in I^\circ} \Ran^\untl\big) \times \Ran_{x_0}^\untl,
        \]
        see \secref{sss Ran untl disjoint loci}.
      \end{itemize}
      \item $\CY_I$ is the categorical prestack 
        \[
          \big(\prod_{i\in I} \Gr_{G,\Ran_\circ^\untl}\big)_\disj \times \wt\Gr_{G,x_0},
        \]
        where $\Gr_{G,\Ran_\circ^\untl}$ is the space encoding the unital factorization structure of $\Gr_G$, but for the punctured curve $X_\circ:=X\setminus x_0$. See \secref{sss punctured Ran}.
      \item 
        The morphism 
        \[
            \jmath_I: \CY_I\to \CZ_I
        \]
        is the base-change of the morphism \eqref{eqn-jI}. The functor
        \[
            \jmath_I^!:\Dmod^\lax(\CZ_I) \to \Dmod^\lax(\CY_I)
        \]
        is the $!$-pullback functors for lax D-modules (see \secref{sss lax D-module} and \secref{sss !-pullback lax global section}).

    \end{itemize}

\sssec{}
  \label{sss hol is adapt translate}
      Let $\CF\in \Dmod_{\on{hol}}(\wt\Gr_{G,x_0})$ and write 
      \[
          \omega\in \Dmod^\lax(\Gr_{G,\Ran_\circ^\untl})
      \]
      for the dualizing D-module. To prove Lemma \ref{lem hol is adapt}, we need to check:
    \begin{itemize}
      \item 
        For any marked finite set $I$, the partially defined left adjoint $\jmath_{I,!}$ of $\jmath_I^!$ is defined on the object
        \begin{equation}
          \label{eqn-hol-adapt-1}
            (\boxt_{i\in I} \omega)|_\disj  \boxt \CF  \in \Dmod^\lax(\CZ_I).
        \end{equation}
      \item 
        The canonical morphism
        \[
            \jmath_{I,!}(   (\boxt_{i\in I} \omega)|_\disj  \boxt \CF   ) \to  \big( (\boxt_{i\in I} \omega) \boxt \jmath_!( \omega\boxt \CF  ) \big)|_\disj
        \]
        is invertible. 
    \end{itemize}

\sssec{}
  To verify the claims in \secref{sss hol is adapt translate}, we need some preparations.

  \medskip
  Let $\CY$ be any categorical prestack. We say a lax D-module $\CF\in \Dmod^\lax(\CY)$ is ind-holonomic if its $!$-pullback along any affine point $S\to \CY$ is contained in $\Dmod_{\on{hol}}(S)$. Note that \eqref{eqn-hol-adapt-1} is an ind-holonomic object.

  \medskip 
  By definition, $!$-pullback functors preserve ind-holonomic lax D-modules.

  \medskip 
  Let $\CY$ be a categorical prestack. We say a collection of (finite type) indschemes $(f_\alpha:Y_\alpha\to \CY)_{\alpha\in A}$ over $\CY$ is \emph{adapted to $!$-direct images} if
  \begin{itemize}
    \item 
      The functors 
      \[
          f_\alpha^!:\Dmod^\lax(\CY) \to \Dmod(Y_\alpha)
      \]
      are jointly conservative.
    \item 
      The left adjoint of $f_\alpha^!$ exists, i.e., we have the $!$-direct image functor
      \[
        f_{\alpha,!}:\Dmod(Y_\alpha) \to \Dmod^{\lax}(\CY).
      \]
    \item 
      The functor $f_{\alpha,!}$ preserves ind-holonomic lax D-modules, i.e., we have a functor
      \begin{equation}
        \label{eqn-falpha!}
        f_{\alpha,!}:\Dmod_{\on{hol}}(Y_\alpha) \to \Dmod^{\lax}_{\on{hol}}(\CY).
      \end{equation}
  \end{itemize}
  It is clear these conditions imply that $\Dmod^\lax_{\on{hol}}(\CY)$ is generated by the images of \eqref{eqn-falpha!}.

\sssec{}
  \label{sss ! direct image hol}
  Let $\jmath:\CY\to \CZ$ be a morphism between categorical prestacks, and $(f_\alpha:Y_\alpha\to \CY)_{\alpha\in A}$ and $(g_\beta:Z_\beta \to \CZ)_{\beta\in B}$ be collections of indschemes that are adapted to $!$-direct images. We say these two collections are \emph{compatible with $\jmath$} if for any $\alpha$, there exists $\beta$ such that the composition $Y_\alpha\to \CY \to \CZ$ factors through $Z_\beta$.

  \medskip 
  We claim the above condition implies the partially defined left adjoint $\jmath_!$ of
  \[
      \jmath^!: \Dmod^\lax(\CZ) \to \Dmod^\lax(\CY)
  \]
  is defined on $\Dmod^\lax_{\on{hol}}(\CY)$. Namely, we only need to show $(\jmath\circ f_\alpha)_!$ is defined on $\Dmod_{\on{hol}}(Y_\alpha)$ for any $\alpha\in A$. Note that $\jmath\circ f_\alpha$ factors as $\jmath\circ f_\alpha$ as 
  \[
    Y_\alpha \xrightarrow{\jmath_{\alpha\beta}} Z_\beta \xrightarrow{g_\beta} \CZ
  \] 
  for some $\beta$. Now the claim follows from the following two facts:
  \begin{itemize}
    \item 
      The partially defined left adjoint $\jmath_{\alpha\beta,!}$ of 
      \[
          \jmath_{\alpha\beta}^!: \Dmod(Z_\beta) \to \Dmod(Y_\alpha)
      \] 
      is defined on $\Dmod_{\on{hol}}(Y_\alpha)$, because $Y_\alpha$ and $\beta$ are (ind-finite type) indschemes;
    \item 
      The functor $g_{\beta,!}:\Dmod(Z_\beta) \to \Dmod^\lax(Z)$ left adjoint to $g_\beta^!$ exists by assumption.
  \end{itemize}

\sssec{}
  \label{sss ! direct image hol continue}
  Let $f:\CY\to \CZ$ be a morphism between categorical prestacks and $\CW$ be another categorical prestack. Let $(f_\alpha:Y_\alpha\to \CY)_{\alpha\in A}$, $(g_\beta:Z_\beta \to \CZ)_{\beta\in B}$ and $(h_\gamma:W_\gamma \to \CW)_{\gamma\in C}$ be collections of indschemes that are adapted to $!$-direct images. We say these collections are \emph{compatble with $f$ and $(f,\on{id}_\CW)$} if
  \begin{itemize}
    \item 
      $(f_\alpha,h_\gamma)_{(\alpha,\gamma)\in A\times C}$ is a collection of schemes over $\CY \times \CW$ that is adapted to $!$-direct images;
    \item 
      $(g_\beta,h_\gamma)_{(\beta,\gamma)\in B\times C}$ is a collection of schemes over $\CZ \times \CW$ that is adapted to $!$-direct images;
    \item 
      $(f_\alpha)_{\alpha\in A}$ and $(g_\beta)_{\beta\in B}$ are compatible with $f$.
  \end{itemize}

  \medskip 
  As in \secref{sss ! direct image hol}, one can show these conditions imply for $\CM \in \D_{\on{hol}}^\lax(\CY)$ and $\CN \in \D_{\on{hol}}^\lax(\CW)$, we have
  \[
      (\jmath,\on{id}_\CW)_!( \CM\boxt \CN  ) \xrightarrow{\simeq} \jmath_!(\CM) \boxt \CN.
  \]

\sssec{}
  \label{sss ! direct image hol last}
  Finally, let us apply the above paradigm to the claims in \secref{sss hol is adapt translate}. We only need to find collections of indschemes over
  \[
      \CW:=\Gr_{G,\Ran_\circ^\untl},\, \CY:= \Gr_{G,\Ran_\circ^\untl} \times \wt{\Gr_{G,x_0}}, \, \CZ:= \Gr_{G,\Ran_{x_0}^\untl}
  \]
  satisfying the conditions in \secref{sss ! direct image hol continue}. Note that here we can ignore the functor $(-)|_\disj$ because it commutes with arbitrary $!$-direct images.

  \medskip

  Note that $\CY$, $\CZ$ and $\CW$ are defined over $\Ran_{x_0}^\untl$ (resp. $\Ran_\circ^\untl$). Now the desired collections of indschemes over them can be given by applying base-change to the schemes
  \[
    \begin{aligned}
        X^I \times x_0 &\to& \Ran_{x_0}^\untl,\; |I|<\infty  \\ 
        (X-x)^I  &\to& \Ran_{\circ}^\untl ,\; |I|<\infty
    \end{aligned}
  \]
  respectively. Namely, the conditions in \secref{sss ! direct image hol continue} can be verified by combining the following two arguments:
  \begin{itemize}
    \item 
      The maps $X^I\times x_0 \to \Ran_{x_0}$ into the \emph{non-unital} marked Ran space are pseudo-proper, hence there is a $!$-direct image functor along
      \[
          \CT|_{X^I\times x_0} \to  \CT|_{\Ran_{\ul{x}_0}}
      \]
      for $\CT=\CY$ or $\CZ$. Moreover, these functors preserve ind-holonomic objects and commute with external tensor products.
    \item 
      $\CT=\CY$ or $\CZ$ (resp. $\CT=\CW$) is a \emph{coCartesian space} over $\Ran_{x_0}^\untl$ (resp. $\Ran_\circ^\untl$), see \secref{sss Cart catprestk}. Also, for any morphism $\ul{x}\subseteq \ul{x}'$ in $\Ran_{x_0}^\untl$ (resp. in $\Ran_\circ^\untl$), the structural morphism
      \[
          \CT_{\ul{x}} \to \CT_{\ul{x}'}
      \]
      is ind-proper. Then we can mimic the construction in \cite[Sect. 4.3]{Ga1} to obtain a $!$-direct image functor along
      \[
          \CT|_{\Ran_{\ul x_0}} \to \CT,
      \]
      which preserves ind-holonomic objects and commute with external tensor products.
  \end{itemize}

\qed[Lemma \ref{lem hol is adapt}]

\sssec{}
  \label{sss ind in application}

As a consequence of Lemma \ref{lem hol is adapt}, we can apply Proposition \ref{prop fact ind} to $V:= \delta_1$, where $1\in \wt\Gr_{G,x_0}$ 
is a unit point. This gives an object
\[
    \ind_{\omega_{\Gr_G}}(\delta_1) \simeq \jmath_!( \omega_{\Gr_G,\circ} \boxt \delta_1 ) \in \omega_{\Gr_G}\on{-mod}^\onfact(  \Dmod(\wt\Gr_{G,x_0})^{\on{fact}_{x_0},\Dmod(\Gr_G)} ).
\]
Note that the underlying lax D-module of this object is
\[
    \jmath_!( \omega_{\Gr_{G,\Ran_\circ^\untl}} \boxt \delta_1  ) \in \Dmod^\lax(  \wt\Gr_{G,\Ran_{x_0}^\untl}  ),
\]
where we recall that $\jmath$ is the map
\[
    \jmath: \Gr_{G,\Ran_\circ^\untl} \times \wt\Gr_{G,x_0} \to  \wt\Gr_{G,\Ran_{x_0}^\untl}.
\]

\ssec{Verification of Properties (2) and (3)}

\sssec{}

Using the method in \secref{sss ! direct image hol} - \secref{sss ! direct image hol last}, one can show the partially defined functor $(\wt\pi^{\on{fact}_{x_0},\Gr_G})_!$, left adjoint to 
$$(\wt\pi^{\on{fact}_{x_0},\Gr_G})^!:\Dmod^\lax(\Gr_{G,\Ran^\untl_{x_0}}/\on{Hecke}_{x_0})\to \Dmod^\lax(\wt\Gr_{G,\Ran^\untl_{x_0}}),$$
is defined on ind-holonomic objects, and it sends objects in 
$$\omega_{\Gr_G}\mod^\onfact(\Dmod(\wt\Gr_{G,x_0})^{\on{fact}_{x_0},\Dmod(\Gr_G)})_{x_0},$$
whose underlying object of $\Dmod^\lax(\wt\Gr_{G,\Ran_{x_0}})$ is ind-holonomic to objects in 
$$\omega_{\Gr_G}\mod^\onfact(\Vect^{\on{fact}_{x_0},\Dmod(\Gr_G)})_{x_0},$$
thereby providing a left adjoint to the lower-right vertical functor
in \eqref{e:big monad diagram}.

\medskip

This establishes Property (2) in \secref{sss:properties}. 

\sssec{}

We now proceed to establishing Property (3). We need to establish that the Beck-Chevalley natural transformation
corresponding to the diagram
\begin{equation} \label{e:BC setup prelim}
\CD
\Vect @>{\wt\Phi}>>  \omega_{\Gr_G}\mod^\onfact(\Dmod(\wt\Gr_{G,x_0})^{\on{fact}_{x_0},\Dmod(\Gr_G)})_{x_0}  \\
@A{\oblv_{\fL(G)_{x_0}}}AA @AA{(\wt\pi^{\on{fact}_{x_0},\Gr_G})^!}A \\
\Vect^{\fL(G)_{x_0}}  @>{\Phi}>> \omega_{\Gr_G}\mod^\onfact(\Vect^{\on{fact}_{x_0},\Dmod(\Gr_G)})_{x_0}
\endCD
\end{equation} 
is an isomorphism.

\medskip

Note that we can view \eqref{e:BC setup prelim} as a commutative diagram involving the functors \eqref{e:main functor C}:
\begin{equation} \label{e:BC setup}
\CD
\on{inv}_{\fL(G)_{x_0}}(\Dmod(\fL(G)_{x_0})^K) @>{\wt\Phi}>>  \omega_{\Gr_G}\mod^\onfact(\Dmod(\fL(G)_{x_0})^K)^{\on{fact}_{x_0},\Dmod(\Gr_G)})_{x_0}  \\
@AAA @AAA \\
\on{inv}_{\fL(G)_{x_0}}(\Dmod(\fL(G)_{x_0})^{\fL(G)_{x_0}}) @>{\Phi}>> \omega_{\Gr_G}\mod^\onfact((\Dmod(\fL(G)_{x_0})^{\fL(G)_{x_0}})^{\on{fact}_{x_0},\Dmod(\Gr_G)})_{x_0},
\endCD
\end{equation} 
where $\Dmod(\fL(G)_{x_0})^K$ and $\Dmod(\fL(G)_{x_0})^{\fL(G)_{x_0}}$ are considered as objects of $\fL(G)_{x_0}\mmod$ with respect to the \emph{left}
action of $\fL(G)_{x_0}$, and $(-)^K$ and $(-)^{\fL(G)_{x_0}}$ are taken with respect to the \emph{right} action. 

\medskip

The vertical arrows in \eqref{e:BC setup} are induced by
the 1-morphism 
$$\Dmod(\fL(G)_{x_0})^{\fL(G)_{x_0}}\to \Dmod(\fL(G)_{x_0})^K$$
in $\fL(G)_{x_0}\mmod$, given by $\oblv_{\fL(G)_{x_0}\to K}$ with respect to the \emph{right} action. 

\sssec{}

Let $\Phi^+$ be the functor
$$\on{inv}_{\fL(G)_{x_0}}(\Dmod(\fL(G)_{x_0})^{\fL^+(G)_{x_0}})  \to \omega_{\Gr_G}\mod^\onfact((\Dmod(\fL(G)_{x_0})^{\fL^+(G)_{x_0}})^{\on{fact}_{x_0},\Dmod(\Gr_G)})_{x_0},$$
which is a particular case of \eqref{e:main functor C} for $\bC:=(\Dmod(\fL(G)_{x_0})^{\fL^+(G)_{x_0}}=\Dmod(\Gr_{G,x_0})$.

\medskip

We expand diagram \eqref{e:BC setup} as 
\begin{equation} \label{e:BC setup expanded}
\CD
\on{inv}_{\fL(G)_{x_0}}(\Dmod(\fL(G)_{x_0})^K) @>{\wt\Phi}>>  \omega_{\Gr_G}\mod^\onfact(\Dmod(\fL(G)_{x_0})^K)^{\on{fact}_{x_0},\Dmod(\Gr_G)})_{x_0}  \\
@A{\oblv_{\fL^+(G)_{x_0}\to K}}AA @AA{\oblv_{\fL^+(G)_{x_0}\to K}}A \\
\on{inv}_{\fL(G)_{x_0}}(\Dmod(\fL(G)_{x_0})^{\fL^+(G)_{x_0}}) @>{\Phi^+}>> \omega_{\Gr_G}\mod^\onfact((\Dmod(\fL(G)_{x_0})^{\fL^+(G)_{x_0}})^{\on{fact}_{x_0},\Dmod(\Gr_G)})_{x_0} \\
@A{\oblv_{\fL(G)_{x_0}\to \fL^+(G)_{x_0}}}AA @AA{\oblv_{\fL(G)_{x_0}\to \fL^+(G)_{x_0}}}A \\
\on{inv}_{\fL(G)_{x_0}}(\Dmod(\fL(G)_{x_0})^{\fL(G)_{x_0}}) @>{\Phi}>> \omega_{\Gr_G}\mod^\onfact((\Dmod(\fL(G)_{x_0})^{\fL(G)_{x_0}})^{\on{fact}_{x_0},\Dmod(\Gr_G)})_{x_0}.
\endCD
\end{equation}

It is enough to show that the Beck-Chevalley natural transformations 
\begin{equation} \label{e:BC again 1}
(\oblv_{\fL(G)_{x_0}\to \fL^+(G)_{x_0}})^L\circ \Phi^+_{\oblv} \to \Phi \circ (\oblv_{\fL(G)_{x_0}\to \fL^+(G)_{x_0}})^L
\end{equation}
and 
\begin{equation} \label{e:BC again 2}
(\oblv_{\fL^+(G)_{x_0}\to K})^L\circ \wt\Phi \to \Phi^+\circ (\oblv_{\fL^+(G)_{x_0}\to K})^L
\end{equation}
are both isomorphisms.

\sssec{}

The assertion for \eqref{e:BC again 1} follows from the fact that the 1-morphism
$$\oblv_{\fL(G)_{x_0}\to \fL^+(G)_{x_0}}: (\Dmod(\fL(G)_{x_0})^{\fL(G)_{x_0}}\to (\Dmod(\fL(G)_{x_0})^{\fL^+(G)_{x_0}}$$
admits a left adjoint already in $\fL(G)_{x_0}\mmod$, see \secref{sss:Av Gr}. 

\sssec{}

To prove the assertion for \eqref{e:BC again 2}, we expand the upper portion of \eqref{e:BC setup expanded}:

\medskip

\begin{equation} \label{e:BC setup expanded bis}
\CD
\on{inv}_{\fL(G)_{x_0}}(\Dmod(\fL(G)_{x_0})^K) @>{\wt\Phi}>>  \omega_{\Gr_G}\mod^\onfact(\Dmod(\fL(G)_{x_0})^K)^{\on{fact}_{x_0},\Dmod(\Gr_G)})_{x_0}  \\
@AAA @AAA \\
\on{inv}_{\fL(G)_{x_0}}(\Dmod(\fL(G)_{x_0})^{\on{alm-}\fL^+(G)_{x_0}}) 
@>{\Phi^{\on{alm-}+}}>> \omega_{\Gr_G}\mod^\onfact((\Dmod(\fL(G)_{x_0})^{\on{alm-}\fL^+(G)_{x_0}})^{\on{fact}_{x_0},\Dmod(\Gr_G)})_{x_0}  \\
@AAA @AAA  \\
 \on{inv}_{\fL(G)_{x_0}}(\Dmod(\fL(G)_{x_0})^{\fL^+(G)_{x_0}}) @>{\Phi^+}>> \omega_{\Gr_G}\mod^\onfact((\Dmod(\fL(G)_{x_0})^{\fL^+(G)_{x_0}})^{\on{fact}_{x_0},\Dmod(\Gr_G)})_{x_0},
 \endCD
\end{equation} 
where the symbol $\on{alm-}\fL^+(G)_{x_0}$ refers to \emph{almost invariants} with respect to $\fL^+(G)_{x_0}$.

\medskip 

It suffices to show that the Beck-Chevalley natural transformations corresponding to both subdiagrams are isomorphisms. 

\medskip

For the lower square this follows from the fact that
the 1-morphism
$$(\Dmod(\fL(G)_{x_0})^{\fL^+(G)_{x_0}}\to (\Dmod(\fL(G)_{x_0})^{\on{alm-}\fL^+(G)_{x_0}}$$
admits a left adjoint already in $\fL(G)_{x_0}\mmod$, see \secref{sss:Av !}.

\sssec{}

For the upper square in \eqref{e:BC setup expanded bis} we note:

\medskip

\begin{itemize}

\item The left vertical arrow is an equivalence;

\smallskip

\item The right vertical arrow is fully faithful. Indeed, the 1-morphism
$$(\Dmod(\fL(G)_{x_0})^{\on{alm-}\fL^+(G)_{x_0}}\to \Dmod(\fL(G)_{x_0})^K$$
in $\fL(G)_{x_0}\mmod$ is fully faithful and admits a \emph{right} adjoint. 

\end{itemize}

Now, the fact that the Beck-Chevalley natural transformation is an isomorphism follows from the next
general assertion:

\begin{lem} \label{l:BC}
Let 
$$
\CD
\bC_1 @>{\Phi}>> \bC_2 \\
@A{\iota_1}AA @A{\iota_2}AA \\
\bC'_1 @>{\Phi'}>> \bC'_2
\endCD
$$
be a commutative diagram, in which $\iota_1$ is an equivalence and $\iota_2$ is fully faithful.
Then the Beck-Chevalley natural transformation
$$(\iota_2)^L\circ \Phi\to \Phi'\circ (\iota_1)^L$$
is an isomorphism.
\end{lem}

\begin{proof}

It suffices to show that the natural transformation
$$(\iota_2)^L\circ \Phi\circ \iota_1\to \Phi'\circ (\iota_1)^L\circ \iota_1$$
is an isomorphism.

\medskip

We have a commutative diagram of functors
$$
\CD
(\iota_2)^L\circ \Phi \circ \iota_1 @>>> \Phi'\circ (\iota_1)^L\circ \iota_1 \\
@V{\sim}VV @VVV \\
(\iota_2)^L\circ \iota_2\circ \Phi' @>>> \Phi'.
\endCD
$$

We have to show that the top horizontal arrow is an isomorphism. 
However, the conditions of the lemma imply that all three other arrows 
are isomorphisms.

\end{proof} 

\ssec{Verification of Property (4)}

In this subsection we perform the key calculation involved in the proof of \thmref{t:main}.

\sssec{}

We rewrite the functor \eqref{e:complicated forgetful} as
\begin{multline} \label{e:complicated forgetful 1}
\omega_{\Gr_G}\mod^\onfact(\Dmod(\wt\Gr_{G,x_0})^{\on{fact}_{x_0},\Dmod(\Gr_G)})_{x_0} \overset{\oblv}\to 
\Dmod^\lax(\wt\Gr_{G,\Ran^\untl_{x_0}}) \overset{(\wt\pi^{\on{fact}_{x_0},\Gr_G})_!}\longrightarrow \\ 
\to \Dmod^\lax(\Gr_{G,\Ran^\untl_{x_0}}/\on{Hecke}_{x_0}) 
\overset{\iota_{x_0}^!}\longrightarrow \Vect,
\end{multline} 
where 
\[
    \iota_{x_0}: \on{pt} \to \Gr_{G,\Ran^\untl_{x_0}}/\on{Hecke}_{x_0}
\] 
is the base change of the map $x_0 \to \Ran_{x_0}^\untl$.

\medskip

Note that by \secref{sss ind in application}, the image of the natural transformation
$$(\iota_1^!\circ \oblv_{\omega_{\Gr_G}})^L\to \wt\Phi$$
evaluated on $k\in \Vect$ under the first arrow in \eqref{e:complicated forgetful 1} is the map 
\begin{multline} \label{e:decode BC} 
\jmath_!(\omega_{\Gr_{G,\Ran^\untl_\circ}}\boxt \delta_1)\to
\jmath_!(\omega_{\Gr_{G,\Ran^\untl_\circ}}\boxt \omega_{\wt\Gr_{G,x_0}})
\simeq \jmath_!\circ \jmath^!(\omega_{\wt\Gr_{G,\Ran^\untl_{x_0}}})\to \omega_{\wt\Gr_{G,\Ran^\untl_{x_0}}}.
\end{multline} 

\sssec{}

Denote by $\pi_{\Ran^\untl}/\on{Hecke}_{x_0}$ the projection
$$\Gr_{G,\Ran^\untl_{x_0}}/\on{Hecke}_{x_0}\to \Ran^\untl_{x_0}.$$

Note that this morphism is pseudo-proper (see \cite[Sect. 1.5.3]{Ga1}); hence, the functor 
$$(\pi_{\Ran^\untl}/\on{Hecke}_{x_0})_!:\Dmod^\lax(\Gr_{G,\Ran^\untl_{x_0}}/\on{Hecke}_{x_0})\to \Dmod^\lax(\Ran^\untl_{x_0}),$$
left adjoint to $(\pi_{\Ran^\untl}/\on{Hecke}_{x_0})^!$, is well-defined and satisfies base change. 

\medskip

Hence, we obtain that the functor
$$\Dmod^\lax(\wt\Gr_{G,\Ran^\untl_{x_0}}) \overset{(\wt\pi^{\on{fact}_{x_0},\Gr_G})_!}\longrightarrow 
\Dmod^\lax(\Gr_{G,\Ran^\untl_{x_0}}/\on{Hecke}_{x_0}) 
\overset{(\iota_{x_0})^!}\longrightarrow \Vect,$$
appearing in \eqref{e:complicated forgetful 1} can be rewritten as
\begin{equation} \label{e:complicated forgetful 2}
\Dmod^\lax(\wt\Gr_{G,\Ran^\untl_{x_0}})  \overset{(\wt\pi_{\Ran^\untl})_!}\longrightarrow \Dmod^\lax(\Ran^\untl_{x_0}) \overset{(-)_{x_0}}\longrightarrow \Vect,
\end{equation}
where:

\begin{itemize}

\item $\wt\pi_{\Ran^\untl}$ denotes the projection $\wt\Gr_{G,\Ran^\untl_{x_0}}\to \Ran_{x_0}^\untl$;

\item $(\wt\pi_{\Ran^\untl})_!$ denotes the partially defined left adjoint to $(\wt\pi_{\Ran^\untl})^!$.

\end{itemize}

\medskip

Thus, we obtain that we need to show that the morphism \eqref{e:decode BC} becomes an isomorphism after applying the functor
\eqref{e:complicated forgetful 2}.

\sssec{}

The question is local; hence we can assume that $(X,x_0)=(\BA^1,0)$. Consider the objects of $\Dmod(\Ran_{x_0})$ obtained 
by applying the functor 
$$\Dmod^\lax(\wt\Gr_{G,\Ran^\untl_{x_0}})  \overset{(\wt\pi_{\Ran^\untl})_!}\longrightarrow \Dmod^\lax(\Ran^\untl_{x_0}) \xrightarrow{ \imath^! }\Dmod(\Ran_{x_0})
$$
to the two sides of \eqref{e:decode BC}, where $\Ran_{x_0}$ is the \emph{non-unital} marked Ran space (which is a non-categorical prestack) and 
\[
  \imath: \Ran_{x_0} \to \Ran_{x_0}^\untl
\]
is the obvious map. Note that both these objects are equivariant with respect to the action of $\BG_m$
on $\Ran_{x_0}$ by dilations.

\sssec{}
Now we apply the contraction principle to the $\BG_m$-action on $\Ran_{x_0}$, which says the functors
\[
    x_0^!:\Dmod(\Ran_{x_0}) \to \Vect
\]
and 
\[
    \on{C}^\cdot_c(\Ran_{x_0},-): \Dmod(\Ran_{x_0}) \to \Vect
\]
are canonically equivalent when restricted to $\BG_m$-equivariant D-modules on $\Ran_{x_0}$\footnote{Proof: the map $x_0:\on{pt}\to \Ran_{x_0}$ is right inverse to $p:\Ran_{x_0} \to \on{pt}$. Hence $x_0^!\circ p^!\simeq \on{Id}$. We only need to show this natural isomorphism exhibits $x_0^!$ and the left adjoint of $p^!$, when restricted to $\BG_m$-equivariant objects. Note that we have 
\[
    \Ran_{x_0}\simeq \on{colim}_{I\in \on{Fin}} (\Ran_{x_0} \times_{\Ran} X^I)
\] 
where $\on{Fin}$ is the category of finite sets. Moreover, this isomorphism is compatible with the $\BG_m$-actions on both sides. Hence the desired claim follows from the contraction principle for schemes (\cite[Appendix A]{DG}).}. This implies the functors
\[
    \Dmod^\lax(\Ran_{x_0}^\untl) \overset{(-)_{x_0}}\longrightarrow \Vect
\]
and
\[
    \Dmod^\lax(\Ran_{x_0}^\untl) \xrightarrow{\imath^!} \Dmod(\Ran_{x_0}) \xrightarrow{\on{C}^\cdot_c(\Ran_{x_0},-)} \Vect,
\]
are canonically equivalent when restricted to $\BG_m$-equivariant D-modules on $\Ran_{x_0^\untl}$. Note that by \cite[Lemma C.5.12]{GLC2}, the latter functor is canonically equivalent to
\[
    \Dmod^\lax(\Ran_{x_0}^\untl)  \xrightarrow{\on{C}^\cdot_c(\Ran^\untl_{x_0},-)} \Vect.
\]

\sssec{}

\medskip

Thus, it remains to show that the morphism \eqref{e:decode BC} becomes an isomorphism after applying the functor
\begin{equation} \label{e:complicated forgetful 3}
\Dmod^\lax(\wt\Gr_{G,\Ran^\untl_{x_0}})  \overset{(\wt\pi_{\Ran^\untl})_!}\longrightarrow \Dmod^\lax(\Ran^\untl_{x_0}) \overset{\on{C}^\cdot_c(\Ran^\untl_{x_0},-)}\longrightarrow \Vect,
\end{equation}

\sssec{}

Let us again assume that $(X,x_0)$ is arbitrary. We obtain that it suffices to show that the morphism \eqref{e:decode BC}
induces an isomorphism after applying the functor
$$\on{C}^\cdot_c(\wt\Gr_{G,\Ran^\untl_{x_0}},-).$$

I.e., we need to show that the locally closed embedding
\begin{equation} \label{e:locally closed}
\Gr_{G,\Ran^\untl_\circ}\times \{1\} \to \Gr_{G,\Ran^\untl_\circ}\times \wt\Gr_{G,x_0} \to \wt\Gr_{G,\Ran^\untl_{x_0}}
\end{equation} 
induces an isomorphism on homology.

\medskip

As we will show in the next subsection, the latter assertion follows from the \emph{homological contractibility}
statement from \cite{Ga1}. 

\ssec{The contractibility statement}

\sssec{} \label{sss:contr 1}

Write
$$X=\ol{X}-\ul{x},$$
where $\ol{X}$ is a complete (and smooth) curve and $\ul{x}=\{x_1,...,x_n\}$ is a finite collection
of points on $\ol{X}$.

\medskip

Let $\Bun_G$ denote the moduli stack of $G$-bundles on $\ol{X}$. Let $\Bun^{\on{level}^1_{\ul{x}}}_G$ 
(resp., $\wt\Bun^{\on{level}^1_{\ul{x}}}_G$) be the moduli stack of $G$-bundles with structure of level
$1$ at $\{x_1,...,x_n\}$ (resp., additional structure of level $K$ at $x_0$). 

\medskip

We have a commutative diagram
\begin{equation} \label{e:NAPD}
\CD 
\Gr_{G,\Ran^\untl_\circ} @>{\text{\eqref{e:locally closed}}}>> \wt\Gr_{G,\Ran^\untl_{x_0}} \\
@VVV @VVV \\
\wt\Bun^{\on{level}^1_{\ul{x}}}_G @>{\on{id}}>> \wt\Bun^{\on{level}^1_{\ul{x}}}_G.
\endCD
\end{equation} 

We will show that both vertical maps in \eqref{e:NAPD} induce isomorphisms at the level of $\on{C}_\cdot(-)$. 
This will establish the corresponding fact for the morphism \eqref{e:locally closed}.

\medskip

We will show that both vertical maps in \eqref{e:NAPD} are in fact universal homological equivalences. 

\sssec{} \label{sss:contr 2}

The fact that the left vertical arrow in \eqref{e:NAPD} is a universal homological equivalence is
the statement of the contractibility theorem from \cite{Ga1}, applied to the complete curve $\ol{X}$
and $\{x_0,x_1,...,x_n\}$ as marked points. 

\sssec{} \label{sss:contr 3}

Note that the right vertical arrow in \eqref{e:NAPD} is the base change of the map
\begin{equation} \label{e:NAPD marked}
\Gr_{G,\Ran^\untl_{x_0}}\to \Bun^{\on{level}^1_{\ul{x}}}_G.
\end{equation}

Hence, it suffices to show that \eqref{e:NAPD marked} is a universal homological equivalence.

\sssec{} \label{sss:contr 4}

Consider the forgetful map
$\Ran^\untl_{x_0} \to \Ran^\untl$ and the map $\on{add}_{x_0}:\Ran^\untl\to \Ran^\untl_{x_0}$. 
They realize $\Ran_{x_0}$ as a retract of $\Ran$. Similarly, the map \eqref{e:NAPD marked}
is a retract of the map 
$$\Gr_{G,\Ran^\untl}\to \Bun^{\on{level}^1_{\ul{x}}}_G.$$

Hence, it suffices to show that the latter map is a universal homological equivalence.
However, this is again an instance of the contractibility theorem, applied to the complete curve $\ol{X}$
and $\{x_1,...,x_n\}$ as marked points. 

\qed

\section{Proof of \thmref{t:red to alm const case} for a torus} \label{s:torus}

For the next three sections we will be concerned with the proof of  \thmref{t:red to alm const case}.

\medskip

In this section we take $G=T$ to be a torus, and we will prove \thmref{t:red to alm const case}
using local geometric class field theory. 

\medskip

This special case will also be used in the proof of the general case of \thmref{t:red to alm const case}. 

\ssec{Reduction to the case of characters}

\sssec{}

Write 
$$\fL^+(T)_{x_0}\simeq T\times \on{ker}(\fL^+(\ft)\to \ft).$$

\medskip

Fourier-Laumon transforms identifies the monoidal category $\Dmod(\fL^+(T)_{x_0})$ with
$$\QCoh(\ct/\check\Lambda)\otimes \Dmod(\on{ker}(\fL^+(\ft)\to \ft)^*),$$
equipped with the \emph{pointwise} tensor product.

\sssec{}

Note that for $\bC\in \fL^+(T)_{x_0}\mmod$, the inclusion 
$$\on{alm-inv}_{\fL^+(T)_{x_0}}(\bC)\to \bC$$
is an equivalence if and only if for every \emph{non-zero} geometric point
$$\chi^+\in (\ct/\check\Lambda)\times (\on{ker}(\fL^+(\ft)\to \ft)^*),$$
the fiber $\bC_{\chi^+}$ is zero, see \cite[Lemma 21.4.6]{AGKRRV}.

\sssec{}

Let $\bC$ be an object of $\fL(T)_{x_0}\mmod$. Thus, we obtain that in order to prove 
\thmref{t:red to alm const case}, it suffices to show that for any such $\chi^+$, 
$$\omega_{\Gr_T}\mod^\onfact(\bC_{\chi^+}^{\on{fact}_{x_0,\Dmod(\Gr_G)}})_{x_0}=0.$$

Up to changing the ground field, we can assume that $\chi^+$ is a rational point. 

\sssec{}

Note that $k$-rational points of $(\ct/\check\Lambda)\times (\on{ker}(\fL^+(\ft)\to \ft)^*)$ can be thought 
of character sheaves on $\fL^+(T)_{x_0}$.

\medskip

For $\bC\in \fL(T)_{x_0}\mmod$ and a $k$-rational point $\chi^+$ as above, we have 
$$\bC_{\chi^+}\simeq \Dmod(\fL(T)_{x_0})^{(\fL^+(T)_{x_0},\chi^+)}\otimes \bC_0,$$
where $\bC_0\in \DGCat$ and the action of $\fL(T)_{x_0}$ is via the first factor.

\medskip

Thus, we need to show that for a non-trivial $\chi^+$ and any $\bC_0$, the category 
$$\omega_{\Gr_T}\mod^\onfact((\Dmod(\fL(T)_{x_0})^{(\fL^+(T)_{x_0},\chi^+)})^{\on{fact}_{x_0},\Dmod(\Gr_T)}\otimes \bC_0)_{x_0}$$
is zero.

\sssec{}

We will prove: 

\begin{thm} \label{t:main torus zero}
For a non-trivial $\chi^+$, 
the category $$\omega_{\Gr_T}\mod^\onfact((\Dmod(\fL(T)_{x_0})^{(\fL^+(T)_{x_0},\chi^+)})^{\on{fact}_{x_0},\Dmod(\Gr_T)})_{x_0}$$
is zero.
\end{thm} 

The proof in the presence of $\bC_0$ is the same. The rest of this section is devoted to the proof of \thmref{t:main torus zero}. 

\ssec{Character sheaves on \texorpdfstring{$\fL(T)_{x_0}$}{LT}}

\sssec{}

Let $\chi$ be an extension of $\chi^+$ to a character sheaf on all of $\fL(T)_{x_0}$. (Note that the space of such extensions 
is a torsor over 
$$\Hom(\Lambda,B(k^\times))\simeq B(\cT(k)),$$
where $\cT$ is the Langlands-dual torus.)

\medskip

The choice of $\chi$ equips the category $\Dmod(\fL(T)_{x_0})^{(\fL^+(T)_{x_0},\chi^+)}$ with an action of $\Lambda$, so that
$$(\Dmod(\fL(T)_{x_0})^{(\fL^+(T)_{x_0},\chi^+)})^\Lambda\simeq \Vect.$$

Let us denote the above copy of $\Vect$, viewed as an object of $\fL(T)_{x_0}\mmod$, by $\Vect_\chi$. 

\sssec{}

Using the equivalence\footnote{Note that the category $\Rep(\Lambda)\mmod$, which appears in the formula below, identifies with $\QCoh(\cT)$, 
viewed as a monoidal category under \emph{convolution}.}
$$\Lambda\mmod \simeq \Rep(\Lambda)\mmod$$
of \cite{Ga3}, we can recover $\Dmod(\fL(T)_{x_0})^{(\fL^+(T)_{x_0},\chi^+)}$ as 
$$\on{Funct}_{\Rep(\Lambda)}(\Vect,\Vect_\chi).$$

\medskip

Hence, we can reformulate \thmref{t:main torus zero} as:

\begin{thm} \label{t:main torus zero char}
The category $$\omega_{\Gr_T}\mod^\onfact(\Vect_\chi^{\on{fact}_{x_0},\Dmod(\Gr_T)})_{x_0}$$
is zero.
\end{thm} 

\sssec{}

The rest of this section is devoted to the proof of \thmref{t:main torus zero char}. The proof will amount 
to a simple computation, once we input the assertion of geometric class field theory (gCFT), reviewed
in the next subsection.

\ssec{Geometric class field theory}

\sssec{}

Consider the map
$$X\to \Ran_{x_0}, \quad x \mapsto \{x,x_0\}.$$

Let $\Gr^{\on{level}^\infty_{x_0}}_{T,X}$ denote the pullback of $\Gr^{\on{level}^\infty_{x_0}}_{G,\Ran_{x_0}}$ along this map.
Denote by $j$ and $i$ the open and closed embeddings
$$\Gr^{\on{level}^\infty_{x_0}}_{T,X-x_0}\hookrightarrow \Gr^{\on{level}^\infty_{x_0}}_{T,X} \hookleftarrow \Gr^{\on{level}^\infty_{x_0}}_{T,x_0},$$
respectively.

\medskip

Note that we have the identifications
$$\Gr^{\on{level}^\infty_{x_0}}_{T,X-x_0}\simeq \Gr_{T,X-x_0}\times \fL(T)_{x_0} \text{ and } \Gr^{\on{level}^\infty_{x_0}}_{T,x_0}\simeq \fL(T)_{x_0}.$$

\sssec{}

For $\lambda\in \Lambda$ consider the connected component
$$(\Gr_{T,X-x_0})^\lambda\times \fL^+(T)_{x_0}=
(\Gr_{T,X-x_0})^\lambda\times (\fL(T)_{x_0})^0\subset (\Gr_{T,X-x_0})^\lambda\times \fL(T)_{x_0}.$$ 

Its closure in $\Gr^{\on{level}^\infty_{x_0}}_{T,X}$, denoted $\overline{(\Gr_{T,X-x_0})^\lambda\times \fL^+(T)_{x_0}}$, is a  $\fL^+(T)_{x_0}$-torsor over $X$,
and its special fiber identifies with
$$(\fL(T)_{x_0})^\lambda\subset \fL(T)_{x_0}.$$

\sssec{}

We will need the following statement from gCFT:

\begin{thm} \label{t:gCFT}
Up to replacing $X$ by an open subset containing $x_0$, there exists a $\cT$-local system $\sigma$ on $X-x$ such that
for every $\lambda\in \Lambda$, the local system 
$$\lambda(\sigma)\boxt \chi^+$$
on 
$$X\times \fL^+(T)_{x_0}\simeq (\Gr_{T,X-x_0})^\lambda\times \fL^+(T)_{x_0}$$
extends to a local system on $\overline{(\Gr_{T,X-x_0})^\lambda\times \fL^+(T)_{x_0}}$. 

\medskip

Moreover, $\sigma$ is unique, up to tensoring with $\cT$-local systems unramified near $x$. 
\end{thm} 

\sssec{}

Note that a choice of $\sigma$ as in the theorem determines an extension $\chi^+\rightsquigarrow \chi$. Indeed, 
the restriction of $\chi$ to $(\fL(T)_{x_0})^\lambda$ equals the restriction of the extended local system to
$$\overline{(\Gr_{T,X-x_0})^\lambda\times \fL^+(T)_{x_0}} \hookleftarrow (\fL(T)_{x_0})^\lambda.$$

\sssec{}

Note that the assignment 
$$\lambda\rightsquigarrow \lambda(\sigma)$$
extends to a local system on $\Gr_{T,\Ran_\circ}$, to be denoted $\Lambda(\sigma)$, which is equipped with a natural factorization
structure. 

\medskip

It follows formally from \thmref{t:gCFT} that the local system
$$\Lambda(\sigma)\boxt \chi$$
on 
$$\Gr_{T,\Ran_\circ}\times \fL(T)_{x_0}\simeq (\Ran_\circ\times \{x_0\})\underset{\Ran_{x_0}}\times \Gr^{\on{level}^\infty_{x_0}}_{G,\Ran_{x_0}}$$
extends (uniquely) to a local system, to be denoted $\chi_{\Ran_{x_0}}$, on
$\Gr^{\on{level}^\infty_{x_0}}_{G,\Ran_{x_0}}$.

\medskip

Moreover, $\chi_{\Ran_{x_0}}$ has a natural factorization structure with respect to $\Lambda(\sigma)$. 

\ssec{The module \texorpdfstring{$\Vect_\chi^{\on{fact}_{x_0},\Dmod(\Gr_T)}$}{Vect chi}}

Recall the object
\begin{equation} \label{e:Vect chi}
\Vect_\chi^{\on{fact}_{x_0},\Dmod(\Gr_T)}\in \Dmod(\Gr_T)\mmod^\onfact_{x_0}.
\end{equation} 

In this subsection we will describe it as the factorization restriction 
of the tautological (i.e., vacuum) object $\Vect^{\on{fact}_{x_0}}\in \Vect\mmod^\onfact_{x_0}$. 

\sssec{}

Denote by $\pi_!^\sigma$ the factorization functor $\Dmod(\Gr_T)\to \Vect$, given by the precomposition of $\pi_!$ (see \secref{sss:pi}) with 
the operation of tensoring by the inverse of $\Lambda(\sigma)$. 

\medskip

We claim:

\begin{prop} \label{p:Vect chi}
The object \eqref{e:Vect chi} identifies canonically with $\Rres_{\pi_!^\sigma}(\Vect\mmod^{\on{fact}_{x_0}})$.
\end{prop}

\begin{proof}

Note that the operation of tensoring by $\Lambda(\sigma)$ is a factorization automorphism of $\Dmod(\Gr_T)$.

\medskip

Moreover, tensoring by $\chi_{\Ran_{x_0}}$ defines an \emph{isomorphism} 
$$(\Dmod(\fL(T)_{x_0})^{\on{fact}_{x_0},\Dmod(\Gr_T)})^{\fL(T)_{x_0}}\to 
(\Dmod(\fL(T)_{x_0})^{\on{fact}_{x_0},\Dmod(\Gr_T)})^{(\fL(T)_{x_0},\chi)},$$
i.e., 
$$\Vect^{\on{fact}_{x_0},\Dmod(\Gr_T)}\to \Vect_\chi^{\on{fact}_{x_0},\Dmod(\Gr_T)},$$
compatible with the above automorphism of $\Dmod(\Gr_T)$.

\medskip

Now, the assertion of the proposition follows from that of \propref{p:Vect over Gr}.

\end{proof} 

\sssec{}

Denote 
$$R_{\cT,\sigma}:=\pi_!^\sigma(\omega_{\Gr_T})\in \on{FactAlg}(\Vect).$$

\medskip

Thus, by \eqref{e:equiv fact mod}, we can reformulate \thmref{t:main torus zero} as follows:

\begin{thm} \label{t:main torus zero ult} 
The category $R_{\cT,\sigma}\mod^\onfact_{x_0}$ is zero.
\end{thm} 

\begin{rem}
The assertion of \thmref{t:main torus zero ult} follows easily from that of \cite[Equation (4.10) and/or Theorem 3.8]{Bogd}.
Below will give an alternative (in a sense, more elementary) proof.
\end{rem}

\ssec{Proof of \thmref{t:main torus zero ult} via chiral algebras}

\sssec{}

Let $R^{\on{ch}}_{\cT,\sigma}$ be the chiral algebra (on $X-x_0$) corresponding to  $R_{\cT,\sigma}$,
so that 
$$R_{\cT,\sigma}\mod^\onfact_{x_0}\simeq R^{\on{ch}}_{\cT,\sigma}\mod^{\on{ch}}_{x_0},$$
see \cite[Sect. D.1]{GLC2}.

\medskip

Thus, we need to show that $R^{\on{ch}}_{\cT,\sigma}\mod^{\on{ch}}_{x_0}=0$. 

\sssec{} \label{sss:chiral bracket}

Note that the D-module on $X-x_0$ underlying $R^{\on{ch}}_{\cT,\sigma}$ identifies canonically with
$$\underset{\lambda}\oplus\, R^{\on{ch},\lambda}_{\cT,\sigma}, \quad R^{\on{ch},\lambda}_{\cT,\sigma}=\lambda(\sigma)\otimes \omega_{X-x_0}[-1].$$

The chiral operation is given by 
\begin{multline*} 
(j_{x_1\neq x_2})_*((\lambda(\sigma)\otimes \omega_{X-x_0}[-1])\boxt (\mu(\sigma)\otimes \omega_{X-x_0}[-1]))\to \\
\to \Delta_*((\lambda(\sigma)\otimes \omega_{X-x_0}[-1])\sotimes (\mu(\sigma)\otimes \omega_{X-x_0}[-1]))[1]\simeq
\Delta_*((\lambda+\mu)(\sigma)\otimes \omega_{X-x_0}[-1]).
\end{multline*} 

\sssec{}

Let $\CM$ be an object of $R^{\on{ch}}_{\cT,\sigma}\mod^{\on{ch}}_{x_0}$; let $M$ denote the underlying the vector space,
so that the D-module underlying $\CM$ is $i_*(M)$. 

\medskip

The chiral action is given by
$$\on{act}:(j_{x\neq x_0})_*(R^{\on{ch}}_{\cT,\sigma})\otimes M\to i_*(M).$$

\medskip

The axiom of chiral action implies that the (signed) sum of the following three morphisms
\begin{equation} \label{e:chiral action}
(j_{x_1\neq x_2,x_1\neq x_0,x_2\neq x_0})_*(R^{\on{ch}}_{\cT,\sigma} \boxt R^{\on{ch}}_{\cT,\sigma})\otimes M\to 
\Delta_*\circ i_*(M)
\end{equation} 
(as D-modules on $X^2$) is zero:

\begin{itemize}

\item 
$(j_{x_1\neq x_2,x_1\neq x_0,x_2\neq x_0})_*(R^{\on{ch}}_{\cT,\sigma} \boxt R^{\on{ch}}_{\cT,\sigma})\otimes M
\overset{\{-,-\}\otimes \on{Id}}\to \Delta_*\circ (j_{x\neq x_0})_*(R^{\on{ch}}_{\cT,\sigma})\otimes M \overset{\on{act}}\to  \Delta_*\circ i_*(M)$; 

\item 
\begin{multline*} 
(j_{x_1\neq x_2,x_1\neq x_0,x_2\neq x_0})_*(R^{\on{ch}}_{\cT,\sigma} \boxt R^{\on{ch}}_{\cT,\sigma})\otimes M
\overset{\on{Id}\otimes \on{act}}\longrightarrow \\
\to (\on{id}\times i)_* \circ (j_{x\neq x_0})_*(R^{\on{ch}}_{\cT,\sigma})\otimes M \overset{\on{act}}\to 
(\on{id}\times i)_* \circ i_*(M)=(i\times i)_*(M)=\Delta_*\circ  i_*(M);
\end{multline*}

\item The map, obtained from the previous one by interchanging the roles of $x_1$ and $x_2$.

\end{itemize}

\sssec{}

Since $\sigma$ is non-trivial at $x$, we can choose $\lambda\in \Lambda$, so that the 1-dimensional local system 
$\lambda(\sigma)$ is non-trivially ramified at $x$. 

\medskip

Let us restrict the three maps in \eqref{e:chiral action} to 
\begin{equation} \label{e:lambda summand}
(j_{x_1\neq x_2,x_1\neq x_0,x_2\neq x_0})_*(R^{\on{ch},\lambda}_{\cT,\sigma} \boxt R^{\on{ch},-\lambda}_{\cT,\sigma})\otimes M.
\end{equation} 

Note that the space of maps
$$(j_{x\neq x_0})_*(\lambda(\sigma)\otimes \omega_{X-x_0})\to i_*(M)$$
is zero, and similarly for $-\lambda$. 

\medskip

Hence, the restriction of the 2nd and 3rd maps to \eqref{e:lambda summand} are zero. Hence, so is the restriction of the 1st map. 

\sssec{}

By definition, the first map identifies with
\begin{multline} \label{e:action on M}
(j_{x_1\neq x_2,x_1\neq x_0,x_2\neq x_0})_*\left((\lambda(\sigma)\otimes \omega_{X-x_0}[-1])\boxt (-\lambda(\sigma)\otimes \omega_{X-x_0}[-1])\right)\otimes M\to \\
\to \Delta_*\circ (j_{x\neq x_0})_*(\omega_{X-x_0}[-1])\otimes M\overset{\on{act}}\to \Delta_*\circ i_*(M).
\end{multline} 

Note however, that the above copy of 
$$\omega_{X-x_0}[-1]=R^{\on{ch},0}_{\cT,\sigma}$$ 
is the chiral unit in $R^{\on{ch}}_{\cT,\sigma}$, and since $\CM$ is a unital chiral module, the last arrow in \eqref{e:action on M}
comes from the canonical map
$$(j_{x\neq x_0})_*(\omega_{X-x_0}[-1]) \to i_*(k).$$

\medskip

Since \eqref{e:action on M} is zero, this means that the first arrow in \eqref{e:action on M} factors via a map
$$(j_{x_1\neq x_2,x_1\neq x_0,x_2\neq x_0})_*\left((\lambda(\sigma)\otimes \omega_{X-x_0}[-1])\boxt (-\lambda(\sigma)\otimes \omega_{X-x_0}[-1])\right)\otimes M\to 
\Delta_*(\omega_X[-1])\otimes M.$$

\medskip

If $M\neq 0$, this would mean that the canonical map
$$(j_{x_1\neq x_2,x_1\neq x_0,x_2\neq x_0})_*\left((\lambda(\sigma)\otimes \omega_{X-x_0}[-1])\boxt (-\lambda(\sigma)\otimes \omega_{X-x_0}[-1])\right) \to
\Delta_*\circ (j_{x\neq x_0})_*(\omega_{X-x_0}[-1]))$$
factors as
\begin{multline*}
(j_{x_1\neq x_2,x_1\neq x_0,x_2\neq x_0})_*\left((\lambda(\sigma)\otimes \omega_{X-x_0}[-1])\boxt (-\lambda(\sigma)\otimes \omega_{X-x_0}[-1])\right) \dashrightarrow \\
\dashrightarrow \Delta_*(\omega_X[-1])\to \Delta_*\circ (j_{x\neq x_0})_*(\omega_{X-x_0}[-1])),
\end{multline*}
which is false.

\qed[\thmref{t:main torus zero ult} ]

\section{Towards the proof of \thmref{t:red to alm const case}: the key geometric lemma} \label{s:proof of main: key}

In this subsection we supply a key geometric input for the proof of \thmref{t:red to alm const case}, 
incarnated by \lemref{l:main}. 

\ssec{Extending the loop group action}

\sssec{} \label{sss:factorizable LG}

Consider the (corr-unital) factorization group ind-scheme $\fL(G)$ and its (co-unital) factorization 
group subscheme $\fL^+(G)$, so that their fibers at $x_0$ are $\fL(G)_{x_0}$ and $\fL^+(G)_{x_0}$,
respectively and 
$$\Gr_G\simeq \fL(G)/\fL^+(G).$$

We consider also the corresponding tautological (i.e., vacuum) factorization module spaces at $x_0$, denoted
$\fL(G)^{\on{fact}_{x_0}}$ and $\fL^+(G)^{\on{fact}_{x_0}}$ over $\fL(G)$ and $\fL^+(G)$, respectively.

\sssec{}

Unwinding the definitions, the $!$-pullback along the multiplication map of $\fL(G)$ defines a \emph{lax-unital} factorization functor
\[
    \Dmod(\fL(G)) \to \Dmod(\fL(G))\otimes \Dmod(\fL(G)),
\]
while the similar functor for $\fL^+(G)$ is a strictly unital factorization functor (see \secref{exam-LG-L+G-GrG} and \secref{ss Mor factalgcat}). In other words, we obtain (associative) coalgebra objects
\begin{equation}
  \label{eqn-DLG-as-coalg}
    \Dmod(\fL(G)) \in \mathbf{coAlg}( \mathbf{UntlFactCat}^{\on{lax-untl}}  )
\end{equation}
while
\begin{equation}
  \label{eqn-DL+G-as-coalg}
   \Dmod(\fL^+(G)) \in \mathbf{coAlg}( \mathbf{UntlFactCat}  ).
\end{equation}
Note that we have a homomorphism between these coalgebra objects
\[
    i^!:\Dmod(\fL(G))\to \Dmod(\fL^+(G))
\]
given by $!$-pullback along the factorization map $i:\fL^+(G) \to \fL(G)$.

\sssec{} \label{sss:factorizable LG acts}

We define a unital factorization action of $\fL(G)$ on a unital factorization category $\bA$ to be a comodule structure of $\bA$ with respect to the coalgebra \eqref{eqn-DLG-as-coalg}, such that the composition
\[
    \bA \to \Dmod(\fL(G))\otimes \bA \xrightarrow{i^!\otimes \on{Id}_\bA} \Dmod(\fL^+(G))\otimes \bA
\]
is \emph{strictly} unital. 

\medskip

If $\fL(G)$ acts on a unital factorization category $\bA$, and $\bD$ is a unital factorization $\bA$-module
category at $x_0$, we can similarly define the notion of unital factorization actions of $\fL(G)^{\on{fact}_{x_0}}$ on $\bD$ that are compatible with the given $\fL(G)$-action on $\bA$.

\medskip

Note that in this case, the fiber $\bD_{x_0}$ at $x_0$ acquires an action of $\fL(G)_{x_0}$ as a DG category. 

\sssec{Example} \label{sss:level w action}

An example of a unital factorization category equipped with an action of $\fL(G)$ is $\bA:=\Dmod(\Gr_G)$.

\medskip

For this choice of $\bA$, an example of a unital factorization module category $\bD$ equipped with a compatible $\fL(G)^{\on{fact}_{x_0}}$-action
is $\Dmod(\fL(G)_{x_0})^{\on{fact}_{x_0},\Dmod(\Gr_G)}$. 

\medskip

Note that the resulting $\fL(G)_{x_0}$-action on 
$$(\Dmod(\fL(G)_{x_0})^{\on{fact}_{x_0},\Dmod(\Gr_G)})_{x_0}\simeq \fL(G)_{x_0}$$
is given by \emph{left} translations. 

\sssec{} \label{sss:fact sph}

For $\bA$ as above, denote 
$$\bA^0:=\bA^{\fL^+(G)},$$
which is defined as the cosimplicial limit of
\[
    \bA \rightrightarrows \Dmod(\fL^+(G)) \otimes \bA \cdots
\]
in $\mathbf{UntlFactCat}$. Note that the forgetful functor
$$\iota:\bA^0\to \bA$$
is a \emph{strictly} unital factorization functor.

\medskip

For $\bD$ as in \secref{sss:factorizable LG acts}, consider
$$\bD^0:=\Rres_\iota(\bD)\in \bA^0\mmod^\onfact_{x_0}.$$

We claim:

\begin{propconstr} \label{p:action of sph}
Under the above circumstances, we have a natural action of $\fL(G)_{x_0}$ on $\bD^0$
\emph{as an object of} $\bA^0\mmod^\onfact_{x_0}$, 
so that the action on
$$\bD^0_{x_0}\simeq \bD_{x_0}$$
is the action from \secref{sss:factorizable LG acts}.
\end{propconstr}

\begin{proof} 

Let 
\begin{equation} \label{e:mer to reg}
\fL^{\on{mer}\rightsquigarrow \on{reg}}(G)_{\Ran_{x_0}}\subset \fL(G)_{\Ran_{x_0}}
\end{equation}
be the group ind-scheme over $\Ran_{x_0}$ defined in \cite[Sect. C.10.10]{GLC2}.

\medskip

By construction, $\fL^{\on{mer}\rightsquigarrow \on{reg}}(G)_{\Ran_{x_0}}$ comes equipped
with a projection to the constant group ind-scheme with fiber $\fL(G)_{x_0}$. 

\medskip

By the construction of the operation of factorization restriction (see \cite[Secrt. B.9.28]{GLC2}), the action of $\fL(G)_{\Ran_{x_0}}$
on the crystal of categories underlying $\bD$ restricts to an action of 
$\fL^{\on{mer}\rightsquigarrow \on{reg}}(G)_{\Ran_{x_0}}$ on $\Rres_\iota(\bD)$.
Moreover, this action factors via 
$$\fL^{\on{mer}\rightsquigarrow \on{reg}}(G)_{\Ran_{x_0}}\to \fL(G)_{x_0}\times \Ran_{x_0}.$$

\end{proof}

\begin{rem} 
Informally, \propref{p:action of sph} reads as follows: the action of $\fL(G)_{x_0}$ on $\bD_{x_0}$ 
commutes with the factorization module structure with respect to $\bA^0$. 
\end{rem} 

\begin{cor} \label{c:action of sph}
For a factorization algebra $\CA\in \bA^0$, the 
category
$$\iota(\CA)\mod^\onfact(\bD)_{x_0}\overset{\text{\cite[Lemma B.12.12]{GLC2}}}\simeq \CA\mod^\onfact(\bD^0)_{x_0}$$
carries an action of $\fL(G)_{x_0}$ compatible with the forgetful functor
$$\CA\mod^\onfact(\bD^0)_{x_0}\to \bD_{x_0}.$$
\end{cor} 

\sssec{Example}

Let us consider what may be a familiar situation in which \propref{p:action of sph} is applicable.

\medskip

Consider the factorization category $\KM(\fg,\kappa)$ of Kac-Moody representations
(at a given level). It carries an action of  $\fL(G)_{\Ran}$ (twisted by the level), defined
as in \cite[Sect. B.14.22]{GLC2}.

\medskip

The corresponding
category $(\KM(\fg,\kappa))^{\fL^+(G)}$ is by the definition the factorization version of the
Kazhdan-Lusztig category, denoted $\KL(G,\kappa)$, see  \cite[Sect. B.14.28]{GLC2}.

\medskip

Consider 
$$\KM(\fg,\kappa)^{\on{fact}_{x_0}}\in \KM(\fg,\kappa)\mmod^\onfact_{x_0}.$$

\medskip

Then \propref{p:action of sph} says that the action of $\fL(G)_{x_0}$ on $\KM(\fg,\kappa)_{x_0}$
commutes with fusion against objects of $\KL(G,\kappa)$.

\medskip

In particular, \corref{c:action of sph} says that for a factorization algebra
$\CA\in \KL(G,\kappa)_{\Ran}$, the category
$$\CA\mod^\onfact(\KM(\fg,\kappa))_{x_0},$$
carries a natural action of $\fL(G)_{x_0}$.

\sssec{}

Let $\bC$ be an object of $\fL(G)_{x_0}\mmod$. By the construction of the functor \eqref{e:functor in one direction},
the example in \secref{sss:level w action} shows that the resulting object
$$\bC^{\on{fact}_{x_0},\Dmod(\Gr_G)}\in \Dmod(\Gr_G)\mmod^\onfact_{x_0}$$
carries a compatible action of $\fL(G)^{\on{fact}_{x_0}}$.

\medskip

Applying \corref{c:action of sph} to 
$$\omega_{\Gr_G} \in \on{Alg}^\onfact(\Dmod(\Gr_G)),$$
viewed as a $\fL(G)$- (and hence $\fL^+(G)$)-equivariant object,
we obtain that the category
$$\omega_{\Gr_G}\mod^\onfact(\bC^{\on{fact}_{x_0},\Dmod(\Gr_G)})_{x_0}$$
carries an action of $\fL(G)_{x_0}$, which commutes with the forgetful functor
$$\omega_{\Gr_G}\mod^\onfact(\bC^{\on{fact}_{x_0},\Dmod(\Gr_G)})_{x_0}\to
(\bC^{\on{fact}_{x_0},\Dmod(\Gr_G)})_{x_0}\simeq \bC.$$

\ssec{The key geometric lemma}

\sssec{} \label{sss:i and j not}

We continue to be in the context of \secref{sss:factorizable LG acts}. Consider the restrictions 
$$\bA_X:=\bA|_X \text{ and } \bD_{X\times x_0}:=\bD|_{X\times x_0},$$
along $X\to \Ran$ and $X\times x_0\to \Ran_{x_0}$, respctively. 

\medskip

Denote by
$$(X-x_0) \overset{j}\hookrightarrow X \overset{i}\hookleftarrow \{x_0\}$$
the corresponding morphisms, and by
$$\bA_{X-x_0}\otimes \bD_{x_0} \overset{j_*}\hookrightarrow \bD_{X\times x_0} \overset{i_*}\hookleftarrow \bD_{x_0}$$
the corresponding functors.

\sssec{}

For a triple of objects $\ba\in \bA_{X-x_0}$ and $\bd,\bd'\in \bD_{x_0}$, consider the space
$$\CHom_{\bD_{X\times x_0}}(j_*(\ba\boxt \bd),i_*(\bd')).$$

For future reference, we note that if $\ba\in \bA^0_{X-x_0}$, we have a canonical identification
\begin{equation} \label{e:restr doesnt matter}
\CHom_{\bD_{X\times x_0}}(j_*(\iota(\ba)\boxt \bd),i_*(\bd'))\simeq
\CHom_{\bD^0_{X\times x_0}}(j_*(\ba\boxt \bd),i_*(\bd')),
\end{equation}
see \secref{sss:fact sph} for the notation. This follows from the 
construction of the operation of factorization restriction (see \cite[Secrt. B.9.28]{GLC2}). 

\sssec{}

More generally, for an affine test-scheme $S$ and an $S$-point $g$ of $\fL(G)_{x_0}$, consider the spaces
\begin{equation} \label{e:mor before action}
\CHom_{\bD_{X\times x_0}\otimes \Dmod^!(S)}\left(\on{pr}^!(j_*(\ba\boxt \bd)),\on{pr}^!(i_*(\bd'))\right)
\end{equation} 
and 
\begin{equation} \label{e:mor after action}
\CHom_{\bD_{X\times x_0}\otimes \Dmod^!(S)}\left(\on{pr}^!(j_*(\ba\boxt \bd)), \on{act}^!(i_*(\bd'))\right),
\end{equation} 
where $\on{pr}^!$ and $\on{act}^!$ are the two functors
$$\bD_{X\times x_0}\to \bD_{X\times x_0}\otimes \Dmod^!(S).$$

\sssec{} \label{sss:act g} 

In what follows, we will abuse the notation slightly and omit $S$. So we will simply write  
$j_*(\ba\boxt \bd)$ instead of $\on{pr}^!(j_*(\ba\boxt \bd))$, and we will write
$$g\cdot i_*(\bd'):=\on{act}^!(i_*(\bd')).$$

Note also that 
$$g\cdot i_*(\bd')\simeq i_*(g\cdot \bd'),$$
where in the right-hand side $g\cdot -$ refers to the $\fL(G)_{x_0}$-action on $\bD_{x_0}$.

\medskip

Thus, instead of \eqref{e:mor before action} we will write 
\begin{equation} \label{e:mor before action simple}
\CHom_{\bD_{X\times x_0}}(j_*(\ba\boxt \bd),i_*(\bd')),
\end{equation}
and instead of \eqref{e:mor after action} we will write
\begin{equation} \label{e:mor after action simple}
\CHom_{\bD_{X\times x_0}}(j_*(\ba\boxt \bd),i_*(g\cdot \bd')).
\end{equation} 

\sssec{}

Let us come back for a moment to the statement of \propref{p:action of sph}. It implies that for
$\ba,\bd,\bd'$ and $g$ as above, we have a canonical isomorphism 
\begin{equation} \label{e:action of sph}
\CHom_{\bD_{X\times x_0}}(j_*(\ba\boxt \bd),i_*(\bd')) \simeq \CHom_{\bD_{X\times x_0}}(j_*(\ba\boxt g\cdot \bd),i_*(g\cdot \bd')).
\end{equation} 

\medskip

Hence, by \propref{p:action of sph}, the expression in \eqref{e:mor after action simple} is canonically
isomorphic to 
\begin{equation} \label{e:mor after action simple bis}
\CHom_{\bD_{X\times x_0}}(j_*(\ba\boxt (g^{-1}\cdot \bd)),i_*(\bd')). 
\end{equation} 

\sssec{}

The key assertion behind the proof of \thmref{t:red to alm const case} is the following:

\begin{mainlem} \label{l:main}
Suppose that the object $\ba$ belongs to $(\bA_{X-x_0})^{\fL(G)_{X-x_0}}$. 

\medskip

\noindent{\em(a)} If $g$ is a point of\footnote{In the formula below, $N$ is the maximal unipotent subgroup of $G$.}
 $\fL(N)_{x_0}$, there exists a canonical isomorphism
between the spaces \eqref{e:mor before action simple} and \eqref{e:mor after action simple}.

\medskip

\noindent{\em(b)} If $\bd\in (\bD_{x_0})^{\fL^+(G)_{x_0}}$, then for any $g\in \fL(G)_{x_0}$, 
there exists a natural isomorphism, to be denoted $\alpha_{g,\bd,\bd'}$, 
between the spaces \eqref{e:mor before action simple} and \eqref{e:mor after action simple}, 
such that for
$g_1\in \fL^+(G)_{x_0}$ the diagram
$$
\CD
\CHom_{\bD_{X\times x_0}}(j_*(\ba\boxt \bd),i_*(\bd'))  @>{\alpha_{g_1\cdot g,\bd,\bd'}}>> \CHom_{\bD_{X\times x_0}}(j_*(\ba\boxt \bd),i_*((g_1\cdot g)\cdot \bd')) \\
& &  @V{=}VV \\
@V{\bd \text{ is spherical}}V{\sim}V \CHom_{\bD_{X\times x_0}}(j_*(\ba\boxt \bd),i_*(g_1\cdot (g\cdot \bd'))) \\
& &  @V{\sim}V{\text{\eqref{e:action of sph}}}V \\
\CHom_{\bD_{X\times x_0}}(j_*(\ba\boxt g_1^{-1}\cdot \bd),i_*(\bd')) @>{\alpha_{g,g_1^{-1}\cdot \bd,\bd'}}>> 
\CHom_{\bD_{X\times x_0}}(j_*(\ba\boxt g_1^{-1}\cdot \bd),i_*(g\cdot \bd'))
\endCD
$$
commutes.

\end{mainlem}

In what follows, adopting the conventions of Remark \secref{sss:act g}, we will write the sought-for isomorphism in
Main \lemref{l:main} as
\begin{equation} \label{e:main iso}
\CHom_{\bD_{X\times x_0}}(j_*(\ba\boxt \bd),i_*(\bd')) \overset{\alpha_{g,\bd,\bd'}}\simeq \CHom_{\bD_{X\times x_0}}(j_*(\ba\boxt \bd),i_*(g\cdot \bd')).
\end{equation} 
\begin{rem}
Let us emphasize the difference between the assertion of Main \lemref{l:main} and that of \propref{p:action of sph}:

\medskip

In \eqref{e:action of sph}, both $\bd$ and $\bd'$ are moved by $g$ (while $\ba$ is only required to be $\fL^+(G)$-equivariant).

\medskip

By contrast, in Main \lemref{l:main} only $\bd$ (or $\bd'$) is moved by $g$, but $\ba$ is required to be $\fL(G)$-equivariant.

\end{rem}

%
%

\begin{rem}

One can show that when $G$ is semi-simple, one can get rid of the condition that  
$\bd$ be $\fL^+(G)_{x_0}$-equivariant. However, when $G$ has a non-trivial connected center, the
$\fL^+(G)_{x_0}$-equivariance condition on $\bd$ is necessary. 

\end{rem} 

\sssec{}

Here is how Main Lemma \ref{l:main} will be used:

\begin{cor} \label{c:main cor} Let $\ba$, $\bd$ and $\bd'$ be as in \lemref{l:main}.

\medskip

\noindent{\em(a)} We have an isomorphism
$$\CHom_{\bD_{X\times x_0}}\Bigl(j_*(\ba\boxt \bd),i_*(\bd')\Bigr) \simeq
\CHom_{\bD_{X\times x_0}}\left(j_*(\ba\boxt \bd),i_*(\on{Av}^{\fL^+(N)_{x_0}}_*(\bd'))\right).$$

\medskip

\noindent{\em(b)} For $G$ reductive, under the additional assumptions of \lemref{l:main}(b), we have
\begin{multline*} 
\CHom_{\bD_{X\times x_0}}\Bigl(j_*\left(\ba\boxt \bd\right)\otimes \on{C}_\cdot(\Gr_{G,x_{x_0}}),i_*(\bd')\Bigr) \simeq \\
\simeq \CHom_{\bD_{X\times x_0}}\left(j_*\Bigl(\ba\boxt \on{Av}^{\fL^+(G)_{x_0}\to \fL(G)_{x_0}}_!(\bd)\Bigr),i_*(\bd')\right),
\end{multline*} 
where $\on{Av}^{\fL^+(G)_{x_0}\to \fL(G)_{x_0}}_!$ is as in \secref{sss:Av Gr}. 

\end{cor}

\begin{proof}

We prove point (a), as point (b) is similar. By adjunction,
\begin{multline} \label{e:Hom and Av}
\CHom_{\bD_{X\times x_0}}\left(j_*(\ba\boxt \bd),i_*(\on{Av}^{\fL^+(N)_{x_0}}_*(\bd'))\right)\simeq \\
\simeq \CHom_{\bD_{X\times x_0}\otimes \Dmod^!(\fL^+(N))}\left(j_*(\ba\boxt \bd),i_*(g_{\on{univ}}\cdot \bd')\right),
\end{multline}
where $g_{\on{univ}}$ is the tautological $\fL^+(N)$-point of $\fL^+(N)$. 

\medskip

Now, Main \lemref{l:main}(a) implies that the left-hand side in \eqref{e:Hom and Av} identifies with
\begin{equation} \label{e:Hom and Av bis}
\CHom_{\bD_{X\times x_0}\otimes \Dmod^!(\fL^+(N))}\left(j_*(\ba\boxt \bd),i_*(\bd')\right).
\end{equation}

Both objects 
$$j_*(\ba\boxt \bd) \text{ and } i_*(\bd')\in \bD_{X\times x_0}\otimes \Dmod^!(\fL^+(N))$$
are the same-named objects in $\bD_{X\times x_0}$, tensored with $\omega_{\fL^+(N)}\in \Dmod^!(\fL^+(N))$.
Since the latter object is compact, the expression in \eqref{e:Hom and Av bis} is isomorphic to
$$\CHom_{\bD_{X\times x_0}}\left(j_*(\ba\boxt \bd),i_*(\bd')\right)\otimes \End_{\Dmod^!(\fL^+(N))}(\omega_{\fL^+(N)})\simeq
\CHom_{\bD_{X\times x_0}}\left(j_*(\ba\boxt \bd),i_*(\bd')\right) \otimes \on{C}^\cdot(\fL^+(N)).$$

Now, since $N$ is unipotent, $\fL^+(N))$ is contractible, and hence $\on{C}^\cdot(\fL^+(N))\simeq k$. 

\end{proof}

\ssec{Proof of Main Lemma \ref{l:main}: reduction to a particular congruence level} \hfill \label{ss:red to congr}

\smallskip 

In this subsection we will show that we can assume that both $\bd$ and $\bd'$ are
invariant with respect to some congruence subgroup $K_n\subset \fL^+(G)_{x_0}$.

\sssec{}

First, as the category $\bD_{x_0}$ is acted on by $\fL(G)_{x_0}\supset \fL^+(G)_{x_0}$, every object can be written as a colimit 
of objects invariant with respect to congruence subgroups. 

\medskip

Thus, we can assume that $\bd$ is invariant with respect to some $K_n$, $n\geq 1$. 

\sssec{} \label{sss:Av up}

We now use the fact that $\ba$ is $\fL^+(G)_{X-x_0}$-equivariant, and we apply isomorphism
\eqref{e:restr doesnt matter}: 

\medskip 

By \propref{p:action of sph}, the category $\bD^0_{X\times x_0}$ is acted on by $\fL(G)_{x_0}$, and in particular
by $K_n$. By assumption, with respect to this action, the object 
$$j_*(\ba\boxt \bd)\in \bD^0_{X\times x_0}$$
is $K_n$-invariant. 

\medskip

Hence, the left-hand side in \eqref{e:main iso} remains unchanged if we replace $\bd'$ 
by $\on{Av}_*^{K_n}(\bd')$.

\medskip

Similarly, the right-hand side in \eqref{e:main iso} remains unchanged if we replace 
$g\cdot \bd'$ by $\on{Av}_*^{K_n}(g\cdot \bd')$.

\sssec{}

Now, for a given point $g$ of $\fL(G)_{x_0}$, let $m\gg 1$ be large enough so that
$$\on{Ad}_g(K_m)\subset K_n.$$

In this case, 
$$\on{Av}_*^{K_n}(g\cdot \bd') \simeq \on{Av}_*^{K_n}(g\cdot \on{Av}_*^{K_m}(\bd')).$$ 

Therefore, the right-hand side in \eqref{e:main iso} remains unchanged if we replace 
$\bd'$ by $\on{Av}_*^{K_m}(\bd')$. 

\sssec{}

Thus, we can assume that $\bd'$ is also invariant with respect to some congruence subgroup. 

\ssec{Proof of Main Lemma: reduction of point (a) to point (b)} \label{ss:b=>a}

\sssec{}

Given an $S$-point $g$ of $\fL(N)_{x_0}$, we need to establish isomorphism \eqref{e:main iso}.

\medskip

By the previous subsection, we can assume that both $\bd$ and $\bd'$ are invariant with respect to
the subgroup $K_n(N):=K_n\cap \fL(N)_{x_0}$ for some $n\geq 1$.

\sssec{}

Let $\lambda$ be a dominant coweight so that 
$$\on{Ad}_{t^\lambda}(\fL^+(N)_{x_0})\subset K_n(N).$$

I.e., 
$$\fL^+(N)_{x_0} \subset \on{Ad}_{t^{-\lambda}}(K_n(N)).$$

\medskip

Using the $\fL(G)_{x_0}$-action on $\bD^0_{X\times x_0}$, we can identify the two sides of \eqref{e:main iso}
with ones, where we replace $\bd$ by $t^{-\lambda}\cdot \bd$ and $\bd'$ by $t^{-\lambda}\cdot \bd'$,
and $g$ by $\on{Ad}_{t^{-\lambda}}(g)$. 

\medskip

However, now $t^{-\lambda}\cdot \bd$ and $t^{-\lambda}\cdot \bd'$ are $\on{Ad}_{t^{-\lambda}}(K_n(N))$-equivariant,
and hence $\fL^+(N)_{x_0}$-equivariant, by construction.

\sssec{}

This reduces the assertion of point (a) of \lemref{l:main} to point (b) (for $G$ replaced by $N$).

\ssec{Proof of Main Lemma \ref{l:main}(b): the mechanism} \label{ss:idea}

In this subsection we will explain the main geometric idea behind the proof of Main Lemma \ref{l:main}(b).

\medskip

We will describe a paradigm in which one obtains an isomorphism \eqref{e:main iso}. 

\sssec{}

Assume that $\bd$ is invariant with respect to a subgroup $K\subset \fL(G)_{x_0}$. Note that using the
averaging functor as in \secref{sss:Av up}, we can assume that $\bd'$ is also $K$-invariant\footnote{In 
practice, for the proof of Main Lemma \ref{l:main}(b), we will take $K$ to be $\fL(G)_{x_0}$ itself.}. 

\medskip

Let $g$ be an element of $\fL(G)_{x_0}$, i.e., a map $\cD^\times_{x_0}\to G$. Let $g'$ be a map 
\begin{equation} \label{e:rat map}
g':X\times X-\Delta\to G
\end{equation} 
with the following properties:

\begin{itemize} 

\item For every $x\neq x_0$, the restriction of map $$g'_x:=g'|_{(X-x)\times x}:X-x\to G$$ along
$\cD_{x_0}\hookrightarrow X-x$ lies in $K$; 

\medskip

\item The restriction of the map $$g'_{x_0}:=g'|_{(X-x_0)\times x_0}:X-x_0\to G$$ along $\cD^\times_{x_0}\hookrightarrow X-x_0$ equals 
$g$ modulo $K$. 

\end{itemize} 

We claim that a choice of a map $g'$ as above gives rise to an isomorphism \eqref{e:main iso}.

\sssec{} \label{sss:act by global 1}

For any map as in \eqref{e:rat map}, its Laurent expansion in the first coordinate around the divisor
$$\Delta \cup (x_0\times X)$$
defines a section of 
$$\fL(G)_{X\times x_0}:=\fL(G)_{\Ran_{x_0}}|_{X\times x_0}$$ over $X\simeq X\times x_0$, and hence acts by a self-equivalence on $\bD_{X\times x_0}$. 

\medskip

We claim that the action of the above element $g'$ is such that
\begin{equation} \label{e:action away from point}
g'\cdot (j_*(\ba\boxt \bd))\simeq j_*(\ba\boxt \bd)
\end{equation}
and 
\begin{equation} \label{e:action at point}
g'\cdot (i_*(\bd'))\simeq i_*(g\cdot \bd).
\end{equation}

This would give rise to an isomorphism \eqref{e:main iso}. 

\sssec{}  \label{sss:act by global 2}

To prove \eqref{e:action away from point}, we need to establish the corresponding isomorphism 
over $X-x_0$. 

\medskip

Note that a map \eqref{e:rat map} acts on an object
$$\ba\boxt \bd\in \bA_{X-x_0}\otimes \bD_{x_0}$$
via:

\begin{itemize}

\item The Laurent expansion of $g'$ around the diagonal divisor and the action of
the resulting element of $\fL(G)_{X-x_0}$ on $\ba$;

\medskip

\item The Laurent expansion of $g'$ around $x_0\times X$ and the resulting action of
$\fL(G)_{x_0}$ on $\bd$. 

\end{itemize} 

Now, the condition that $\ba$ is $\fL(G)_{X-x_0}$-equivariant implies that
$$g'\cdot \ba\simeq \ba.$$

And the condition that the $g'|_{\cD^\times_{x_0}\times x}\in K$ combined with 
the assumption that $\bd$ in $K$-invariant implies that $g'\cdot \bd\simeq \bd$. 

\sssec{}  \label{sss:act by global 3}

To prove \eqref{e:action at point}, we need to show that the action of 
$$g'|_{\cD^\times_{x_0}\times x_0}\in \fL(G)$$
on $\bd'$ produces the same result as $g\cdot \bd'$. 

\medskip

However, this follows from the assumption that
$$g'|_{\cD^\times_{x_0}\times x_0}=g\,\on{mod}\,K$$
and the fact that $\bd'$ is $K$-invariant. 

\sssec{} \label{sss:two isos coincide}

Assume now that $g'$ is in fact a map
$$X\times X\to G,$$
i.e., that it is regular around the diagonal.

\medskip

Unwinding the construction, we obtain that the resulting isomorphism \eqref{e:main iso} equals
\begin{multline*}
\CHom_{\bD_{X\times x_0}}(j_*(\ba\boxt \bd),i_*(\bd')) \overset{\text{\eqref{e:action of sph}}}\simeq  \\
\CHom_{\bD_{X\times x_0}}(j_*(\ba\boxt (g\cdot \bd)),i_*(g\cdot \bd')) \overset{\bd\text{ is spherical}}\simeq
\simeq \CHom_{\bD_{X\times x_0}}(j_*(\ba\boxt \bd),i_*(g\cdot \bd')).
\end{multline*}

\sssec{}

Note also that the isomorphism \eqref{e:main iso} is naturally compatible with the group structure on 
$\Maps(X\times X-\Delta,G)$, i.e., for a pair of points $g'_1,g'_2\in \Maps(X\times X-\Delta,G)$, the diagram
$$
\CD
\CHom_{\bD_{X\times x_0}}(j_*(\ba\boxt \bd),i_*(\bd'))  @>{\alpha_{g_2,\bd,\bd'}}>> 
\CHom_{\bD_{X\times x_0}}(j_*(\ba\boxt \bd),i_*(g_2\cdot \bd'))  \\
@V{\alpha_{g_1\cdot g_2,\bd,\bd'}}VV @VV{\alpha_{g_1,g_2\cdot \bd,\bd'}}V \\
\CHom_{\bD_{X\times x_0}}(j_*(\ba\boxt \bd),i_*((g_1\cdot g_2)\cdot \bd')) @>{=}>>
\CHom_{\bD_{X\times x_0}}(j_*(\ba\boxt \bd),i_*(g_1\cdot (g_2\cdot \bd'))) 
\endCD
$$
commutes.

\ssec{Proof of Main Lemma \ref{l:main}(b): approximating the loop group} \label{ss:approx loop}

The assertion of the Main Lemma \ref{l:main}(b) is \'etale-local, so in what follows we will 
assume that the pair $(X,x_0)$ is $(\BA^1,0)$. We will implement the mechanism from \secref{ss:idea}
or $K=\fL^+(G)_{x_0}$. 

\sssec{} 

We will construct a group ind-scheme (ind-affine, of ind-finite type)
$\Gamma$, equipped with a map
$$\phi:\Gamma\to \Maps(X\times X-\Delta,G),$$
such that the composite map
\begin{equation} \label{e:Gamma eval}
\Gamma\overset{\phi}\to \Maps(X\times X-\Delta,G)\to \fL(G)_{x_0} 
\end{equation} 
satisfies:

\begin{itemize}

\item For $\Gamma^+:=\Gamma\underset{\fL(G)_{x_0}}\times \fL^+(G)_{x_0}$
the resulting map 
\begin{equation} \label{e:Gamma eval Gr}
\Gamma/\Gamma^+\to \fL(G)_{x_0}/\fL^+(G)_{x_0}=\Gr_{G,x_0}
\end{equation} 
is an isomorphism.

\end{itemize}

Given such a map, the construction in \secref{ss:idea} will imply the assertion of Main \lemref{l:main}(b). 

\sssec{}

We take
$$\Gamma:=G(t,t^{-1}):=\bMaps(\BA^1-0,G).$$

We let $\phi$ be defined by precomposition with the projection 
$$X\times X-\Delta=\BA^1\times \BA^1-\Delta \overset{t_1,t_2\mapsto t_1-t_2}\longrightarrow \BA^1-0.$$

\sssec{}

Note that the map \eqref{e:Gamma eval} is the standard map corresponding to $k[t,t^{-1}]\hookrightarrow k\ppart$.
Hence, 
$$\Gamma^+=\bMaps(\BA^1,G)=G[t].$$

\medskip

Hence, the map \eqref{e:Gamma eval Gr} is the map
$$G(t,t^{-1})/G[t]\to G\ppart/G\qqart,$$ 
which is known to be an isomorphism\footnote{This is a consequence of the Beauville-Laszlo theorem,
combined with the fact that every $G$-bundle on $\BP^1$ can be trivialized on $\BA^1$, \'etale-locally with
respect to the
scheme of parameters.}.

\section{Proof of \thmref{t:red to alm const case}} \label{s:proof of main-end}

\ssec{What to we need to show?}

\sssec{}

Recall that by \corref{c:action of sph}, the category 
\begin{equation} \label{e:category of interest}
\omega_{\Gr_G}\mod^\onfact(\bC^{\on{fact}_{x_0},\Dmod(\Gr_G)})_{x_0}
\end{equation} 
carries an action of $\fL(G)_{x_0}$.

\medskip

Since the forgetful functor
$$\omega_{\Gr_G}\mod^\onfact(\bC^{\on{fact}_{x_0},\Dmod(\Gr_G)})_{x_0}\to \bC$$
is conservative and compatible with the $\fL(G)_{x_0}$-actions, by \lemref{l:alm triv cons}, it suffices to show that the functor
\begin{equation} \label{e:alm eq omega mod}
\on{alm-inv}_{\fL(G)_{x_0}}\left(\omega_{\Gr_G}\mod^\onfact(\bC^{\on{fact}_{x_0},\Dmod(\Gr_G)})_{x_0}\right)\to
\omega_{\Gr_G}\mod^\onfact(\bC^{\on{fact}_{x_0},\Dmod(\Gr_G)})_{x_0}
\end{equation}
is an equivalence. 

\sssec{} \label{sss:proof of alm strategy}

The strategy of the proof will be as follows:

\medskip

\begin{enumerate}

\item We will show that the category \eqref{e:category of interest} is \emph{spherically generated}
(see \secref{sss:sph gen} for what this means); 

\medskip

\item We will show that the functor $\on{Av}^{\fL^+(G)_{x_0}\to \fL(G)_{x_0}}_!$ on the spherical subcategory
of \eqref{e:category of interest} (see \secref{sss:Av Gr}) is conservative.
By \propref{p:when al triv}, this will imply that \eqref{e:alm eq omega mod} is an equality.

\end{enumerate}

%
%
%
%
%
%
%
%

\ssec{Functoriality revisted}

In this subsection we briefly revisit the set-up of \secref{ss:functoriality}. 

\sssec{}

Note that we have a canonical map of factorization algebras in $\Dmod(\Gr_G)$
$$\omega_\phi:(\Gr_\phi)_*(\omega_{\Gr_{G'}})\to \omega_{\Gr_G}.$$

Hence, we obtain a functor
\begin{multline} \label{e:res functor grps omega}
\omega_{\Gr_G}\mod^\onfact(\bC^{\on{fact}_{x_0,\Dmod(\Gr_G)}})_{x_0}\overset{\Res_{\omega_\phi}}\longrightarrow
(\Gr_\phi)_*(\omega_{\Gr_{G'}})\mod^\onfact(\bC^{\on{fact}_{x_0},\Dmod(\Gr_G)})_{x_0}\simeq \\
\simeq \omega_{\Gr_{G'}}\mod^\onfact(\Rres_{\Gr_\phi}(\bC^{\on{fact}_{x_0},\Dmod(\Gr_G)}))_{x_0}\simeq 
 \omega_{\Gr_{G'}}\mod^\onfact((\bC')^{\on{fact}_{x_0},\Dmod(\Gr_{G'})})_{x_0}.
\end{multline} 

Note that the source (resp., target) of the above functor carries a natural action of $\fL(G)_{x_0}$ (resp., $\fL(G')_{x_0}$),
see \corref{c:action of sph}. 

\sssec{}

Unwinding the constructions, we obtain:

\begin{lem} \label{l:compat actions}
The functor \eqref{e:res functor grps omega} intertwines the $\fL(G')_{x_0}$-action
on the target with the $\fL(G')_{x_0}$-action on the source obtained from the $\fL(G)_{x_0}$-action by 
precomposing with $\phi$.
\end{lem}

\ssec{Proof of spherical generation}

\sssec{}

Replacing $\bC$ by the kernel of the right adjoint to \eqref{e:Sph gen}, we can assume that
$\bC^{\fL^+(G)_{x_0}}=0$.

\medskip

We will show that the category \eqref{e:category of interest} is zero in this case. 

\sssec{}

By lemmas \ref{l:res functor grps} and \ref{l:compat actions}, and using the fact that \thmref{t:main}
has already been proved for the group $T$, we know that the action of $\fL(T)_{x_0}$
(and in particular of $\fL^+(T)_{x_0}$) has been canonically trivialized. 

\medskip

In particular, the inclusion
$$\on{alm-inv}_{\fL^+(T)_{x_0}}\left(\omega_{\Gr_G}\mod^\onfact(\bC^{\on{fact}_{x_0},\Dmod(\Gr_G)})_{x_0}\right)
\hookrightarrow \omega_{\Gr_G}\mod^\onfact(\bC^{\on{fact}_{x_0},\Dmod(\Gr_G)})_{x_0}$$
is an equality.

\sssec{}

We claim that the inclusion 
$$\on{alm-inv}_{\fL^+(N)_{x_0}}\left(\omega_{\Gr_G}\mod^\onfact(\bC^{\on{fact}_{x_0},\Dmod(\Gr_G)})_{x_0}\right)
\hookrightarrow \omega_{\Gr_G}\mod^\onfact(\bC^{\on{fact}_{x_0},\Dmod(\Gr_G)})_{x_0}$$
is also an equality.

\medskip

To prove this, it suffices to show that the functor $\on{Av}^{\fL^+(N)_{x_0}}_*$ is conservative on \eqref{e:category of interest}. Let 
$$\bc\in \omega_{\Gr_G}\mod^\onfact(\bC^{\on{fact}_{x_0},\Dmod(\Gr_G)})_{x_0}$$ 
be an object in the kernel of $\on{Av}^{\fL^+(N)_{x_0}}_*$. Then by \corref{c:main cor}(a), we obtain that 
\begin{equation} \label{e:contr 1}
\CHom_{\bD_{X\times x_0}}\Bigl(j_*(\omega_{\Gr_G}\boxt \bc_{x_0}),i_*(\bc_{x_0})\Bigr)=0,
\end{equation} 
where:

\begin{itemize}

\item $\bD:=\bC^{\on{fact}_{x_0},\Dmod(\Gr_G)}$;

\item $\bc_{x_0}=\oblv_{\omega_{\Gr_G}}(\bc)$. 

\end{itemize} 

\sssec{} \label{sss:units}

We claim that \eqref{e:contr 1} implies that $\bc_{x_0}=0$. Indeed, the structure of factorization $\omega_{\Gr_G}$-module on $\bc$ 
gives rise to a map
$$j_*(\omega_{\Gr_G}\boxt \bc_{x_0}) \to i_*(\bc_{x_0})[1],$$
and we claim that if this map is zero, then $\bc_{x_0}$ is zero.

\medskip

Indeed, this follows from the fact that $\bc$ was a \emph{unital} factorization module, and hence
the composition
$$j_*(\one_{\Dmod(\Gr_G)} \boxt \bc_{x_0})\to j_*(\omega_{\Gr_G}\boxt \bc_{x_0}) \to i_*(\bc_{x_0})[1]$$
is the canonical morphism. 

\sssec{}

We now claim that the inclusion 
\begin{multline*} 
\on{inv}_{K_1}\left(\omega_{\Gr_G}\mod^\onfact(\bC^{\on{fact}_{x_0},\Dmod(\Gr_G)})_{x_0}\right)\overset{\sim}\to 
\on{alm-inv}_{K_1}\left(\omega_{\Gr_G}\mod^\onfact(\bC^{\on{fact}_{x_0},\Dmod(\Gr_G)})_{x_0}\right) \hookrightarrow \\
\hookrightarrow \omega_{\Gr_G}\mod^\onfact(\bC^{\on{fact}_{x_0},\Dmod(\Gr_G)})_{x_0}
\end{multline*} 
is also an equality.

\medskip

Indeed, this follows from the triangular decomposition
$$K_1=(K_1\cap \fL^+(N)_{x_0})\cdot (K_1\cap \fL^+(T)_{x_0})\cdot (K_1\cap \fL^+(N^-)_{x_0}),$$
and the fact that the equality holds for each of the factors.

\sssec{}

Thus, the action of $\fL^+(G)_{x_0}$ on 
$$\bC':=\omega_{\Gr_G}\mod^\onfact(\bC^{\on{fact}_{x_0},\Dmod(\Gr_G)})_{x_0}$$
factors through 
$$\fL^+(G)_{x_0}\twoheadrightarrow G,$$
while we have:

\begin{itemize}

\item $(\bC')^N=\bC'$

\item The action of $T$ on $\bC'$ is trivialized;

\item $(\bC')^G=0$.

\end{itemize}

\medskip

In \secref{ss:toric localization} below we will show that any $\bC'$ with the above properties is zero. 

\begin{rem}

Note that if we weakened the second assumption to $\on{alm-inv}_T(\bC')=\bC'$, the assertion that
$\bC'=0$ would be false: a counterexample is provided by \secref{sss:bad example}. 

\end{rem}

\begin{rem}

As we shall see, the first condition (i.e., that $(\bC')^N=\bC'$) is actually superfluous. So, in fact we are proving the following: 

\end{rem}

\begin{prop}
Let us be a given a categorical representation of $G$, such that its restriction to $T$ can be trivialized. Then the initial
categorical representation is almost trivial.
\end{prop}

\ssec{The toric localization argument} \label{ss:toric localization}

To simplify the notation, for the duration of this subsection we will perform the notational change,
$$\bC'\rightsquigarrow \bC.$$

\sssec{}

Let $\bD$ be a category equipped with an action of $T$. Then the category 
$\bD^T$ is acted on by $\Vect^T$. 

\medskip

In particular, the (commutative algebra)
$$\on{C}^\cdot(BT)\simeq \Sym(\ft^*[-2])$$
maps to the Bernstein center of $\bD^T$. 

\sssec{}

Let $\bD^T_0\subset \bD$ be the \emph{non-cocomplete} subcategory consisting of objects, on which
some non-zero graded ideal in $\Sym(\ft^*[-2])$ acts trivially. 

\medskip

Denote by $\wt\bD$ the quotient 
$$\bD^T/\bD^T_0,$$
\emph{taken in the world of non-cocomplete categories}.

\sssec{}

We consider $\bC^N$ and $\bC$ itself as acted on by $T$ (note that by \cite[Remark 1.2]{BZGO}, if $\bC\neq 0$, then $\bC^N\neq 0$). 
Let $\CF$ be an object of $\Dmod(T\backslash G/B)$. Convolution with $\CF$ can be thought of as a functor
$$\bC^B=(\bC^N)^T\to \bC^T.$$

\medskip

The following assertion is an abstract version of the toric localization principle:

\begin{lem} \label{l:loc princ}
For $\bc\in \bC^B$, the image of $\CF\star \bc\in \bC^T$ under
$$\bC^T\to \wt\bC$$
is canonically isomorphic to the image of 
$$\underset{w\in W}\oplus\, (w\cdot \bc)\otimes \CF_w,$$
where:

\begin{itemize}

\item $W$ denotes the Weyl group;

\item For $w\in W$ we denote by the same symbol the corresponding $T$-fixed point in $G/B$;

\item $w\cdot \bc:=\delta_w\star \bc$, where $\delta_w$ is viewed as an object of $\Dmod(G/B)^T=\Dmod(T\backslash G/B)$. 

\item $\CF_w$ is the !-fiber of $\CF$ at $w\in G/B$. 

\end{itemize}

\end{lem} 

The proof is given in \secref{ss:proof T loc} below. 

\begin{cor} \label{c:loc princ}
The image of $\bc$ under $\bC^T\to \wt\bC$ is a retract of the image of 
$$\ul{k}_{G/B}\star \bc.$$
\end{cor}

\sssec{} 

Note now that the assumption that $\bC^G=0$ implies that the functor 
$$\ul{k}_{G/B}\star (-): \bC^B\to \bC^T$$
vanishes.

\medskip

Hence, from \corref{c:loc princ} we obtain that this assumption forces that the inclusion
$$(\bC^N)^T_0\subset (\bC^N)^T=\bC^B$$
is an equality.

\sssec{}

We now use the assumption that the $T$-action on $\bC$ is trivialized. Hence so is the $T$-action on $\bC^N$. 
This assumption implies that
$$(\bC^N)^T\simeq \bC^N\otimes \Vect^T,$$
where the action of $\Vect^T$ is via the second factor.

\medskip

Now, for $0\neq \bc\in \bC^N$, the object
$$\bc\otimes k\in \bC^N\otimes \Vect^T$$
does \emph{not} belong to $(\bC^N)^T_0$. 

\qed

\ssec{Proof of \lemref{l:loc princ}} \label{ss:proof T loc}

\sssec{}

It is enough to show that the image of $\CF\in \Dmod(T\backslash G/B)\simeq \Dmod(G/B)^T$ under
$$\Dmod(G/B)^T\to \wt{\Dmod(G/B)}$$
is isomorphic to the image of 
$$\underset{w}\oplus\, (\delta_w \otimes \CF_w).$$

We have a canonical map
$$\underset{w}\oplus\, (\delta_w \otimes \CF_w)\to \CF,$$
whose cone has the property that its !-fibers at the points $w$ are zero. 

\medskip

Hence, it is enough to show that if $\CF$ has vanishing !-fibers at all the points $w$, then
its projection to $\wt{\Dmod(G/B)}$ vanishes. 

\sssec{}

Using Cousin decomposition, we can assume that $\CF$ is the *-extension from an object on
a single Schuber cell $(G/B)_w\in G/B$. Furthermore, by assumption, it is the *-extension from
the open subscheme $(G/B)_w-w$. Hence, it is enough to show that the inclusion 
$$\Dmod((G/B)_w-w)^T_0\subseteq \Dmod((G/B)_w-w)^T$$
is an equality.

\medskip

We claim that this is the case for any scheme $Y$ on which a torus $T$ acts without fixed points.

\sssec{}

First, in order to show that the inclusion 
$$\Dmod(Y)^T_0\subseteq \Dmod(Y)^T$$
is an equality, we can replace the action of the original $T$ by the action of any 
$\BG_m$ that maps to $T$. 

\medskip

The fact that $T$ has no fixed points on $Y$ implies that we can find $\BG_m\to T$
that acts on $Y$ with finite stabilizers. We will show that for this copy of $\BG_m$, the action of 
$$\on{C}^\cdot(B\BG_m)\simeq k[\eta], \quad\deg(\eta)=2$$
on $\Dmod(Y)^{\BG_m}$ factors though a non-zero ideal.

\sssec{}

The action
of $\on{C}^\cdot(B\BG_m)$ on 
$$\Dmod(Y)^{\BG_m}\simeq \Dmod(Y/\BG_m)$$
factors via a homomorphism
\begin{equation} \label{e:DM}
\on{C}^\cdot(B\BG_m)\to \on{C}^\cdot(Y/\BG_m).
\end{equation} 

Hence, it is enough to show that this homomorphism factors though a non-zero ideal.

\sssec{}

The assumption on the $\BG_m$-action on $Y$ implies that $Y/\BG_m$ is a Deligne-Mumford stack. 
Hence, $\on{C}^\cdot(Y/\BG_m)$ is finite-dimensional.

\medskip

Hence, the homomorphism \eqref{e:DM} has a non-trivial kernel. 

\qed

\ssec{Proof of the conservativity of \texorpdfstring{$\on{Av}^{\fL^+(G)_{x_0}\to \fL(G)_{x_0}}_!$}{Av!}}

\sssec{}

Let $$\bc\in \left(\omega_{\Gr_G}\mod^\onfact(\bC^{\on{fact}_{x_0},\Dmod(\Gr_G)})_{x_0}\right)^{\fL^+(G)_{x_0}}$$ 
be an object in the kernel of $\on{Av}^{\fL^+(G)_{x_0}\to \fL(G)_{x_0}}_!$. 

\medskip

Then by \corref{c:main cor}(b), we obtain that 
\begin{equation} \label{e:contr 2}
\CHom_{\bD_{X\times x_0}}\Bigl(j_*(\omega_{\Gr_G}\boxt \bc_{x_0}),i_*(\bc_{x_0})\Bigr)=0.
\end{equation} 

\sssec{} This implies that $\bc_{x_0}=0$ by the same argument as in \secref{sss:units}.

\qed[\thmref{t:red to alm const case}]

\section{An application: integrable Kac-Moody representations} \label{s:KM}

In this section we discuss an application of \thmref{t:initial} (rather, of its incarnation as 
\thmref{t:main}). Namely, we establish an equivalence between integrable Kac-Moody
representations and representations of the integrable quotient of the Kac-Moody
chiral algebra. 

\ssec{Integrable Kac-Moody representations and the integrable quotient}

\sssec{}

We define integrable Kac-Moody representations, following the framework of \cite[Sect. 7.4]{Ro}. Namely,
we let $\CL_{\Gr_G}$ be a factorization line bundle on $\Gr_G$. 

\medskip

Consider the factorization functor
\begin{equation} \label{e:L-twisted sects}
\Dmod(\Gr_G)\to \Vect, \quad \CF\mapsto \Gamma^{\IndCoh}(\Gr_G,\CF\otimes \CL_{\Gr_G}).
\end{equation}

\medskip

We define the \emph{integrable Kac-Moody factorization algebra} in $\Vect$ to be
$$\BV^{\on{Int}}_{G,\kappa}:=\Gamma^{\IndCoh}(\Gr_G,\omega_{\Gr_G}\otimes \CL_{\Gr_G}).$$

\begin{rem} \label{r:int exact}

Recall that to the datum of $\CL_{\Gr_G}$ there corresponds a discrete invariant, called a level,
and denoted $\kappa$, which is a $W$-invariant $\BZ$-valued quadratic form on the coweight lattice 
$\Lambda$ of $G$.

\medskip

It is shown in \cite[Theorem 7.4.3]{Ro} that if the restriction of $\kappa$ to one of the simple factors of
$G$ is negative-definite, then $\BV^{\on{Int}}_{G,\kappa}=0$.

\medskip

If the restriction of $\kappa$ to all simple factors of
$G$ is non-negative definite, then 
$$\BV^{\on{Int,ch}}_{G,\kappa}:=\BV^{\on{Int}}_{G,\kappa}|_X[-1]\in \Dmod(X)$$
lives in single cohomological degree $0$, so can be regarded as a classical
chiral algebra. 

\medskip

In the latter case, if $G$ is semi-simple and simply-connected, $\BV^{\on{Int,ch}}_{G,\kappa}$ is what is
usually called \emph{the integrable quotient} of the Kac-Moody chiral algebra at level $\kappa$, denoted 
$$\BV^{\on{ch}}_{\fg,\kappa}:=\BV_{\fg,\kappa}|_X[-1]$$ 
(per our conventions, we reserve the symbol $\BV_{\fg,\kappa}$ for the corresponding factorization algebra). 

\medskip

In the opposite case, namely, when $G=T$ is a torus, $\BV^{\on{Int,ch}}_{G,\kappa}$ is the lattice chiral algebra. 

\end{rem} 

\sssec{}

Recall that the pullback of
$\CL_{\Gr_{G,\Ran}}$ along the unit section
$$\on{unit}_{\Gr_G}:\Ran\to \Gr_G$$
is canonically trivialized (as a factorization line bundle over $\on{pt}$). 

\medskip

We quote the following result (see \cite[Corollary C]{Zhao}):

\begin{thm} \label{t:mult line bundle}
The pullback $\CL_{\fL(G)}$  of $\CL_{\Gr_G}$ along the projection
$$\fL(G)\to \Gr_G$$
carries a uniquely defined multiplicative structure, compatible with the
trivialization of its restriction to $\fL^+(G)$, given by the trivialization of $\on{unit}_{\Gr_G}^*(\CL_{\Gr_G})$. 
\end{thm} 

\sssec{}

Consider the resulting central extension of $\fL(G)_{x_0}$; denote it by
$$1\to \BG_m\to \wh{\fL(G)}_{\kappa,x_0}\to \fL(G)_{x_0}\to 1.$$

The datum of such a central extension is equivalent to that of a \emph{weak} action of $\fL(G)_{x_0}$ on $\Vect$.
Denote the resulting object\footnote{We refer the reader to \cite[Sect. B.14.19]{GLC2}, where weak actions of loop groups on categories are discussed.}
of $\fL(G)_{x_0}\mmod^{\on{weak}}$ by $\Vect_\kappa$. 

\medskip

The category $\Rep(\fL(G)_{x_0},\kappa)$ of integrable $\fL(G)$-representations at level $\kappa$ is by definition
$$\on{Funct}_{\fL(G)_{x_0}\mmod^{\on{weak}}}(\Vect,\Vect_\kappa).$$ 

\begin{rem} \label{r:neg def}

The category $\Rep(\fL(G)_{x_0},\kappa)$ is well-defined for \emph{any} $\kappa$. However, when $\kappa$ is negative-definite,
the natural forgetful functor 
\begin{equation} \label{e:forget from int}
\Rep(\fL(G)_{x_0},\kappa)\to \Vect
\end{equation}
is zero\footnote{However, the forgetful functor $\Rep(\fL(G)_{x_0},\kappa)\to \KL(G,\kappa)$ is conservative. There is no contradiction here, since the
forgetful functor  $\KL(G,\kappa)\to \Vect$ is not conservative.}. 

\medskip

Indeed, for an object of $\Rep(\fL(G)_{x_0},\kappa)$, the $n$-th cohomology of its image under \eqref{e:forget from int}
would be an \emph{integrable representation} of $\fL(G)_{x_0}$ at level $\kappa$ in the classical sense. But for $\kappa$ negative-definite, there
are non-zero such, because the sign of $\kappa$ makes the dominance condition on the highest weight impossible
to satisfy. 

\medskip

Yet, the category $\Rep(\fL(G)_{x_0},\kappa)$ is non-zero. In fact, for any $\kappa$, we 
have a natural identification
$$\Rep(\fL(G)_{x_0},\kappa')\simeq (\Rep(\fL(G)_{x_0},\kappa))^\vee, \quad \kappa'=-\kappa-\kappa_{\on{Killing}}.$$

\end{rem}

\sssec{}

In this section we will prove:

\begin{thm} \label{t:int}
Suppose that the level $\kappa$ is non-negative definite on each simple factor. 
Then there is a canonical equivalence 
$$\Rep(\fL(G)_{x_0},\kappa) \simeq \BV^{\on{Int}}_{G,\kappa}\mod^\onfact_{x_0},$$
commuting with the tautological forgetful functors of both sides to $\Vect$.
\end{thm}

The theorem will be proved in Sects. \ref{ss:proof of int 1}-\ref{ss:proof of int 2}.

\begin{rem}

The statement of \thmref{t:int} is false for $\kappa$ negative-definite:

\medskip

Indeed, according
to Remark \ref{r:int exact}, in this case $\BV^{\on{Int}}_{G,\kappa}=0$, so the right-hand side
in \thmref{t:int} is zero. However, according to Remark \ref{r:neg def}, the left-hand side is non-zero.

\end{rem}

\begin{rem}

The statement of \thmref{t:int} is known at the level of abelian categories, in two cases:
either when $G$ is semi-simple and simply-connected or when $G$ is a torus. 

\medskip

When $G$ is semi-simple and simply-connected, this is the statement that for a module $\CM$ over the affine Kac-Moody Lie algebra
at level $\kappa$,
the following two conditions are equivalent:

\smallskip

\noindent(i) The action of the affine Kac-Moody Lie algebra on $\CM$ integrates to an action of $\wh{\fL(G)}_{\kappa,x_0}$;

\smallskip

\noindent(ii) When we view $\CM$ as a chiral module over $\BV^{\on{ch}}_{\fg,\kappa}$ (i.e., a
factorization module over $\BV_{\fg,\kappa}$ at $x_0$), the action factors through the integrable quotient
chiral module over $$\BV^{\on{ch}}_{\fg,\kappa}\twoheadrightarrow \BV^{\on{Int,ch}}_{G,\kappa}.$$

\medskip

When $G=T$ is a torus, the statement of \thmref{t:int} at the abelian level is equivalent to that
of \cite[Theorem 3.10.14]{BD1}.

\end{rem} 

\sssec{} \label{sss:int ss}

Under the assumption on the level $\kappa$, the category $\Rep(\fL(G)_{x_0},\kappa)$ is known to be semi-simple 
(e.g., this can be proved by the same method
as in \cite[Appendix D]{Ro}). Hence, from \thmref{t:int}, we obtain:

\begin{cor} \label{c:int semi-simple}
The category $\BV^{\on{Int}}_{G,\kappa}\mod^\onfact_{x_0}$ is semi-simple.
\end{cor}

\begin{rem}
The statement of \corref{c:int semi-simple} is known at the abelian level. The innovation here
is that it continues to hold at the derived level, i.e., that there are no higher Exts between
objects in the heart. 

\medskip

One may (somewhat recklessly) conjecture that the same holds for any rational VOA.

\end{rem} 

\ssec{Factorization categories with an action of the loop group} \label{ss:fact w loop}


%


Recall the setting of Sects. \ref{sss:factorizable LG}-\ref{sss:factorizable LG acts}. 

\medskip

Thus, let $\bA$ be a factorization algebra category, equipped with an action of $\fL(G)_{\Ran}$,
compatible with the factorization structure. 

\medskip

In particular, $\bA_{x_0}$ is a category acted on by $\fL(G)_{x_0}$, and we can form
\begin{equation} \label{e:action Gr 1}
(\bA_{x_0})^{\on{fact}_{x_0},\Dmod(\Gr_G)}\in \Dmod(\Gr_G)\mmod^\onfact_{x_0}.
\end{equation} 

In this subsection we will discuss an additional feature of the above construction
in the unital setting. 

\sssec{}

Let us assume that $\bA$ is unital and that its unit object $\one_{\bA}$ is $\fL^+(G)$-equivariant. 
In this case, the action on the unit gives rise to a unital factorization functor
$$\Phi:\Dmod(\Gr_G)\to \bA.$$

\medskip

In particular, we can form the object
\begin{equation} \label{e:action Gr 2}
\Rres_\Phi(\bA^{\on{fact}_{x_0}})\in  \Dmod(\Gr_G)\mmod^\onfact_{x_0}.
\end{equation}

\sssec{}
\label{sss:tight}

We now add the following  technical condition: we assume that $\bA$ is \emph{tight}, i.e., the functor of the insertion of the unit
$$\on{ins.unit}_{\ul{x}_1\subseteq \ul{x}_2}:\bA_{\ul{x}_1} \to \bA_{\ul{x}_2}, \quad \ul{x}_1\subseteq \ul{x}_2$$
admits a colimit-preserving right adjoint (see \cite[Sect. C.16.1]{GLC2}). 

\medskip
In particular, this implies that for any $\ul{x}$, the object $(\one_\bA)_{\ul{x}}\in \bA_{\ul{x}}$ is compact. We impose
an even stronger condition, namely, that $(\one_\bA)_{\ul{x}}$ is compact as an object of $(\bA_{\ul{x}})^{\fL^+(G)_{\ul{x}}}$.

\begin{propconstr}  \label{p:fact act loop}
Under the above circumstances, the objects \eqref{e:action Gr 1} and \eqref{e:action Gr 2}
are canonically isomorphic.
\end{propconstr}

The rest of this subsection is devoted to the proof of this proposition.

\sssec{}

We first construct a map in one direction, \eqref{e:action Gr 1} $\to$  \eqref{e:action Gr 2}. By the definition
of factorization restriction, the datum of such a map is equivalent to the datum of a functor
\begin{equation} \label{e:action Gr construct}
\Phi_m:(\bA_{x_0})^{\on{fact}_{x_0},\Dmod(\Gr_G)}_{\ul{x}}\to \bA^{\on{fact}_{x_0}}_{\ul{x}}\simeq \bA_{\ul{x}}, \quad \ul{x}\in \Ran_{x_0}
\end{equation} 
compatible with factorization via $\Phi$.

\sssec{}

Recall the group ind-scheme $\fL^{\on{mer}\rightsquigarrow \on{reg}}(G)_{\Ran_{x_0}}$, see 
\eqref{e:mer to reg}. It projects onto $\fL(G)_{x_0}$, and let us denote by $\ol\fL^+(G)_{\Ran_{x_0}}$
the kernel. In particular, we have a short exact sequence
\begin{equation} \label{e:congr fact}
1\to \ol\fL^+(G)_{\ul{x}}\to \fL^{\on{mer}\rightsquigarrow \on{reg}}(G)_{\ul{x}}\to \fL(G)_{x_0}\to 1.
\end{equation}

\noindent(Note that $\ol\fL^+(G)_{\ul{x}}$ identifies also with the kernel of 
$\fL^+(G)_{\ul{x}}\to \fL^+(G)_{x_0}$.)

\medskip

We can identify
$$\Gr^{\on{level}^\infty_{x_0}}_{G,\ul{x}}\simeq \fL(G)_{\ul{x}}/\ol\fL^+(G)_{\ul{x}},$$
where the \emph{right} action of $\fL(G)_{x_0}$ on $\Gr^{\on{level}^\infty_{x_0}}_{G,\ul{x}}$ comes
from the short exact sequence \eqref{e:congr fact}. 

\sssec{}

Consider the functor
\begin{equation} \label{e:insert unit x0}
\on{ins.unit}_{\{x_0\}\subseteq \ul{x}}:\bA_{x_0}\to \bA_{\ul{x}}.
\end{equation} 

Both categories are acted on by $\fL(G)_{x_0}$, and the functor \eqref{e:insert unit x0} is compatible with these actions.
Moreover, \eqref{e:insert unit x0} factors via a functor
\begin{equation} \label{e:insert unit x0 bis}
\on{ins.unit}_{\{x_0\}\subseteq \ul{x}}:\bA_{x_0}\to (\bA_{\ul{x}})^{\ol\fL^+(G)_{\ul{x}}}.
\end{equation} 

The functor \eqref{e:insert unit x0 bis} is also compatible with the $\fL(G)_{x_0}$-actions, where the action
on the right-hand side is via the short exact sequence \eqref{e:congr fact}.

\sssec{} \label{sss:tight bis}

For future reference we note that it follows from the assumptions on $\bA$ in \secref{sss:tight} that the functor \eqref{e:insert unit x0 bis}
admits a colimit-preserving right adjoint. 

\sssec{}

Recall that 
$$(\bA_{x_0})^{\on{fact}_{x_0},\Dmod(\Gr_G)}_{\ul{x}}=
\Dmod(\Gr^{\on{level}^\infty_{x_0}}_{G,\ul{x}})\underset{\fL(G)_{x_0}}\otimes \bA_{x_0}.$$

Thus, from \eqref{e:insert unit x0 bis} we obtain a functor 
\begin{equation} \label{e:action Gr construct 1}
(\bA_{x_0})^{\on{fact}_{x_0},\Dmod(\Gr_G)}_{\ul{x}}\to
\Dmod(\Gr^{\on{level}^\infty_{x_0}}_{G,\ul{x}})\underset{\fL(G)_{x_0}}\otimes  (\bA_{\ul{x}})^{\ol\fL^+(G)_{\ul{x}}}
\end{equation}

\sssec{}

We rewrite the right-hand side in \eqref{e:action Gr construct 1} as
\begin{equation} \label{e:action Gr construct 2}
\Dmod(\fL(G)_{\ul{x}})^{\ol\fL^+(G)_{\ul{x}}}\underset{\fL^{\on{mer}\rightsquigarrow \on{reg}}(G)_{\ul{x}}/\ol\fL^+(G)_{\ul{x}}}\otimes (\bA_{\ul{x}})^{\ol\fL^+(G)_{\ul{x}}} 
\simeq \Dmod(\fL(G)_{\ul{x}}) \underset{\fL^{\on{mer}\rightsquigarrow \on{reg}}(G)_{\ul{x}}}\otimes (\bA_{\ul{x}})^{\ol\fL^+(G)_{\ul{x}}}.
\end{equation}

Now, the action of $\fL(G)_{\ul{x}}$ on $\bA_{\ul{x}}$ gives rise to a functor
\begin{equation} \label{e:action Gr construct 3}
\Dmod(\fL(G)_{\ul{x}}) \underset{\fL^{\on{mer}\rightsquigarrow \on{reg}}(G)_{\ul{x}}}\otimes (\bA_{\ul{x}})^{\ol\fL^+(G)_{\ul{x}}}\to \bA_{\ul{x}}.
\end{equation}

Composing \eqref{e:action Gr construct 1}, \eqref{e:action Gr construct 2} and \eqref{e:action Gr construct 3}, we obtain the sought-for functor
\eqref{e:action Gr construct}. 

\medskip

The compatibility with factorization against $\Phi$ follows from the construction. 

\sssec{}

We now show that the resulting functor
\begin{equation} \label{e:action Gr constructed} 
(\bA_{x_0})^{\on{fact}_{x_0},\Dmod(\Gr_G)}\to \Rres_\Phi(\bA^{\on{fact}_{x_0}})
\end{equation} 
is an equivalence.

\medskip

In order to do se, we apply Proposition \ref{sss:adj test for factres}. We need to show:

\begin{itemize} 

\item The functor \eqref{e:action Gr constructed} induces an equivalence between the fibers at $x_0$;

\medskip

\item The functor $\Phi$ admits a colimit-preserving right adjoint;

\medskip

\item The functor $\Phi_m$ admits a colimit-preserving right adjoint.

\end{itemize}

\medskip

The fact that the first condition is satisfied is automatic. Indeed, at the level of fibers at $x_0$, the functor
\eqref{e:action Gr constructed} is the identity functor $\bA_{x_0}\to \bA_{x_0}$. 

\medskip

The fact that $\Phi$ admits a colimit-preserving right adjoint follows from the fact that $(\one_\bA)_{\ul{x}}$
is compact as an object of $\bA^{\fL^+(G)}_{\ul{x}}$, combined with the ind-properness of $\Gr_{G,\ul{x}}$.

\medskip

Finally, let us show that $\Phi$ admits a colimit-preserving right adjoint. For that it is sufficient show that
both functors \eqref{e:action Gr construct 1} and \eqref{e:action Gr construct 3} admit a colimit-preserving right adjoints. 

\sssec{}

The fact that \eqref{e:action Gr construct 1} admits a colimit-preserving right adjoint follows from the corresponding fact for
\eqref{e:insert unit x0 bis}, see \secref{sss:tight bis}. 

\medskip

For the functor \eqref{e:action Gr construct 3}, we write it as a composition
$$\Dmod(\fL(G)_{\ul{x}})^{\fL^{\on{mer}\rightsquigarrow \on{reg}}(G)_{\ul{x}}}\otimes (\bA_{\ul{x}})^{\ol\fL^+(G)_{\ul{x}}}\to
\Dmod(\fL(G)_{\ul{x}})^{\fL^{\on{mer}\rightsquigarrow \on{reg}}(G)_{\ul{x}}}\otimes \bA_{\ul{x}} \to \bA_{\ul{x}},$$
and it suffices to show that the second arrow admits a colimit-preserving right adjoint.

\medskip

Using the $\fL(G)_{\ul{x}}$-action on $\bA_{\ul{x}}$, we rewrite this arrow as
$$\Dmod(\fL(G)_{\ul{x}})^{\fL^{\on{mer}\rightsquigarrow \on{reg}}(G)_{\ul{x}}}\otimes \bA_{\ul{x}}\to \bA_{\ul{x}},$$
induced by the functor 
$$\Dmod(\fL(G)_{\ul{x}})^{\fL^{\on{mer}\rightsquigarrow \on{reg}}(G)_{\ul{x}}} \simeq \Dmod(\Gr_{G,\ul{x}}/\on{Hecke}_{x_0})
\overset{\on{C}^\cdot(\Gr_{G,\ul{x}}/\on{Hecke}_{x_0},-)}\longrightarrow \Vect.$$

This implies the desired assertion since $\Gr_{G,\ul{x}}$ is ind-proper and $\on{Hecke}_{x_0}$ is pseudo-proper.

\qed[\propref{p:fact act loop}]

\ssec{Proof of \thmref{t:int}} \label{ss:proof of int 1}

\sssec{} \label{sss:enter KM}

We consider the factoriztion category $\KM(\fg,\kappa)$ of Kac-Moody representations at level $\kappa$,
see \cite[Sect. B.14.22]{GLC2}. It is naturally equipped with a (strong) action of $\fL(G)$ \emph{at level} $\kappa$. 

\medskip

Now, the datum of the central extension $\wh{\fL(G)}_\kappa$ allows us to modify the weak action of $\fL(G)$
on $\KM(\fg,\kappa)$, so that the resulting strong action occurs at level $0$. The resulting object
$$\KM(\fg,\kappa)\in \fL(G)\mmod$$
has the universal property that
\begin{equation} \label{e:weak vs strong}
\on{Funct}_{\fL(G)_{\ul{x}}}(\bC,\KM(\fg,\kappa)_{\ul{x}}) \simeq \on{Funct}_{\fL(G)_{\ul{x}}\on{-weak}}(\bC,\Vect_\kappa), \quad \bC\in \fL(G)_{\ul{x}}\mmod, \quad \ul{x}\in \Ran,
\end{equation} 
see \cite[Sect. B.14.12]{GLC2}.

\sssec{}

Taking $\bC$ in \eqref{e:weak vs strong} to be $\Vect$, we obtain
\begin{equation} \label{e:int via KM}
\Rep(\fL(G)_{x_0},\kappa):=\on{Funct}_{\fL(G)_{x_0}\on{-weak}}(\Vect,\Vect_\kappa) \simeq 
\on{Funct}_{\fL(G)_{x_0}\mmod}(\Vect,\KM(\fg,\kappa)_{x_0}).
\end{equation} 

\sssec{}

Consider the object
\begin{equation} \label{e:KM as mod}
\KM(\fg,\kappa)^{\on{fact}_{x_0},\Dmod(\Gr_G)}\in \Dmod(\Gr_G)\mmod^\onfact_{x_0}.
\end{equation} 

\medskip

We now perform the crucial step in the proof of \thmref{t:int}. Namely, we combine \eqref{e:int via KM} with \thmref{t:main}, and obtain an equivalence
\begin{equation} \label{e:apply main to int}
\Rep(\fL(G)_{x_0},\kappa)\simeq 
\omega_{\Gr_{G,\Ran}}\mod^\onfact(\KM(\fg,\kappa)^{\on{fact}_{x_0},\Dmod(\Gr_G)})_{x_0}.
\end{equation} 

\medskip

The equivalence \eqref{e:apply main to int} commutes with the forgetful functors to $\Vect$, where in the
right-hand side, the corresponding functor is 
$$\omega_{\Gr_{G,\Ran}}\mod^\onfact(\KM(\fg,\kappa)^{\on{fact}_{x_0},\Dmod(\Gr_G)})_{x_0}\to  \KM(\fg,\kappa)_{x_0}\overset{\oblv_{\KM}}\longrightarrow \Vect,$$
where the second arrow is the tautological forgetful functor. 

\sssec{}

Note now that $\KM(\fg,\kappa)$, viewed as a factorization category equipped with an action of $\fL(G)$, fits into the paradigm of 
\secref{ss:fact w loop}.

\medskip

The corresponding factorization functor
$$\Phi:\Dmod(\Gr_G)\to \KM(\fg,\kappa)$$
is the functor of $\CL_{\Gr_G}$-twisted IndCoh sections, to be denoted $\Gamma^{\on{enh}}_\kappa$. 
 I.e., its composition with the forgetful functor 
\begin{equation} \label{e:forget KM}
\oblv_{\KM}:\KM(\fg,\kappa)\to \Vect
\end{equation}
is the functor \eqref{e:L-twisted sects}, to be denoted $\Gamma_\kappa$.

\medskip

Thus, applying \propref{p:fact act loop}, we obtain that the object \eqref{e:KM as mod} identifies with
$$\Rres_{\Gamma^{\on{enh}}_\kappa}(\KM(\fg,\kappa)^{\on{fact}_{x_0}}).$$

\sssec{}

Let $\on{Vac}^{\on{Int}}_{G,\kappa}$ be the factorization algebra in $\KM(\fg,\kappa)$ equal to the image of
$\omega_{\Gr_G}$ under the functor $\Gamma^{\on{enh}}_\kappa$. 

\medskip

Applying \eqref{e:equiv fact mod}, we rewrite 
$$\omega_{\Gr_{G,\Ran}}\mod^\onfact\left(\Rres_{\Gamma^{\on{enh}}_\kappa}(\KM(\fg,\kappa)^{\on{fact}_{x_0}})\right)_{x_0}\simeq
\on{Vac}^{\on{Int}}_{G,\kappa}\mod^\onfact(\KM(\fg,\kappa)^{\on{fact}_{x_0}})_{x_0}.$$

\sssec{}

Thus, combining, we obtain an equivalence
\begin{equation} \label{e:almost there}
\Rep(\fL(G)_{x_0},\kappa)\simeq \on{Vac}^{\on{Int}}_{G,\kappa}\mod^\onfact_{x_0}.
\end{equation} 
(see \secref{sss:vac} for the notation), which commutes with the natural forgetful functors of both
sides to $\Vect$. 

\begin{rem}
Note that the equivalence \eqref{e:almost there} did not use the assumption that $\kappa$ is non-negative definite.
\end{rem} 

\sssec{}

By construction
$$\BV^{\on{Int}}_{G,\kappa}\simeq \oblv_{\KM}(\on{Vac}^{\on{Int}}_{G,\kappa}).$$

Hence, in order to prove \thmref{t:int}, it remains to show the following:

\begin{prop} \label{p:convergence}
The functor
$$\oblv^{\on{Int}}_{\KM}:\on{Vac}^{\on{Int}}_{G,\kappa}\mod^\onfact_{x_0}\to \BV^{\on{Int}}_{G,\kappa}\mod^\onfact_{x_0},$$
induced by the factorization functor $\oblv_{\KM}$, is an equivalence.
\end{prop}

\qed[\thmref{t:int}]

\begin{rem} 
As we will see, \propref{p:convergence} would be almost tautological, if not not for some
homological algebra ``issues" (the idea is that the category $\KM(\fg,\kappa)$ is ``almost" the same
as $\BV_{\fg,\kappa}\mod^\onfact$). 

\medskip

Yet, this issues become fatal when $\kappa$ is negative-definite. 
(Indeed, as was remarked above, the equivalence \eqref{e:almost there} holds for any $\kappa$.)  
\end{rem} 

\ssec{Proof of \propref{p:convergence}: reduction to the bounded below category} 

\sssec{}

Note that the chiral algebra $\BV^{\on{Int,ch}}_{G,\kappa}$ is concentrated in non-positive cohomological 
degrees\footnote{In fact, according to Remark \ref{r:int exact}, it is actually concentrated in degree $0$, 
a fact that will be used later.}. Hence, the category 
$\BV^{\on{Int}}_{G,\kappa}\mod^\onfact_{x_0}$ acquires a t-structure,
characterized by the property that the forgetful functor
$$\oblv_{\BV^{\on{Int}}_{G,\kappa}}:\BV^{\on{Int}}_{G,\kappa}\mod^\onfact_{x_0}\to \Vect$$
(which is by definition conservative) is t-exact. Moreover, $\BV^{\on{Int}}_{G,\kappa}\mod^\onfact$ is left-complete in its t-structure,
see \cite[Proposition B.9.18]{GLC2}.

\sssec{}

Note that the factorization category $\KM(\fg,\kappa)$ carries a t-structure in the sense of \cite[Sect. B.11.11]{GLC2}
(see \cite[Sect. B.14.22]{GLC2}). 

\medskip

Since the functor $\oblv_{\KM}$ is t-exact and conservative \emph{on the bounded below category}, we obtain 
that the object 
$$\on{Vac}^{\on{Int}}_{G,\kappa}|_X[-1]\in \KM(\fg,\kappa)_X$$
is connective.

\medskip

Hence, by \cite[Sect. B.11.16]{GLC2}, the category $\on{Vac}^{\on{Int}}_{G,\kappa}\mod^\onfact_{x_0}$ also acquires 
a t-structure, characterized by the property that the forgetful functor
$$\oblv_{\on{Vac}^{\on{Int}}_{G,\kappa}}:\on{Vac}^{\on{Int}}_{G,\kappa}\mod^\onfact_{x_0}\to \KM(\fg,\kappa)_{x_0}$$
(which is by definition conservative) is t-exact.

\begin{rem} \label{r:vacuum in heart}

For future reference, we note that the object $\on{Vac}^{\on{Int}}_{G,\kappa}|_X[-1]$
lies in the heart of the t-structure of $\KM(\fg,\kappa)_X$. 

\medskip

Let us consider $\on{Vac}^{\on{Int}}_{G,\kappa}|_X[-1]$ as an object of $\KL(G,\kappa)_X$, where
$$\KL(G,\kappa)=\KM(\fg,\kappa)^{\fL^+(G)}.$$

Recall that the level was assumed \emph{non-negative definite}
(on every simple factor of $\fg$). We claim that in this case, the functor $\oblv_{\KM}$, restricted to $\KL(G,\kappa)_{x_0}$, 
is actually conservative\footnote{This observation is due to G.~Dhillon.}.

\medskip

Indeed, in this case, the compact generators of $\KL(G,\kappa)_{x_0} $ have a finite cohomological dimension
(this can be seen, e.g., from the Kashiwara-Tanisaki localization); hence $\KL(G,\kappa)_{x_0} $ has no non-zero
infinitely connective objects.

%

\medskip

Hence, the fact that $\BV^{\on{Int}}_{G,\kappa}|_X[-1]=\BV^{\on{Int,ch}}_{G,\kappa}$ lies in the heart of
the t-structure implies that the same is true for $\on{Vac}^{\on{Int}}_{G,\kappa}|_X[-1]$. 

\end{rem} 

\sssec{} \label{sss:truncated conclusion}

In particular, we obtain that the functor $\oblv^{\on{Int}}_{\KM}$, appearing in \propref{p:convergence}, is t-exact
(since its composition with $\oblv_{\BV^{\on{Int}}_{G,\kappa}}$, which is t-exact and conservative, is t-exact). 

\medskip

Recall that the source category in \propref{p:convergence}, being equivalent to $\Rep(\fL(G)_{x_0},\kappa)$,
is semi-simple (see \secref{sss:int ss}). Moreover, the same argument shows that its irreducible objects are bounded below
(in fact, that they lie in the heart of the t-structure). 

\medskip

Since the target category is left-complete in its t-structure, we obtain that in order to prove that $\oblv^{\on{Int}}_{\KM}$ is an equivalence, 
it suffices to show that it induces an equivalence on the corresponding
bounded below categories.

\ssec{Proof of \propref{p:convergence}: the bounded below part} \label{ss:proof of int 2}

\sssec{}

Let $\on{Vac}_{\fg,\kappa}$ denote the factorization unit in $\KM(\fg,\kappa)$. Denote 
$$\BV_{\fg,\kappa}:=\oblv_{\KM}(\on{Vac}_{\fg,\kappa}).$$

This is the usual factorization algebra attached to $\fg$ at level $\kappa$. 

\sssec{}

The unit for $\on{Vac}^{\on{Int}}_{G,\kappa}$, which is a map
$$\on{Vac}_{\fg,\kappa}\to \on{Vac}^{\on{Int}}_{G,\kappa},$$
induces a homomorphism
\begin{equation} \label{e:two vacs}
\BV_{\fg,\kappa}\to \BV^{\on{Int}}_{G,\kappa}
\end{equation}
as unital factorization algebras in $\Vect$.

\medskip

Let us regard $\BV_{\fg,\kappa}\mod^\onfact$ as a unital \emph{lax factorization category}, see \cite[Sect. B.11.12]{GLC2}.
The homomorphism \eqref{e:two vacs} allows us to upgrade $\BV^{\on{Int}}_{G,\kappa}$ to a unital factorization algebra in
$\BV_{\fg,\kappa}\mod^\onfact$ (see \cite[Setct. C.11.18]{GLC2}); when viewed as such we will denote it by $\BV^{\on{Int,enh}}_{G,\kappa}$.

\medskip

The forgetful factorization functor 
$$\oblv_{\BV_{\fg,\kappa}}:\BV_{\fg,\kappa}\mod^\onfact\to \Vect$$
induces a functor
\begin{equation} \label{e:two algs}
\BV^{\on{Int,enh}}_{G,\kappa}\mod^\onfact_{x_0}\to \BV^{\on{Int}}_{G,\kappa}\mod^\onfact_{x_0},
\end{equation}
to be denoted $\oblv^{\on{Int}}_{\BV_{\fg,\kappa}}$.

\medskip

According to \cite[Lemma C.11.19]{GLC2}, the functor \eqref{e:two algs} is an equivalence. 

\sssec{}

The functor $\oblv_{\KM}$ induces a factorization functor
\begin{equation} \label{e:oblv KM enh}
\KM(\fg,\kappa)\simeq \on{Vac}_{\fg,\kappa}\mod^\onfact(\KM(\fg,\kappa))
\to \BV_{\fg,\kappa}\mod^\onfact,
\end{equation}
to be denoted $\oblv^{\on{enh}}_{\KM}$.

\medskip

We have
$$\oblv^{\on{enh}}_{\KM}(\BV^{\on{int}}_{G,\kappa})\simeq \BV^{\on{Int,enh}}_{G,\kappa}$$
as factorization algebras in $\BV_{\fg,\kappa}\mod^\onfact$. 

\medskip

In particular, $\oblv^{\on{enh}}_{\KM}$ induces a functor
\begin{equation} \label{e:oblv KM enh Int}
\on{Vac}^{\on{Int}}_{G,\kappa}\mod^\onfact_{x_0}\to \BV^{\on{Int,enh}}_{G,\kappa}\mod^\onfact_{x_0},
\end{equation}
to be denoted $\oblv^{\on{Int,enh}}_{\KM}$.

\medskip

Unwinding, we obtain that the functor $\oblv^{\on{Int}}_{\KM}$ identifies with the composition
$$\oblv^{\on{Int}}_{\BV_{\fg,\kappa}}\circ \oblv^{\on{Int,enh}}_{\KM}.$$

\sssec{}

By \cite[Sect. B.11.15]{GLC2}, the lax factorization category $\BV_{\fg,\kappa}\mod^\onfact$
carries a t-structure in the sense of \cite[Sect. B.11.11]{GLC2}. It is characterized by the property 
that the (conservative) forgetful functor $\oblv_{\BV_{\fg,\kappa}}$ is t-exact. 

\medskip

Since $\BV^{\on{Int,enh}}_{G,\kappa}$ is connective, by \cite[Sect. B.11.16]{GLC2}, we obtain that the category
$\BV^{\on{Int,enh}}_{G,\kappa}\mod^\onfact_{x_0}$ acquires a t-structure, characterized by the property 
that the (conservative) forgetful functor 
$$\BV^{\on{Int,enh}}_{G,\kappa}\mod^\onfact_{x_0}\to \BV_{\fg,\kappa}\mod^\onfact_{x_0}\to \Vect$$
is t-exact.

\medskip

In particular, we obtain that the equivalence $\oblv^{\on{Int}}_{\BV_{\fg,\kappa}}$ of \eqref{e:two algs}
t-exact.

\sssec{}

Hence, in order to prove \propref{p:convergence}, it remains to show that 
the functor $\oblv^{\on{Int,enh}}_{\KM}$ of \eqref{e:oblv KM enh Int} induces an equivalence on the bounded below subcategories
of the two sides. 

\medskip

Recall (see Remark \ref{r:vacuum in heart}) that the object $\on{Vac}^{\on{Int}}_{G,\kappa}|_X[-1]$ lies in the heart of the t-structure\footnote{It is here
that we crucially use the assumption that $\kappa$ is non-negative definite.}. 
Hence, its values over all powers of $X$ are bounded below. Hence, for an eventually coconnective object in $\KM(\fg,\kappa)_{x_0}$, the morphisms
that define on it a structure of object of $\on{Vac}^{\on{Int}}_{G,\kappa}\mod^\onfact_{x_0}$ take place in the bounded 
below subcategories of values $\KM(\fg,\kappa)$ on powers of $X$.

\medskip

The same is true for $\BV^{\on{Int}}_{G,\kappa}|_X[-1]$ and $\BV_{\fg,\kappa}\mod^\onfact$. 

\medskip

Therefore, in order to prove that $\oblv^{\on{Int,enh}}_{\KM}$ induces an equivalence on the bounded below subcategories,
it suffices to show that the functor $\oblv^{\on{enh}}_{\KM}$ of \eqref{e:oblv KM enh} induces an equivalence 
between the bounded below subcategories of the two sides (evaluated on powers of $X$). 

\medskip

However, the latter is the
assertion of \cite[Lemma 4.2.3(a)]{GLC2}.

\qed[\propref{p:convergence}]

\ssec{An addendum: failure of the coherent version of \thmref{t:initial}} \label{ss:coh}

\sssec{}

Parallel to the setting of \secref{ss:constr of funct}, one can consider:

\medskip

\begin{itemize}

\item The category $\fL(G)_{x_0}\mmod^{\on{weak}}$
of DG categories, equipped with a weak action of $\fL(G)_{x_0}$;

\item The factorization category
$\IndCoh(\Gr_G)$;

\item The functor
\begin{equation} \label{e:the functor coh}
\fL(G)_{x_0}\mmod^{\on{weak}} \to \IndCoh(\Gr_G)\mmod^\onfact_{x_0},\quad \bC\mapsto \bC^{\on{fact}_{x_0},\IndCoh(\Gr_G)}. 
\end{equation}

\end{itemize} 

\sssec{}

However, we claim that, unlike \thmref{t:initial}, the functor \eqref{e:the functor coh} fails to be fully faithful.
Namely, as we shall presently explain, the functor

\smallskip

\begin{multline} \label{e:functor to disprove}
\on{Funct}_{\fL(G)_{x_0}\mmod^{\on{weak}}}(\bC_1,\bC_2)\to \\
\to \on{Funct}_{\IndCoh(\Gr_G)\mmod^\onfact_{x_0}}((\bC_1)^{\on{fact}_{x_0},\IndCoh(\Gr_G)},(\bC_2)^{\on{fact}_{x_0},\IndCoh(\Gr_G)})
\end{multline}
fails to be an equivalence.

\medskip

Namely, the right-hand side admits a natural conservative functor to $\on{Funct}_{\DGCat}(\bC_1,\bC_2)$, whose composition with the functor
in \eqref{e:functor to disprove} is the natural forgetful functor
\begin{equation} \label{e:forget weak}
\on{Funct}_{\fL(G)_{x_0}\mmod^{\on{weak}}}(\bC_1,\bC_2)\to \on{Funct}_{\DGCat}(\bC_1,\bC_2).
\end{equation} 

However, we claim that we can find $\bC_1,\bC_2$ so that the functor \eqref{e:forget weak} fails to be conservative.

\sssec{}

Namely, we take $\bC_1=\Vect$ and $\bC_2=\Vect_\kappa$, where $\kappa$ is negative-definite. The left-hand side
in \eqref{e:forget weak} is $\Rep(\fL(G)_{x_0},\kappa)$, and \eqref{e:forget weak} is the natural forgetful functor
$$\Rep(\fL(G)_{x_0},\kappa)\to \Vect.$$

However, the above functor is zero for negative-definite (see Remark \ref{r:neg def}). 

\appendix

\section{D-modules in infinite type} \label{s:inf type}

In this section we (re)collect some facts pertaining to the extension of the theory of D-modules
to algebro-geometric objects of infinite type. The goal is to make sense of the category of D-modules
on the loop group $\fL(G)_{x_0}$. 

\ssec{The case of affine schemes}

In this subsection we develop the theory of D-modules on affine schemes (not necessarily of finite type).
We will mostly follow \cite[Sect. A.4-A.5]{GLC2}.

\sssec{}

Let $S$ be a scheme (not necessarily of finite type). Set
\begin{equation} \label{e:D!}
\Dmod^!(S):=\underset{S\to S_0}{\on{colim}}\, \Dmod(S_0),
\end{equation} 
where:

\begin{itemize}

\item The colimit is taken over the (opposite of the) category of affine schemes of finite type
that receive a map from $S$;

\smallskip

\item The transition functors are given by !-pullback.

\end{itemize}

\medskip

Equivalently, for a fixed presentation
\begin{equation} \label{e:present affine}
S\simeq \underset{\alpha}{\on{lim}}\, S_{0,\alpha},
\end{equation}
where $S_{0,\alpha}$ are affine schemes of finite type, we have
\begin{equation} \label{e:present affine D!}
\Dmod^!(S)\simeq \underset{\alpha}{\on{colim}}\, \Dmod(S_{0,\alpha}).
\end{equation} 

\medskip

For an arbitrary $S$, there is no reason for $\Dmod^!(S)$ defined in the above way to be 
compactly generated or dualizable.

\sssec{} \label{sss:ind-hol aff sch}

We shall say that an object $\Dmod^!(S)$ is \emph{ind-holonomic} it lies in the essential image
of 
$$\underset{\alpha}{\on{colim}}\, \Dmod^{\on{hol}}(S_{0,\alpha})\to \underset{\alpha}{\on{colim}}\, \Dmod(S_{0,\alpha})=\Dmod^!(S).$$

\sssec{}

We set
$$\Dmod_*(S):=\on{Funct}_{\DGCat}(\Dmod^!(S),\Vect).$$

Using Verdier duality, we obtain:
$$\Dmod_*(S)\simeq \underset{S\to S_0}{\on{lim}}\, \Dmod(S_0),$$
where the limit is taken with respect to the *-pushforward functors. 

\medskip

In terms of the presentation \eqref{e:present affine}, we have
$$\Dmod_*(S)\simeq \underset{\alpha}{\on{lim}}\, \Dmod(S_{0,\alpha}).$$

\sssec{}

Since the transition functors in \eqref{e:D!} are symmetric monoidal, the category $\Dmod^!(S)$
carries a natural symmetric monoidal structure. Its unit object, denoted $\omega_S$, is the image
of $\omega_{S_0}$ for any $S\to S_0$. 

\medskip

Note that we have a canonical isomorphism
$$\on{C}^\cdot(S):=\underset{S\to S_0}{\on{colim}}\, \on{C}^\cdot(S_0) \simeq \underset{S\to S_0}{\on{colim}}\, \End_{\Dmod(S_0)}(\omega_{S_0})\simeq 
\End_{\Dmod^!(S)}(\omega_S).$$

\medskip

In addition, we have a natural action of $\Dmod^!(S)$ on $\Dmod_*(S)$.

\medskip

The category $\Dmod_*(S)$ has a distinguished object, denoted $\ul{k}_S$; it corresponds to the compatible
family of functors 
$$\Dmod(S_0)\to \Vect, \quad S\to S_0$$ equal to
$$\underset{S\to S'_0\overset{f}\to S_0}{\on{colim}}\, \CHom(\omega_{S'_0},f^!(\CF)).$$

\begin{rem}

One can describe the object $\ul{k}_S$ more explicitly as follows: let
$$\Dmod^{\on{hol}}_*(S)\subset \Dmod_*(S)$$
be the full subcategory equal to
$$\underset{S\to S_0}{\on{lim}}\, \Dmod^{\on{hol}}(S_0)\subset \underset{S\to S_0}{\on{lim}}\, \Dmod(S_0),$$
where the limits are taken with respect to the *-pushforward functors. 

\medskip

Since on the holonomic category, the *-pushforward admit left adjoints, we can rewrite $\Dmod^{\on{hol}}_*(S)$ also as
$$\underset{S\to S_0}{\on{colim}}\, \Dmod^{\on{hol}}(S_0),$$
where the colimit is taken with respect to the *-pullbacks.

\medskip

In terms of the latter presentation, $\ul{k}_S$ equals the image of $\ul{k}_{S_0}$ 
for some/any $S_0$.

\end{rem}

\sssec{} \label{sss:funct !}

For a map $f:S_1\to S_2$ between affine schemes, we have a tautologically defined functor 
$$f^!:\Dmod^!(S_2)\to \Dmod^!(S_1).$$

Dually, we have a functor
$$f_*:\Dmod_*(S_1)\to \Dmod_*(S_2).$$

\sssec{} \label{sss:funct *}
 
Assume now that $f$ is of finite presentation. In this case we can define a functor
$$f_*:\Dmod^!(S_1)\to \Dmod^!(S_2),$$
which satisfies base change against !-pullbacks.
 
\medskip
 
Dually, we can define
$$f^!:\Dmod_*(S_2)\to \Dmod_*(S_1),$$
which satisfies base change against *-pushforwards. When $f$ is \'etale, this functor sends $\ul{k}_{S_2}\mapsto \ul{k}_{S_1}$. 

\sssec{}

Assume now that $f$ is a closed embedding of finite presentation. In this case the functors
$$f_*:\Dmod^!(S_1)\rightleftarrows \Dmod^!(S_2):f^!$$
and 
$$f_*:\Dmod_*(S_1)\rightleftarrows \Dmod_*(S_2):f^!$$
form adjoint pairs. 

\ssec{Extension to prestacks--the !-version} 

\sssec{} \label{sss:Dmod prestck}

The functoriality of $\Dmod^!(-)$ on affine schemes with respect to !-pullbacks allows us to extend the functor
$$\Dmod(-)^!:(\affSch)^{\on{op}}\to \DGCat$$
to arbitrary prestacks by the \emph{procedure of right Kan extension}. 

\medskip

Explicitly, for $\CY\in \on{PreStk}$, we have
\begin{equation} \label{e:Dmod! prestacks}
\Dmod^!(\CY)=\underset{S\to \CY}{\on{lim}}\, \Dmod^!(S),
\end{equation} 
where:

\begin{itemize}

\item The limit is taken over the (opposite of the) category of affine schemes mapping to $\CY$;

\smallskip

\item The transition functors are given by !-pullback.

\end{itemize}

\sssec{}

We shall say that an object of $\Dmod^!(\CY)$ is \emph{ind-holonomic} if its value on each $S\to \CY$
belongs to the ind-holonomic subcategory (see \secref{sss:ind-hol aff sch}).

\sssec{}

Since the transition functors in \eqref{e:Dmod! prestacks} are symmetric monoidal, the category 
$\Dmod^!(\CY)$ acquires a symmetric monoidal structure.

\medskip

Its unit is the object $\omega_\CY\in \Dmod^!(\CY)$, whose value on every $S\to \CY$ is
$\omega_S\in \Dmod^!(S)$. 

\sssec{}

By \secref{sss:funct *}, if $f:\CY_1\to \CY_2$ is map between prestacks that is affine and of finite presentation,
we have a well-defined functor
$$f_*:\Dmod(\CY_1)\to \Dmod(\CY_2)$$
that satisfies base change against !-pullbacks.

\medskip

Assume now that $f$ is a closed embedding of finite presentation, In this the functors
$$f_*:\Dmod^!(\CY_1)\rightleftarrows \Dmod^!(\CY_2):f^!$$
are an adjoint pair. 

\sssec{}

Let $f$ again be arbitrary. The functor $f^!$ has a partially defined left adjoint, to be denoted $f_!$. 

\medskip

In particular, we can consider the partially defined functor
$$\on{C}^\cdot_c(\CY,-):\Dmod^!(\CY)\to \Vect,$$
left adjoint to 
$$k\mapsto \omega_\CY.$$

\begin{lem} \label{l:ccs on ind-hol}
The functor $\on{C}^\cdot_c(\CY,-)$ is defined on ind-holonomic objects.
\end{lem}

\begin{proof}

We first consider the case when $\CY=S$ is an affine scheme. Let $\CF\in \Dmod^!(S)$ be obtained
as $g^!(\CF_0)$ for $g:S\to S_0$, where $S_0$ is an affine scheme of finite type and 
$\CF_0\in \Dmod^{\on{hol}}(S_0)$. 

\medskip

Then the sought-for value of $\on{C}^\cdot_c(S,\CF)$ is given by
$$\underset{S\to S'_0\overset{f}\to S_0}{\on{colim}}\, \on{C}^\cdot_c(S'_0,f^!(\CF_0)).$$

\medskip

Let now $\CY$ be a general prestack and $\CF\in \Dmod^!(\CY)$ be ind-holonomic. Then it is easy to see that 
$$\underset{f:S\to \CY}{\on{colim}}\, \on{C}^\cdot_c(S,f^!(\CF))$$
provides the value of the sought-for left adjoint.

\end{proof} 

\sssec{}

We will denote
$$\on{C}_\cdot(\CY):=\on{C}^\cdot_c(\CY,\omega_\CY).$$

Note that the dual of $\on{C}_\cdot(\CY)$ identifies with
$$\on{C}^\cdot(\CY)\simeq \End_{\Dmod^!(\CY)}(\omega_\CY)\simeq
\underset{S\to \CY}{\on{lim}}\, \End_{\Dmod^!(S)}(\omega_S)\simeq \underset{S\to \CY}{\on{lim}}\, \on{C}^\cdot(S).$$

\ssec{The case of (non-affine) schemes}

In this subsection we study/define the categories $\Dmod^!(-)$ and $\Dmod_*(-)$ on \emph{schemes}. 

\sssec{}

Let now $Y$ be a (not necessarily affine) scheme, assumed quasi-compact and separated. Let $S\to Y$ be a Zariski cover, where $S$ is affine, 
and let $S^\bullet$ be its Cech nerve. Consider $\Dmod^!(S^\bullet)$ as a cosimplicial category under !-pullbacks. 

\medskip

A standard argument shows that the restriction functor
$$\Dmod^!(Y) \to \on{Tot}(\Dmod^!(S^\bullet))$$
is an equivalence.

\medskip

In addition, since the morphism in $S^\bullet$ are of finite presentation, we can view $\Dmod^!(S^\bullet)$ as a \emph{simplicial}
category under *-pushforwards. The functor of *-pushforward gives rise to a functor
\begin{equation} \label{e:glue scheme *}
|\Dmod^!(S^\bullet)|\to \Dmod^!(Y),
\end{equation}
and another standard argument shows that \eqref{e:glue scheme *} is also an equivalence. 

\sssec{}

We continue to assume that $Y$ is a scheme. Note that we can consider $\Dmod_*(S^\bullet)$ as simplicial
(resp., cosimplicial) category via *-pushforwards (resp., !-pullbacks). 

\medskip

Set
$$\Dmod_*(Y) := |(\Dmod_*(S^\bullet)|.$$

A standard argument shows that the restriction functor
\begin{equation} \label{e:glue scheme !}
\Dmod_*(Y) \to \on{Tot}(\Dmod_*(S^\bullet)),
\end{equation}
given by !-pullback, is an equivalence.

\medskip

The category $\Dmod_*(Y)$ contains a canonically defined object, denoted $\ul{k}_Y$, whose value on the terms of $\Dmod_*(S^\bullet)$
is $\ul{k}_{S^\bullet}$ (see the last sentence in \secref{sss:funct *}). 

\medskip

From the equivalences \eqref{e:glue scheme *} and \eqref{e:glue scheme !}, we obtain that we have a canonical equivalence
$$\on{Funct}_{\DGCat}(\Dmod^!(Y),\Vect) \simeq \Dmod_*(Y),$$

\medskip 

One shows that the above constructions is canonically independent of the choice of the cover $S\to Y$.

\sssec{}

For a map $f:Y_1\to Y_2$ we have a naturally defined functor
$$f_*:\Dmod_*(Y_1)\to \Dmod_*(Y_2).$$

This functor can be also thought of as obtained from
$$f^!:\Dmod_*(Y_2)\to \Dmod_*(Y_1)$$
by duality. 

\sssec{}

We have a natural action of $\Dmod^!(Y)$ on $\Dmod_*(Y)$. 

\medskip

We shall call an object 
$$\omega_Y^{\on{fake},*}\in \Dmod_*(Y)$$
a \emph{fake} dualizing sheaf, if the functor
$$\Dmod^!(Y)\to \Dmod_*(Y), \quad \CF\mapsto \CF\sotimes \omega_Y^{\on{fake},*}$$
is an equivalence. 

\sssec{}

Let $Y$ be written
\begin{equation} \label{e:sch pres}
Y\simeq \underset{\alpha}{\on{lim}}\, Y_\alpha,
\end{equation} 
where the transition maps
$$Y_\alpha \overset{f_{\alpha,\beta}}\to Y_\beta$$
are affine. 

\medskip

In this case, it is easy to see that the functor
\begin{equation} \label{e:sch pres D!}
\underset{\alpha}{\on{colim}}\, \Dmod^!(Y_\alpha)\to \Dmod^!(Y),
\end{equation} 
defined by !-pullback, is an equivalence. 

\medskip

Similarly, the functor
\begin{equation} \label{e:sch pres D*}
\Dmod_*(Y)\to \underset{\alpha}{\on{lim}}\, \Dmod_*(Y_\alpha),
\end{equation} 
defined by *-pushforward, is an equivalence. 

\sssec{}

Let $f:Y_1\to Y_2$ be a map of finite presentation between schemes. 

\medskip

Using the equivalence \eqref{e:sch pres D!}, we construct
a functor
\begin{equation} \label{e:pushforward sch}
f_*:\Dmod^!(Y_1)\to \Dmod^!(Y_2),
\end{equation} 
which satisfies base change against $!$-pullbacks. 

\medskip

This construction allows us to extend the definition of *-pushforward 
on $\Dmod^!(-)$ for schematic maps between prestacks. 

\medskip

Similarly, we construct a functor 
$$f^!:\Dmod^*(Y_2)\to \Dmod^*(Y_1),$$
which can also be thought of as obtained from \eqref{e:pushforward sch}
by duality. 

\medskip

When $f$ is a closed embedding, the functors
$$f_*:\Dmod_*(Y_1)\rightleftarrows \Dmod_*(Y_2):f^!$$
form an adjoint pair. 

\sssec{}

\label{sss placid}

We shall say that $Y$ is \emph{placid} if there exists a presentation \eqref{e:sch pres}, where:

\begin{itemize}

\item The schemes $Y_\alpha$ are of finite type;

\smallskip

\item The maps $f_{\alpha,\beta}$ are smooth,

\end{itemize}

\medskip

In this case, the transition functors in \eqref{e:sch pres D!} preserve compactness, so $\Dmod^!(Y)$
is compactly generated, and in particular, dualizable. Hence, in this case $\Dmod_*(Y)$ is also compactly
generated and 
$$\Dmod_*(Y)\simeq (\Dmod^!(Y))^\vee.$$

\medskip

For a presentation \eqref{e:sch pres} as above, we can rewrite $\Dmod_*(Y)$ as 
\begin{equation} \label{e:D* placid}
\Dmod_*(Y)\simeq \underset{\alpha}{\on{colim}}\, \Dmod(Y_\alpha),
\end{equation} 
where the limit is taken with respect to the *-pullback functors (they are defined thanks to the smoothness
assumption).

\medskip

In terms of this presentation, the object $\ul{k}_Y$ equals to the image of $\ul{k}_{Y_\alpha}$ for any $\alpha$.

\sssec{}

Suppose that $Y$ is placid, and assume that in the presentation \eqref{e:sch pres}, the transition
maps $f_{\alpha,\beta}$ are equidimensional. 

\medskip

A dimension theory for such a presentation is an assignment
$$\alpha\mapsto d_\alpha\in \BZ$$ such that for a pair of indices and an arrow $\alpha\to \beta$ in the category of indices,
we have
$$d_\alpha-d_\beta=\on{dim.rel.}(f_{\alpha,\beta}).$$

Note that a choice of a dimension theory gives rise to a fake dualizing object $\omega_Y^{\on{fake},*}\in \Dmod_*(Y)$. Namely,
in terms of \eqref{e:D* placid}, $\omega_Y^{\on{fake},*}$ is the image of $\omega_{Y_\alpha}[-2d_\alpha]$ for 
some/any $\alpha$. 

\sssec{} \label{sss:smooth fake}

Note that if the schemes $Y_\alpha$ are smooth and connected, 
we have a distinguished dimension theory, given by $d_\alpha=\dim(Y_\alpha)$. In this case $\omega_Y^{\on{fake},*}=\ul{k}_Y$. 

\ssec{The case of ind-schemes}

\sssec{}

Let $\CY$ be an ind-scheme. Write
\begin{equation} \label{e:pres indsch}
\CY=\underset{i}{``\on{colim}"}\, Y_i,
\end{equation} 
where $Y_i$ are schemes, and the transition maps $Y_i\overset{f_{i,j}}\to Y_j$ are closed embeddings, and
the index category is filtered. 

\medskip

Recall that the category $\Dmod^!(\CY)$ is a priori defined, see \secref{sss:Dmod prestck}. Note, however, that the functor
\begin{equation} \label{e:Dmod! indsch}
\Dmod^!(\CY)\to  \underset{i}{\on{lim}}\, \Dmod^!(Y_i),
\end{equation} 
given by !-pullback is an equivalence, where the limit is formed also using the !-pullback functors. 

\sssec{}

We define
\begin{equation} \label{e:Dmod* indsch}
\Dmod_*(\CY):=\underset{i}{\on{colim}}\, \Dmod_*(Y_i),
\end{equation} 
where the colimit is taken with respect to *-pushforwards. 

\medskip

It is clear that the definition does not depend on the presentation of $\CY$ as in \secref{e:pres indsch}. 

\sssec{}

For a map between ind-schemes 
$$f:\CY_1\to \CY_2,$$
we obtain a well-defined functor
$$f_*:\Dmod_*(\CY_1)\to \Dmod_*(\CY_2).$$

In particular, we have a well-defined functor
$$\on{C}^\cdot(\CY,-):\Dmod(\CY)\to \Dmod(\on{pt})=\Vect.$$

\medskip

Suppose now that $f$ is schematic of finite presentation. In this case we also obtain a functor
$$f^!:\Dmod_*(\CY_2)\to \Dmod_*(\CY_1).$$

When $f$ is a closed embedding, the functors
$$f_*:\Dmod_*(\CY_1)\rightleftarrows \Dmod(\CY_2):f^!$$
form an adjoint pair. 

\sssec{}

We have a naturally defined action of $\Dmod^!(\CY)$ on $\Dmod_*(\CY)$. 

\medskip

We shall call an object 
$$\omega_\CY^{\on{fake},*}\in \Dmod_*(\CY)$$
a \emph{fake} dualizing sheaf, if the functor
$$\Dmod^!(\CY)\to \Dmod_*(\CY), \quad \CF\mapsto \CF\sotimes \omega_\CY^{\on{fake},*}$$
is an equivalence. 

\sssec{}

Following \cite[Sect. 7]{BD2}, we shall say that $\CY$ is \emph{reasonable} if it admits a presentation \eqref{e:pres indsch}, where the transition maps
$f_{i,j}$ are of finite presentation.

\medskip

If $\CY$ is reasonable, and $Z$ is a closed subscheme of $\CY$, we shall say that $Z$ is \emph{reasonable}
if for some/any $i$ such that $Z\subset Y_i$, this closed embedding is of finite presentation.

\medskip

Reasonable subschemes of $\CY$ form a filtered category. A presentation of $\CY$ as in \eqref{e:pres indsch}
with $Y_i$ reasonable will be called a \emph{reasonable presentation}.

\sssec{}

Let \eqref{e:pres indsch} be a reasonable presentation. In this case the functors
$$f_{i,j}^!:\Dmod^!(Y_j)\to \Dmod^!(Y_i)$$
admit left adjoints. 

\medskip

Hence, using the equivalence \eqref{e:Dmod! indsch}, we can rewrite
\begin{equation} \label{e:Dmod! indsch bis}
\Dmod^!(\CY)\simeq \underset{i}{\on{colim}}\, \Dmod^!(Y_i),
\end{equation} 
where the colimit is taken with respect to *-pushforwards. 

\medskip

Similarly, in the formation of the colimit in \eqref{e:Dmod* indsch}, the transition functors $(f_{i,j})_*$
admit right adjoints. Hence, we can rewrite 
\begin{equation} \label{e:Dmod* indsch bis}
\Dmod^*(\CY)\simeq  \underset{i}{\on{lim}}\, \Dmod^*(Y_i),
\end{equation} 
where the limit is formed using the !-pullback functors. 

\medskip

Using \eqref{e:Dmod! indsch bis} and \eqref{e:Dmod* indsch bis}, we obtain that for a reasonable ind-scheme $\CY$,
we still have an identification
$$\on{Funct}_{\DGCat}(\Dmod^!(\CY),\Vect) \simeq \Dmod_*(\CY).$$

\sssec{} \label{sss:ind-placid}

We shall say that an ind-scheme $\CY$ is \emph{ind-placid} if it admits a reasonable presentations 
who terms are placid schemes.

\medskip

Note that using \eqref{e:Dmod! indsch bis} we obtain that in this case $\Dmod^!(\CY)$ is compactly generated
(and hence dualizable) and $\Dmod_*(\CY)$ is also compactly generated, and we have
$$\Dmod_*(\CY)\simeq (\Dmod^!(\CY))^\vee.$$

\ssec{D-modules on the loop and arc groups} \label{ss:arcs and loops}

In this section we let $G$ be a connected affine algebraic group (i.e., we are not assuming $G$ to be reductive).

\medskip 

For the duration of this section we will unburden the notation and replace
$$\fL^+(G)_{x_0}\rightsquigarrow \fL^+(G),\,\, \fL(G)_{x_0}\rightsquigarrow \fL(G),\,\, \Gr_{G,x_0} \rightsquigarrow \Gr_G.$$

\sssec{}

Write 
$$\fL^+(G)\simeq \underset{n}{\on{lim}}\, \fL^+(G)/K_n.$$

This presentation exhibits $\fL^+(G)$ as a placid scheme. Thus we have:
$$\Dmod^!(\fL^+(G))\simeq \underset{n,(-)^!}{\on{colim}}\, \Dmod(\fL^+(G)/K_n)$$
and
$$\Dmod_*(\fL^+(G))\simeq  \underset{n,(-)_*}{\on{lim}}\, \Dmod(\fL^+(G)/K_n)\simeq \underset{n,(-)^*}{\on{colim}}\, \Dmod(\fL^+(G)/K_n).$$

We have a canonical identification
\begin{equation} \label{e:L+G duality}
\Dmod_*(\fL^+(G))\simeq (\Dmod^!(\fL^+(G)))^\vee.
\end{equation}

\sssec{} \label{sss:self-duality L+G}

That said, by \secref{sss:smooth fake}, we have a canonical identification
$$\Dmod^!(\fL^+(G))\overset{\sim}\to \Dmod_*(\fL^+(G)), \quad \omega_{\fL^+(G)}\mapsto \ul{k}_{\fL^+(G)}.$$

\medskip

Using this identification, we will simply write 
$$\Dmod^!(\fL^+(G)=:\Dmod(\fL^+(G)):=\Dmod_*(\fL^+(G)).$$

\sssec{}

The group structure on $\fL^+(G)$ makes $\Dmod^!(\fL^+(G))$ into a commutative Hopf algebra in $\DGCat$ (using !-pullbacks),
and it makes $\Dmod_*(\fL^+(G))$ into a cocommutative Hopf algebra in $\DGCat$ (under *-pushforwards). 

\medskip

The counit and unit in $\Dmod^!(\fL^+(G))$ are given by
$$\Dmod^!(\fL^+(G)) \overset{!\text{-fiber at }1}\longrightarrow \Dmod(\on{pt})=\Vect \text{ and }
\Vect \overset{k\mapsto \omega_{\fL^+(G)}}\longrightarrow \Dmod^!(\fL^+(G)),$$
respectively.

\medskip

The unit and counit in $\Dmod_*(\fL^+(G))$ are given by
$$\Vect \overset{k\mapsto \delta_1}\longrightarrow \Dmod_*(\fL^+(G)) \text{ and }
\Dmod_*(\fL^+(G)) \overset{\on{C}^\cdot(\fL^+(G),-)}\longrightarrow \Vect,$$
respectively, where $\delta_1$ is the *-direct image of $k\in \Dmod(\on{pt})$ under the unit map
$\on{pt}\to \fL^+(G)$. 

\medskip

These two structures are obtained from one another by duality, using \eqref{e:L+G duality}.

\begin{rem}

When using the notation $\Dmod(\fL^+(G))$, we will view it either as $\Dmod^!(\fL^+(G))$
(in the comonoidal incarnation) or $\Dmod_*(\fL^+(G))$ (in the monoidal incarnation), depending on the context. 

\end{rem}

\sssec{}

We now consider the case of $\fL(G)$. First off, we claim that it is ind-placid as an ind-scheme. Indeed, write
$$\Gr_G=\underset{i}{``\on{colim}"}\, Y_i,$$
where $Y_i\subset \Gr_G$ are closed $\fL^+(G)$-invariant subschemes. 

\medskip

Denote by $\wt{Y}_i$ the preimage of $Y_i$ in $\fL(G)$. The closed embeddings 
$$\wt{Y}_i\to \wt{Y}_j,$$
being obtained by base change from $Y_i\to Y_j$, are automatically of finite presentation.

\medskip

We claim that each $\wt{Y}_i$ is placid. Indeed, we can write it as
\begin{equation} \label{e:Yi left}
\underset{n}{\on{lim}}\, \wt{Y}_i/K_n,
\end{equation}
and/or
\begin{equation} \label{e:Yi right}
\underset{n}{\on{lim}}\, K_n\backslash \wt{Y}_i.
\end{equation}

Note that these two inverse families are automatically equivalent: for every $n$ there exists $n'$ such that the projection
$$\wt{Y}_i\to  \wt{Y}_i/K_n$$ factors as 
\begin{equation} \label{e:nn'}
\wt{Y}_i\to K_{n'}\backslash \wt{Y}_i\to \wt{Y}_i/K_n 
\end{equation}
and vice versa.

\sssec{}

Hence, we obtain that the categories 
$$\Dmod^!(\fL(G)) \text{ and } \Dmod_*(\fL(G))$$
are compactly generated and dual to each other.

\medskip

Explicitly,
$$\Dmod^!(\fL(G))\simeq \underset{i,(-)^!}{\on{lim}}\, \Dmod^!(\wt{Y}_i)\simeq \underset{i,(-)_*}{\on{colim}}\,  \Dmod^!(\wt{Y}_i),$$
while for every $i$,
$$\Dmod^!(\wt{Y}_i)\simeq \underset{n,(-)^!}{\on{colim}}\, \Dmod(\wt{Y}_i/K_n),$$
and 
$$\Dmod_*(\fL(G))\simeq \underset{i,(-)^!}{\on{lim}}\, \Dmod_*(\wt{Y}_i)\simeq \underset{i,(-)_*!}{\on{colim}}\,  \Dmod_*(\wt{Y}_i),$$
while for every $i$,
$$\Dmod_*(\wt{Y}_i)\simeq \underset{n,(-)_*}{\on{lim}}\, \Dmod(\wt{Y}_i/K_n)\simeq \underset{n,(-)^*}{\on{colim}}\, \Dmod(\wt{Y}_i/K_n).$$

In the above presentations of $\Dmod^!(\wt{Y}_i)$ and $\Dmod_*(\wt{Y}_i)$, one can replace the family $\wt{Y}_i/K_n$ by 
$K_n\backslash \wt{Y}_i$. 

\sssec{}

Note also that we have
\begin{equation} \label{e:Dmod! LG via Gr}
\Dmod^!(\fL(G))\simeq \underset{n,(-)^!}{\on{colim}}\, \Dmod(\fL(G)/K_n)
\end{equation}
and
\begin{equation} \label{e:Dmod* LG via Gr}
\Dmod_*(\fL(G))\simeq \underset{n,(-)_*}{\on{lim}}\, \Dmod(\fL(G)/K_n)\simeq \underset{n,(-)^*}{\on{colim}}\, \Dmod(\fL(G)/K_n),
\end{equation}
and we can also replace the family
$$n\mapsto \fL(G)/K_n$$
by
$$n\mapsto K_n\backslash \fL(G).$$

\sssec{}

Let
$$\omega^{\on{fake},*,L}_{\fL(G)}\in \Dmod_*(\fL(G))$$
be the object equal to the image of 
$$\omega_{\fL(G)/K_n}[-2\dim(\fL^+(G)/K_n)]\in \Dmod(\fL(G)/K_n)$$
for some/any $n$ under the presentation
$$\Dmod_*(\fL(G))\simeq \underset{n,(-)^*}{\on{colim}}\, \Dmod(\fL(G)/K_n).$$

It is clear that $\omega^{\on{fake},*,L}_{\fL(G)}$ is indeed a fake dualizing sheaf, i.e., it gives rise
to an equivalence
$$\Dmod^!(\fL(G))\to \Dmod_*(\fL(G)), \quad \CF\mapsto \CF\sotimes \omega^{\on{fake},*,L}_{\fL(G)}.$$

We define 
$$\omega^{\on{fake},*,R}_{\fL(G)}\in \Dmod_*(\fL(G))$$
similarly, using the presentation
$$\Dmod_*(\fL(G))\simeq \underset{n,(-)^*}{\on{colim}}\, \Dmod(K_n\backslash \fL(G)).$$

\sssec{}

Let $\mu:G\to \BG_m$ be the modular character (i.e., the determinant of the adjoint action). Let
$\deg(\mu)$ be the corresponding function
$$\pi_0(\fL(G))\overset{\mu}\to \pi_0(\fL(\BG_m))\simeq \BZ.$$

\begin{prop} \label{p:LR omega}
The objects $\omega^{\on{fake},*,L}_{\fL(G)}$ and $\omega^{\on{fake},*,R}_{\fL(G)}[2\deg(\mu)]$ are canonically isomorphic. 
\end{prop}

\begin{proof}

For every $i$, let
\begin{equation} \label{e:fake i}
\omega^{\on{fake},*,L}_{\wt{Y}_i} \text{ and } \omega^{\on{fake},*,R}_{\wt{Y}_i}
\end{equation} 
be the objects of $\Dmod_*(\wt{Y}_i)$ defined by a similar procedure, using the presentation of $\wt{Y}_i$ 
as \eqref{e:Yi left} and \eqref{e:Yi right}, respectively.

\medskip

It suffices to exhibit a compatible family of isomorphisms
\begin{equation} \label{e:LR i}
\omega^{\on{fake},*,L}_{\wt{Y}_i}  \simeq \omega^{\on{fake},*,R}_{\wt{Y}_i}[2\deg(\mu)].
\end{equation}

Note that the objects \eqref{e:fake i} are associated to the two dimension theories on $\wt{Y}_i$: one
attaches to 
$$\wt{Y}_i\to \wt{Y}_i/K_n$$ 
the integer $d_n:=\dim(\fL^+(G)/K_n)$, and another attaches the same integer to
$$\wt{Y}_i\to K_n\backslash \wt{Y}_i.$$

It suffices to show that these two dimension theories are in fact equivalent up to the shift by $\deg(\mu)$. 

\medskip

To do this, fix an integer $n$ and let $n'$ be as in \eqref{e:nn'}. The required equality follows from the fact that the resulting map
$$K_{n'}\backslash \wt{Y}_i\to \wt{Y}_i/K_n$$
is smooth of relative dimension $\dim(K_n/K_{n'})+\deg(\mu)$. 

\end{proof}

\sssec{} \label{sss:fake omega loop}

Assume that $G$ is unimodular. In this case, using the identification of \propref{p:LR omega}, we will use the notation
$$\omega^{\on{fake},*,L}_{\fL(G)}=:\omega^{\on{fake},*}_{\fL(G)}:=\omega^{\on{fake},*,R}_{\fL(G)}.$$

Thus, $\omega^{\on{fake},*}_{\fL(G)}$ gives rise to a canonical identification 
\begin{equation} \label{e:self-duality loop}
\Dmod^!(\fL(G)) \simeq \Dmod^*(\fL(G)).
\end{equation} 

We will use the notation:
$$\Dmod^!(\fL(G))=:\Dmod(\fL(G)):=\Dmod^*(\fL(G)).$$

\sssec{}

The structure of group-object in ind-schemes defines on $\Dmod^!(\fL(G))$ a structure of commutative Hopf algebra,
and on $\Dmod_*(\fL(G))$ a structure of cocommutative Hopf algebra. 

\medskip

These two structures are obtained from one another
by duality.

\begin{rem}

When using the notation $\Dmod(\fL(G))$, we will view it either as $\Dmod^!(\fL(G))$
(in the comonoidal incarnation) or $\Dmod_*(\fL(G))$ (in the monoidal incarnation), depending on the context. 

\end{rem}

\section{Categorical representations of (loop) groups} \label{s:1-cat rep}

In this section we (re)collect some facts pertaining to the notion of action of a group on a category.
We start with groups of finite type, and then develop the theory for the arc and loop groups. 

\ssec{The case of groups of finite type} \label{ss:fd groups}

In this subsection, we let $H$ be an affine algebraic group (of finite type). 

\sssec{}

We consider $\Dmod(H)$ as a monoidal category in $\DGCat$ under convolution. We set
$$H\mmod:=\Dmod(H)\mmod.$$

\medskip

Note that Verdier duality identifies $\Dmod(H)$ with its own dual. This allows us to view $\Dmod(H)$
as comonoidal category via 
$$\Dmod(H)\overset{\on{mult}^!}\longrightarrow \Dmod(H\times H)\simeq \Dmod(H)\otimes \Dmod(H).$$

We can tautologically interpret $H\mmod$ as $\Dmod(H)\commod$ for this structure.

\sssec{} \label{sss:sym mon H-mmod}

Push-forward (resp., pullback) with respect to the diagonal map extends the above monoidal (resp., comonoidal) 
structure on $\DGCat$ to a structure of cocommutative (resp., commutative)  Hopf algebra object in $\DGCat$. 
This structure endows
$H\mmod$ with a symmetric monoidal structure, compatible with the forgetful functor
\begin{equation} \label{e:forget H-mmod}
H\mmod\to \DGCat.
\end{equation} 

\medskip

The unit of this symmetric monoidal structure is a copy of $\Vect$, equipped with the trivial action of $H$, i.e. the coaction map
$$\Vect \to \Dmod(H)\otimes \Vect\simeq \Dmod(H)$$
is the functor $k\mapsto \omega_H$.

\medskip

The corresponding action map 
$$\Dmod(H)\simeq \Dmod(H)\otimes \Vect\to \Vect$$
is the functor of de Rham cochains, denoted $\on{C}^\cdot(H)$. 

\sssec{}

As is the case for modules over Hopf algebras, an object of $H\mmod$ is dualizable if and only if its
image under \eqref{e:forget H-mmod} is dualizable, and the functor \eqref{e:forget H-mmod}, being 
symmetric monoidal, commutes with duality.

\sssec{} \label{sss:adj H-mmod}

Let $F:\bC_1\to \bC_2$ be 1-morphism in $H\mmod$, and suppose that $F$ admits a left (resp., right) adjoint,
when viewed as a 1-morphism in $\DGCat$. We claim that this adjoint then exists in $H\mmod$.

\medskip

Indeed, a priori $F^L$ (resp., $F^R$) will be op-lax (resp., lax) compatible with the action of $\Dmod(H)$.
However, it is easy to see that the fact that $H$ is is a group (as opposed to a monoid) forces any
op-lax (resp., lax) compatible functor between module categories to be strictly compatible.

\medskip

Indeed, for a point $S$-point $h$ of $H$ and a lax-compatible functor $F:\bC_1\to \bC_2$, the morphism
$$F\circ h\overset{\alpha_h}\to h\circ F$$
(viewed as a natural transformation beween functors $\bC_1\otimes \Dmod(S)\rightrightarrows \bC_2\otimes \Dmod(S)$)
admits an inverse given by
$$h\circ F\simeq h\circ F\circ h^{-1}\circ h \overset{\alpha_{h^{-1}}}\to h\circ h^{-1}\circ F\circ h\simeq F\circ h,$$
and similarly for op-lax functors. 

\sssec{}

Note that the Verdier duality identification
\begin{equation} \label{e:Verdier duality H}
\Dmod(H)^\vee\simeq \Dmod(H)
\end{equation}
is compatible with the structure of object of $(H\times H)\mmod$ on the two sides. 

\medskip

In particular, the left and right adjoints of the forgetful functor \eqref{e:forget H-mmod} are canonically isomorphic. 

\begin{rem} \label{r:Frobenius fd}

Note that we can also interpret \eqref{e:Verdier duality H} as follows: the functor
$$\Dmod(H)\otimes \Dmod(H) \overset{\text{convolution}}\longrightarrow \Dmod(H)  
\overset{\text{counit}}\longrightarrow \Vect$$
is a duality pairing, which differs from the Verdier one by the inversion operation on $H$.

\medskip

This pairing makes $\Dmod(H)$ into a Frobenius algebra in $\DGCat$. 

\end{rem} 

\sssec{}

For $\bC\in H\mmod$, set
$$\bC^H:=\on{Funct}_{H\mmod}(\Vect,\bC).$$

We have a pair of adjoint functors
$$\oblv_H:\bC^H\rightleftarrows \bC:\on{Av}^H_*.$$

\sssec{}

Consider the category $\Vect$. The forgetful functor
$$\oblv_H:\Vect^H\to \Vect$$
is comonadic with the comonad in question is given by tensor product with $\on{C}^\cdot(H)$, on
which the coalgebra structure is induced by the group structure on $H$.

\medskip

Equivalently,
$$\Vect^H\simeq \on{C}_\cdot(H)\mod,$$
where $\on{C}_\cdot(H)$ is an algebra via the group structure on $H$.

\medskip

The category $\Vect^H$ carries a natural (symmetric) monoidal structure. It corresponds to the structure
on $\on{C}^\cdot(H)$ (resp., $\on{C}_\cdot(H)$) of commutative (resp., cocommutative) Hopf algebra.

\sssec{}

For any $\bC$, the category $\bC^H$ is naturally tensored over $\Vect^H$. The comonad 
$$\oblv_H\circ \on{Av}^H_*$$
on $\bC^H$ is given by tensoring with the coalgebra object 
$$ \on{Av}^H_*(k)\in \Vect^H.$$

Note that the underlying object of $\Vect$, i.e., $\oblv_H\circ \on{Av}^H_*(k)$, identifies with $\on{C}^\cdot(H)$. 

\medskip

In particular, this implies that the functor $\oblv_H:\bC^H\to \bC$ is fully faithful if $H$ is unipotent.

\sssec{}

For $\bC\in H\mmod$, let $\bC_H\in \DGCat$ be the object such that
$$\on{Funct}_{\DGCat}(\bC_H,\bC_0)=\on{Funct}_{H\mmod}(\bC,\bC_0), \quad \bC_0\in \DGCat,$$
where $H$ acts trivially on $\bC_0$.

\medskip

Note that the functor
$$\on{Av}^H_*:\bC\to \bC^H$$
is $H$-invariant, and hence gives rise to a functor
\begin{equation} \label{e:inv to coinv}
\bC_H\to \bC^H.
\end{equation}

\sssec{}

We have the following basic assertion, see, e.g., \cite[Theorem B.1.2]{Ga4}:
\begin{prop}
The functor \eqref{e:inv to coinv} is an equivalence.
\end{prop}

\begin{cor} \label{c:inv self-dual} 
Let $\bC\in H\mmod$ be dualizable as a DG category. Then $\bC^H$ is also dualizable and 
we have a canonical equivalence
$$(\bC^H)^\vee\simeq (\bC^\vee)^H.$$
\end{cor}

\begin{cor} \label{c:ten vs coten over H}
For $\bC_1,\bC_2\in H\mmod$, we have a canonical equivalence
$$\bC_1\underset{\Dmod(H)}\otimes \bC_2\simeq (\bC_1\otimes \bC_2)^H.$$
\end{cor}

\begin{cor} \label{c:inv and colimits}
The functor 
$$\bC\mapsto \bC^H, \quad H\mmod\to \DGCat$$
commutes with colimits and tensor products by objects of $\DGCat$.
\end{cor} 

\ssec{The case of arc groups}

In this and the next subsections we keep the notational change from \secref{ss:arcs and loops}. 

\sssec{}

We set
$$\fL^+(G)\mmod:=\Dmod^!(\fL^+(G))\commod \simeq \Dmod_*(\fL^+(G))\mmod.$$

The entire discussion in \secref{ss:fd groups} is applicable in this case. Moreover, as we shall 
see below, the study of $\fL^+(G)\mmod$ reduces to that of $H\mmod$ for finite-dimensional
quotients $\fL^+(G)\rightrightarrows H$. 

\sssec{}

For $\bC\in \fL^+(G)\mmod$ and an integer $n$, let $\be_n$ be the endofunctor of $\bC$
equal to $\oblv_{K_n}\circ \on{Av}^{K_n}_*$. 

\medskip

Note that $\be_n$ can be thought of as the action of 
$$\ul{k}_{K_n}\in \Dmod_*(\fL^+(G)).$$

\medskip

The essential image of $\be_n$ is the \emph{full subcategory}
$$\bC^{K_n}\subset \bC.$$

Since $K_n$ is normal in $\fL^+(G)$, the category $\bC^{K_n}$ is stable under the $\fL^+(G)$-action, i.e., forms
a subobject in $\fL^+(G)\mmod$. Moreover, the action of $\fL^+(G)$ on $\bC^{K_n}$ factors through 
$\fL^+(G)/K_n$. 

\sssec{}

Consider the colimit
\begin{equation} \label{e:inv under congr}
\underset{n}{\on{colim}}\, \bC^{K_n},
\end{equation}
formed in $\fL^+(G)\mmod$, where the transition functors are the natural inclusions. 

\medskip

Since the index family (i.e., $\bZ^{\geq 0}$) is filtered, the underlying DG category of \eqref{e:inv under congr} 
is a similar colimit taken in $\DGCat$.

\medskip

Note also that since for $n_1\leq n_2$, the embedding
$$\bC^{K_{n_1}}\hookrightarrow \bC^{K_{n_2}}$$
admits a right adjoint, we can rewrite \eqref{e:inv under congr} also as
\begin{equation} \label{e:inv under congr bis}
\underset{n}{\on{lim}}\, \bC^{K_n},
\end{equation}
where the limit is taken with respect to the above averaging functors.

\sssec{}

We have a tautologically defined map in $\fL^+(G)\mmod$
\begin{equation} \label{e:inv under congr to C}
\underset{n}{\on{colim}}\, \bC^{K_n} \to \bC.
\end{equation} 

We claim:

\begin{prop}
The map \eqref{e:inv under congr to C} is an isomorphism.
\end{prop}

\begin{proof}

Since the forgetful functor
$$\fL^+(G)\mmod\to \DGCat$$ 
is conservative, it suffices to show that the functor \eqref{e:inv under congr to C} is an equivalence
of DG categories. 

\medskip

Since the category of indices is filtered, and each $\bC^{K_n}\to \bC$ is fully faithful
admitting a colimit-preserving right adjoint, it follows automatically that the functor
\eqref{e:inv under congr to C} is fully faithful.  

\medskip

Hence, it remains to show that it is essentially surjective. I.e., it suffices to exhibit any object
of $\bC$ as a colimit of objects belonging to $\bC^{K_n}$ for some $n$. In fact, we claim that 
for any $\bc\in \bC$, the naturally defined map
\begin{equation} \label{e:object as limit}
\underset{n}{\on{colim}}\, \be_n(\bc)\to \bC
\end{equation} 
is an isomorphism.

\medskip

Indeed, \eqref{e:object as limit} follows from the isomorphism
$$\underset{n}{\on{colim}}\, \ul{k}_{K_n}\simeq \delta_1$$
in $\Dmod_*(\fL^+(G))$. 

\end{proof} 

\ssec{The case of loop groups} \label{ss:cat rep loop group}

We finally consider the case of $\fL(G)$, and define what we mean by categories acted on by it.

\sssec{}

We define 
$$\fL(G)\mmod:=\Dmod^!(\fL(G))\commod=\Dmod_*(\fL(G))\mmod.$$

Remarks as in \secref{sss:sym mon H-mmod}-\ref{sss:adj H-mmod} apply
equally well to $\fL(G)\mmod$.

\sssec{}

For $\bC\in \fL(G)\mmod$, set
$$\bC^{\fL(G)}:=\on{Funct}_{\fL(G)\mmod}(\Vect,\bC).$$

In particular, we can consider the (symmetric monoidal) category $\Vect^{\fL(G)}$. 

\sssec{}

We claim:

\begin{lemconstr}
There is a canonical equivalence of symmetric monoidal categories
$$\Vect^{\fL(G)}\simeq \on{C}_\cdot(\fL(G))\mod,$$
compatible with the forgeftul functors to $\Vect$. 
\end{lemconstr}

\begin{proof}

For $V\in \Vect$, its lift to an object of $\Vect^{\fL(G)}$ is a datum of structure of an object of
$$\on{Tot}(\Dmod^!(\fL(G)^\bullet))$$
with terms $\omega_{\fL(G)^\bullet}\otimes V$, where $\fL(G)^\bullet$ is the simplicial ind-scheme
constructed out of $\fL(G)$, viewed as a group-object in ind-schemes.

\medskip

Such a datum amounts to an isomorphism
\begin{equation} \label{e:V omega}
\omega_{\fL(G)}\otimes V\simeq \omega_{\fL(G)}\otimes V,
\end{equation} 
in $\Dmod^!(\fL(G))$, satisfying the natural associativity conditions.

\medskip

The map $\to$ in \eqref{e:V omega} gives rise by adjunction to a map
\begin{equation} \label{e:V omega bis}
\on{C}_\cdot(\fL(G))\otimes V=\on{C}^\cdot_c(\fL(G),\omega_{\fL(G)})\otimes V\to V,
\end{equation} 
and the associativity datum on \eqref{e:V omega} is equivalent to the associativity datum for
\eqref{e:V omega bis}.

\medskip

Vice versa, starting from \eqref{e:V omega bis} we construct a map $\to$ \eqref{e:V omega},
and the associativity for \eqref{e:V omega bis} implies that \eqref{e:V omega} is automtically
an isomorphism.

\end{proof} 

\sssec{}

Consider $\Dmod_*(\fL(G))$ as an object of $(\fL(G)\times \fL(G))\mmod$. This can be either done by viewing
$\Dmod_*(\fL(G))$ as a bimodule over itself. Or, equivalently, we can view $\Dmod_*(\fL(G))$ as the dual 
(inside $(\fL(G)\times \fL(G))\mmod$) of $\Dmod^!(\fL(G))$, viewed as a bi-comodule over itself. 

\medskip

Let us now assume that $G$ is unimodular. Recall the object 
$$\omega^{\on{fake},*}_{\fL(G)}\in \Dmod_*(\fL(G)),$$
see \secref{sss:fake omega loop}.

\medskip

Its interpretation as $\omega^{\on{fake},*,L}_{\fL(G)}$ implies that it is naturally an object of
$$\Dmod_*(\fL(G))^{\fL(G)-L},$$
where the superscript ``L" refers to the left action of $\fL(G)$ on $\Dmod_*(\fL(G))$.

\medskip

Similarly, its interpretation as $\omega^{\on{fake},*,R}_{\fL(G)}$ implies that it is naturally an object of
$$\Dmod_*(\fL(G))^{\fL(G)-R}.$$

\medskip

However, we claim:

\begin{lem} \label{l:fake LG bi inv}
The object $\omega^{\on{fake},*}_{\fL(G)}$ naturally lifts to an object of
$$\Dmod_*(\fL(G))^{\fL(G)\times \fL(G)}.$$
\end{lem}

\begin{proof} 

We have to show that for points $g_1,g_2\in \fL(G)$, the diagram
\begin{equation} \label{e:fake LG bi inv}
\CD
(g_1\cdot \omega^{\on{fake},*}_{\fL(G)})\cdot g_2 @>{\sim}>> g_1\cdot (\omega^{\on{fake},*}_{\fL(G)})\cdot g_2) \\
@V{\sim}VV @VV{\sim}V \\
(g_1\cdot \omega^{\on{fake},*,L}_{\fL(G)})\cdot g_2 & & g_1\cdot (\omega^{\on{fake},*,R}_{\fL(G)})\cdot g_2) \\
@V{\sim}VV @VV{\sim}V \\
\omega^{\on{fake},*,L}_{\fL(G)}\cdot g_2 & & g_1\cdot \omega^{\on{fake},*,R}_{\fL(G)} \\
@V{\sim}VV @VV{\sim}V \\
\omega^{\on{fake},*,R}_{\fL(G)}\cdot g_2 & & g_1\cdot \omega^{\on{fake},*,L}_{\fL(G)}  \\
@V{\sim}VV  @VV{\sim}V \\
\omega^{\on{fake},*,R}_{\fL(G)} @>>>  \omega^{\on{fake},*,L}_{\fL(G)} 
\endCD
\end{equation}
commutes, along with the higher compatibilities. 

\medskip

First, we claim that the higher compatibilities hold automatically, since 
$$\CHom(\omega^{\on{fake},*,L}_{\fL(G)},\omega^{\on{fake},*,L}_{\fL(G)})\simeq \on{C}^\cdot(\fL(G))$$
is coconnective. 

\medskip

Hence, we only need to check the commutation of \eqref{e:fake LG bi inv} up to homotopy. However, the latter 
follows immediately from the construction of the isomorphism of \propref{p:LR omega}.

\end{proof}

As a corollary, we obtain: 

\begin{cor}
The isomorphism \eqref{e:self-duality loop} lifts to an identification of objects in $(\fL(G)\times \fL(G))\mmod$.
\end{cor}

\begin{cor}
The left and right adjoints to the forgetful functor 
$$\fL(G)\mmod\to \DGCat$$
are canonically isomorphic.
\end{cor}

\begin{rem} \label{r:Frobenius loop}

Remark \ref{r:Frobenius fd} applies here as well for the pairing
$$\Dmod_*(\fL(G))\otimes \Dmod_*(\fL(G))\overset{\text{convolution}}\longrightarrow \Dmod_*(\fL(G))
\simeq \Dmod^!(\fL(G))\overset{\on{counit}}\longrightarrow \Vect.$$

This pairing makes $\Dmod_*(\fL(G))$ into a Frobenius algebra in $\DGCat$.

\end{rem}

\ssec{The case of a reductive \texorpdfstring{$G$}{G}}

Up to now, the discussion of $\fL(G)$ was valid for \emph{any} affine algebraic group $G$. We will now assume that
$G$ is reductive, and exploit some special features of this case. The key to what we are about to say is that for $G$ 
reductive, the affine Grassmannian
$$\Gr_G:=\fL(G)/\fL^+(G)$$
is ind-proper.

\sssec{} \label{sss:Av Gr App}

We claim:

\begin{prop} 
The forgetful functor 
$$\oblv_{\fL(G)\to \fL^+(G)}:\bC^{\fL(G)}\to \bC^{\fL^+(G)}$$
admits a left adjoint, to be denoted $\on{Av}^{\fL^+(G)\to \fL(G)}_!$. 
\end{prop} 

\begin{proof}

Note that we can interpret $\bC^{\fL^+(G)}$ as
$$(\Dmod(\Gr_G)\otimes \bC)^{\fL(G)}.$$

Hence, it suffices to show that the functor
$$\Vect\to \Dmod(\Gr_G), \quad k\mapsto \omega_{\Gr_G}$$ 
admits a left adjoint in $\fL(G)\mmod$, or equivalently, in $\Vect$.

\medskip

However, the left adjoint in question is provided by $\on{C}^\cdot(\Gr_G,-)$.

\end{proof} 

\sssec{}

Set
$$\Sph_G:=\Dmod_*(\fL(G))^{\fL^+(G)\times \fL^+(G)}.$$

The usual Hecke algebra construction shows that the monoidal structure on $\Dmod_*(\fL(G))$ induces 
one on $\Sph_G$. Moreover, for every $\bC\in \fL(G)\mmod$, we have a canonical action of $\Sph_G$ on
$\bC^{\fL^+G}$. 

\sssec{} \label{sss:conv omega}

Note that we can consider $\omega^{\on{fake},*}_{\fL(G)}$ as an associative algebra object of $\Sph_G$; 
when viewed as such, let us denote it by $\omega^{\on{fake},*}_{\on{Hecke}}$.

\medskip

Unwinding, we obtain that the monad
\begin{equation} \label{e:Av ! monad}
\oblv_{\fL(G)\to \fL^+(G)}\circ \on{Av}^{\fL^+(G)\to \fL(G)}_!
\end{equation} 
acting on $\bC^{\fL^+(G)}$ is given by the action of $\omega^{\on{fake},*}_{\on{Hecke}}$ . 

\medskip

Therefore, by the Barr-Beck-Lurie theorem, we can identify
$$\bC^{\fL(G)}\simeq \omega^{\on{fake},*}_{\on{Hecke}}\mod(\bC^{\fL^+(G)}).$$

%
%
%

\sssec{}

Let $\bC$ be an object $\fL(G)\mmod$. Let $\bC_{\fL(G)}$ be the object of $\DGCat$ such that
$$\on{Funct}_{\DGCat}(\bC_{\fL(G)},\bC_0)\simeq \on{Funct}_{\fL(G)\mmod}(\bC,\bC_0), \quad \bC_0\in \DGCat,$$
where in the right-hand side, $\bC_0$ is considered as equipped with the trivial $\fL(G)$-action. 

\sssec{}

For $\bC\in \fL(G)\mmod$ and $\bC_0\in \DGCat$, let us view $\on{Funct}_{\DGCat}(\bC,\bC_0)$
as an object of $\fL(G)\mmod$ via the action on the source. 

\medskip

By \propref{e:inv to coinv} (applied to $\fL^+(G)$), we have
\begin{equation} \label{e:Sph in Hom}
\on{Funct}_{\DGCat}(\bC,\bC_0)^{\fL^+(G)}\simeq \on{Funct}_{\DGCat}(\bC^{\fL^+(G)},\bC_0).
\end{equation} 

Unwinding, we obtain that the action of $\omega^{\on{fake},*}_{\on{Hecke}}$ on the left-hand side
corresponds to the action of $\omega^{\on{fake},*}_{\on{Hecke}}$ on the source in the right-hand side.

\sssec{}

Take $\bC_0=\bC^{\fL(G)}$, and consider the object of the category $\on{Funct}_{\DGCat}(\bC^{\fL^+(G)},\bC^{\fL(G)})$
given by $\on{Av}^{\fL^+(G)\to \fL(G)}_!$. Since the functor $\on{Av}^{\fL^+(G)\to \fL(G)}_!$ is acted on by
the monad \eqref{e:Av ! monad}, we obtain that the above object lifts to an object of
\begin{multline*} 
\omega^{\on{fake},*}_{\on{Hecke}}\mod(\on{Funct}_{\DGCat}(\bC^{\fL^+(G)},\bC_0))\simeq \\
\simeq \omega^{\on{fake},*}_{\on{Hecke}}\mod(\on{Funct}_{\DGCat}(\bC,\bC_0)^{\fL^+(G)})\simeq 
\on{Funct}_{\DGCat}(\bC,\bC_0)^{\fL(G)}.
\end{multline*} 

Thus, the functor $\on{Av}^{\fL^+(G)\to \fL(G)}_!$ may be viewed as a point in
$$\on{Funct}_{\fL(G)\mmod}(\bC,\bC^{\fL(G)}) \simeq \on{Funct}(\bC_{\fL(G)},\bC^{\fL(G)}).$$

\sssec{}

We have (see \cite[Theorem D.1.4(b)]{Ga4}):

\begin{prop} \label{p:inv vs coinv loop}
The above functor
$$\bC_{\fL(G)}\to \bC^{\fL(G)}.$$
is an equivalence. 
\end{prop}

\sssec{}

From \propref{p:inv vs coinv loop} we obtain:

\begin{cor}
Let $\bC\in \fL(G)\mmod$ be dualizable as a category. Then $\bC^{\fL(G)}$ is also dualizable and 
we have a canonical equivalence
$$(\bC^{\fL(G)})^\vee\simeq (\bC^\vee)^{\fL(G)}.$$
\end{cor}

\begin{cor} \label{c:tensor over LG}
For $\bC_1,\bC_2\in \fL(G)\mmod$, we have a canonical isomorphism
$$\bC_1\underset{\Dmod_*(\fL(G))}\otimes \bC_2\simeq (\bC_1\otimes \bC_2)^{\fL(G)}.$$
\end{cor}

\begin{cor} \label{c:inv and colimits LG}
The functor 
$$\bC\mapsto \bC^{\fL(G)}, \quad \fL(G)\mmod\to \DGCat$$
commutes with colimits and tensor products by objects of $\DGCat$.
\end{cor} 

\begin{rem}

One can rewrite the natural transformation in \propref{p:inv vs coinv loop} also as follows.

\medskip

Let $G$ be an arbitrary unimodular group. As is the case for an arbitrary augmented algebra $A$,
which is isomorphic to its own dual as a bimodule, there is a canonical natural transformation
\begin{equation} \label{e:inv to coinv map A}
M_A\to M^A, \quad M\in A\mod.
\end{equation} 

Unwinding, one can see that \eqref{e:inv to coinv loop} is the map \eqref{e:inv to coinv loop} for
$A:=\Dmod_*(\fL(G))\in \DGCat$. 

Note, however, that the resulting natural transformation
\begin{equation} \label{e:inv to coinv loop}
\bC_{\fL(G)}\to \bC^{\fL(G)}
\end{equation} 
will \emph{not} in general be an isomorphism, unless $G$ is reductive. 

\medskip

For example, if $G=N$ is unipotent and $\bC=\Vect$, the natural transformation \eqref{e:inv to coinv loop}
will be zero.

\end{rem}

\newpage 

\section{Factorization categories and modules} \label{s:fact cat} Most of the material in this appendix, except \secref{ss proof of thm the fibration} - \secref{ss fact via operad}, follows \cite[Sect. C]{GLC2}. The reader is referred to \emph{loc.cit.} for more details.

\ssec{Unital factorization spaces}

\label{ss unital factspc}

\sssec{}
\label{sss categorical prestack}

A \emph{categorical prestack} is a functor 
\[
    (\on{Sch}^{\on{aff}})^{\on{op}} \to \on{Cat}.
\]
Let 
\[
    \on{CatPreStk} := \on{Fun}(  (\on{Sch}^{\on{aff}})^{\on{op}},\on{Cat} )
\]
be the \emph{2-category} of categorical prestacks.

\sssec{}
  \label{sss maximal subprestack}
Given a categorical prestack $\CY$, its value at an affine scheme $S$ is denoted by $\CY(S)$. Let $\CY(S)^\simeq$ be the maximal subgroupoid of $\CY(S)$. We see that $S\mapsto \CY(S)^\simeq$ defines a (non-categorical) prestack
\[
    \CY^\simeq \in \on{PreStk} := \on{Fun}(  (\on{Sch}^{\on{aff}})^{\on{op}},\on{Spc} ).
\]

\medskip 

By Yoneda's lemma, objects in $\CY(S)$ are identified with morphisms $S\to \CY$, while morphisms in $\CY(S)$ can be identified with 2-morphisms
\[
      \begin{tikzcd}
    S \ar[r, bend  left=60, ""{name=U}]
      \ar[r, bend right=60, ""{name=D,below}]
            & \CY \ar[Rightarrow, from=U, to=D, "\varphi" description]
    \end{tikzcd}
\]
in $\on{CatPreStk}$.

\sssec{} \label{sss:unital Ran space}

The \emph{unital Ran space} is the categorical prestack $\Ran^\untl$ that attaches to an affine test scheme $S$ the \emph{category} of finite subsets of $\Hom(S,X)$, where morphisms are given by inclusions of subsets.

\medskip
In particular, a $k$-point $\underline{x}\in \Ran^\untl$ is just a finite subset of closed points of $X$. 

\sssec{Remark}
We warn the readers that we do not require $\underline{x}$ to be nonempty. In particular, there is a canonical point $\emptyset\in \Ran^\untl$ corresponding to the empty subset. See \secref{ss fact via operad} for the reason we make this choice.

\medskip
Note that $ \Ran^\untl$ is denoted $\Ran^{\on{untl},*}$ in \cite[Sect. C.5]{GLC2}.

\sssec{}
  \label{sss Cart catprestk}
Let $\CY$ be a categorical prestack. A \emph{coCartesian space over $\CY$} is a categorical prestack $\CZ \to \CY$ such that for any affine scheme $S$, the functor $\CZ(S) \to \CY(S)$ is a left fibration, i.e., a coCartesian fibration in groupoids. We also say $\CZ\to \CY$ is a \emph{coCartesian morphism of categorical prestacks}.

\medskip 
Dually, we define the notion of \emph{Cartesian spaces over $\CY$}. 

\sssec{Remark}
	Roughly speaking, a coCartesian space $\CT_{\Ran^\untl}$ over $\Ran^\untl$ is an assignment as follows.
	\begin{itemize}
		\item 
			For any point $\ul{x}\in \Ran^\untl$, a (usual) prestack $\CT_{\ul{x}}$;
		\item 
			For any $\ul{x}\subseteq \ul{x}'$, a morphism $\CT_{\ul{x}}\to \CT_{\ul{x}'}$
			that is compatible with compositions.
	\end{itemize}
	The above data should depend ``algebraically'' on $\underline{x}$ and $\underline{x}'$.

\sssec{Example}

  \label{sss GrG untl}
	We have a coCartesian space $\Gr_{G,\Ran^\untl}$ over $\Ran^\untl$ such that $\Gr_{G,\ul{x}}$ is as in \secref{sss GrG factspc}, while the morphism
	\[
		\Gr_{G,\ul{x}} \to \Gr_{G,\ul{x}'}\;
	\]
	(for $\ul{x}\subseteq\ul{x}'$) sends $(\CP_G^{\on{glob}},\beta)$ to $(\CP_G^{\on{glob}},\beta|_{X\setminus \ul{x}'})$.

\sssec{Example}

	We have a Cartesian space $\fL^+(G)_{\Ran^\untl}$ over $\Ran^\untl$ such that $\fL^+(G)_{\ul{x}}:=G( \mathcal{D}_{\ul{x}}  )$, while the morphism
	\[
		\fL^+(G)_{\ul{x}'} \to \fL^+(G)_{\ul{x}}
	\]
	(for $\ul{x}\subseteq\ul{x}'$) is induced by restriction along $\mathcal{D}_{\ul{x}} \subseteq \mathcal{D}_{\ul{x}'}  $.

\sssec{}

Let $\CY$ be a categorical prestack. There is a common generalization of coCartesian spaces and Cartesian spaces over $\CY$. A \emph{corr-space over $\CY$} is a categorical prestack $\CZ \to \CY$ such that for any affine scheme $S$, the functor $\CZ(S) \to \CY(S)$ is a \emph{fibration in correspondences} (see \cite[C.10]{GLC2}).

\sssec{Remark}
	\label{sss corrspace}
Roughly speaking, a corr-space $\CT_{\Ran^\untl}$ over $\Ran^\untl$ is an assignment as follows.
	\begin{itemize}
		\item 
			For any point $\ul{x}\in \Ran^\untl$, a (usual) prestack $\CT_{\ul{x}}$;
		\item 
			For any $\ul{x}\subseteq \ul{x}'$, a correspondence $\CT_{\ul{x}}\gets \CT_{\ul{x}\subseteq \ul{x}'} \to \CT_{\ul{x}'}$ or prestacks that is compatible with compositions.
	\end{itemize}
	In particular, we have
	\[
		\CT_{\ul{x}} \simeq \CT_{\ul{x}\subseteq \ul{x}}
	\]
	and
	\[
		\CT_{\ul{x}\subseteq \ul{x}'} \times_{\CT_{\ul{x}'}} \CT_{\ul{x}'\subseteq \ul{x}''} \simeq \CT_{\ul{x}\subseteq \ul{x}''}.
	\]

\medskip 
	The corr-space $\CT_{\Ran^\untl}$ is coCartesian (resp. Cartesian) over $\Ran^\untl$ when the morphisms $\CT_{\ul{x}}\gets \CT_{\ul{x}\subseteq \ul{x}'}$ (resp. $\CT_{\ul{x}\subseteq \ul{x}'} \to \CT_{\ul{x}'}$) are invertible.

\sssec{Example}

	We have a corr-space $\fL(G)_{\Ran^\untl}$ over $\Ran^\untl$ such that $\fL(G)_{\ul{x}}:=G( \ocD_{\ul{x}}  )$, $\fL(G)_{\ul{x}\subseteq \ul{x}'} := G( \cD_{\ul{x}'}\setminus\ul{x}  )$, while the correspondence
	\[
		\fL(G)_{\ul{x}} \gets \fL(G)_{\ul{x}\subseteq \ul{x}'} \to \fL(G)_{\ul{x}'}
	\]
	(for $\ul{x}\subseteq\ul{x}'$) is induced by restriction along
	\[
		 \ocD_{\ul{x}} \to \cD_{\ul{x}'}\setminus\ul{x} \gets \ocD_{\ul{x}'}.
	\]

\sssec{} 
\label{sss Ran untl monoid}

Note that $\Ran^\untl$ is an abelian monoid object in $\on{CatPreStk}$, with the addition map given by
\[
	\on{union}: \Ran^\untl\times \Ran^\untl \to \Ran^\untl,\; (\ul{x},\ul{y}) \mapsto \ul{x}\cup\ul{y}.
\]
For any finite set $I$, we obtain a map
\[
	\on{union}_I: \prod_{i\in I} \Ran^\untl \to \Ran^\untl.
\]
Note that when $I=\emptyset$, this is the map $\emptyset:\on{pt} \to \Ran^\untl$.

\sssec{}
\label{sss Ran untl disjoint loci}

As in the non-unital case, we consider the disjoint locus
\[
	(\Ran^\untl \times \Ran^\untl)_\disj \subseteq \Ran^\untl\times \Ran^\untl,
\]
which sends an affine scheme $S$ to the category 
\[
	\{ (\ul{x},\ul{y})\in (\Ran^\untl\times \Ran^\untl)(S)\;\vert\; \on{Graph}_{\ul{x}}\cap \on{Graph}_{\ul{y}} = \emptyset\}.
\]
Similarly, for any finite set $I$, we can define
\[
	 \big (\prod_{i\in I} \Ran^\untl\big)_\disj \subseteq  \prod_{i\in I} \Ran^\untl.
\]

We record the following result for future references.

\begin{lem}
  \label{lem-union-local-iso}
    For any finite set $I$,
    \[
        \on{union}_I: \big (\prod_{i\in I} \Ran^\untl\big)_\disj \to \Ran^\untl
    \]
    is a Cartesian morphism of categorical prestacks (See \secref{sss Cart catprestk}). For any affine test scheme $S\to \Ran^\untl$, the base-change
    \begin{equation}
      \label{eqn-fiber-of-union}
        \big (\prod_{i\in I} \Ran^\untl\big)_\disj\times_{ \Ran^\untl } S
    \end{equation}
    is a finite coproduct of open subschemes of $S$ (taking in the category of prestacks). 
\end{lem}

\proof 
  It is enough to treat the case $I=\emptyset$ and $I=\{1,2\}$. In both cases, the first claim is obvious. For the second claim, when $I=\emptyset$, the base-change \eqref{eqn-fiber-of-union} is either empty or $S$. When $I=\{1,2\}$, we only need to show that
  \[
      (\Ran^\untl \times \Ran^\untl)_\disj\times_{ \Ran^\untl } X^J
  \]
  is a finite coproduct of open subschemes of $X^J$. Indeed, it is given by
  \[
      \bigsqcup_{J_1\sqcup J_2=J} (X^{J_1}\times X^{J_2})_{\on{disj}},
  \]
  where $(X^{J_1}\times X^{J_2})_{\on{disj}}\subseteq X^{J_1}\times X^{J_2}$ is the intersection of the preimages of $(X^{\{j_1\}} \times X^{\{j_2\}})\setminus \Delta$ for all the pairs $(j_1,j_2)\in J_1\times J_2$.

\qed

\sssec{}
	Now a \emph{unital} (resp. \emph{counital}, \emph{corr-unital}) factorization space $\CT$ is a \emph{coCartesian space} (resp. \emph{Cartesian space}, \emph{corr-space}) $\CT_{\Ran^\untl}$ over $\Ran^\untl$ equipped with a \emph{multiplicative structure over the disjoint loci}. In other words, for any finite set $I$, we have an isomorphism
	\[
		\on{union_I}^{-1}(  \CT_{\Ran^\untl} ) |_{   \big (\prod_{i\in I} \Ran^\untl\big)_\disj  } \simeq ( \prod_{i\in I}\CT_{\Ran^\untl} ) |_{ \big (\prod_{i\in I} \Ran^\untl\big)_\disj }
	\]
and a homotopy-coherent datum of associativity and commutativity. Here $\on{union}_I^{-1}$ is the change-of-base along the map $\on{union}_I$.

\sssec{Remark}
	Roughly speaking, a corr-unital factorization space $\CT$ consists of the following data
	\begin{itemize}
		\item 
			Prestacks $\CT_{\ul{x}}$ and correspondences $\CT_{\ul{x}}\gets \CT_{\ul{x}\subseteq \ul{x}'} \to \CT_{\ul{x}'}$ as in \secref{sss corrspace};
		\item 
			For any finite collection $\ul{x}_i\subseteq \ul{x}_i'$ ($i\in I$) such that $\ul{x}_i'\cap \ul{x}_j'=\emptyset$ ($i\neq j$), ca ommutative diagram
			\[
				\xymatrix{
					\prod_{i\in I} \CT_{\ul{x}_i} \ar[d]_-{\on{mult}_{(\ul{x}_i)}} 
					& \prod_{i\in I} \CT_{\ul{x}_i\subseteq \ul{x}_i'} \ar[l] \ar[r] \ar[d]^-{\on{mult}_{(\ul{x}_i\subseteq \ul{x}_i')}}
					& \prod_{i\in I} \CT_{\ul{x}_i'} \ar[d]^-{\on{mult}_{(\ul{x}_i')}}  \\ 
					 \CT_{\sqcup \ul{x}_i} &
					  \CT_{\sqcup \ul{x}_i\subseteq \sqcup \ul{x}_i'}
					  \ar[l] \ar[r]
					  &  \CT_{ \sqcup \ul{x}_i'}
				}
			\]
			that is compatible with compositions.
	\end{itemize}

	\medskip
	Such a $\CT$ is unital (resp. counital) means when the leftward (resp. rightward) morphisms are invertible.

\sssec{Example}
  \label{sss unital factspc example}
	The corr-space $\fL(G)_{\Ran^\untl} \to \Ran^\untl$ has a natural corr-unital factorization structure. We denote the resulting corr-unital factorization space by $\fL(G)$.

	\medskip 
	Similarly, we have a unital factorization space $\Gr_G$ and a counital factorization space $\fL^+(G)$.

\ssec{Crystal of categories over the unital Ran space}

\sssec{}

Given a categorical prestack $\CY$, a \emph{crystal of categories} $\underline{\bC}$ over $\CY$ is an assignment as follows.
\begin{itemize}
    \item For an affine scheme $S$ and a morphism $y:S\to \CY$, assign a $\Dmod(S)$-module category $\bC_y$;
    \item For affine schemes $S_1$, $S_2$ and a 2-morphism        
            \begin{equation}
                \label{eqn-2-morphism-cat-prestk}
                \xymatrix{
                    & S_1 \ar[rd]^-{y_1} \ar@{=>}[d]^-\alpha \\ 
                    S_2 \ar[rr]_-{y_2} \ar[ru]^-f & & \CY,
                    }
            \end{equation}
            (i.e. a morphism $\alpha:y_1\circ f\to y_2$ in $\CY(S_2)$), assign a functor 
            \[
                \bC_\alpha:\bC_{y_1} \to \bC_{y_2}
            \]
            intertwining the action of $f^!:\Dmod(S_1)\to \Dmod(S_2)$, such that the induced functor
            \[
                \bC_{y_1}\otimes_{\Dmod(S_1)} \Dmod(S_2) \to \bC_{y_2}
            \]
            is an equivalence when $\alpha$ is invertible.
    \item A homotopy-coherent system of compatibilities for compositions.
\end{itemize}
    
\sssec{Remark} \label{sss actual defn of cryscat:label}
Let 
\[
    \widetilde{\CY} \to (\on{Sch}^{\on{aff}})^{\on{op}}  
\] 
be the coCartesian fibration corresponding to the functor $\CY$. Let 
\[
    \widetilde{\mathbf{CrysCat}} \to (\on{Sch}^{\on{aff}})^{\on{op}}
\] 
be the coCartesian fibration of 2-categories corresponding to the functor $S\mapsto \Dmod(S)\on{-mod}$. Then a crystal of categories $\underline{\bC}$ over $\CY$ is \emph{defined} to be a $(\on{Sch}^{\on{aff}})^{\on{op}}$-functor $\widetilde{\CY} \to \widetilde{\mathbf{CrysCat}}$ preserving coCartesian arrows.

\sssec{Example}
    We have the \emph{constant crystal of categories} $\underline{\Dmod}(\CY)$ over $\CY$ which assignes to $y\in \CY(S)$ the cateogry $\Dmod(S)$ and assignes to a 2-morphism $\alpha$ the functor $f^!:\Dmod(S_1)\to \Dmod(S_2)$.

\sssec{Remark}
    \label{rem cryscat on Ran:label}
    Roughly speaking, a crystal of categories $\bA$ over $\Ran^\untl$ is an assignment as follows.
    \begin{itemize}
        \item For any point $\underline{x}\in \Ran^\untl$, assign a DG category $\bA_{\underline{x}}$;
        \item For any inclusion $\underline{x}\subseteq \underline{x}'$, assign a functor
        \[
            \on{ins}_{\underline{x}\subseteq \underline{x}'}: \bA_{\underline{x}}\to \bA_{\underline{x}'}
        \]
        that is compatible with compositions.
    \end{itemize}
    The above data should depend ``algebraically'' on $\underline{x}$ and $\underline{x'}$. This means we also allow $\ul{x}$ to be affine points $\Ran^\untl(S)$, and the above data should be contravariantly functorial in $S$.

\sssec{}
	Given a corr-space $\CT_{\Ran^\untl}$ over $\Ran^\untl$, under some mild finiteness assumptions, we can construct a crystal of categories $\ul{\Dmod}(\CT_{\Ran^\untl})$ over $\Ran^\untl$ such that
	\begin{itemize}
		\item For any $\ul{x}\in \Ran^\untl$, 
		\[
			\ul{\Dmod}(\CT_{\Ran^\untl})_{\ul{x}} \simeq \Dmod(\CT_{\ul{x}})
		\]
		\item For any inclusion $\underline{x}\subseteq \underline{x}'$, the functor
		\[
			\on{ins}_{\ul{x}\subseteq \ul{x}'}: \ul{\Dmod}(\CT_{\Ran^\untl})_{\ul{x}}\to \ul{\Dmod}(\CT_{\Ran^\untl})_{\ul{x}'}
		\]
		is given by $!$-pull-$*$-push along the correspondence
		\[
			\CT_{\ul{x}} \gets \CT_{\ul{x}\subseteq \ul{x}'} \to \CT_{\ul{x}'}.
		\]
	\end{itemize}
	Here the finiteness assumptions are required such that: 
	\begin{itemize}
		\item The category $\Dmod(\CT_{\ul{x}})$ is well-defined\footnote{In fact, we need to consider all finite type affine points $\ul{x}:S\to \Ran^\untl$.};
		\item The $!$-pullback functor $\Dmod(\CT_{\ul{x}}) \to \Dmod(\CT_{\ul{x}\subseteq \ul{x}'})$ is well-defined, and the $*$-pushforward functor $\Dmod(\CT_{\ul{x}\subseteq \ul{x}'})\to \Dmod(\CT_{\ul{x}'})$ is well-defined.
		\item The above $!$-pullback functors and $*$-pushforward functors have base-change isomorphisms.
	\end{itemize}

\sssec{Example}
  \label{exam-LG-L+G-GrG}
	The coCartesian space $\Gr_{G,\Ran^\untl}\to \Ran^\untl$ produces a crystal of categories $\ul{\Dmod}(\Gr_G)$ over $\Ran^\untl$, due to the fact that each $\Gr_{G,\ul{x}}$ is ind-finite type.

	\medskip 
	The Cartesian space $\fL^+(G)_{\Ran^\untl} \to \Ran^\untl$ produces crystals of categories $\ul{\Dmod}^!(\fL^+(G))$ and $\ul{\Dmod}_*(\fL^+(G))$, due to the fact that each $\fL^+(G)_{\ul{x}}$ is placid (see \secref{sss placid}). Note however that we can write
	\[
		\ul{\Dmod}^!(\fL^+(G))=: \ul{\Dmod}(\fL^+(G)):= \ul{\Dmod}_*(\fL^+(G))
	\]
	because of \secref{sss:self-duality L+G}.

	\medskip 
	Similarly, the corr-space $\fL(G)_{\Ran^\untl} \to \Ran^\untl$ produces a crystal of categories $\ul{\Dmod}(\fL(G))$.

\ssec{Morphisms between crystals of categories}
    There are two notions of morphisms between crystals of categories over a categorical prestack $\CY$: \emph{lax functors} and \emph{strict functors}.

\sssec{}
    A \emph{lax functor} $F: \bC\to \bC'$ is an assignment as follows.
\begin{itemize}
    \item For any $y:S\to \CY$, assign a $\Dmod(S)$-linear functor $F_y: \bC_y \to \bC_y'$;
    \item For any 2-morphism \eqref{eqn-2-morphism-cat-prestk}, assign a $\Dmod(S)$-linear natural transformation
            \[
                \xymatrix{
                    \bC_y \ar[r]^-{\bC_\alpha} \ar[d]_-{F_y}
                    & \bC_{y'} \ar[d]^-{F_{y'}} \\ 
                    \bC'_y \ar[r]_-{\bC'_\alpha} \ar@{=>}[ru]^-{F_\alpha}
                    & \bC'_{y'},
                }
            \]
            such that it is invertible if $\alpha$ is so.
    \item A homotopy-coherent system of compatibilities for compositions.
\end{itemize}

    \medskip
    A lax functor $F: \bC\to \bC'$ is \emph{strict} if the above natural transformations are all invertible. 

    \medskip
    The totality of crystals of categories over $\CY$ and lax functors gives a 2-category $\mathbf{CrysCat}^{\on{lax}}(\CY)$. There is a 1-full subcategory 
    \[
        \mathbf{CrysCat}^{\on{strict}}(\CY) \subseteq \mathbf{CrysCat}^{\on{lax}}(\CY)
    \]
    with morphisms being strict functors.

    \medskip 
    We equip $\mathbf{CrysCat}^{\on{lax}}(\CY)$ with the natural symmetric monoidal structure given by the formula
    \[
        (\bC\otimes\bD)_y := \bC_y \otimes_{\Dmod(S)} \bD_y,\; (y:S\to Y).
    \]

\sssec{Remark} 
    Following \secref{sss actual defn of cryscat:label}, a lax functor $F: \bC\to \bC'$ is \emph{defined} to be a (right-)lax natural transformation (over $(\on{Sch}^{\on{aff}})^{\on{op}}$) between the corresponding functors
    \[
        \bC,\bC':\widetilde{\CY} \to \widetilde{\mathbf{CrysCat}}
    \]
    such that its value at any coCartesian arrow in $\widetilde{\CY}$ is strict, while a strict functor $F: \bC\to \bC'$ is just a strict natural transformation.

    \medskip
    A 2-morphism in $\mathbf{CrysCat}^{\on{lax}}(\CY)$ is \emph{defined} to be a (strict) \emph{modification}\footnote{There is no room for laxness for modifications because 3-morphisms in $\widetilde{\mathbf{CrysCat}}$ are invertible.} between such lax natural transformations. 
    In other words, they are 2-morphisms in $\on{Fun}_{(\on{Sch}^{\on{aff}})^{\on{op}}}(  \widetilde{\CY}, \widetilde{\mathbf{CrysCat}} )$. 

\sssec{Example}
	\label{sss lax global section}

    A lax functor $\underline{\Dmod}(\CY) \to \bC$ is called a \emph{lax global section} of $\bC$. It is an assignment as follows.
    \begin{itemize}
        \item For any $y:S\to \CY$, assign an object $\CF_y \in \bC_y$;
        \item For any 2-morphism \eqref{eqn-2-morphism-cat-prestk}, assign a morphism $\CF_\alpha:\bC_\alpha(\CF_y) \to \CF_{y'}$ in $\bC_{y'}$ such that it is invertible if $\alpha$ is so.
        \item A homotopy-coherent system of compatibilities for compositions.
    \end{itemize}
    The collection of lax global sections of $\bC$ form a DG category $\bC_\CY^{\on{lax}}$, which is denoted $\Gamma^{\on{lax}}(\CY,\bC)$ in \cite[Sect. C.2]{GLC2}.

    \medskip
    Similarly, a strict functor $\underline{\Dmod}(\CY) \to \bC$ is called a \emph{strict global section} of $\bC$. It is an assignment as above such that $\CF_\alpha$ is \emph{always} invertible. The collection of strict global sections of $\bC$ form a full subcategory $\bC_\CY^{\on{strict}} \subseteq \bC_\CY^{\on{lax}}$, which is denoted by $\Gamma^{\on{strict}}(\CY,\bC)$ in \emph{loc.cit.}.

\sssec{Example}
  \label{sss lax D-module}
    Objects in 
    \[
        \Dmod^{\on{lax}}(\CY) := \underline{\Dmod}(\CY)_\CY^{\on{lax}} \simeq \Gamma^{\on{lax}}(\CY,\underline{\Dmod}(\CY))
    \]
    are called \emph{lax D-modules} on $\CY$, while those in
    \[
        \Dmod^{\on{strict}}(\CY) := \underline{\Dmod}(\CY)_\CY^{\on{strict}} \simeq \Gamma^{\on{strict}}(\CY,\underline{\Dmod}(\CY))
    \]
    are \emph{strict D-modules} on $\CY$.

\sssec{Remark}
	\label{rem description lax unital morphism over Ran}
    Following \secref{rem cryscat on Ran:label}, roughly speaking, a lax (resp. strict) functor $F:\bA \to \bA'$ over $\Ran^\untl$ consists of the following data:
    \begin{itemize}
        \item For any point $\underline{x}\in \Ran^\untl$, assign a functor $F_{\ul{x}}:\bA_{\underline{x}} \to \bA_{\underline{x}}'$;
        \item For any inclusion $\underline{x}\subseteq \underline{x}'$, assign a natural transformation (resp. isomorphism)
            \[
                \xymatrix{
                    \bA_{\ul{x}} \ar[rr]^-{\on{ins}_{\ul{x}\subseteq\ul{x}'}} \ar[d]_-{F_{\ul{x}}}
                    & &\bA_{\ul{x}'} \ar[d]^-{F_{\ul{x}'}} \\ 
                    \bA'_{\ul{x}} \ar[rr]_-{\on{ins}_{\ul{x}\subseteq\ul{x}'}} \ar@{=>}[rru]^-{F_{ \ul{x}\subseteq \ul{x}' }}
                    & &\bA'_{\ul{x}'},
                }
            \]
        that is compatible with compositions.
    \end{itemize}
 
 	\medskip
 	In particular, a lax (resp. strict) global section $\CA$ of $\bA$ over $\Ran^\untl$ consists of the following data:
 	\begin{itemize}
 		\item 
 			For any point $\underline{x}\in \Ran^\untl$, assign an object $\CA_{\ul{x}} \in \bA_{\ul{x}}$;
 		\item 
 			For any inclusion $\underline{x}\subseteq \underline{x}'$, assign a morphism (resp. isomorphism)
 			\[
 				\CA_{\ul{x}\subseteq \ul{x}'}:\on{ins}_{\ul{x}\subseteq \ul{x}'} (\CA_{\ul{x}}) \to \CA_{\ul{x}'}
 			\]
 			that is compatible with compositions.
 	\end{itemize}

\sssec{}
    Given a morphism $f:\CY \to \CZ $ between categorical prestacks, there is a symmetric monoidal functor
    \[
    	f^*: \mathbf{CrysCat}^{\on{lax}}(\CZ) \to \mathbf{CrysCat}^{\on{lax}}(\CY)
    \]
    given by the formula
    \[
    	(f^*\bC)_y \simeq \bC_{f(y)},
    \]
    where $y:S\to \CY$ is an affine point of $\CY$ and $f(y)=f\circ y$ is an affine point of $\CZ$. Note that $f^*$ preserves strict functors, i.e., we have a functor
    \[
    	f^*: \mathbf{CrysCat}^{\on{strict}}(\CZ) \to \mathbf{CrysCat}^{\on{strict}}(\CY)
    \]

\sssec{}
	\label{sss !-pullback lax global section}
    In particular, for any $\bC \in \mathbf{CrysCat}^{\on{lax}}(\CZ)$, we obtain a functor
    \[
    	f^!:\bC_\CZ^{\on{lax}} \to (f^*\bC)^{\on{lax}}_\CY
    \]
    given by
    \[
    	\on{Fun}_{\mathbf{CrysCat}^{\on{lax}}(\CZ)}( \ul{\Dmod}(\CZ), \bC  ) \to \on{Fun}_{\mathbf{CrysCat}^{\on{lax}}(\CY)}( f^*\ul{\Dmod}(\CZ), f^*\bC  ) \simeq \on{Fun}_{\mathbf{CrysCat}^{\on{lax}}(\CY)}( \ul{\Dmod}(\CY), f^*\bC  ).
    \]
    Note that $f^!$ sends $\bC_\CZ^{\on{strict}}$ into $(f^*\bC)^{\on{strict}}_\CY$.

    \medskip
    When $\bC = \ul{\Dmod}(\CZ)$, the functor $f^!$ sends (lax) D-modules of $\CZ$ to (lax) D-modules of $\CY$. This construction generalizes the $!$-pullback functors for usual (i.e. non-categorical) prestacks.

\sssec{Remark}
	Following \secref{sss actual defn of cryscat:label}, the functor $f^*$ is given by precomposing with $\widetilde{\CY} \to \widetilde{\CZ}$.

\sssec{Example}
	There is an obvious map $\Ran \to \Ran^\untl$, which induces a restriction functor
	\[
		\mathbf{CrysCat}^{\on{lax}}(\Ran^\untl) \to \mathbf{CrysCat}(\Ran).
	\]
	Here we do not need to distinguish $\mathbf{CrysCat}^{\on{lax}}(\Ran)$ and $\mathbf{CrysCat}^{\on{strict}}(\Ran)$ because $\Ran$ is a usual prestack. 

\sssec{}
  Let $\CY$ and $\CZ$ be categorical prestacks. The \emph{external tensor product} functor
  \[
      -\boxtimes-: \mathbf{CrysCat}^{\on{lax}}(\CY)\times  \mathbf{CrysCat}^{\on{lax}}(\CZ) \to  \mathbf{CrysCat}^{\on{lax}}(\CY\times \CZ)
  \]
  is defined to be
  \[
    \CC\boxt \mathcal{D} := \on{pr}_1^*(\CC) \otimes \on{pr}_2^*(\mathcal{D}).
  \]
  Note that
  \[
    (\CC \boxt \mathcal{D})_{(y,z)}\simeq \CC_y \otimes \mathcal{D}_z
  \]
  for $y:S\to \CY$ and $z:T\to \CZ$.

\sssec{Warning}
  For a 2-morphism $f_1\to f_2: \CY\to \CZ$ in $\mathbf{CatPreStk}$, we only have a \emph{left-lax} natural transformation
  \[
    f_1^* \to f_2^*: \mathbf{CrysCat}^{\on{lax}}(\CZ) \to \mathbf{CrysCat}^{\on{lax}}(\CY).
  \] 
  In other words, given a morphism $F:\bC\to \bC'$ in $\mathbf{CrysCat}^{\on{lax}}(\CY)$, we have a canonical 2-morphism
  \[
    \xymatrix{
        f_1^*(\bC) \ar[r] \ar[d]
        & f_2^*(\bC) \ar[d] \\
        f_1^*(\bC') \ar[r] \ar@{=>}[ru]
        & f_2^*(\bC').
    }
  \]

\ssec{Unital factorization categories}

\sssec{}
\label{sss:unital factcat}
Recall that $\Ran^\untl$ is an abelian monoid object in $\on{CatPreStk}$ (see \secref{sss Ran untl monoid}).

\medskip
A \emph{unital factorization category} $\bA$ is a crystal of categories $\ul{\bA}$ over $\Ran^\untl$ equipped with a \emph{multiplicative structure over the disjoint loci}. In other words, for any finite set $I$, we have an equivalence
\begin{equation}
	\label{eqn-defn-fact-cat}
	\on{mult}_I: ( \boxt_{i\in I} \ul{\bA}  )|_{ \on{disj} }
	 \xrightarrow{\simeq} 
	 \on{union}_I^*(\ul{\bA})|_{ \on{disj} }
\end{equation}
and a homotopy-coherent datum of associativity and commutativity. Here $(-)|_\disj$ means restriction along 
\[
	 \big (\prod_{i\in I} \Ran^\untl\big)_\disj  \to  \prod_{i\in I} \Ran^\untl,
\]
see \secref{sss Ran untl disjoint loci}.

\sssec{Remark}
	\label{sss unital factcat explicit}
	Roughly speaking, a unital factorization category consists of the following data:
	\begin{itemize}
		\item 
			DG categories $\bA_{\ul{x}}$ and functors $\on{ins}_{\ul{x}\subseteq \ul{x}'}$ as in \secref{rem cryscat on Ran:label};
		\item 
			For any finite collection of \emph{disjoint} points in $\Ran^\untl$, i.e., 
			\[
				(\ul{x}_i)_{i\in I}\in  \big (\prod_{i\in I} \Ran^\untl\big)_\disj,\; |I|<\infty
			\] 
			assign an equivalence
			\begin{equation}
				\label{eqn-factcat-structure-fiber}
				\on{mult}_{(\ul{x}_i)}: \otimes \bA_{\ul{x}_i} \simeq \bA_{\bigsqcup \ul{x}_i}
			\end{equation}
			(and a datum of associativity and commutatitvity)
	\end{itemize}
	such that for two collections $(\ul{x}_i)_{i\in I}$, $(\ul{x}_i')_{i\in I}$ of disjoint points satisfying $\ul{x}_i\subseteq \ul{x}_i'$, the following diagram commutes\footnote{This is also a structure rather than a property, and there is a datum of compatibility with associativity and commutatitvity.}
	\begin{equation}
		\label{eqn-unital-factcat-explicit}
		\xymatrix{
			\otimes \bA_{\ul{x}_i} 
			\ar[rr]^-{\on{mult}_{(\ul{x}_i)}}_\simeq  \ar[d]_-{\otimes  \on{ins}_{\ul{x}_i\subseteq \ul{x}_i' }}
			& &  \bA_{\bigsqcup \ul{x}_i} 
			\ar[d]^-{ \on{ins}_{ \sqcup \ul{x}_i \subset  \sqcup \ul{x}_i' }  } \\ 
			\otimes \bA_{\ul{x}_i'} 
			\ar[rr]^-\simeq_-{\on{mult}_{(\ul{x}_i')}} 
			&&  \bA_{\bigsqcup \ul{x}_i'}.
		}
	\end{equation}

	\medskip 
	Note that for $I=\emptyset$, we have a canonical equivalence $\on{mult}_{\emptyset}: \Vect \simeq \bA_\emptyset$\footnote{Note that for $I=\{1\}$ and $\ul{x}_1=\emptyset$, we also have an equivalence $\on{mult}_{(\emptyset)}: \bA_\emptyset \simeq \bA_\emptyset$, which is just the identity functor.}.

\sssec{Example}
	\label{exam unital factcat Vect}
	The constant crystal $\ul{\Dmod}(\Ran^\untl)$ is a unital factorization cateogry by identifying both sides of \eqref{eqn-defn-fact-cat} with 
	\[
		\ul{\Dmod}\big(  \big (\prod_{i\in I} \Ran^\untl\big)_\disj  \big).
	\]
	We often denote this factorization category by $\Vect$ (because its fiber at any $\ul{x}\in \Ran^\untl$ is $\Vect$).

\sssec{Example}
  \label{exam-LG-L+G-GrG-fact}
	The factorization structures on $\Gr_{G,\Ran^\untl}$, $\fL^+(G)_{\Ran^\untl}$ and $\fL(G)_{\Ran^\untl}$ upgrade the crystals of categories $\ul{\Dmod}(\Gr_G)$, $\ul{\Dmod}(\fL^+(G))$, $\ul{\Dmod}(\fL(G))$ to unital factorization categories.

\sssec{Variant}
	\label{sss lax unital factcat}
	We define a unital \emph{lax-factorization} cateogry to be a crystal of categories $\ul{\bA}$ over $\Ran^\untl$ equipped with a \emph{(right-)lax} multiplicative structure over the disjoint loci. This means we replace the equivalence \eqref{eqn-defn-fact-cat} with a morphism in $\mathbf{CrysCat}^{\on{strict}}(  \big (\prod_{i\in I} \Ran^\untl\big)_\disj )$:
	\begin{equation}
	\label{eqn-defn-laxfact-cat}
		\on{mult}_I: ( \boxt_{i\in I} \ul{\bA}  )|_{ \on{disj}} \to \on{union}_I^*(\ul{\bA})|_{ \on{disj}  }.
	\end{equation}

	\medskip
	One may also consider an even weaker notion, where \eqref{eqn-defn-laxfact-cat} is only required to be a morphism in $\mathbf{CrysCat}^{\on{lax}}(  \big (\prod_{i\in I} \Ran^\untl\big)_\disj )$. However, we do not see any application of such a structure.

\sssec{Remark}

	The restriction functor $\mathbf{CrysCat}^{\on{lax}}(\Ran^\untl) \to \mathbf{CrysCat}(\Ran)$ sends unital factorization categories to non-unital ones.

\ssec{Morphisms between unital factorization categories} 
\label{ss Mor factalgcat}
There are (at least) two notions of morphisms between unital factorization categories: \emph{lax-unital functors} and \emph{(strictly) unital functors}.

\sssec{}
	\label{sss:lax unital factfun}
	Let $\bA$ and $\bA'$ be unital factorization categories. A \emph{lax-unital factorization functor} $F:\bA \to \bA'$ is a morphism $\ul{F}: \ul{\bA}\to \ul{\bA}'$ in $\mathbf{CrysCat}^{\on{lax}}(\Ran^\untl)$ equipped with commutative diagrams (for any finite set $I$)
	\begin{equation}
		\label{eqn-factfun-defn}
		\xymatrix{
			( \boxt_{i\in I} \ul{\bA}  )|_{ \on{disj} } 
			\ar[d]_-{ (\boxt \ul{F})|_\disj  }
			\ar[r]^-\simeq 
			& \on{union}_I^*(\ul{\bA})|_{ \on{disj} } 
			\ar[d]^-{ \on{union}_I^*(\ul{F})|_\disj  } \\
			( \boxt_{i\in I} \ul{\bA}'  )|_{ \on{disj} }
			\ar[r]^-\simeq 
			& \on{union}_I^*(\ul{\bA}')|_{\on{disj} } 
		}
	\end{equation}
	and a homotopy-coherent datum of associativity and commutativity. Here the horizontal arrows are the structural equivalences for $\bA$ and $\bA'$ (see \eqref{eqn-defn-fact-cat}).

	\medskip 
	We say $F$ is (strictly) unital if $\ul{F}$ is contained in $\mathbf{CrysCat}^{\on{strict}}(\Ran^\untl)$.

\sssec{}
	Unital factorization categories and lax-unital factorization functors between them form a 2-category, which is denoted by
	\[
		\mathbf{UntlFactCat}^{\on{lax-untl}}.
	\]
	The 1-full subcategory of strictly unital factorization functors is denoted by
	\[
		\mathbf{UntlFactCat} \subseteq \mathbf{UntlFactCat}^{\on{lax-untl}}.
	\]

\sssec{}
  \label{sss defn factalg in factcat}
	Let $\bA$ be a unital factorization category. A \emph{(unital)\footnote{When talking about factorization algebras in a \emph{unital} factorization category, we always assume the algebra is unital.} factorization algebra in $\bA$} is \emph{defined} to be a lax-unital factorization functor $\CA:\Vect \to \bA$, where $\Vect$ is the unital factorization category defined in \secref{exam unital factcat Vect}.

	\medskip 
	Factorization algebras in $\bA$ form a category, which is denoted by $\on{FactAlg}(\bA)$.

\sssec{}
  \label{sss traddesc factalg}
	Below is a concrete description of factorization algebras $\CA$ in a unital factorization category $\bA$.

	\medskip
	For each finite set $I$, we have a functor
	\[
		\prod_{i\in I}\Gamma^{\on{lax}}(\Ran^\untl, \ul{\bA})
		\to \Gamma^{\on{lax}}( \prod_{i\in I}  \Ran^\untl , \boxt_{i\in I} \ul{\bA} ) 
		\to \Gamma^{\on{lax}}\big(  \big(\prod_{i\in I}  \Ran^\untl\big)_\disj , (\boxt_{i\in I} \ul{\bA})|_\disj \big) 
	\]
	that sends
	\[
		(\CM_i)_{i\in I} \to (\boxt \CM_i)|_\disj,
	\]
	and a functor
	\[
		\Gamma^{\on{lax}}\big(  \Ran^\untl , \bA \big) \to \Gamma^{\on{lax}}\big(  \big(\prod_{i\in I}  \Ran^\untl\big)_\disj , \on{union}_I^*(\bA)|_\disj \big)
	\]
	that in turn sends
	\[
		\CN \mapsto \on{union}_I^!(\CN)|_\disj,
	\]
	see \secref{sss !-pullback lax global section}. Note that the structural equivalence \eqref{eqn-defn-fact-cat} induces an equivalence
	\[
		\on{mult}_I: \Gamma^{\on{lax}}\big(  \big(\prod_{i\in I}  \Ran^\untl\big)_\disj , (\boxt_{i\in I} \ul{\bA})|_\disj \big)  \xrightarrow{\simeq}  \Gamma^{\on{lax}}\big(  \big(\prod_{i\in I}  \Ran^\untl\big)_\disj , \on{union}_I^*(\bA)|_\disj \big).
	\]

	\medskip
	Unwinding the definitions, a factorization algebra $\CA$ in $\bA$ is an object
	\[
		\ul{\CA} \in \Gamma^{\on{lax}}(\Ran^\untl, \ul{\bA})=:\ul{\bA}_{\Ran^\untl}^{\on{lax}}
	\]
	equipped with isomorphisms (for any finite set $I$)
  \begin{equation}
    \label{eqn-structure-iso-factalg}
    \on{mult}_I(  (\boxt_{i\in I} \ul{\CA})|_\disj ) \xrightarrow{\simeq} \on{union}_I^!(\ul{\CA})|_\disj
  \end{equation}
	and a homotopy-coherent datum of associativity and commutativity.

\sssec{Remark} 
	\label{sss unital factalg explicit}
	Roughly speaking, a factorization algebra $\CA$ in a given unital factorization category $\bA$ consists of the following data:
	\begin{itemize}
		\item[(i)]
			Objects $\CA_{\ul{x}}\in \bA_{\ul{x}}$ and morphisms 
			\[
				\CA_{\ul{x}\subseteq \ul{x}'}: \on{ins}_{\ul{x}\subseteq \ul{x}'}(\CA_{\ul{x}}) \to \CA_{\ul{x}'}
			\]
			as in \secref{rem description lax unital morphism over Ran};
		\item[(ii)]
			For any collection $(\ul{x}_i)_{i\in I}$ of \emph{disjoint} points in $\Ran^\untl$, an isomorphism
			\begin{equation}
        \label{eqn-structure-iso-factalg-fiber}
				m_{(\ul{x}_i)}: \on{mult}_{(\ul{x}_i)}(\boxt \CA_{\ul{x}_i}) \xrightarrow{\simeq}\CA_{\bigsqcup \ul{x}_i}
			\end{equation}
      that is associative and commutative, where the (invertible) functor $\on{mult}_{(\ul{x}_i)}$ is as in \secref{sss unital factcat explicit};
	\end{itemize}
	such that for two collections $(\ul{x}_i)_{i\in I}$, $(\ul{x}_i')_{i\in I}$ of disjoints points satisfying $\ul{x}_i\subseteq \ul{x}_i'$, the following diagram commutes
	\[
		\xymatrix{
			\on{ins}_{ \sqcup \ul{x}_i \subset  \sqcup \ul{x}_i' } \circ \on{mult}_{(\ul{x}_i)}(\boxt \CA_{\ul{x}_i}) 
			\ar[d]^-\simeq 
			\ar[rr]^-\simeq
			& & \on{ins}_{ \sqcup \ul{x}_i \subset  \sqcup \ul{x}_i' } (\CA_{\bigsqcup \ul{x}_i}) \ar[d]
			\\ 
			\on{mult}_{(\ul{x}_i')}( \boxt \on{ins}_{\ul{x}_i\subseteq \ul{x}_i' } (\CA_{\ul{x}_i})  ) 
			\ar[r]
			& \on{mult}_{(\ul{x}_i')}( \boxt  \CA_{\ul{x}_i'})
			\ar[r]_-\simeq
			& \CA_{\bigsqcup \ul{x}_i'},
		}
	\]
	where
	\begin{itemize}
		\item The left vertical isomorphism is due to \eqref{eqn-unital-factcat-explicit};
		\item The two horizontal isomorphisms are induced by those in (ii);
		\item The remaining two arrows are induced by those in (i).
	\end{itemize}

	\medskip 
	Note that for $I=\emptyset$, we have a canonical isomorphism
	\[
		m_\emptyset: \on{mult}_\emptyset(k) \xrightarrow{\simeq} \CA_\emptyset,
	\]
	where $\on{mult}_\emptyset$ is the canonical identification $\Vect\simeq \bA_\emptyset$.

\sssec{Example}

	For any unital factorization category $\bA$, there is a unique \emph{unital} factorization functor 
	\[
		\CA: \Vect \to \bA.
	\]
	Namely, unitality implies each $\CA_{\ul{x}} \in \bA_{\ul{x}}$ is canonically identified with $\on{ins}_{\emptyset\subseteq \ul{x}}(\CA_\emptyset)$, while $\CA_\emptyset$ is canonically identified with $k$ via the equivalence $\bA_\emptyset \simeq \Vect$.

	\medskip
	We denote this unital factorization functor by
	\[
		\on{unit}_\bA: \Vect \to \bA
	\]
	and view it as an object
	\[
		\on{unit}_\bA \in \on{FactAlg}(\bA).
	\]

  \medskip 
  The following result is obvious modulo homotopy coherence. A rigorious proof will be provided in \cite{CFZ}.

	\begin{lem}
		\label{lem-unit-is-initial}
		Let $\bA$ be a unital factorization category. Then $\on{unit}_\bA$ is an initial object in $\on{FactAlg}(\bA)$.
	\end{lem}

\sssec{}
	By definition, for any lax unital factorization functor $F:\bA \to \bA'$, we have a functor
	\begin{equation}
    \label{eqn-laxuntl-factalg-to-factalg}
		\on{FactAlg}(\bA) \to \on{FactAlg}(\bA'),\; \CA \mapsto F\circ \CA.
	\end{equation}
	We also write $F(\CA):=F\circ \CA$.

	\medskip
	In particular, we obtain an object $F(\on{unit}_\bA) \in \on{FactAlg}(\bA')$. By Lemma \ref{lem-unit-is-initial}, there is a unique morphism
	\[
		\on{unit}_{\bA'} \to F(\on{unit}_\bA) .
	\]

  \medskip
	The following result is obvious modulo homotopy coherence. A rigorious proof will be provided in \cite{CFZ}.

	\begin{lem}
    \label{lem-criterion-for-unital-factfun}
		Let $F:\bA\to \bA'$ be a lax-unital factorization functor between unital factorization categories. Then $F$ is strictly unital iff the canonical morphism $\on{unit}_{\bA'} \to F(\on{unit}_\bA)$ is invertible.
	\end{lem}

\sssec{Variant}
  \label{sss laxfact functors}
	The notion of (lax-)unital factorization functors makes sense also for unital lax-factorization categories (see \secref{sss lax unital factcat}). In particular, we can define the object $\on{unit}_\bA$ for any unital lax-factorization cateogry $\bA$\footnote{Warning: Lemma \ref{lem-criterion-for-unital-factfun} needs $\bA$ to be a unital (strict-)factorization category, although $\bA'$ can be lax.}.
	
	\medskip 
	One can consider an even weaker notion of factorization functors, where the commutative diagram \eqref{eqn-factfun-defn} is replaced with a 2-morphism
	\[
		\xymatrix{
			( \boxt_{i\in I} \ul{\bA}  )|_{ \on{disj} } 
			\ar[d]_-{ (\boxt \ul{F})|_\disj  }
			\ar[r]
			& \on{union}_I^*(\ul{\bA})|_{ \on{disj} } 
			\ar[d]^-{ \on{union}_I^*(\ul{F})|_\disj  } \\
			( \boxt_{i\in I} \ul{\bA}'  )|_{ \on{disj} }
			\ar[r] \ar@{=>}[ru]
			& \on{union}_I^*(\ul{\bA}')|_{\on{disj} }.
		}
	\]
	A morphism $\ul{F}:\ul{\bA} \to \ul{\bA}'$ equipped with such a structure is called a \emph{lax-unital lax-factorization functor}. Taking $\bA:=\Vect$, we obtain the notion of \emph{lax-factorization algebras} in a unital lax-factorization category $\bA'$. One can describe these objects as in \secref{sss unital factalg explicit} by replacing the isomorphism $m_{(\ul{x}_i)_{i\in I}}$ with a morphism
	 \[
		m_{(\ul{x}_i)}: \on{mult}_{(\ul{x}_i)}(\boxt \CA_{\ul{x}_i}) \to \CA_{\bigsqcup \ul{x}_i}.
	\]

\ssec{Unital factorization module categories}
\label{ss factmodcat}

\sssec{} \label{sss:marked unital Ran}
	Let $\ul{x}_0\in \Ran^\untl(S_0)$ be an affine point. The \emph{$\ul{x}_0$-marked unital Ran space} is the category prestack $\Ran_{\ul{x}_0}^\untl$ over $S_0$ that attaches to an affine test $S_0$-scheme $S$ the category of finite subsets $\ul{y}\subseteq\Hom(S,X)$ that contain the image of $\ul{x}_0$ under the restriction map
	\[
		\Hom(S_0,X) \to \Hom(S, X).
	\]

	\medskip 
	Note that there is a canonical morphism $S_0 \to \Ran_{\ul{x}_0}$ corresponding to $\ul{x}_0\subseteq \Hom(S_0,X)$.

	\medskip 
	Also note that $\Ran_{\emptyset}^\untl=\Ran^\untl$.

\sssec{Remark}
	Let $\widetilde{\Ran^\untl} \to (\on{Sch}^{\on{aff}})^{\on{op}}$ be the coCartesian fibration corresponding to $\Ran^\untl$ (see \secref{sss actual defn of cryscat:label}). We can view $\ul{x}_0$ as an object $\widetilde{\Ran^\untl}$ lying over $S_0$. By definition, $\Ran_{\ul{x}_0}^\untl$ is the functor corresponding to the coCartesian fibration
	\[
		\big( \widetilde{\Ran^\untl}\big)_{\ul{x}_0/} \to (\on{Sch}^{\on{aff}})^{\on{op}},
	\]
	where the source is the slice (a.k.a. comma) category of arrows out of $\ul{x}_0$.

\sssec{}
  \label{sss markRan module}
	Recall $\Ran^\untl$ is an abelian monoid object in $\on{CatPreStk}$, with addition morphism given by $(\ul{x},\ul{y}) \mapsto \ul{x}\cup \ul{y}$. It is clear that $\Ran_{\ul{x}_0}^\untl$ is a $\Ran^\untl$-module object in $\on{CatPreStk}$, with the action morphism given by the same formula
	\[
		\on{union}: \on{Ran}^\untl \times \on{Ran}^\untl_{\ul{x}_0} \to \on{Ran}^\untl_{\ul{x}_0},\; (\ul{x},\ul{y}) \mapsto \ul{x}\cup \ul{y}.
	\]
	This morphism is well-defined because if $\ul{y}$ contains (the image of) $\ul{x}_0$, so does $ \ul{x}\cup \ul{y}$.

	\medskip 
	As in the non-marked case, for any \emph{marked} finite set $I=I^\circ\sqcup\{*\}$, we define the subfunctor
	\begin{equation}
		\label{eqn-disj-loci-marked-Ran}
		\big(\big( \prod_{i\in I^\circ}\on{Ran}^\untl\big) \times \on{Ran}^\untl_{\ul{x}_0} \big)_\disj \subseteq  \big( \prod_{i\in I^\circ}\on{Ran}^\untl\big) \times \on{Ran}^\untl_{\ul{x}_0} 
	\end{equation}
	that contains those $(\ul{x}_i, \ul{y})$ such that any pair of graphs of $\ul{x}_i$ and $\ul{y}$ is disjoint.

\sssec{} \label{sss:unital factmodcat}
	Let $\bA$ be a unital factorization category. Recall its underlying crystal of categories $\ul{\bA}$ over $\Ran^\untl$ has a multiplicative structure over the disjoint loci.

	\medskip 
	A \emph{unital factorization $\bA$-module category $\bC$ at $\ul{x}_0$} is a crystal of categories $\ul{\bC}$ over $\Ran^\untl_{\ul{x}_0}$ equipped with a \emph{multiplicative $\ul{\bA}$-module structure over the disjoint loci, with respect to the $\Ran^\untl$-module structure on $\Ran^\untl_{\ul{x}_0}$.}

  \medskip
  In other words, for any marked finite set $I=I^\circ\sqcup\{0\}$, we have an isomorphism:
	\begin{equation}
    \label{eqn-structure-equivalence-factmodcat}
		\on{act}_{I}: \big((\boxt_{i\in I^\circ} \ul{\bA})\boxt \ul{\bC}\big)|_\disj \xrightarrow{\simeq} \on{union}_I^*(\ul{\bC})|_\disj,
	\end{equation}
	and a homotopy-coherent datum of compatibility with \eqref{eqn-defn-fact-cat}. Here
	\[
		\on{union}_I: \big( \prod_{i\in I^\circ}\on{Ran}^\untl\big) \times \on{Ran}^\untl_{\ul{x}_0}  \to  \on{Ran}^\untl_{\ul{x}_0} 
	\]
	is the map $(\ul{x}_i, \ul{y})\mapsto (\cup \ul{x}_i) \cup \ul{y}$, and $(-)|_\disj$ means restriction along \eqref{eqn-disj-loci-marked-Ran}.

  \medskip

\sssec{Remark}
  \label{sss untl factmodcat desc}
	Roughly speaking, a unital factorization $\bA$-module category consists of the following data\footnote{To simplify the notations, in below we pretend $\ul{x}_0$ is a $k$-point. Otherwise certain base-changes are necessary.}:
	\begin{itemize}
		\item For any $\ul{y}\in \Ran_{\ul{x}_0}^\untl$, assign a DG category $\bC_{\ul{y}}$;
		\item For any $\ul{y}\subseteq \ul{y}'$ in $\Ran_{\ul{x}_0}^\untl$, assign a functor $\on{ins}_{\ul{y}\subseteq \ul{y}'}: \bC_{\ul{y}} \to \bC_{\ul{y}'}$;
		\item For any finite collection of \emph{disjoint} points $(\ul{x}_i)_{i\in I^\circ}$ in $\Ran^\untl$ and $\ul{y}$ in $\Ran_{\ul{x}_0}^\untl$, assign an equivalence
		\[
			\on{act}_{ (\ul{x}_i,\ul{y})  }: \big(\otimes \bA_{\ul{x}_i}\big) \otimes \bC_{\ul{y}} \xrightarrow{\simeq} \bC_{ (\sqcup \ul{x}_i) \sqcup \ul{y}  }
		\]
		compatible with the equivalences \eqref{eqn-factcat-structure-fiber}.
		\item Commutative squares similar to \eqref{eqn-unital-factcat-explicit}.
		\item Datum of higher compatibilities.
	\end{itemize}

\sssec{Example}
  \label{exam-module-at-empty-set}
  Note that the above definitions make sense even for $\ul{x}_0=\emptyset$. In this case, it is easy to see the following data are equivalent:
  \begin{itemize}
    \item[(i)] A unital factorization $\bA$-module category $\bM$;
    \item[(ii)] A (plain) DG category $\bM_0$.
  \end{itemize}
  Indeed, given $\bM$, we can consider its fiber at $\emptyset\in \Ran_\emptyset^\untl$, which is a DG category; conversely given $\bM_0$, the tensor product 
  \[
      \ul{\bA}\otimes \bM_0  \in \mathbf{CrysCat}(\Ran^\untl) \simeq  \mathbf{CrysCat}(\Ran^\untl_\emptyset)
  \]
  has an obvious unital factorization $\bA$-module structure. One can check these two constructions are inverse to each other.

\sssec{Example}
  \label{sss vacuum factmod cat}
	Let $\bA$ be a unital factorization category. The pullback of $\ul{\bA}$ along the morphism $\Ran_{\ul{x}_0} \to \Ran$ is naturally a factorization $\bA$-module at $\ul{x}_0$. We denote the resulting object by $\bA^{\on{fact}_{\ul{x}_0}}$.

\sssec{Variant}
  \label{sss laxfact modcat}
  As in \secref{sss lax unital factcat}, for a unital lax-factorization category $\bA$, we can define the notion of unital lax-factorization $\bA$-module categories by allowing $\on{act}_{I}$ (see \eqref{eqn-structure-equivalence-factmodcat}) to be non-invertible.

\ssec{Morphisms between unital factorization module categories}
\label{ss factfun factmodcat}
\sssec{}
  \label{sss:lax unital factfun between modules}
	Let $F:\bA\to \bA'$ be a lax-unital factorization functor between unital factorization categories. For unital factorization $\bA$-module category $\bC$ and $\bA'$-module cateogry $\bC'$ at $\ul{x}_0$, a \emph{lax-unital $F$-linear factorization functor $G:\bC \to \bC'$} is a morphism $\ul{G}: \ul{\bC} \to \ul{\bC}'$ in $\mathbf{CrysCat}^{\on{lax}}(\Ran_{\ul{x}_0}^\untl)$ equipped with commutative diagrams (for any marked finite set $I=I^\circ \sqcup\{0\}$)
	\[
		\xymatrix{
			\big((\boxt_{i\in I^\circ} \ul{\bA})\boxt \ul{\bC}\big)|_\disj
			\ar[d]_-{ ((\boxt \ul{F})\boxt \ul{G})|_\disj  }
			\ar[r]^-\simeq 
			& \on{union}_I^*(\ul{\bC})|_\disj
			\ar[d]^-{ \on{union}_I^*(\ul{G})|_\disj  } \\
			\big((\boxt_{i\in I^\circ} \ul{\bA}')\boxt \ul{\bC}'\big)|_\disj
			\ar[r]^-\simeq 
			& \on{union}_I^*(\ul{\bC}')|_\disj
		}
	\]
	and a homotopy-coherent datum of compatibility with \eqref{eqn-factfun-defn}.

	\medskip 
	When $F$ is (strictly) unital, we say $G$ is (strictly) unital if $\ul{G}$ is contained in $\mathbf{CrysCat}^{\on{strict}}(\Ran_{\ul{x}_0}^\untl)$. In fact, $G$ is \emph{automatically} unital.

\begin{lem}
  \label{lem-laxun-mod-is-untl}
  Let $F:\bA\to \bA'$ be a unital factorization functor, and $G:\bC\to \bC'$ be a lax-unital $F$-linear factorization functor between unital factorization module categories at $\ul{x}_0$. Then $G$ is strictly unital.
\end{lem}

\proof[Sketch]
    For any affine points $\ul{y},\ul{y}':S \to \Ran_{\ul{x}_0}^\untl$ such that $\ul{y}\subseteq \ul{y}'$, we need to show the following natural transformation is invertible:
    \[
        \xymatrix{
            \bC_{\ul{y}} 
            \ar[rr]^-{\on{ins}_{\ul{y}\subseteq \ul{y}'}}
            \ar[d]_-{G_{\ul{y}}}
            & & 
            \bC_{\ul{y}'} 
            \ar[d]^-{G_{\ul{y}'}}\\ 
            \bC'_{\ul{y}} 
            \ar[rr]_-{\on{ins}_{\ul{y}\subseteq \ul{y}'}}
            \ar@{=>}[rru]^-{G_{\ul{y}\subseteq \ul{y}'}}
            & & 
            \bC'_{\ul{y}'}.
        }
    \]

    \medskip
    Note that if $S=\cup S_\alpha$ is a finite covering by locally closed subschemes such that $G_{\ul{y}|_{S_\alpha}\subseteq \ul{y}'|_{S_\alpha}}$ is invertible for each $\alpha$, then $G_{\ul{y}\subseteq \ul{y}'}$ is also invertible. Hence without lose of generality, we can assume $\ul{y'} = \ul{y}\sqcup \ul{z}$ such that $\ul{y}\cap \ul{z}=\emptyset$. Using the factorization structure, the natural transformation $G_{\ul{y}\subseteq \ul{y}'}$ can be identified with
    \[
        \xymatrix{
            \bA_{\emptyset}\otimes\bC_{\ul{y}} 
            \ar[rr]^-{\on{ins}_{\emptyset\subseteq \ul{z}} \otimes \on{Id}}
            \ar[d]_-{F_\emptyset \otimes G_{\ul{y}}}
            & & 
            \bA_{\ul{z}}\otimes  \bC_{\ul{y}} 
            \ar[d]^-{F_{\ul{z}}\otimes G_{\ul{y}}}\\ 
            \bA'_{\emptyset}\otimes\bC'_{\ul{y}} 
            \ar[rr]_-{\on{ins}_{\emptyset\subseteq \ul{z}} \otimes \on{Id}}
            \ar@{=>}[rru]^-{F_{\emptyset\subseteq \ul{z}} \otimes\on{Id} }
            & & 
            \bA'_{\ul{z}}\otimes  \bC'_{\ul{y}} .
        }
    \]
    This implies $G_{\ul{y}\subseteq \ul{y}'}$ is invertible because $F_{\emptyset\subseteq \ul{z}}$ is invertible by assumption.

\qed

\sssec{}
  \label{sss the fibration}
	Let 
	\[
	     \mathbf{UntlFactModCat}^{\on{lax-untl}}_{\ul{x}_0}.
	\]
	be the 2-category such that:
	\begin{itemize}
		\item An object is a pair $(\bA,\bC)$, where $\bA$ is a unital factorization category and $\bC$ is a unital factorization $\bA$-module category at $\ul{x}_0$;
    \item A morphism $(\bA,\bC)\to (\bA',\bC')$ is a pair $(F,G)$, where $F$ is a lax-unital factorization functor, and $G$ is a lax-unital $F$-linear factorization functor.
	\end{itemize}
  Let
  \[
      \mathbf{UntlFactModCat}_{\ul{x}_0} \subseteq \mathbf{UntlFactModCat}^{\on{lax-untl}}_{\ul{x}_0}
  \]
  be the 1-full subcategory such that morphisms are strictly unital factorization functors.

  \medskip 
  We have a forgetful functor
  \begin{equation}
    \label{eqn-factmodcatpair-forget-factalgcat}
      \mathbf{UntlFactModCat}^{\on{lax-untl}}_{\ul{x}_0} \to \mathbf{UntlFactCat}^{\on{lax-untl}},\; (\bA,\bC) \mapsto \bA.
  \end{equation}
  The fiber of this functor at $\bA$ is
  \[
      \bA\mathbf{-mod}^\onfact_{\ul{x}_0},
  \]
  the 2-category of unital factorization $\bA$-module categories at $\ul{x}_0$. Note that this category is also the fiber of the forgetful functor
  \[
      \mathbf{UntlFactModCat}_{\ul{x}_0} \to \mathbf{UntlFactCat},\; (\bA,\bC) \mapsto \bA.
  \]
  See Lemma \ref{lem-laxun-mod-is-untl}.

\sssec{}
  \label{sss cryscat of base act}
  Note that $\Ran_{\ul{x}_0}^\untl$, together with its $\Ran^\untl$-module structure, is defined over $S_0$. It follows that the symmetric monoidal 2-category
  \[
      \mathbf{CrysCat}(S_0) \simeq \Dmod(S_0)\mathbf{-mod}
  \]
  acts on the fibers of \eqref{eqn-factmodcatpair-forget-factalgcat}. In particular, it acts on $\bA\mathbf{-mod}^\onfact_{\ul{x}_0}$.

\sssec{}
  Let $\bA$ be a unital factorization category and $\bC$ be a unital factorization $\bA$-module category at $\ul{x}_0$. Recall a factorization algebra $\CA$ in $\bA$ is a lax-unital factorization functor $\CA:\Vect \to \bA$ (see \secref{sss defn factalg in factcat}).

  \medskip
  A \emph{factorization $\CA$-module $\CC$ in $\bC$} is \emph{defined} to be a lax-unital factorization $\CA$-linear functor 
  \[
      \CC:\Vect^{\on{fact}_{\ul{x}_0}} \to \bC.
  \]
  Here $\Vect^{\on{fact}_{\ul{x}_0}}$ is the unital factorization $\Vect$-module category in \secref{sss vacuum factmod cat}.

  \medskip
  Pairs $(\CA,\CC)$ of factorization algebras and modules in $(\bA,\bC)$ form a category, which is denoted by $\on{FactMod}(\bA,\bC)_{\ul x_0}$. There is a forgetful functor
  \begin{equation}
    \label{eqn factmodpair to factalg}
    \on{FactMod}(\bA,\bC)_{\ul{x}_0} \to \on{FactAlg}(\bA),\; (\CA,\CC)\mapsto \CA
  \end{equation}
  whose fiber at $\CA$ is the category
  \[
    \CA\on{-mod}^\onfact(\bC)_{\ul{x}_0}
  \]
   of factorization $\CA$-modules in $\bC$.

\sssec{}
  The action in \secref{sss cryscat of base act} induces an action of $\Dmod(S_0)$ on the fibers of the functor \eqref{eqn factmodpair to factalg}. In particular $\Dmod(S_0)$ acts on $\CA\on{-mod}^\onfact(\bC)_{\ul{x}_0}$.

\sssec{}
  For $\bC=\bA^{\on{fact}_{\ul{x}_0}}$, we write
  \[
    \CA\on{-mod}^\onfact_{\ul{x}_0}:= \CA\on{-mod}^\onfact(\bA^{\on{fact}_{\ul{x}_0}})_{\ul{x}_0}.
  \]
  See \secref{sss vacuum factmod cat}.

\sssec{}
  \label{sss trad factmod object}
  Below is a concrete description of factorization $\CA$-modules in a unital factorization $\bA$-module category $\bC$.

  \medskip
  As in \secref{sss traddesc factalg}, for each marked finite set $I=I^\circ \sqcup\{0\}$, we have a functor
  \[
    \big(\prod_{i\in I^\circ}\Gamma^{\on{lax}}(\Ran^\untl, \ul{\bA})\big) \times \Gamma^{\on{lax}}(\Ran^\untl_{\ul{x}_0}, \ul{\bC})
    \to \Gamma^{\on{lax}}\big(  \big( (\prod_{i\in I}  \Ran^\untl) \times \Ran^\untl_{\ul{x}_0}\big)_\disj , \big((\boxt_{i\in I} \ul{\bA}) \boxt \ul{\bC}\big) |_\disj \big) 
  \]
  that sends
  \[
    (\CM_i, \CF) \to \big((\boxt \CM_i)\boxt \CF\big)|_\disj,
  \]
  and a functor
  \[
    \Gamma^{\on{lax}}(\Ran^\untl_{\ul{x}_0}, \ul{\bC}) \to \Gamma^{\on{lax}}\big(  \big( (\prod_{i\in I}  \Ran^\untl) \times \Ran^\untl_{\ul{x}_0}\big)_\disj , \on{union}_I^*(\bC)|_\disj \big) 
  \]
  that sends
  \[
    \CN \mapsto \on{union}_I^!(\CN)|_\disj.
  \]
  The factorization structure on $\bC$ provides an equivalence
  \[
    \begin{aligned}
    \on{act}_I: \Gamma^{\on{lax}}\big(  \big( (\prod_{i\in I}  \Ran^\untl) \times \Ran^\untl_{\ul{x}_0}\big)_\disj , \big((\boxt_{i\in I} \ul{\bA}) \boxt \ul{\bC}\big) |_\disj \big)  \to  \\
    \to \Gamma^{\on{lax}}\big(  \big( (\prod_{i\in I}  \Ran^\untl) \times \Ran^\untl_{\ul{x}_0}\big)_\disj , \on{union}_I^*(\bC)|_\disj \big).
    \end{aligned}
  \]
  Then a factorization $\CA$-module $\CC$ in $\bC$ is an object
  \[
    \ul{\CC} \in  \Gamma^{\on{lax}}(\Ran^\untl_{\ul{x}_0}, \ul{\bC}) =: \ul{\bC}_{\Ran^\untl_{\ul{x}_0}}^{\on{lax}}
  \]
  equipped with isomorphisms (for any marked finite set $I=I^\circ \sqcup\{0\}$)
  \[
    \on{act}_I\big(  \big((\boxt_{i\in I^\circ} \ul{\CA})\boxt \ul{\CC}\big) |_\disj \big) \xrightarrow{\simeq} \on{union}_I^!(\ul{\CC})|_\disj
  \]
  and a homotopy-coherent datum of compatibility with \eqref{eqn-structure-iso-factalg}.

\sssec{}
  Unwinding the definitions, we have the following result. A rigorious proof will be provided in \cite{CFZ}.

\begin{lem}
  Let $\bA$ be a unital factorization category and $\bC$ be a unital factorization $\bA$-module category at $\ul{x}_0$. For a factorization algebra $\CA$ in $\bA$, the category $\CA\on{-mod}^\onfact(\bC)_{\ul{x}_0}$ is naturally a DG $\Dmod(S_0)$-module category\footnote{In other words, $\CA\on{-mod}^\onfact(\bC)_{\ul{x}_0}$ is presentable and stable, and the acting functors are compatible with colimits.} and the forgetful functors
  \[
      \CA\on{-mod}^\onfact(\bC)_{\ul{x}_0} \to \Gamma^{\on{lax}}(\Ran^\untl_{\ul{x}_0}, \ul{\bC}) \to  \bC_{\ul{x}_0},\; \CC\mapsto \ul{\CC} \mapsto \CC_{\ul{x}_0}
  \]
  are $\Dmod(S_0)$-linear and colimit-preserving.
\end{lem}

\sssec{Example}
  \label{sss vac factmod}
  Let $(\bA,\bA^{\on{fact}_{\ul{x}_0}})$ be as in \secref{sss vacuum factmod cat}. For a factorization algebra $\CA$ in $\bA$, its restriction along $\Ran_{\ul{x}_0}^\untl \to \Ran^\untl$ defines a factorization $\CA$-module in $\bA^{\on{fact}_{\ul{x}_0}}$. We denote this object by
  \[
      \CA^{\on{fact}_{\ul{x}_0}} \in \CA\on{-mod}^\onfact(\bC)_{\ul{x}_0}.
  \]

\sssec{}
  The following result appears in \cite[Lemma C.14.10]{GLC2}. We will provide a more homotopy-coherent proof in \cite{CFZ}.

  \begin{lem}
    \label{lem module for unit is nothing}
    Let $\bA$ be a unital factorization category and $\bC$ be a unital factorization $\bA$-module category at $\ul{x}_0$. Then the forgetful functor
    \[
      \on{unit}_\bA\on{-mod}^\onfact(\bC)_{\ul{x}_0} \to  \bC_{\ul{x}_0},\; \CC\mapsto \CC_{\ul{x}_0}
    \]
    is an equivalence.
  \end{lem}

\sssec{} \label{sss:lax unital sends factmod to factmod}
  By definition, for any lax unital factorization functors $(F,G):(\bA,\bC)\to (\bA',\bC')$ between unital factorization categories and their modules, we have a functor
  \[
      \on{FactMod}(\bA,\bC)_{\ul{x}_0} \to \on{FactMod}(\bA',\bC')_{\ul{x}_0}, \; (\CA,\CC) \mapsto (F\circ \CA,G\circ \CC)
  \]
  compatible with \eqref{eqn-laxuntl-factalg-to-factalg}. In particular, we obtain a functor
  \[
      \CA\on{-mod}^\onfact(\bC)_{\ul{x}_0} \to F(\CA)\on{-mod}^\onfact(\bC')_{\ul{x}_0},\; \CC \mapsto G\circ \CC. 
  \]
  We also write $G(\CC):=G\circ \CC$.

\sssec{Variant}
  As in \secref{sss laxfact functors}, we can define \emph{lax-factorization functors} between unital factorization module categories, and use such a notion to define \emph{lax-factorization module objects}.

\ssec{Change of base}
  \label{ss change of base}
  In \secref{ss factmodcat} and \secref{ss factfun factmodcat}, we introduced factorization \emph{module} structures at a fixed affine point $\ul{x}_0\in \Ran^\untl(S_0)$. In this subsection, we explian how such structures depend on $\ul{x}_0$.

\sssec{}
  Throughout this subsection, we fix the following notations.

  \medskip
  Let $\ul{x}_0\in \Ran^\untl(S_0)$ and $\ul{x}_0'\in \Ran^\untl(S_0')$ be two affine points and 
  \begin{equation}
    \label{eqn-2-cell-change-of-base}
      \xymatrix{
          & S_0 \ar[rd]^-{\ul{x}_0} \ar@{=>}[d]^-\alpha \\ 
          S_0' \ar[rr]_-{\ul{x}_0'} \ar[ru]^-{f_\alpha} & & \Ran^\untl
      }
  \end{equation}
  be a 2-morphism in $\on{CatPreStk}$. In other words, $\alpha$ is a morphism in $\widetilde{ \Ran^\untl  }$ (see \secref{sss actual defn of cryscat:label}).

\sssec{}
  \label{sss propogation factmodcat}
By construction, there is a canonical morphism
  \[
      \Ran_\alpha: \Ran_{\ul{x}_0'}^\untl \to \Ran_{\ul{x}_0}^\untl
  \]
  defined over the morphism $f_\alpha:S_0'\to S_0$ and compatible with the forgetful morphisms to $\Ran^\untl$. Moreover, this morphism is compatible with the $\Ran^\untl$-module structures on $\Ran_{\ul{x}_0}^\untl$ and $\Ran_{\ul{x}_0'}^\untl$ (see \secref{sss markRan module}).

  \medskip
  It follows that pullback along $\Ran_{\ul{x}_0'}^\untl \to \Ran_{\ul{x}_0}^\untl$ defines a functor
  \begin{equation}
    \label{eqn-change-of-base-factmodcat}
      \alpha_\dagger:\mathbf{UntlFactModCat}_{\ul{x}_0}^{\on{lax-untl}} \to \mathbf{UntlFactModCat}_{\ul{x}'_0}^{\on{lax-untl}}
  \end{equation}
  compatible with the forgetful functors to $\mathbf{UntlFactCat}^{\on{lax-untl}}$\footnote{Note that the definition of factorization (algebra) categories is independent of the points $\ul{x}_0$.}. 

  \medskip
  In particular, for $\bA\in \mathbf{UntlFactCat}^{\on{lax-untl}}$, we obtain a functor
  \[
    \alpha_\dagger: \bA\mathbf{-mod}^\onfact_{\ul{x}_0} \to \bA\mathbf{-mod}^\onfact_{\ul{x}_0'}.
  \]
  We write
  \[
      \bC|_{\Ran_{\ul x_0'}^\untl} := \alpha_\dagger(\bC).
  \]

  \medskip 
  In the case when $f_\alpha=\on{id}_{S_0}$ is the identity morphism, we write
   \[
      \mathbf{prop}_{\ul{x}_0 \subseteq \ul{x}_0'}:= \alpha_\dagger: \bA\mathbf{-mod}^\onfact_{\ul{x}_0} \to \bA\mathbf{-mod}^\onfact_{\ul{x}_0'}
   \]
   and call $\mathbf{prop}_{\ul{x}_0 \subseteq \ul{x}_0'}(\bC)$ the \emph{propogation of $\bC$ along $\ul{x}_0 \subseteq \ul{x}_0'$}.

\sssec{Remark}
  The construction $\alpha\mapsto \alpha_\dagger$ is compatible with compositions. In fact, one can construct a coCartesian fibration of 2-categories
  \[
      \mathbf{UntlFactModCat}^{\on{lax-untl}} \to \widetilde{ \Ran^\untl  }
  \]
  such that its fiber at $\ul{x}_0$ is $\mathbf{UntlFactModCat}_{\ul{x}_0}^{\on{lax-untl}} $, and the covariant transport functor along $\alpha$ is $\alpha_\dagger$. A rigorious construction of this coCartesian fibration will be provided in \cite{CFZ}.
  
\sssec{Example}
  Let $\bA$ be a unital factorization category. The functor
  \[
      \alpha_\dagger: \bA\mathbf{-mod}^\onfact_{\ul{x}_0} \to \bA\mathbf{-mod}^\onfact_{\ul{x}_0'}
  \]
  sends $\bA^{\on{fact}_{\ul{x}_0}}$ to $\bA^{\on{fact}_{\ul{x}_0'}}$.   

\sssec{}
  Let $\bA$ be a unital factorization category and $\bC\in \bA\mathbf{-mod}^\onfact_{\ul{x}_0}$. Write 
  \[
    \bC|_{ \Ran_{\ul{x}_0'}^\untl  }:=\alpha_\dagger(\bC) \in \bA\mathbf{-mod}^\onfact_{\ul{x}_0'}.
  \]
  The functor \eqref{eqn-change-of-base-factmodcat} induces a functor
  \[
    \begin{aligned}
    \alpha_\dagger: \on{Fun}_{\mathbf{UntlFactModCat}_{\ul{x}_0}^{\on{lax-untl}}}( (\Vect,\Vect^{\on{fact}_{\ul{x}_0}}) ,(\bA,\bC) ) \to \\
    \to \on{Fun}_{\mathbf{UntlFactModCat}_{\ul{x}'_0}^{\on{lax-untl}}}( (\Vect,\Vect^{\on{fact}_{\ul{x}_0'}}) ,(\bA,\bC|_{ \Ran_{\ul{x}_0'}^\untl  }) ).
    \end{aligned}
  \]
  By definition, this is a functor
  \begin{equation}
    \label{eqn-change-of-base-factmod}
      \alpha_\dagger:\on{FactMod}(\bA,\bC)_{\ul{x}_0} \to \on{FactMod}(\bA,\bC|_{ \Ran_{\ul{x}_0'}^\untl  })_{\ul{x}'_0}
  \end{equation}
  compatible with forgetful functors to $\on{FactAlg}(\bA)$. 

  \medskip 
  In particular, for $\CA\in \on{FactAlg}(\bA)$, we obtain a functor
  \begin{equation}
    \label{eqn-change-of-base-factmod-fixalg}
      \alpha_\dagger: \CA\on{-mod}^\onfact(\bC)_{\ul{x}_0} \to \CA\on{-mod}^\onfact(\bC|_{ \Ran_{\ul{x}_0'}^\untl  })_{\ul{x}'_0}.
  \end{equation}
  When $\bC=\bA^{\on{fact}_{\ul x_0}}$, this gives 
  \begin{equation}
    \label{eqn-change-of-base-factmod-fixalg-vac}
      \alpha_\dagger: \CA\on{-mod}^\onfact_{\ul{x}_0} \to \CA\on{-mod}^\onfact_{\ul{x}'_0}
  \end{equation}

  \medskip
  In the case when $f_\alpha=\on{id}_{S_0}$ is the identity morphism, recall 
  \[
    \mathbf{prop}_{\ul{x}_0 \subseteq \ul{x}_0'}(\bC):=\bC|_{ \Ran_{\ul{x}_0'}^\untl  }.
  \]
  Hence we also denote the functor \eqref{eqn-change-of-base-factmod-fixalg} by
   \[
      \on{prop}_{\ul{x}_0 \subseteq \ul{x}_0'}:= \alpha_\dagger: \CA\on{-mod}^\onfact(\bC)_{\ul{x}_0} \to \CA\on{-mod}^\onfact(\mathbf{prop}_{\ul{x}_0 \subseteq \ul{x}_0'}(\bC))_{\ul{x}'_0}
   \]
   and call $\on{prop}_{\ul{x}_0 \subseteq \ul{x}_0'}(\CC)$ the \emph{propogation of $\CC$ along $\ul{x}_0 \subseteq \ul{x}_0'$}.

\sssec{Example}
  The functor \eqref{eqn-change-of-base-factmod-fixalg-vac} sends $\CA^{\on{fact}_{\ul{x}_0}}$ to $\CA^{\on{fact}_{\ul{x}_0'}}$. See \secref{sss vac factmod}.

\sssec{}
  By construction, the functor \eqref{eqn-change-of-base-factmod-fixalg} fits into a canonical commutative diagram
  \[
      \xymatrix{
          \CA\on{-mod}^\onfact(\bC)_{\ul{x}_0}  \ar[r] \ar[d]
          & \CA\on{-mod}^\onfact(\bC|_{ \Ran_{\ul{x}_0'}^\untl  })_{\ul{x}'_0} \ar[d] \\
          \Gamma^{\on{lax}}(\Ran^\untl_{\ul x_0},  \ul\bC) \ar[r]^-{!\on{-pull}}
          &
          \Gamma^{\on{lax}}(\Ran^\untl_{\ul x_0'}, \ul\bC|_{\Ran^\untl_{\ul x_0'}}),
      }
  \]
  where the vertical arrows are the forgetful functors. 

  \medskip
  On the other hand, the 2-morphim \eqref{eqn-2-cell-change-of-base} induces a 2-morphism
  \[
      \xymatrix{
          S_0' \ar[r]^-{f_\alpha} \ar[d]_-{\ul{x}_0'}
          & S_0  \ar[d]^-{\ul{x}_0} \ar@{=>}[ld] \\ 
          \Ran^\untl_{\ul x_0'} \ar[r]_-{\Ran_\alpha}
          & \Ran^\untl_{\ul x_0},
      }
  \]
  which induces a natural transformation
  \[
      \xymatrix{
          \Gamma^{\on{lax}}(\Ran^\untl_{\ul x_0}, \ul\bC) \ar[r]^-{!\on{-pull}} \ar[d]_-{!\on{-pull}}
          &
          \Gamma^{\on{lax}}(\Ran^\untl_{\ul x_0'}, \ul\bC|_{\Ran^\untl_{\ul x_0'}}) \ar[d]^-{!\on{-pull}} \\ 
          \bC_{\ul x_0} \ar[r]_-{\bC_{\ul{x}_0\subseteq \ul x_0'}} \ar@{=>}[ru] & \bC_{\ul x_0'}
      }
  \]
  (by the definition of lax sections). Combining these squares, we obtain a canonical natural transformation
  \begin{equation}
    \label{eqn-A-modfact-is-lax-unital}
      \xymatrix{
          \CA\on{-mod}^\onfact(\bC)_{\ul{x}_0} \ar[r] \ar[d]
          & \CA\on{-mod}^\onfact(\bC|_{ \Ran_{\ul{x}_0'}^\untl  })_{\ul{x}'_0} \ar[d] \\
          \bC_{\ul x_0} \ar[r]_-{\bC_{\ul{x}_0\subseteq \ul x_0'}} \ar@{=>}[ru] & \bC_{\ul x_0'}
      }
  \end{equation}
  such that the vertical arrows are the forgetul functors.

\sssec{Remark}
  The natural transformation \eqref{eqn-A-modfact-is-lax-unital} can be concretely described as follows. For simplicity, we assume $S_0=S_0'=\on{pt}$.

  \medskip
  Let $\CC\in \CA\on{-mod}^\onfact(\bC)_{\ul{x}_0} $ be a factorization $\CA$-module in $\bC$.

  \medskip 
  The top horizontal arrow sends $\CC$ to its $!$-pullback $\CC|_{\Ran_{\ul{x}_0'}^\untl}$, which is lax global section $\bC|_{ \Ran_{\ul{x}_0'}^\untl  }$ equipped with a canonical factorization $\CA$-module structure. Hence the clockwise arch sends $\CC$ to its fiber 
  \[
      \ul\CC|_{\ul{x}_0}' \in \bC_{\ul{x}_0'}
  \]
  at the point $\ul x_0'\in \Ran_{\ul{x}_0}$

  \medskip 
  On the other hand, the counterclockwise arch sends $\CC$ to the object
  \[
      \on{ins}_{\ul{x}_0\subseteq \ul{x}_0'}( \ul\CC_{\ul{x}_0}  )\in \bC_{\ul{x}_0'},
  \]
  where $\on{ins}_{\ul{x}_0\subseteq \ul{x}_0'}:\bC_{\ul{x}_0} \to \bC_{\ul{x}_0}'$ is part of the structure of $\ul\bC$ as a crystal of categories (see \secref{rem cryscat on Ran:label}).

  \medskip
  Now the value of \eqref{eqn-A-modfact-is-lax-unital} at $\CC$ is given by the morphism
  \[
    \on{ins}_{\ul{x}_0\subseteq \ul{x}_0'}( \ul\CC_{\ul{x}_0}  )\to \ul\CC|_{\ul{x}_0}',
  \]
  which is part of the structure of $\ul \CC$ as a lax global section (see \secref{rem description lax unital morphism over Ran}).

\sssec{}
  Since the morphism 
  \[
      \Ran_\alpha: \Ran_{\ul{x}_0'}^\untl \to \Ran_{\ul{x}_0}^\untl
  \]
  is defined over the morphism $f_\alpha:S_0'\to S_0$. The functor \eqref{eqn-change-of-base-factmodcat} intertwines the action of the symmetric monoidal functor
  \[
    f_\alpha^*: \mathbf{CrysCat}(S_0) \to \mathbf{CrysCat}(S_0'),
  \]
  see \secref{sss cryscat of base act}. It follows that the functor \eqref{eqn-change-of-base-factmod-fixalg} intertwines the action of the symmetric monoidal functor
  \[
    f_\alpha^!: \Dmod(S_0) \to \Dmod(S_0').
  \]
  In particular, we have a canonical functor
  \begin{equation}
    \label{eqn-change-of-base-factmod-fixalg-tensor}
    \CA\on{-mod}^\onfact(\bC)_{\ul{x}_0}\otimes_{\Dmod(S_0)} \Dmod(S_0') \to \CA\on{-mod}^\onfact(\bC|_{ \Ran_{\ul{x}_0'}^\untl  })_{\ul{x}'_0}.
  \end{equation}

  \medskip 
  The following result is stated without proof in \cite[Sect. C.11.9]{GLC2} (but its non-unital analog is proved, see \cite[Lemma B.9.11]{GLC2}). We will provide a proof in \cite{CFZ}.

  \begin{lem}
    \label{lem change of base factcat mod}
      In the above setting, the functor \eqref{eqn-change-of-base-factmod-fixalg-tensor} is invertible if $\alpha$ is so. In particular, the functor
      \[
          \widetilde{ \Ran^\untl_{\ul x_0}} \to  \widetilde{\mathbf{CrysCat}},\; \ul{y} \mapsto \CA\on{-mod}^\onfact(\bC|_{ \Ran_{\ul{y}}^\untl  })_{\ul{y}}
      \]
      defines a crystal of category
      \[
          \CA\ul{\on{-mod}}^\onfact(\bC)
      \]
      over $\Ran^\untl_{\ul x_0}$ (see \secref{sss actual defn of cryscat:label}).
  \end{lem}

\sssec{}
  By construction, \eqref{eqn-A-modfact-is-lax-unital} provides a morphism
  \begin{equation}
    \label{eqn-oblv-CA-is-lax-morphism}
      \ul\oblv_\CA:\CA\ul{\on{-mod}}^\onfact(\bC) \to \ul{\bC}
  \end{equation}
  in $\mathbf{CrysCat}^{\on{lax}}(\Ran^\untl_{\ul x_0})$.

\sssec{}
  \label{sss factmod form cryscat}
  In particular, for $\ul{x}_0=\emptyset\in \Ran^\untl$, we obtain a crystal of category
  \[
      \CA\ul{\on{-mod}}^\onfact:= \CA\ul{\on{-mod}}^\onfact( \bA^{\on{fact}_\emptyset}  )
  \]
  over $\Ran^\untl$, equipped with a forgetful morphism
  \[
    \ul\oblv_\CA: \CA\ul{\on{-mod}}^\onfact\to \ul\bA.
  \]

\ssec{External fusion}
  \label{ss external fusion}
  The main goal of this subsection is to explain $\CA\ul{\on{-mod}}^\onfact$ (see Lemma \ref{lem change of base factcat mod}) is naturally a unital \emph{lax-factorization} category via \emph{external fusion}.

  \medskip
  The \emph{construction} of external fusion for factorization modules was sketched in \cite[Sect. 6.22]{Ra} and \cite[Sect. B.11.14]{GLC2}. However, to work with external fusion, especially in a homotopy-coherent way, it is better to characterize them via universal properties. Recall the tensor product of usual modules can be \emph{defined} as the object that corepresents multilinear morphisms. Following this idea, we will \emph{define} the fusion product of factorization modules as the object that corepresents \emph{factorization multi-functors}. In fact, this approach to external fusion was alluded to in \cite[Sect. 6.26]{Ra}.

\sssec{}
  Let $S_0$ be an affine scheme and $\ul{x}_i \in \Ran^\untl(S_0)$ ($i\in I$) be a finite collection of \emph{disjoint} $S_0$-points. Consider the $S_0$-morphism
  \begin{equation}
    \label{eqn fusion Ran}
    \on{union}_{(\ul{x}_i)}:\big(\prod_{i\in I} \Ran^\untl_{\ul{x}_i}\big)_{/S_0} \to \Ran^\untl_{\sqcup \ul{x}_i},\; (\ul{y}_i) \mapsto \cup \ul{y}_i,
  \end{equation}
  where the source is the fiber product of $\Ran^\untl_{\ul{x}_i}$ relative to $S_0$. By definition, \eqref{eqn fusion Ran} is a $\Ran^\untl$-multilinear morphism, i.e., it is $\Ran^\untl$-linear in each factor of the source.

  \medskip
  Let
  \begin{equation}
    \label{eqn fusion Ran disj}
       \big( \big(\prod_{i\in I} \Ran^\untl_{\ul{x}_i}\big)_{/S_0}  \big)_\disj \subseteq \big(\prod_{i\in I} \Ran^\untl_{\ul{x}_i}\big)_{/S_0}
  \end{equation}
  be the subfunctor containing \emph{disjoint} points $(\ul{y}_i)_{i\in I}$.

\sssec{}
    Let $\bA$ and $\bA'$ be unital factorization categories and
    \[
         \bC_i \in \bA\mathbf{-mod}^\onfact_{\ul{x}_i}\;\;(i\in I), \;\; \bC'\in \bA'\mathbf{-mod}^\onfact_{\sqcup \ul{x}_i},
    \]
    where $I$ is a finite set.

    \medskip
    Consider the external product of $\ul{\bC}_i$ relative to $S_0$:
  \[
      (\boxt \ul{\bC}_i)_{/S_0}:= \otimes \on{pr}_i^*( \ul{\bC}_i  ) \in \mathbf{CrysCat}^{\on{lax}}\big( \big(\prod_{i\in I} \Ran^\untl_{\ul{x}_i}\big)_{/S_0}  \big)
  \]
  and its restriction to the disjoint locus \eqref{eqn fusion Ran disj}:
  \[
      \big((\boxt \ul{\bC}_i)_{/S_0}\big)|_\disj.
  \]
  Note that for each $i\in I$, the factorization $\bA$-module structure on $\ul{\bC}_i$ induces a factorization $\bA$-module structure on $(\boxt \ul{\bC}_i)_{/S_0}$ with respect to the $\Ran^\untl$-module structure on 
  \[
      \big( \big(\prod_{i\in I} \Ran^\untl_{\ul{x}_i}\big)_{/S_0}  \big)_\disj
  \]
  that comes from the $i$-th factor\footnote{Using the language in \secref{sss:unital factmodcat}, this means $(\boxt \ul{\bC}_i)_{/S_0}$ has a multiplicative $\ul{\bA}$-module structure over the disjoint loci with respect to the $i$-th $\Ran^\untl$-module structure on $\big( \big(\prod_{i\in I} \Ran^\untl_{\ul{x}_i}\big)_{/S_0}  \big)_\disj$.}.

  \medskip
  Similarly, for each $i\in I$, the factorization $\bA'$-module structure on $\ul{\bC}'$ induces a same-typed structure on
  \[
      \on{union}_{(\ul{x}_i)}^*(\ul{\bC}')|_\disj 
  \]
  because the map \eqref{eqn fusion Ran} is $\Ran^\untl$-multilinear.

\sssec{}
  A \emph{lax-unital factorization multifunctor}
  \[
      (F,G): (\bA,(\bC_i)_{i\in I}) \to (\bA',\bC')
  \] 
  consists of the following data:
  \begin{itemize}
    \item 
        A lax-unital factorization functor $F:\bA\to \bA'$;
    \item 
        A morphism
        \[
      \ul{G}: \big((\boxt \ul{\bC}_i)_{/S_0}\big)|_\disj\to \on{union}_{(\ul{x}_i)}^*(\ul{\bC}')|_\disj 
  \]
  in 
  \[
       \mathbf{CrysCat}^{\on{lax}}\big( \big( \big(\prod_{i\in I} \Ran^\untl_{\ul{x}_i}\big)_{/S_0}  \big)_\disj \big),
  \]
  such that for each $i\in I$, $\ul{G}$ is a factorization $F$-linear functor\footnote{This is a structure rather than property.} with respect to the $i$-th factorization module structures on the source and the target.
  \end{itemize}
  For a fixed $F$, we call 
  \[
      G : (\bC_i)_{i\in I} \to \bC' 
  \]
  as above a \emph{lax-unital factorization $F$-linear multi-functor}.

  \medskip 
  We say $(F,G)$ is (strictly) unital if $\ul{F}$ and $\ul{G}$ are strict morphisms.

\sssec{}
  One can mimic the definition of compositions of usual multilinear maps to define compositions of \emph{factorization} multilinear functors. Namely, for a given map $\phi:I\to I'$ between finite sets and (lax-)unital factorization multifunctors
  \[
      (F,G_{i'}): (\bA,(\bC_i)_{i\in \phi^{-1}(i')}) \to (\bA',\bC'_{i'}) \;\; i'\in I'
  \]
  and
  \[
      (F',G'): (\bA',(\bC'_{i'})_{i'\in I'}) \to (\bA'',\bC''),
  \]
  there is a canonical (lax-)unital factorization multifunctor
  \[
      \big(F'\circ F, G'\circ_\phi(G_{i'}) \big) :(\bA,(\bC_i)_{i\in I}) \to   (\bA'',\bC'').
  \]
  A homotopy-coherent construction of these compositions will be provided in \cite{CFZ}.

\sssec{Remark}
  The notion of $\bA$-multilinear functors and their compositions is closely related to the framework of \emph{pseudo-tensor categories} in \cite{BD1}. See \secref{ss fact via operad} for more details.

\sssec{Remark}
  Roughly speaking, a lax-unital factorization $F$-multilinear functor $ G: (\bC_i)_{i\in I} \to \bC'$ consists of the following data:
  \begin{itemize}
    \item 
      For any disjoint collection of points $\ul{y}_i \in \Ran^\untl_{\ul{x}_i}$ ($i\in I$), assign a functor
      \[
          G_{(\ul{y}_i)}: \otimes (\bC_i)_{\ul{y}_i} \to \bC'_{\sqcup \ul{y}_i};
      \]
    \item 
      For two collections $(\ul{y}_i)$ and $(\ul{y}_i')$ as above such that $\ul{y}_i\subseteq \ul{y}_i'$, assign a natural transformation
      \[
          \xymatrix{
              \otimes (\bC_i)_{\ul{y}_i} \ar[rr]^-{\otimes \on{ins}_{\ul{y}_i\subseteq \ul{y}_i'}}
              \ar[d]_-{G_{(\ul{y}_i)}}
              & & \otimes (\bC_i)_{\ul{y}_i'} 
              \ar[d]^-{G_{(\ul{y}_i')}} \\ 
              \bC'_{\sqcup \ul{y}_i} \ar@{=>}[rru]^-{G_{(\ul{y}_i\subseteq \ul{y}_i')}}
              \ar[rr]_-{\on{ins}_{\sqcup \ul{y}_i\subseteq \sqcup \ul{y}_i'}}
              & & \bC'_{\sqcup \ul{y}_i'};
          }
      \]
      \item
        For any disjoint collection of points $\ul{y}_i \in \Ran^\untl_{\ul{x}_i}$ ($i\in I$), $\ul{z}_j \in \Ran^\untl$ ($j\in J$) and a map $\phi:J\to I$, assign a commutative diagram
        \[
          \xymatrix{
            (\otimes_{j} \bA_{\ul z_j}) \bigotimes (\otimes_i (\bC_i)_{\ul{y}_i}) \ar[r]^-\simeq 
            \ar[d]_-{ (\otimes F_{\ul{z}_j}) \otimes G_{(\ul{y}_i)}}
            & 
            \bigotimes_i \big( (\otimes_{j\in \phi^{-1}(i)} \bA_{\ul z_j}) \otimes (\bC_i)_{\ul{y}_i} \big) 
            \ar[r]^-\simeq
            & 
            \bigotimes_i  (\bC_i)_{(\sqcup_{j\in \phi^{-1}(i)} \ul z_j) \sqcup  \ul{y}_i} 
            \ar[d]^-{ G_{  (\sqcup_{j\in \phi^{-1}(i)} \ul z_j) \sqcup  \ul{y}_i    }  } \\ 
            (\otimes_{j} \bA'_{\ul z_j}) \otimes \bC'_{\sqcup \ul{y}_i}
            \ar[rr]_-\simeq & &
            \bC'_{ (\sqcup \ul z_j) \bigsqcup (\sqcup \ul y_i ) },
          }
        \]
        where the horizontal equivalences come from the factorization module structures on $\bC_i$ and $\bD$.
      \item 
        Higher compatibilities between the above structures.
  \end{itemize}

\sssec{Example}
  \label{exam fusion with vac}
  Let $\bA$ be a unital factorization category and $\bC\in \bA\mathbf{-mod}^\onfact_{\ul{x}_0}$. For $\ul{x}\in \Ran^\untl(S_0)$ such that $\ul{x}\cap \ul{x}_0=\emptyset$, consider 
  \[
      \mathbf{prop}_{\ul{x}_0\subseteq \ul{x}_0\sqcup \ul{x}}(\bC) \in \bA\mathbf{-mod}^\onfact_{\ul{x}_0\sqcup\ul{x}}
  \]
  (see \secref{sss propogation factmodcat}). There is a canonical unital factorization $\bA$-linear bifunctor
  \[
      ( \bA^{\on{fact}_{\ul{x}}}, \bC  ) \to \mathbf{prop}_{\ul{x}_0\subseteq \ul{x}_0\sqcup \ul{x}}(\bC)
  \]
  constructed as follows. 

  \medskip
  By definition, we need to define a strict morphism 
  \begin{equation}
    \label{eqn-bilinear}
      (\ul\bA|_{\Ran^\untl_{\ul{x}}} \boxt_{S_0} \ul\bC)|_\disj \to \ul\bC|_{(\Ran_{\ul{x}}^\untl \times_{S_0} \Ran_{\ul{x}_0}^\untl)_\disj  )}
  \end{equation}
  between crystal of categories over $(\Ran_{\ul{x}}^\untl \times_{S_0} \Ran_{\ul{x}_0}^\untl)_\disj  )$, such that it is compatible with the factorization $\bA$-module structures coming from both factors. Here the RHS means pullback of $\ul{\bC}$ along the map
  \[
    (\Ran_{\ul{x}}^\untl \times_{S_0} \Ran_{\ul{x}_0}^\untl)_\disj   \to \Ran_{\ul{x}_0}^\untl,\; (\ul{y},\ul{y}_0) \to \ul y \sqcup \ul y_0.
  \]
  Note that this map factors as
  \[
     (\Ran_{\ul{x}}^\untl \times_{S_0} \Ran_{\ul{x}_0}^\untl)_\disj \to (\Ran^\untl \times \Ran_{\ul{x}_0}^\untl)_\disj \xrightarrow{\on{union}} \Ran_{\ul{x}_0}^\untl.
  \]
  Hence to construct \eqref{eqn-bilinear}, we only need a strict morphism
  \begin{equation}
    \label{eqn-bilinear-2}
     (\ul\bA \boxt \ul\bC)|_\disj \to \ul\bC|_{(\Ran^\untl \times \Ran_{\ul{x}_0}^\untl)_\disj}
  \end{equation}
  compatible with the factorization $\bA$-module structures coming from both factors. However, the factorization $\bA$-module structure of $\bC$ implies there is a canonical \emph{isomorphism} \eqref{eqn-bilinear-2}.

\sssec{}
  Let $\bA$ be a unital factorization category and
  \[
      \bC_i \in \bA\mathbf{-mod}^\onfact_{\ul{x}_i}\;\;(i\in I),\;\; \bD\in \bA\mathbf{-mod}^\onfact_{\sqcup\ul{x}_i}
  \]
  with $I$ being finite. We say a \emph{unital} factorization $\bA$-linear multi-functor
  \[
      G: (\bC_i)_{i\in I} \to \bD
  \]
  \emph{exhibits $\bC'$ as the fusion product of $\bC_i$} if pre-composing with $(\on{id}_\bA,G)$ induces an equivalence between
  \begin{itemize}
    \item the category of lax-unital factorization functors
    \[
      (\bA,\bD) \to (\bA',\bD')
    \] 
    \item the category of lax-unital factorization multi-functors
    \[
      (\bA, (\bC_i)_{i\in I}) \to (\bA',\bD')
    \]
  \end{itemize}
  for any unital factorization category $\bA'$ and $\bD'\in \bA'\mathbf{-mod}^\onfact_{\sqcup \ul{x}_i}$.

  \medskip
  Note that $\bD$, equipped with the multi-functor $G$, is essentially unique if exists. We write
  \[
      \boxt_{\bA}^\onfact \bC_i
  \] 
  for this object.

\begin{thm}
  \label{thm-fusion-exist}
   Let $\bA$ be a unital factorization category and 
   \[
      \bC_i \in \bA\mathbf{-mod}^\onfact_{\ul{x}_i},\;\;i\in I 
   \]
   with $I$ being finite. Then the external fusion product 
   \[
      \boxt_{\bA}^\onfact \bC_i \in \bA\mathbf{-mod}^\onfact_{\sqcup\ul{x}_i}
    \]
  exists, and the structural functor
  \[
      \big((\boxt \ul{\bC}_i)_{/S_0}\big)|_\disj\to \on{union}_{(\ul{x}_i)}^*(\ul{ \boxt_{\bA}^\onfact \bC_i})|_\disj 
  \]
  is an equivalence.

\end{thm}

\sssec{}
  \secref{sss:proof_fusion_start}-\secref{sss:proof_fusion_end} are devoted to the proof of the theorem. It is enough to treat the case $I=\emptyset$ and $I=\{1,2\}$. 
  
\sssec{}  \label{sss:proof_fusion_start}

    We first consider the case $I=\emptyset$. Unwinding the definitions, we see that for any test unital factorization category $\bA'$ and $\bD'\in \bA'\mathbf{-mod}^\onfact_{\emptyset}$, a lax-unital factorization multi-functor
    \[
      (\bA, (\bC_i)_{i\in \emptyset}) \to (\bA',\bD')
    \]
    consists of the following data:
    \begin{itemize}
      \item A lax-unital factorization functor $\bA \to \bA'$;
      \item A morphism $\ul{\Dmod}(S_0) \to \ul\bD'_{\emptyset}$ in $\mathbf{CrysCat}(S_0)$.
    \end{itemize}
    Consider the vacuum object
    \[
       \bA^{\onfact_\emptyset} \in \bA\mathbf{-mod}^\onfact_{\emptyset}.
    \]
    Note that there is a canonical identification $\ul{\Dmod}(S_0)\simeq \ul{\bA^{\onfact_\emptyset}}_\emptyset$. It follows that $\bA^{\onfact_\emptyset} $ is the empty external fusion product.

\sssec{}

    We now consider the case $I=\{1,2\}$. To simplify the notations, we asssume $S_0=\on{pt}$ (otherwise one just replace absolute products, tensor products below by relative ones).

    \medskip
    Consider the Bar simplicial diagram 
    \[
      \cdots \Ran^\untl_{\ul x_1} \times \Ran^\untl \times \Ran^\untl_{\ul x_2} \rightrightarrows \Ran^\untl_{\ul x_1} \times  \Ran^\untl_{\ul x_2} 
    \]
    associated to the $\Ran^\untl$-modules $\Ran^\untl_{\ul x_i}$. We can restrict to the disjoint loci and obtain a simplicial diagram
    \begin{equation}
      \label{eqn-hyperdescent-Ran}
      \cdots (\Ran^\untl_{\ul x_1} \times \Ran^\untl \times \Ran^\untl_{\ul x_2})_\disj \rightrightarrows (\Ran^\untl_{\ul x_1} \times  \Ran^\untl_{\ul x_2})_\disj.
    \end{equation}
    For $n\ge 0$, let 
    \begin{equation}
      \label{eqn-hyperdescent-Ran-6}
        u_n: (\Ran^\untl_{\ul x_1} \times (\Ran^\untl)^{\times n} \times \Ran^\untl_{\ul x_2})_\disj \to \Ran^\untl_{\ul x_1\sqcup \ul x_2}
    \end{equation}
    be the union map. Note that these maps provide an augmentation of the diagram \eqref{eqn-hyperdescent-Ran}. We denote this augmented simplicial diagram by
    \begin{equation}
      \label{eqn-hyperdescent-Ran-3}
      \cdots \CY_1 \rightrightarrows \CY_0 \to \CY_{-1}.
    \end{equation}

\sssec{Warning} 
  In a previous version of this paper, we made a false claim that \eqref{eqn-hyperdescent-Ran-3} is an \'etale hypercover. It is false for two reasons: (i) The connecting morphisms $\CY_i \to \CY_j$ are not schematic; (ii) The canonical morphism $\CY_1 \to \CY_0\times_{\CY_{-1}}\CY_0$ is not surjective even on $k$-points. We warn the readers that there are similar mistakes in \cite[Sect. 6.23]{Ra} and \cite[Sect. B.11.14]{GLC2}.

  \medskip
  Nevertheless, we have the following descent result for crystals of categories on \eqref{eqn-hyperdescent-Ran-3}.

\begin{lem}
  The augmented simplicial diagram \eqref{eqn-hyperdescent-Ran-3} induces an equivalence
    \begin{equation}
      \label{eqn-hyperdescent-Ran-4}
        \mathbf{CrysCat}^{\on{strict}}(\CY_{-1}) \xrightarrow{\simeq} \underset{[n]\in \Delta}\lim \mathbf{CrysCat}^{\on{strict}}(\CY_n).
    \end{equation}
\end{lem}

\proof
  
    Lemma \ref{lem-union-local-iso} implies that:
    \begin{itemize}
      \item[(1)]
          Each arrow $u_\alpha:\CY_i \to \CY_j$ is a Cartesian morphism of categorical prestacks (see \secref{sss Cart catprestk});
      \item[(2)] 
          For any affine test scheme $S$ over $\CY_j$ and any arrow $u_\alpha:\CY_i\to \CY_j$, the fiber product $S\times_{\CY_j} \CY_i$ is a finite coproduct of open subschemes of $S$ (taking in the category of prestacks).
    \end{itemize}
    It follows that the functor
    \[
        u_\alpha^*:\mathbf{CrysCat}^{\on{strict}}(\CY_j) \to \mathbf{CrysCat}^{\on{strict}}(\CY_i)
    \]
    admits a right adjoint
    \[
        u_{\alpha,*}:\mathbf{CrysCat}^{\on{strict}}(\CY_i) \to \mathbf{CrysCat}^{\on{strict}}(\CY_j)
    \]
    such that for any affine test scheme $S$ over $\CY_j$, we have
    \[
        \Gamma(S, u_{\alpha,*}(\bC)) \simeq \Gamma(S\times_{\CY_j} \CY_i, \bC).
    \]
    This implies the functor \eqref{eqn-hyperdescent-Ran-4} has a right adjoint sending an object
    \[
        (\bC_n)_{[n]\in \Delta} \in  \underset{[n]\in \Delta}\lim \mathbf{CrysCat}^{\on{strict}}(\CY_n)
    \]
    to 
    \[
        \underset{[n]\in \Delta}\lim  u_{n,*}(\bC_n)  \in \mathbf{CrysCat}^{\on{strict}}(\CY_{-1}),
    \]
    where $u_n: \CY_n \to \CY_{-1}$ is the morphism \eqref{eqn-hyperdescent-Ran-6}. Note that for any categorical prestack $\CY$, the forgetful functor
    \[
        \mathbf{CrysCat}^{\on{strict}}(\CY)  \to \mathbf{CrysCat}(\CY^\simeq) 
    \]
    is conservative, where $\CY^\simeq$ is the maximal non-categorical subprestack of $\CY$ (see \secref{sss maximal subprestack}). Hence to show that the obtained adjoint functors
    \[
        \mathbf{CrysCat}^{\on{strict}}(\CY_{-1}) \rightleftarrows \underset{[n]\in \Delta}\lim \mathbf{CrysCat}^{\on{strict}}(\CY_n)
    \]
    are inverse to each other, we only need to show the induced adjoint functors
    \[
        \mathbf{CrysCat}(\CY_{-1}^\simeq) \rightleftarrows \underset{[n]\in \Delta}\lim \mathbf{CrysCat}(\CY_n^\simeq)
    \]
    are inverse to each other. This reduces the claim to the following Lemma \ref{lem-descent-crystal-Ran}.

\qed

\begin{lem}
  \label{lem-descent-crystal-Ran}

  The augmented simplicial diagram \eqref{eqn-hyperdescent-Ran-3} induces an equivalence
    \begin{equation}
      \label{eqn-hyperdescent-Ran-7}
        \mathbf{CrysCat}(\CY_{-1}^\simeq) \xrightarrow{\simeq} \underset{[n]\in \Delta}\lim \mathbf{CrysCat}(\CY_n^\simeq).
    \end{equation}
\end{lem}

\proof
    By Zariski descents of crystal of categories (see e.g. \cite[Theorem 1.5.7]{Ga3}), we only need to show
    \[
        \underset{[n]\in \Delta^{\on{op}}}{\on{colim}} \,\CY_n^\simeq \to \CY_{-1}^\simeq
    \]
    becomes an isomorphism after Zariski sheafification.

    \medskip
    Note that $\CY_0^\simeq \to \CY_{-1}^\simeq$ induces a surjection on $k$-points. Hence by Lemma \ref{lem-union-local-iso}, it is an effective epimorphism for the Zariski topology. Hence by \cite[Lemma 6.2.3.16]{Lu0}, we only need to show
    \[
        \underset{[n]\in \Delta^{\on{op}}}{\on{colim}} \,\CY_n^\simeq \times_{\CY_{-1}^\simeq} \CY_0^\simeq \to \CY_{0}^\simeq
    \]
    becomes an isomorphism after Zariski sheafification. In fact, we claim this morphism is an isomorphism even before sheafification. 

    \medskip
    To prove claim, we only need to show that for any affine test scheme $S$, and $\ul{z} \in \on{Im}\big(\CY_0^\simeq(S) \to \CY_{-1}^\simeq(S)\big)$, the groupoid
    \[
        \underset{[n]\in \Delta^{\on{op}}}{\on{colim}} \, \CY_n^\simeq (S)\underset{\CY_{-1}^\simeq (S)}\times \{\ul z\}
    \]
    is contractible. Recall $\ul{z}$ is a finite subset of $X(S)$ that contains $\ul{x}_1|_S$ and $\ul{x}_2|_S$. We define an equivalence relation on $\ul{z}$ such that $u\sim u'$ iff there exists a sequence $u=u_1,u_2,\cdots,u_n=u'$ in $\ul{z}$ such that the intersection of the graphs of $u_k$ and $u_{k+1}$ is nonempty. Let $J:=\ul{z}/\sim$ be the set of equivalence classes. Let $J_1$ and $J_2$ be the images of the subsets $\ul{x}_1|_S, \ul{x}_2|_S\subseteq \ul{z}$ under the map $\ul{z} \to J$. Since $\ul{z}$ is in the image of $\CY_0^\simeq(S) \to \CY_{-1}^\simeq(S)$, we see that $J_1\cap J_2=\emptyset$. Unwinding the definitions, for $[n]\in \Delta$, 
    \[
        \CY_n^\simeq (S)\underset{\CY_{-1}^\simeq (S)}\times \{\ul z\}
    \]
    can be identified with the \emph{set} $A_{J,J_1,J_2,n}$ of maps $\phi: J\to  \{-\infty,1,\cdots,n,\infty\} $ such that $J_1\subseteq \phi^{-1}(-\infty)$ and $J_2\subseteq \phi^{-1}(\infty)$. Hence we only need to show the simplicial set $A_{J,J_1,J_2,\bullet}$ is weakly contractible. Note that we have
    \[
        A_{J,J_1,J_2,\bullet} \simeq A_{J\setminus(J_1\cup J_2),\emptyset,\emptyset,\bullet} \simeq \prod_{j\in J\setminus(J_1\cup J_2)} A_{\{j\},\emptyset,\emptyset,\bullet}.
    \]
    Hence we only need to show $A_{\{*\},\emptyset,\emptyset,\bullet}$ is weakly contratible. However, it is easy to see this simplicial set is isomorphic to $\Delta^1$.

\qed

\sssec{} \label{sss:proof_fusion_end}
  Return to the proof of Theorem \ref{thm-fusion-exist}. The factorization structures on $\bA$ and $\bC_i$ implies the objects
    \begin{equation}
      \label{eqn-hyperdescent-Ran-2}
        (\ul\bC_1 \boxt \ul\bA^{\boxtimes n} \boxt \ul\bC_2)|_{\disj} \in \mathbf{CrysCat}^{\on{strict}}( (\Ran^\untl_{\ul x_1} \times (\Ran^\untl)^{\times n} \times \Ran^\untl_{\ul x_2})_\disj  ).
    \end{equation}
    are compatible with \eqref{eqn-hyperdescent-Ran} and pullback functors. Hence by the equivalence \eqref{eqn-hyperdescent-Ran-4}, there is a unique object
    \[
        \ul{\bC} \in \mathbf{CrysCat}^{\on{strict}}( (\Ran^\untl_{\ul x_1}  \times \Ran^\untl_{\ul x_2})_\disj  )
    \]
    such that $u_\bullet^*(\ul\bC) \simeq  (\ul\bC_1 \boxt \ul\bA^{\boxtimes \bullet} \boxt \ul\bC_2)|_{\disj}$. It is easy to see $\ul{\bC}$ has a natural factorization $\bA$-module structure and the equivalence
    \[
        (\ul\bC_1\boxt \ul\bC_2)|_\disj \xrightarrow{\simeq} u_0^*\ul{\bC}
    \]
    exhibits $\bC$ as the fusion product of $\bC_1$ and $\bC_2$.

\qed[Theorem \ref{thm-fusion-exist}]

\sssec{Example}
  \label{sss fusion vacuum factmodcat}
    It is easy to see (for example from \secref{sss:proof_fusion_end}) that the $\bA$-linear bifunctor
  \[
      ( \bA^{\on{fact}_{\ul{x}}}, \bC  ) \to \mathbf{prop}_{\ul{x}_0\subseteq \ul{x}_0\sqcup \ul{x}}(\bC)
  \]
    constructed in \secref{exam fusion with vac} induces an equivalence
    \[
        \bA^{\on{fact}_{\ul{x}}} \boxt_{\bA}^\onfact \bC \simeq \mathbf{prop}_{\ul{x}_0\subseteq \ul{x}_0\sqcup \ul{x}}(\bC).
    \]
    In particular,
    \[
        \bA^{\on{fact}_{\ul{x}}} \boxt_{\bA}^\onfact \bA^{\on{fact}_{\ul{x}'}} \simeq \bA^{\on{fact}_{\ul{x}\sqcup \ul{x}'}}.
    \]

\sssec{}
  The details for the remaining part of this subsection will be provided in \cite{CFZ}.

\sssec{}

  Let $\bA$ be a unital factorization category and $\CA$ be a factorization algebra in $\bA$. For a finite collection of disjoint points $\ul{x}_i \in \Ran^\untl(S_0)$, $i\in I$, there is a canonical functor
  \begin{equation}
    \label{eqn-external-fusion-factmod}
      \prod \CA\on{-mod}^\onfact_{\ul x_i} \to \CA\on{-mod}^\onfact_{\sqcup \ul x_i} ,\; (\CC_i)_{i\in I} \mapsto \boxt_{\CA}^\onfact \CC_i
  \end{equation}
  that sends $(\CC_i)_{i\in I}$ to the lax-unital factorization $\CA$-linear functor
  \[
      \Vect^{\on{fact}_{\sqcup \ul x_i} } \to \bA^{\on{fact}_{\sqcup \ul x_i} }
  \]
  corresponding to the lax-unital factorization $\CA$-linear \emph{multi-functor}
  \[
      (  \Vect^{\on{fact}_{\ul x_i} } )_{i\in I} \xrightarrow{(\CC_i)} (  \bA^{\on{fact}_{\ul x_i} } )_{i\in I} \to \bA^{\on{fact}_{\sqcup \ul x_i} }.
  \]

  \medskip
  One can show the functor \eqref{eqn-external-fusion-factmod} is $\Dmod(S_0)$-multilinear. Hence we obtain a functor
  \begin{equation}
    \label{eqn-external-fusion-factmod-tensor}
    \bigotimes_{\Dmod(S_0)} \CA\on{-mod}^\onfact_{\ul x_i} \to \CA\on{-mod}^\onfact_{\sqcup \ul x_i} ,\; (\CC_i)_{i\in I} \mapsto \boxt_{\CA}^\onfact \CC_i,
  \end{equation}
  which is called the \emph{external fusion functor} for factorization $\CA$-modules. By construction, we have a canonical commutative diagram:
  \begin{equation}
    \label{eqn-external-fusion-factmod-oblv}
    \xymatrix{
        \bigotimes_{\Dmod(S_0)} \CA\on{-mod}^\onfact_{\ul x_i}  \ar[r] \ar[d]
        &  \CA\on{-mod}^\onfact_{\sqcup \ul x_i} \ar[d] \\ 
        \bigotimes_{\Dmod(S_0)} \bA_{\ul{x}_i} \ar[r]^-\simeq
        & \bA_{\sqcup\ul{x}_i},
    }
  \end{equation}
  where the vertical arrows are the forgetful functors, and the bottom equivalence is due to the factorization structure on $\bA.$

\sssec{} 
   Moreover, using the universal propery of external fusion, one can show the functors \eqref{eqn-external-fusion-factmod-tensor} are compatible with the change-of-base functors \eqref{eqn-change-of-base-factmod-fixalg}. In other words, for any finite set $I$, we obtain a strict morphism
  \[
      \on{mult}_I:(\boxt_{i\in I} \CA\ul{\on{-mod}}^\onfact)|_\disj \to \on{union}_I^*(\CA\ul{\on{-mod}}^\onfact)|_\disj
  \]
  in 
  \[
      \mathbf{CrysCat}^{\on{strict}}( (\prod_{i\in I} \Ran^\untl)_\disj  ).
  \]
  Finally, using the universal propery of external fusion, one can supply a datum of associativity and commutativity for the functors $\on{mult}_I$. In other words, we obtain a structure of unital \emph{lax-factorization} category (see \secref{sss lax unital factcat}) on $\CA\ul{\on{-mod}}^\onfact$. We denote it just by 
  \[
      \CA{\on{-mod}}^\onfact \in \mathbf{UntlFactCat}.
  \]
  Moreover, the \eqref{eqn-oblv-CA-is-lax-morphism} provides a \emph{lax-unital} factorization functor
  \[
    \oblv_\CA: \CA{\on{-mod}}^\onfact \to \bA.
  \] 
  By construction, it sends the unit
  \[
      \on{unit}_{\CA{\on{-mod}}^\onfact}
  \]
   of $\CA{\on{-mod}}^\onfact$ to $\CA\in \on{FactAlg}(\bA)$. Therefore we write
  \[
      \CA^{\on{enh}} :=\on{unit}_{\CA{\on{-mod}}^\onfact} \in \on{FactAlg}(  \CA{\on{-mod}}^\onfact  ). 
  \]

\sssec{}
  \label{sss universal property of A-factmod}
  The construction $(\bA,\CA)\mapsto \CA{\on{-mod}}^\onfact$ is functorial in $\bA$. In other words, for any lax-unital factorization functor $F:\bA\to \bA'$, we have a canonical \emph{unital}\footnote{The functor below is strictly unital by Lemma \ref{lem-criterion-for-unital-factfun}, which is also true for unital lax-factorization categories (including $\CA{\on{-mod}}^\onfact$).} factorization functor
  \[
      F^{\on{enh}}:\CA{\on{-mod}}^\onfact  \to F(\CA){\on{-mod}}^\onfact
  \]
  compatible with $\oblv_\CA$, $\oblv_{F(\CA)}$ and $F$.

  \medskip
  Combining with Lemma \ref{lem module for unit is nothing}, we obtain a canonical unital factorization functor
  \[
      F^{\on{enh}}:\bA \to F(\on{unit}_\bA){\on{-mod}}^\onfact
  \]
  such that $\oblv_{F(\on{unit}_\bA)} \circ F^{\on{enh}}\simeq F$.

\sssec{Variant}
  \label{sss universal property of A-factmod modversion}
  Let $\bA$ be a unital factorization category and $\CA$ be a factorization algebra in $\bA$. For any $\bC\in \bA\mathbf{-mod}^\onfact_{\ul{x}_0}$, one can similarly construct a unital \emph{lax-factorization} $(\CA{\on{-mod}}^\onfact)$-module category at $\ul{x}_0$, denoted by
  \[
      \CA{\on{-mod}}^\onfact(\bC),
  \]
  such that its fiber at $\ul{x}_0$ is the DG category of factorization $\CA$-modules in $\bC$ (see Lemma \ref{lem change of base factcat mod}). It is equipped with a lax-unital $\oblv_\CA$-linear factorization functor
  \[
      \oblv_{\CA,\bC}: \CA{\on{-mod}}^\onfact(\bC) \to \bC.
  \]

\sssec{}
  \label{sss universal property of A-factmod modversion-2}
  Moreover, for a lax-unital factorization functor $(F,G): (\bA,\bC) \to (\bA',\bC')$, we have a \emph{unital} factorization $ F^{\on{enh}}$-linear functor
  \[
      G^{\on{enh}}: \CA{\on{-mod}}^\onfact(\bC) \to F(\CA){\on{-mod}}^\onfact(\bC').
  \]
  In particular, we have a unital factorization $ F^{\on{enh}}$-linear functor
  \[
      G^{\on{enh}}: \bC \to F(\on{unit}_\bA){\on{-mod}}^\onfact(\bC').
  \]

  \medskip
  Conversely, given any $G^{\on{enh}}$ as above, one can recover $G$ as $\oblv_{F(\on{unit}_\bA),\bC} \circ G^{\on{enh}}$. One can check these two constructions are inverse to each other. In other words, for fixed $F:\bA\to \bA'$ and modules $\bC$ and $\bC'$, the following two data are equivalent:
  \begin{itemize}
    \item A unital factorization $ F^{\on{enh}}$-linear functor
      \[
          G^{\on{enh}}: \bC \to F(\on{unit}_\bA){\on{-mod}}^\onfact(\bC').
      \]
    \item A lax unital factorization $F$-linear functor
    \[
        G:\bC \to \bC'.
    \]
  \end{itemize}

\ssec{Restrictions of factorization modules}

\sssec{}
  \label{sss 1-2-Cart fib}
  We say a functor $\bE \to \bB$ between 2-categories is a (1,2)-Cartesian fibration if
  \begin{itemize}
    \item 
      There are enough Cartesian 1-morphisms. In other words, for any morphism $f:u\to v$ in $\bB$ and a lifting $V\in \bE$ of $v$, there exists a lifting $F:U\to V$ of $f$ such that for any $W\in \bE$ over $w\in \bB$, the following square of categories is Cartesian:
      \[
        \xymatrix{
            \Maps_\bE(W,U) \ar[r]^-{F\circ-} \ar[d]
            & \Maps_\bE(W,V) \ar[d] \\
            \Maps_\bB(w,u) \ar[r]^-{f\circ-}
            & \Maps_\bB(w,v)
        }
      \]
      \item 
        There are enough Caresian 2-morphisms. In other words, for objects $U,V\in \bE$ over $u,v\in \bB$, the functor
        \[
            \Maps_\bE(U,V) \to \Maps_\bB(u,v)
        \]
        is a Cartesian fibration.
      \item
        The collection of Cartesian 2-morphisms are closed under horizontal compositions.
  \end{itemize}

  \sssec{Remark}
    \label{sss strictening}
    The theory of (1,2)-Cartesian fibrations is developed in \cite{GHL}, \cite{AS1} and \cite{AS2} under the name of \emph{inner 2-Cartesian fibrations}. In particular, there is a Grothendieck cosntruction in \cite{AS2} which says knowing a (1,2)-Cartesian fibration $\bE\to \bB$ is equivalent to knowing a functor $\bB^{1\on{-op},2\on{-op}} \to 2-\mathbf{Cat}$.

  \sssec{}
    The following result will be one of the main theorems for \cite{CFZ}. 
  \begin{thm}
    \label{thm-the-fibration}
      The forgetful functor
      \begin{equation}
        \label{eqn-the-fibration}
           \mathbf{UntlFactModCat}^{\on{lax-untl}}_{\ul{x}_0} \to \mathbf{UntlFactCat}^{\on{lax-untl}}
      \end{equation}
      is a (1,2)-Cartesian fibration. 
  \end{thm}

  \sssec{}
    We will explain the main ideas of the proof in the next subsection. For now, we deduce some useful results from it.

  \sssec{}  \label{sss:unital restriction}

    The Grothendieck construction in \cite{AS2} provides a functor
    \begin{equation}
      \label{eqn-the-restriction}
         (\mathbf{UntlFactCat}^{\on{lax-untl}})^{2\on{-op},1\on{-op}} \to 2-\mathbf{Cat}.
    \end{equation}
    In other words, for a \emph{lax-unital} factorization functor $\Phi:\bA\to \bB$, we have a contravariant transport functor between the fibers of \eqref{eqn-the-fibration}:
    \[
        \Rres_\Phi^\untl: \bB\mathbf{-mod}^\onfact_{\ul{x}_0} \to  \bA\mathbf{-mod}^\onfact_{\ul{x}_0}.
    \] 
    The construction $\Phi\mapsto \Rres_\Phi^\untl$ is contravariant. In other words, for a 2-morphism $\Phi\to \Phi'$, we have a natural transformation
    \[
        \Rres_{\Phi'}^\untl \to \Rres_{\Phi}^\untl.
    \]
    When $\Phi$ is unital, we also write $\Rres_\Phi:=\Rres_\Phi^\untl$. 

    \sssec{}
      Recall that inside any 2-category, there is a notion of \emph{adjoint pair of 1-morphisms}. Moreover, a functor between 2-categories always sends adjoint pairs to adjoint pairs. Applying to the functor to \eqref{eqn-the-restriction}, we obtain:

  \begin{prop} \label{prop:basic adjunction}
      Suppose $\Phi: \bA \rightleftarrows \bB:\Psi$ is an adjoint pair in $\mathbf{UntlFactCat}^{\on{lax-untl}}$, then we have an adjoint pair 
      \[
        \Rres_\Phi^\untl: \bB\mathbf{-mod}^\onfact_{\ul{x}_0} \rightleftarrows   \bA\mathbf{-mod}^\onfact_{\ul{x}_0}:  \Rres_\Psi^\untl
      \]
      in $2-\mathbf{Cat}$. In particular, for $\bM\in  \bA\mathbf{-mod}^\onfact_{\ul{x}_0}$ and $\bN\in\bB\mathbf{-mod}^\onfact_{\ul{x}_0}  $, we have a canonical equivalence between the following categories:
      \begin{itemize}
        \item The category of unital factorization $\bA$-linear functors $\Rres_\Phi^\untl(\bN) \to  \bM$;
        \item The category of unital factorization $\bB$-linear functors $\bN \to \Rres_\Psi^\untl(\bM) $.
      \end{itemize}

  \end{prop}

  \sssec{}
    For $\bN\in \bB\mathbf{-mod}^\onfact_{\ul{x}_0} $, the object $\Rres_\Phi^\untl(\bN)$ is called the \emph{restriction of $\bD$ along $\Phi$}. By definition, it is equipped with a lax-unital $\Phi$-linear factorization functor
    \[
         \Rres_\Phi^\untl(\bN) \to \bN,
    \]
    which is a Cartesian lifting of $\Phi$ in $\mathbf{UntlFactModCat}^{\on{lax-untl}}_{\ul{x}_0}$.

    \medskip
    In other words, for any test lax-unital factorization functor $F:\bA'\to \bA$ and test object $\bM'\in \bA'\mathbf{-mod}^\onfact_{\ul{x}_0}$, pre-composing with $ \Rres_F^\untl(\bN)\to \bN$ induces an equivalence between the following categories:
    \begin{itemize}
      \item The category of lax-unital $\Phi\circ F$-linear factorization functors $\bM' \to \bN$;
      \item The category of lax-unital $F$-linear factorization functors $\bM' \to  \Rres_F^\untl(\bN)$.
    \end{itemize}
    Taking $\bA'=\Vect$ and $\bM'=\Vect^{\on{fact}_{\ul x_0}}$, we obtain the following result.

  \begin{prop} \label{prop:fiber of unital restriction}
    Let $\Phi:\bA\to \bB$ be a lax-unital factorization functor and $\bN\in  \bB\mathbf{-mod}^\onfact_{\ul{x}_0}$. Then for any $\CA\in \on{FactAlg}(\bA)$, pre-composing with $ \Rres_F^\untl(\bN)\to \bN$ induces an equivalence
    \[
        \CA\on{-mod}^\onfact( \Rres_F^\untl(\bN) )_{\ul{x}_0} \xrightarrow{\simeq} F(\CA)\on{-mod}^\onfact(\bN)_{\ul{x}_0}.
    \]
  \end{prop}

  \sssec{}
    Combining with Lemma \ref{lem module for unit is nothing}, we obtain the following result.

    \begin{cor}
      \label{cor fiber unital restriction}
    Let $\Phi:\bA\to \bB$ be a lax-unital factorization functor and $\bN\in  \bB\mathbf{-mod}^\onfact_{\ul{x}_0}$. There is a canonical dotted equivalence making the following diagram commute
    \[
      \xymatrix{
        ( \Rres_F^\untl(\bN) )_{\ul{x}_0} \ar@{.>}[r]^-\simeq  \ar[d]
        & F(\on{unit}_\bA)\on{-mod}^\onfact(\bN)_{\ul{x}_0} \ar[d]^-{\on{oblv}_{\ul x_0}} \\ 
        \bN_{\ul{x}_0} \ar@{=}[r] &  \bN_{\ul{x}_0}.
      }
    \]
  \end{cor}

  \sssec{}
    In particular, for $\bA=\Vect$, we obtain:

  \begin{cor}
    Let $\bB$ be a unital factorization category and $\CB\in \on{FactAlg}(\bB)$. For any $\bN\in \bB\mathbf{-mod}^\onfact_{\ul{x}_0}$, there is a canonical equivalence
    \[
       ( \Rres_\CB^\untl(\bN) )_{\ul{x}_0} \simeq \CB\on{-mod}^\onfact(\bN)_{\ul{x}_0}.
    \]
  \end{cor}

  \sssec{Remark}
    Let $\Phi:\bA\to \bB$ be a \emph{unital} factorization functor and $\bM\in \bB\mathbf{-mod}^\onfact_{\ul{x}_0}$. Then the underlying crystal of categories for $\Rres_\Phi(\bM)$ can be explicitly calculated as a limit (see \secref{sss sharp} and \secref{sss defn sharp}). For example, the restriction of $\Rres_\Phi(\bM)$ along the map
      \[
          X\times S_0 \to \Ran^{\untl}_{\ul{x}_0},\; y\mapsto y\cup \ul{x}_0
      \]
      fits into the following Cartesian square (\secref{rem descrp restriction}):
      \[
          \xymatrix{
              \Rres_\Phi(\bM)|_{X\times S_0} \ar[r] \ar[d]
              & \bM|_{X\times S_0} \ar[d] \\ 
              j_*j^*(\bA|_X \boxt \bM|_{\ul{x}_0}) \ar[r]
              & j_*j^*(\bM|_{X\times S_0}),
          }
      \]
  where
  \begin{itemize}
    \item $j:(X\times S_0)\setminus \on{graph}_{\ul x_0}\to X\times S_0$ is the complement of the union of the graphs for elements in $\ul{x}_0\subseteq X(S_0)$;
    \item the bottom horizontal functor is provided by the factorization $\bB$-module structure on $\ul{\bM}$ and the functor $\Phi$.
  \end{itemize}

  \sssec{}
    Let $g:\CB\to \CB'$ be a morphism in $\on{FactAlg}(\bB)$. The natural transformation $\Rres_{g}^\untl:\Rres_{\CB'}^\untl\to \Rres_\CB^\untl$ provides a functor
    \[
         ( \Rres_{\CB'}^\untl(\bN) )_{\ul{x}_0} \to  ( \Rres_\CB^\untl(\bN) )_{\ul{x}_0}.
    \]
    By the above corollary, we obtain a canonical functor
    \[
        \on{Res}_g: \CB'\on{-mod}^\onfact(\bN)_{\ul{x}_0} \to  \CB\on{-mod}^\onfact(\bN)_{\ul{x}_0}.
    \]
    For $\CN\in \CB'\on{-mod}^\onfact(\bN)_{\ul{x}_0}$, its image $\on{Res}_g(\CN)$ is called the \emph{restriction of $\CN$ along $g$}.

  \sssec{Remark}
    \label{sss the fibration 1-cat down}
    By definition, a (1,2)-Cartesian fibration $\pi:\bE\to \bF$ induces a Cartesian fibrations 
    \[
        \Maps_\bE(u,v) \to \Maps_\bF(\pi(u),\pi(v))
    \]
    of 1-categories. Hence Theorem \ref{eqn-the-fibration} implies the forgetful functor
    \[
        \on{FactMod}(\bB,\bN)_{\ul x_0} \to \on{FactAlg}(\bB)
    \]
    is a Cartesian fibration between 1-categories. It follows from construction that the functor $\on{Res}_g$ is the contravariant transport functor for this Cartesian fibration.

  \sssec{}
    We now provide a useful criterion to check whether a factorization module category is obtained via restriction. We need the following lemma.

    \begin{lem}
      \label{lem factmodcat 2-conservative}
      The forgetful functor
      \[
          \bA\mathbf{-mod}^\onfact_{\ul{x}_0} \to \mathbf{CrysCat}(S_0),\; \bM \mapsto \bM_{\ul{x}_0}
      \]
      is conservative on 2-morphisms.
    \end{lem}

    \proof[Sketch]
        Let $G_1,G_2:\bM\to \bM'$ be factorization $\bA$-linear functors and $\alpha:G_1\to G_2$ be a 2-morphism between them such that $\alpha_{\ul{x}_0}$ is invetible. For any affine test scheme $\ul{x}:S\to \Ran^\untl_{\ul{x}_0}$, we need to show $\alpha_{\ul{x}}$ is also invertible. 

        \medskip
        Suppose $\ul{x} = \ul{x}_0|_S \sqcup \ul{y}$ can be written as a disjoint union of subsets. Then the factorization structure implies 
        \[
            (G_i)_{\ul{x}}: \bM_{\ul{x}} \to \bM'_{\ul{x}} 
        \]
        can be identified with
        \[
            \on{Id} \otimes (G_i)_{\ul{x}_0}: \bA_{\ul{y}} \otimes_{\Dmod(S)} \bM_{\ul{x}_0|_S} \to \bA_{\ul{y}} \otimes_{\Dmod(S)} \bM_{\ul{x}_0|_S}',
        \]
        and the 2-morphism $\alpha_{\ul{x}}$ can be identified with $\on{Id}\otimes \alpha_{\ul{x}_0|_S}$. This implies $\alpha_{\ul{x}}$ is invertible because $\alpha_{\ul{x}_0}$ is so.

        \medskip 
        For the general case, we can replace $S$ with a covering of locally closed subschemes such that the above property holds on each subscheme. This reduces the general case to the above case.

    \qed

    \begin{prop} \label{sss:adj test for factres}
        Let
        \[
          (\Phi,\Phi^m): (\bA,\bM) \rightleftarrows (\bB,\bN): (\Psi,\Psi^m)
        \]
        be an adjoint pair in $\mathbf{UntlFactModCat}^{\on{lax-untl}}_{\ul{x}_0}$ such that:
        \begin{itemize}
          \item[(i)] The left adjoint $\Phi$ is unital\footnote{In fact, a standard argument shows that this is automatic.};
          \item[(ii)] It induces an equivalence
            \[
                \Phi^m_{\ul x_0}: \bM_{\ul x_0} \rightleftarrows \bN_{\ul x_0}: \Psi^m_{\ul x_0}
            \]
          in $\mathbf{CrysCat}(S_0)$.
        \end{itemize}
          Then the canonical factorization $\bA$-linear functor
          \[
              \bM \to \Rres_\Phi^\untl(\bN)
          \]
          is an equivalence.
    \end{prop}

    \proof[Sketch]
        Using the universal property of $\Rres_\Phi^\untl$, it is easy to show the given adjoint pair can be written as the composition of\footnote{Similar claim is true for any (1,2)-Cartesian fibration.}
        \[
          (\bA,\bM) \rightleftarrows (\bA,\Rres_\Phi^\untl(\bN))
        \]
        and
        \[
          (\bA,\Rres_\Phi^\untl(\bN)) \rightleftarrows (\bB,\bN).
        \]
        Moreover, by Corollary \ref{cor fiber unital restriction} and assumption (i), the second pair induces an equivalence
        \[
          \Rres_\Phi^\untl(\bN)_{\ul x_0} \rightleftarrows \bN_{\ul x_0}.
        \]
        Hence by assumption (ii), the first pair also induces an equivalence
        \[
          \bM_{\ul x_0} \rightleftarrows\Rres_\Phi^\untl(\bN)_{\ul x_0}.
        \]
        Now the claim follows formally from Lemma \ref{lem factmodcat 2-conservative}.

    \qed

\ssec{Sketch of Theorem \ref{thm-the-fibration}}
  \label{ss proof of thm the fibration}
  In this subsection, we explain the main ideas in the proof of Theorem \ref{thm-the-fibration}. A detailed proof will be provided in \cite{CFZ}.

 \sssec{}
  Let $\bE\to \bB$ be a functor between 2-categories. To show it is a (1,2)-Cartesian fibration, one only needs to show:
  \begin{itemize}
     \item[(a)]
        There are enough \emph{locally Cartesian 1-morphisms}. By definition, this means for any arrow $\Delta^1 \to \bB$, the base-change $\bE\times_\bB \Delta^1 \to \Delta^1$ has enough Cartesian 1-morphisms.
      \item[(b)]
        There are enough \emph{locally Cartesian 2-morphisms}.
      \item[(c)]
        The collection of locally Cartesian 1-morphisms are closed under compositions.
      \item[(d)]
        The collection of locally Cartesian 2-morphisms are closed under both horizontal and vertical compositions.
   \end{itemize} 

   \medskip
   In \cite{CFZ}, we will provide a \emph{constructive} proof for (a) and (b), and use these explicit constructions to verify (c) and (d). In this subsection, we only explain the construction of locally Cartsian 1-morphisms in 
    \begin{equation}
        \label{eqn-the-fibration-avatar}
           \mathbf{UntlFactModCat}^{\on{lax-untl}}_{\ul{x}_0} \to \mathbf{UntlFactCat}^{\on{lax-untl}}
      \end{equation}
  The construction for locally Cartesian 2-morphisms is similar but much simpler.

\sssec{}
  Let $\Phi:\bA \to \bB$ be a lax-unital factorization functor, i.e., a moprhism in the base of \eqref{eqn-the-fibration-avatar}. Let $\bN$ be a unital factorization $\bB$-module at $\ul{x}_0$, i.e., an object in the fiber of \eqref{eqn-the-fibration-avatar} over $\bB$. We will construct a locally coCartesian arrow $\Phi_m:\bM \to \bN$ lying over $\Phi$.

  \medskip
  We can enlarge \eqref{eqn-the-fibration-avatar} to allow \emph{lax-factorization} (module)-cateogires:
  \begin{equation}
        \label{eqn-the-fibration-lax}
    \mathbf{UntlLaxFactModCat}^{\on{lax-untl}}_{\ul{x}_0} \to \mathbf{UntlLaxFactCat}^{\on{lax-untl}}.
  \end{equation}
  By \secref{sss universal property of A-factmod} and \secref{sss universal property of A-factmod modversion}, the morphism $\Phi$ factors as
  \[
    \bA \xrightarrow{\Phi^{\on{enh}}} \Phi(\on{unit}_\bA)\on{-mod}^\onfact \xrightarrow{\on{oblv}_{\on{unit}_\bA}} \bB
  \]
  in the base of \eqref{eqn-the-fibration-lax}, and there is a canonical morphism in the source of \eqref{eqn-the-fibration-lax}
  \[
    \oblv_{\on{unit}_\bA,\bN}:\Phi(\on{unit}_\bA)\on{-mod}^\onfact(\bN) \to \bN
  \]
  that lifts the morphism $\on{oblv}_{\on{unit}_\bA}$. 

  \medskip
  By the universal property in \secref{sss universal property of A-factmod modversion-2}, we only need to show there exists an arrow in \eqref{eqn-the-fibration-lax}
  \begin{equation}
    \label{eqn-unital-res-allow-lax}
    \bM\to \Phi(\on{unit}_\bA)\on{-mod}^\onfact(\bN) 
  \end{equation}
  that lifts $\Phi^{\on{enh}}$ such that for any test object $\bM'\in \bA\mathbf{-mod}^\onfact_{\ul x_0}$, it induces an equivalence between:
  \begin{itemize}
    \item 
      The category of unital factorization $\bA$-linear functors $\bM'\to \bM$;
    \item
      The category of unital factorization $\Phi^{\on{enh}}$-linear functors $\bM'\to \Phi(\on{unit}_\bA)\on{-mod}^\onfact(\bN) $.
  \end{itemize}

\sssec{}
  Roughly speaking, the above reduction allows us to get rid of lax-unital functors, with the caveat that $\bB$ (and $\bN$) is allowed to be \emph{lax-factorization}. 

\sssec{}
  Now the desired claim (about existence of \eqref{eqn-unital-res-allow-lax}) follows formally from the following two claims:
  \begin{itemize}
     \item[(i)] The functor
      \begin{equation}
        \label{eqn-the-fibration-lax-but-unital}
        \mathbf{UntlLaxFactModCat}_{\ul{x}_0} \to \mathbf{UntlLaxFactCat}
      \end{equation}
      has enough locally Cartesian 1-morphisms.
      \item[(ii)] Consider the embedding
      \[
         \mathbf{UntlFactModCat}_{\ul{x}_0} \to \mathbf{UntlLaxFactModCat}_{\ul{x}_0}
      \]
      and its fiber at an object $\bA\in \mathbf{UntlFactCat}$:
      \[
          \bA\mathbf{-mod}^\onfact_{\ul x_0} \to \bA\mathbf{-mod}^{\on{laxfact}}_{\ul x_0}.
      \]
      The latter functor admits a right adjoint.
   \end{itemize} 

\sssec{}
  Claim (i) is obvious modulo homotopy-coherent issues. 

  \medskip
  Namely, for any morphism $\Phi:\bA\to \bB$ in $\mathbf{UntlLaxFactCat}$ and an object $\bN \in \bB\mathbf{-mod}^{\on{laxfact}}_{\ul x_0}$, the underlying crystal of categories $\ul\bN$ has a unital \emph{lax-}factorization $\bA$-module structure given by
  \[
    \big((\boxt_{i\in I^\circ} \ul{\bA})\boxt \ul{\bN}\big)|_\disj \to \big((\boxt_{i\in I^\circ} \ul{\bB})\boxt \ul{\bN}\big)|_\disj \to \on{union}_I^*(\ul{\bN})|_\disj,
  \]
  where the first functor is given by $\Phi$, and the second functor is the lax-factorization $\bB$-module structure on $\ul\bN$ (see \secref{sss laxfact modcat}). In other words, we obtain a canonical object
  \[
     \Rres^{\on{lax-fact}}_\Phi(\bN) \in \bA\mathbf{-mod}^{\on{laxfact}}_{\ul x_0}
  \]
  equipped with a unital factorization $\Phi$-linear functor
  \[
     \Rres^{\on{lax-fact}}_\Phi(\bN)\to \bN.
  \]
  One can check this is a locally Cartesian 1-morphism in \eqref{eqn-the-fibration-lax-but-unital}. A homotopy-coherent proof using the language of (generalized) operads will be provided in \cite{CFZ}.

\sssec{}
  \label{sss strictening fact}
 Claim (ii) is proved via an explicit construction of the desired right adjoint \emph{strictening} functor
 \[
    \mathbf{Str}_\bA: \bA\mathbf{-mod}^{\on{laxfact}}_{\ul x_0} \to \bA\mathbf{-mod}^\onfact_{\ul x_0}.
 \]
 In fact, for future reference, we will show the diagram
 \[
   \xymatrix{
      \mathbf{UntlFactModCat}_{\ul{x}_0} \ar[r]^-\subseteq \ar[d]
      & \mathbf{UntlLaxFactModCat}_{\ul{x}_0} \ar[d] \\
      \mathbf{UntlFactCat} \ar[r]^-\subseteq 
      & \mathbf{UntlLaxFactCat}
   }
 \] 
 is right adjiontable along the horizontal directions. In other words, the horizontal functors admit right adjoints, and the Beck--Chevalley transformation is invertible:
  \[
   \xymatrix{
      \mathbf{UntlFactModCat}_{\ul{x}_0} \ar[d]
      & \mathbf{UntlLaxFactModCat}_{\ul{x}_0} \ar[d] \ar[l]_-{\mathbf{Str}}  \\
      \mathbf{UntlFactCat}  
      & \mathbf{UntlLaxFactCat} \ar[l]_-{\mathbf{Str}}.
   }
 \] 
 Note that the desired functor $\mathbf{Str}_\bA$ can be given by the restriction of the top horizontal functor on the fiber over $\bA$.

\sssec{}
  \label{sss sharp}
    We will construct an endo-functor 
    \[
      \sharp:  \mathbf{UntlLaxFactModCat}_{\ul{x}_0}\to  \mathbf{UntlLaxFactModCat}_{\ul{x}_0}
    \]
    equipped with a natural transformation $\mu:\sharp \to \on{Id}$ and define
    \[
        \mathbf{Str}(\bA,\bM) := \lim \big( \cdots\to  (\bA^{\sharp\sharp},\bM^{\sharp\sharp})  \xrightarrow{\mu (\bA^\sharp,\bM^\sharp) }  (\bA^\sharp,\bM^\sharp) \to (\bA,\bM) \big)
    \]
    to be the sequential limit of the $\sharp$-construction. We will show
 \begin{itemize}
   \item[(1)] The objects $\mathbf{Str}(\bA,\bM)$ is contained in $\mathbf{UntlFactModCat}_{\ul{x}_0}$;
   \item[(2)] The functor
   \[
        \on{Fun}(-,\bA^\sharp) \to \on{Fun}(-,\bA)
   \]
   is invertible when restricted to $\mathbf{UntlFactCat}$;
   \item[(3)] The functor
   \[
        \on{Fun}(-,(\bA^\sharp,\bM^\sharp) ) \to \on{Fun}(-,(\bA,\bM))
   \]
   is invertible when restricted to $\mathbf{UntlFactModCat}_{\ul{x}_0}$.
 \end{itemize}
 It is clear that these properties imply the claim in \secref{sss strictening fact}. 

\sssec{Remark}
  The definition of $\sharp$ below might look mysterious, but in fact, it comes from a general construction about operads once we reformulate factorization structures using the language in \secref{ss fact via operad}.

\sssec{}
  Let $(\bA,\bM)\in \mathbf{UntlLaxFactModCat}_{\ul{x}_0}$ be a pair. To define its image under $\sharp$, we need some notations.

  \medskip
  For any finite set $I\in \on{Fin}$, we write 
  \[
      \CR_I:= \big(\prod_{i\in I} \Ran^\untl\big)_\disj
  \]
  and 
  \[
      \ul{\bA}_I:= (\boxt_{i\in I}\ul\bA)|_\disj \in \mathbf{CrysCat}(\CR_I).
  \]
  For any \emph{marked} finite set $I=I^\circ\sqcup\{0\} \in \on{Fin}_*$, we write 
  \[
    \CR_{I}:= \big((\prod_{i\in I^\circ} \Ran^\untl) \times \Ran^\untl_{\ul x_0}\big)_\disj
  \]
  and
  \[
    \ul{\bM}_{I}:=  \big((\boxt_{i\in I^\circ}\ul\bA) \boxt \ul\bM\big)|_\disj \in \mathbf{CrysCat}(\CR_{I}).
  \]
  The readers should be able to distinguish the marked and non-marked notations based on the context.

  \medskip
  For a morphism $\phi:I\to J$ in either $\on{Fin}$ or $\on{Fin}_*$, we have a map
  \[
      \on{union}_{I\to J}:  \CR_I \to \CR_J,\; (\ul{y}_i)_{i\in I} \to (\ul{z}_j)_{j\in J}
  \]
  given by $\ul{z}_j:= \bigsqcup_{i\in \phi^{-1}(j)} \ul{y}_i$. By Lemma \ref{lem-union-local-iso}, $\on{union}_{I\to J}$ is a Cartesian morphism and becomes a quasi-compact open immersion after Zariski sheafification. It follows that the functor
  \[
      T_{I\to J}:= \on{union}_{I\to J}^*: \mathbf{CrysCat}(\CR_J) \to \mathbf{CrysCat}(\CR_I)
  \]
  admits a right adjoint 
  \[
      T_{J\gets I}:= \on{union}_{I\to J,*}: \mathbf{CrysCat}(\CR_I) \to \mathbf{CrysCat}(\CR_J),
  \] 
  and there are base-change isomorphisms between them. In particular, one can prove
  \begin{equation}
    \label{eqn-base-change-cryscat}
      T_{J\to K}\circ T_{K\gets I} \simeq T_{J\gets I\times_K J}\circ T_{I\times_K J \to I}=: T_{J\gets I\times_K J \to I}
  \end{equation}

\sssec{}
  Note that the lax-factorization structure on $(\ul{\bA},\ul{\bM})$ provides canonical morphisms
  \begin{eqnarray}
    \label{eqn-factmod-structure-theta}
      \theta_{I\to J}: \ul{\bA}_I \to T_{I\to J}(\ul{\bA}_J) & \textrm{for} & I,J\in \on{Fin} \\ \nonumber
      \theta_{I\to J}: \ul{\bM}_I \to T_{I\to J}(\ul{\bM}_J) & \textrm{for} & I,J\in \on{Fin}_*.
  \end{eqnarray}

  \medskip
  We have a functor
  \begin{eqnarray}
    \label{eqn-tw-in-factres}
      \on{TwArr}(\on{Fin}) &\to& \mathbf{CryCat}(\Ran^\untl) \\ \nonumber
       (I\xrightarrow{\phi}J) &\mapsto& T_{\{1\}\gets I\to J}(\ul{\bA}_J),
  \end{eqnarray}
  where
  \begin{itemize}
    \item $\on{TwArr}(\on{Fin})$ is the category of twisted arrows in $\on{Fin}$. In other words, an object in $\on{TwArr}(\on{Fin})$ is a morphism $\phi:I\to J$ in $\on{Fin}$, while a morphism in $\on{TwArr}(\on{Fin})$ is a commutative diagram
    \[
      \xymatrix{
          I \ar[r]^-\phi & J \ar[d]^-\beta \\ 
          I' \ar[u]^-\alpha \ar[r]^-{\phi'} & J';
      }
    \]
    \item The functor \eqref{eqn-tw-in-factres} sends the above commutative diagram to the composition
    \[
        T_{\{1\}\gets I\to J}(\ul{\bA}_J)\xrightarrow{T_{\{1\}\gets I\to J}(\theta_{J\to J'})} T_{\{1\}\gets I\to J'}(\ul{\bA}_{J'}) \to T_{\{1\}\gets I'\to J'}(\ul{\bA}_{J'}) ,
    \]
    where the last morphism is induced by the adjunction $(T_{I'\to I}, T_{I\gets I'})$.
  \end{itemize}
  Similarly we have a functor
  \begin{eqnarray}
    \label{eqn-tw-in-factres-mod}
      \on{TwArr}(\on{Fin}_*) &\to& \mathbf{CryCat}(\Ran^\untl_{\ul x_0}) \\ \nonumber
       (I\xrightarrow{\phi}J) &\mapsto& T_{\{0\}\gets I\to J}(\ul{\bM}_J),
  \end{eqnarray}

\sssec{}
  \label{sss defn sharp}
  We now define
  \begin{eqnarray*}
      \ul{\bA^\sharp} := \lim_{I\to J} T_{\{1\}\gets I\to J}(\ul{\bA}_J) & \textrm{indexed by} &  \on{TwArr}(\on{Fin}); \\
      \ul{\bM^\sharp} := \lim_{I\to J} T_{\{0\}\gets I\to J}(\ul{\bM}_J) & \textrm{indexed by} &  \on{TwArr}(\on{Fin}_*).
  \end{eqnarray*}
  Note that there are obvious morphisms
  \begin{equation}
    \label{eqn-sharp-to-id}
      \ul{\bA^\sharp} \to \ul{\bA},\; \ul{\bM^\sharp} \to \ul{\bM}
  \end{equation}
  given by evaluations at $\id_{\{1\}}$ and $\id_{\{0\}}$ respectively.

\sssec{}
  \label{sss fact structure on sharp}
  We now explain that $(\ul{\bA^\sharp},\ul{\bM^\sharp})$ has a canonical unital lax-factorization structure. We will only do this for $\ul{\bA^\sharp}$. The module part can be constructed by replacing non-marked sets with marked ones.

  \medskip
  We will only construct the structure morphisms (for any finite set $K\in \on{Fin}$)
  \begin{equation}
    \label{eqn-lax-on-fix-1}
      (\boxt_{k\in K}\ul\bA^\sharp) |_\disj \to \on{union}_K^*(\ul{\bA^\sharp})|_\disj = T_{K\to \{1\} }(\bA^\sharp),
  \end{equation}
  and leave the higher compatibilities to \cite{CFZ}.

  \medskip
  We have a canonical morphism
  \begin{equation}
    \label{eqn-lax-on-fix-0}
    (\boxt_{k\in K}\ul\bA^\sharp) |_\disj \to  \lim_{(I_k\to J_k)_{k\in K}}  \big(\boxt_{k\in K} T_{\{1\}\gets I_k\to J_k}(\ul{\bA}_{J_k})\big)|_{\disj} 
  \end{equation}
  by exchanging limits with external products and restrictions\footnote{\label{footnote:restriction limit}In fact, the restriction functor $f^*:\mathbf{CrysCat}(\CZ) \to \mathbf{CrysCat}(\CY)$ commutes with limits for any map $f:\CY\to \CZ$ between categorical prestacks. This follows from the fact that $-\otimes_{\Dmod(S_1)}\Dmod(S_2)$ commutes with limits for affine schemes $S_1$ and $S_2$.}. Unwinding the definitions, we have
  \[
     \big(\boxt_{k\in K} T_{\{1\}\gets I_k\to J_k}(\ul{\bA}_{J_k})\big)|_{\disj} \simeq T_{K\gets \sqcup I_k \to \sqcup J_k} \ul\bA_{\sqcup J_k}
  \]
  Hence we obtain a morphism
  \begin{equation}
    \label{eqn-lax-on-fix-2}
       (\boxt_{k\in K}\ul\bA^\sharp) |_\disj \to \lim_{(I_k\to J_k)_{k\in K}} T_{K\gets \sqcup I_k \to \sqcup J_k} \ul\bA_{\sqcup J_k} \simeq \lim_{I\to J\to K} T_{K\gets I\to J} \ul\bA_J,
  \end{equation}
  where the last limit is indexed by $\on{TwArr}(\on{Fin}_{/K})$, which is equivalence to the category of twisted arrows $I\to J$ equipped with a map $J\to K$.

  \medskip
  On the other hand, one can show $T_{K\to \{1\} } $ commutes with limits (see Footnote \ref{footnote:restriction limit}). Hence we have
  \begin{equation}
    \label{eqn-lax-on-fix-3}
      T_{K\to \{1\} }(\bA^\sharp) \simeq \lim_{I\to J} T_{K\to \{1\}}\circ T_{\{1\}\gets I\to J}(\ul{\bA}_J) \simeq \lim_{I\to J} T_{K\gets K\times I \to J} (\ul{\bA}_J),
  \end{equation}
  where the last equivalence is due to the base-change isomorphism \eqref{eqn-base-change-cryscat}. Let $\on{TwArr}(\on{Fin})_K$ be the category of twisted arrows $I\to J$ in $\on{Fin}$ equipped with a map $I\to K$. Note that 
  \[
      \on{TwArr}(\on{Fin}) \to \on{TwArr}(\on{Fin})_K,\; (I\to J) \mapsto (K\gets K\times I \to J)
  \]
  is left adjoint to the forgetful functor. It follows that we have a canonical equivalence
  \[
      \lim_{K\gets I\to J} T_{K\gets I \to J} (\ul{\bA}_J) \xrightarrow{\simeq} \lim_{I\to J} T_{K\gets K\times I \to J} (\ul{\bA}_J),
  \]
  where the first limit is indexed by $\on{TwArr}(\on{Fin})_K$. Combining with \eqref{eqn-lax-on-fix-3}, we obtain an equivalence
  \begin{equation}
    \label{eqn-lax-on-fix-5}
      T_{K\to \{1\} }(\bA^\sharp) \simeq  \lim_{K\gets I\to J} T_{K\gets I \to J} (\ul{\bA}_J).
  \end{equation}

  \medskip
  Finally, the forgetful functor  $\on{TwArr}(\on{Fin}_{/K}) \to  \on{TwArr}(\on{Fin})_K$ admits a right adjoint
  \[
      ( K\gets I \to J ) \mapsto ( I \to K\times J \to K ).
  \]
  This implies we have a canonical equivalence
  \begin{equation}
    \label{eqn-lax-on-fix-4}
      \lim_{K\gets I\to J} T_{K\gets I \to J} (\ul{\bA}_J) \simeq \lim_{I\to J\to K} T_{K\gets I\to J} \ul\bA_J.
  \end{equation}
  Now the desired morphism \eqref{eqn-lax-on-fix-1} is defined to be the composition
  \[
      \eqref{eqn-lax-on-fix-5}^{-1}\circ \eqref{eqn-lax-on-fix-4}^{-1}\circ \eqref{eqn-lax-on-fix-2}.
  \]

  \medskip
  In \cite{CFZ}, we will show these morphisms (when $K$ varies) indeed define an object
  \[
      \bA^\sharp \in \mathbf{UntlLaxFactCat}.
  \]
  Similarly, we have
  \[
      (\bA^\sharp,\bM^\sharp) \in \mathbf{UntlLaxFactModCat}_{\ul x_0}.
  \]
  Moreover, \eqref{eqn-sharp-to-id} can be upgraded to a morphism
  \[
      (\bA^\sharp,\bM^\sharp) \to (\bA,\bM).
  \]

\sssec{}
  We now explain claim (1) in \secref{sss sharp}. We will only do this for 
  \[
      \mathbf{Str}(\bA):= \lim\big( \cdots \to \bA^{\sharp\sharp} \to \bA^\sharp \to \bA \big)
  \] 
  The module part can be constructed by replacing non-marked sets with marked ones.

  \medskip
  For any integer $m$, we say a unital lax-factorization category $\bA$ is \emph{$m$-strict} if for any collection of disjoint affine points $\ul{x}_k\in \Ran^\untl(S)$, $k\in K$ satisfying $|\sqcup \ul{x}_k|\le m$, the structural functor
  \[
      \on{mult}_{(\ul{x}_k)}: \otimes_{\Dmod(S)} \bA_{\ul{x}_k} \to \bA_{\sqcup\ul{x}_k}
  \]
  is invertible. Note that $\bA$ is always $(-1)$-strict. Also note that $\bA$ being $0$-strict is equivalent to $\bA_{\emptyset} \simeq \Vect$.

  \medskip
  To prove claim (1), we only need to show $\bA^\sharp$ is $m$-strict whenever $\bA$ is $(m-1)$-strict. 

  \medskip
  We first consider the case $m=0$. Unwinding the definitions, we have
  \[
      \bA^\sharp_\emptyset \simeq \lim_{I\to J} \big(T_{\{1\}\gets I\to J}(\ul{\bA}_J) \big)_\emptyset \simeq \lim_{I\to J} (\ul{\bA}_J)_{(\emptyset)_{j\in J}} \simeq \lim_{I\to J} (\bA_\emptyset)^{\otimes J}
  \]
  where recall the limit is indexed by $\on{TwArr}(\on{Fin})$. Note that the forgetful functor 
  \[
      \on{TwArr}(\on{Fin})\to \on{Fin},\; (I\to J) \mapsto J
  \]
  is a weak homotopy equivalence. It follows that 
  \[
      \lim_{I\to J} (\bA_\emptyset)^{\otimes J} \simeq \lim_{J\in \on{\on{Fin} }} (\bA_\emptyset)^{\otimes J} \simeq  (\bA_\emptyset)^{\otimes \emptyset} \simeq \Vect.
  \]
  Hence $\bA^\sharp_\emptyset\simeq \Vect$ as desired.

  \medskip
  We now prove the general case when $m>0$. We need to show the structural functor
  \[
      \on{mult}_{(\ul{x}_k)}: \otimes_{\Dmod(S)} \bA_{\ul{x}_k}^\sharp \to \bA_{\sqcup\ul{x}_k}^\sharp
  \]
  is inverible when $|\sqcup \ul{x}_k|\le m$. Since the $m=0$ case is known, we can assume $|K|\ge 2$ and each $\ul{x}_k$ is non-empty. This implies $|\ul{x}_k|< m$.

  \medskip
  By construction in \secref{sss fact structure on sharp}, we only need to show the fiber of \eqref{eqn-lax-on-fix-0} at
  \[
      (\ul{x}_k)_{k\in K} \in \big(\prod_{k\in K} \Ran^\untl\big)_\disj(S) = \CR_K(S)
  \]
  is invertible. In other words, we need to show we can exchange limits with tensor products in the following expression:
  \[
      \bigotimes_{k\in K}   \big(\lim_{I_k\to J_k} \big(T_{ \{k\}\gets I_k \to J_k } \ul\bA_{J_k}  \big)_{\ul x_k}\big)_{/\Dmod(S)} 
  \]
  For this purpose, we prove the following stronger claim: for fixed $k\in K$ and \emph{any} $\Dmod(S)$-module category $\CC$, the functor
  \[
      \CC \otimes_{\Dmod(S)}  \lim_{I_k\to J_k} \big( T_{ \{k\}\gets I_k \to J_k }  \ul\bA_{J_k}  \big)_{\ul{x}_k} \to  \lim_{I_k\to J_k} \CC\otimes_{\Dmod(S)}  \big( T_{ \{k\}\gets I_k \to J_k }  \ul\bA_{J_k}  \big)_{\ul{x}_k} 
  \]
  is invertible. For this purpose, we prove the following stronger claim: for $\ul{x}\in \Ran^\untl(S)$ such that $|\ul{x}|<m$, the functor
  \begin{equation}
    \label{eqn-sharp-construction-stable}
       \lim_{I\to J} \CC\otimes_{\Dmod(S)}  \big( T_{ \{1\}\gets I \to J } \ul{\bA}_J  \big)_{\ul{x}}  \to \CC\otimes_{\Dmod(S)} \bA_{\ul{x}}
  \end{equation}
  is an equivalence.

  \medskip
  Using the assumption that $\bA$ is $(m-1)$-strict, we have
  \[
      \big( T_{ \{1\}\gets I \to J } \ul{\bA}_J   \big)_{\ul{x}} \simeq  \big( T_{ \{1\}\gets I } \ul{\bA}_I   \big)_{\ul{x}}.
  \]
  Since the forgetful functor $\on{TwArr}(\on{Fin})\to \on{Fin}^{\on{op}},\; (I\to J) \mapsto I$ is a weak equivalence, we obtain
  \[
       \lim_{I\to J} \CC\otimes_{\Dmod(S)}  \big( T_{ \{1\}\gets I \to J } \ul{\bA}_J  \big)_{\ul{x}}  \simeq \lim_{I\in \on{Fin}^{\on{op}}} \CC\otimes_{\Dmod(S)} \big( T_{ \{1\}\gets I } \ul{\bA}_I  \big)_{\ul{x}} \simeq \CC\otimes_{\Dmod(S)}\big( T_{ \{1\}\gets \{1\} } \ul{\bA}_{\{1\}}  \big)_{\ul{x}} \simeq  \CC\otimes_{\Dmod(S)} \bA_{\ul{x}}
  \]
  as desired.

\sssec{Remark}
  The above argument is closely related to the notion of \emph{pro-nilpotent operads} in \cite{FG}. We will explain this in \cite{CFZ}.

\sssec{}
  Note that the equivalence \eqref{eqn-sharp-construction-stable} says:

\begin{lem}
  \label{lem-sharp-is-stable}
  Let $\bA$ be a unital $(m-1)$-strict factorization category. Then the functor $\bA^\sharp \to \bA$ induces an equivalence
  \[
      \bA^\sharp_{\ul{x}} \to \bA_{\ul{x}}
  \]
  for any affine point $\ul{x}\in \Ran^\untl(S)$ with $|\ul{x}|<m$.
\end{lem}

\sssec{}
  In particular, we have shown the restriction of the sequence
  \[
      \mathbb{Z}_{\le 0} \to \mathbf{UntlLaxFactCat},\; -n\mapsto \bA^{n\sharp} 
  \] 
  at a point $\ul{x}\in \Ran^\untl(S)$ becomes stable for $-n< -|\ul{x}|$. In particular,
  \[
      \mathbf{Str}(\bA)_{\ul{x}} \simeq (\bA^{n\sharp} )_{\ul x} \textrm{ for } n>|\ul x|.
  \]

\sssec{Remark}
  \label{rem descrp restriction}
  In fact, a more elaborate analysis gives
  \[
      \mathbf{Str}(\bA)_{\ul{x}} \simeq (\bA^{n\sharp} )_{\ul x} \textrm{ for } n\ge |\ul x|-1 \ge 0.
  \]
  For example, a direct calculation shows the restriction $\mathbf{Str}(\bA)$ along $X^2 \to \Ran^\untl$ fits into the following Cartesian square
  \[
      \xymatrix{
          \mathbf{Str}(\bA)|_{X^2} \ar[r] \ar[d]
          & \bA|_{X^2} \ar[d] \\ 
          j_*j^*(\bA|_X\boxt\bA|_X) \ar[r]
          & j_*j^*(\bA|_{X^2}),
      }
  \]
  where
  \begin{itemize}
    \item $j:X^2\setminus X\to X^2$ is the complement of the diagonal embedding;
    \item the bottom horizontal functor is provided by the lax-factorization structure on $\ul{\bA}$.
  \end{itemize}

  \medskip
  Similarly, one can show
  \[
      \mathbf{Str}(\bM)_{\ul{x}} \simeq (\bM^{n\sharp} )_{\ul x} \textrm{ for } n\ge |\ul x|-|\ul{x}_0|.
  \]
  For example, the restriction of $\mathbf{Str}(\bM)$ along the map
  \[
      X\times S_0 \to \Ran^{\untl}_{\ul{x}_0},\; y\mapsto y\cup \ul{x}_0
  \]
  fits into the following Cartesian square
  \[
      \xymatrix{
          \mathbf{Str}(\bM)|_{X\times S_0} \ar[r] \ar[d]
          & \bM|_{X\times S_0} \ar[d] \\ 
          j_*j^*(\bA|_X \boxt \bM|_{\ul{x}_0}) \ar[r]
          & j_*j^*(\bM|_{X\times S_0}),
      }
  \]
  where
  \begin{itemize}
    \item $j:(X\times S_0)\setminus \on{graph}_{\ul x_0}\to X\times S_0$ is the complement of the union of the graphs for elements in $\ul{x}_0\subseteq X(S_0)$;
    \item the bottom horizontal functor is provided by the lax-factorization $\bA$-module structure on $\ul{\bM}$.
  \end{itemize}

\sssec{}
  Finally, we explain claim (2) in \secref{sss sharp}. Claim (3) can be proved similarly by replacing non-marked sets with marked ones.

  \medskip
  Let $\bB$ be any test unital factorization category. We want to show
  \[
      \on{Fun}(\bB,\bA^\sharp) \to \on{Fun}(\bB,\bA)
  \]
  is an equivalence. We only explain how to construct the desired inverse functor
  \begin{equation}
    \label{eqn-universal-property-for-sharp}
      \on{Fun}(\bB,\bA) \to \on{Fun}(\bB,\bA^\sharp),
  \end{equation}
  and leave the verification to \cite{CFZ}.

  \medskip
  Let $\Phi:\bB\to \bA$ be a unital factorization functor. The $\sharp$-construction is functorial, hence we have a functor
  \[
      \Phi^\sharp:\bB^\sharp \to \bA^\sharp.
  \]
  Since $\bB$ is $\infty$-strict by assumption, Lemma implies $\mu_\bB:\bB^\sharp \to \bB$ is an equivalence. We now define the functor \eqref{eqn-universal-property-for-sharp} to be $\Phi\mapsto \mu_\bB^{-1}\circ \Phi^\sharp$.

\qed[Sketch of Theorem \ref{thm-the-fibration}]

\ssec{Induced modules}
  \label{ss induced modules}
  In this subsection, we fix a unital factorization category $\bA$ and a unital factorization $\bA$-module category $\bM$ at $\ul{x}_0$. For simplicity, we assume $\ul{x}_0=x_0$ is a single $k$-point on $X$.

  \medskip
  Let $\CA\in \on{FactAlg}(\bA)$ be a factorization algebra in $\bA$ and
  \[
    \oblv_\CA: \CA\on{-mod}^\onfact(\bM)_{x_0} \to \bM_{x_0}
  \]
  be the forgetful functor. We will study the \emph{partially defined} left adjoint $\ind_\CA$ of this functor.

  \sssec{Warning}
      For general $\CA$, the functor $\oblv_\CA$ does not preserve limits, because general external tensor products and $!$-pullback functors do not preserve limits. In fact, we do not know how to calculate limits in $\CA\on{-mod}^\onfact(\bM)_{x_0}$ (although we know they exist by presentability). As a consequence, $\oblv_\CA$ does not admit a left adjoint.

  \sssec{}
    We are going to provide a sufficient condition on $\CA$ and $V$ such that there exists an \emph{object} $\ind_\CA(V)$ such that
    \[
        \Hom_{\CA\on{-mod}^\onfact(\bM)_{x_0}}(\ind_\CA(V),\CM) \simeq \Hom_{\bM_{x_0}}(V,\CM_{x_0}).
    \]
    Such an object $\ind_\CA(V)$ is called the \emph{induced} (a.k.a. \emph{free}) factorization $\CA$-module in $\bM$. 

  \sssec{Remark}
    More generally, one can ask what are the coCartesian arrows in the Cartesian fibration (see \secref{sss the fibration 1-cat down})
    \[
        \on{FactMod}(\bA,\bM)_{x_0} \to \on{FactAlg}(\bA).
    \]
    We will treat this problem in \cite{CFZ}.

  \sssec{}
    \label{sss punctured Ran}
    To construct $\ind_\CA(V)$, we need some notations.

    \medskip
    Let $X_\circ:=X\setminus x_0$ be the punctured curve and $\Ran_\circ^\untl$ be the unital Ran space for $X_\circ$. Consider the map
    \[
        \j: \Ran_\circ^\untl \times x_0 \to \Ran_{x_0}^\untl, (\ul{y},x_0) \mapsto \ul{y}\sqcup \{x_0\}.
    \]
    Note that $\j$ induces a bijection between $k$-points, but is not an isomorphism. In fact, one can check
    \begin{itemize}
      \item $\Ran_\circ^\untl\times x_0$ is a Cartesian space over $\Ran_{x_0}^\untl$ (see \secref{sss Cart catprestk});
      \item For any affine points $S\to \Ran_{x_0}^\untl$, the base-change of $\j$ is a finite coproduct of locally closed immersions.
    \end{itemize}

    \medskip
    Similarly, for any marked finite set $I=I^\circ \sqcup\{0\}$, we have a map
    \begin{equation}
      \label{eqn-jI}
        \j_I: \big( \prod_{i\in I}  \Ran_\circ^\untl \big)_\disj \times x_0 \to \big( \big( \prod_{i\in I^\circ}  \Ran_\circ^\untl \big)\times \Ran_{x_0}^\untl\big)_\disj \simeq \big( \big( \prod_{i\in I^\circ}  \Ran^\untl \big)\times \Ran_{x_0}^\untl\big)_\disj
    \end{equation}
    given by $((\id)_{i\in I^\circ}, \j)$. To simplify the notations, we write it as
    \[
        \j_I: \CR_{I,\circ} \times x_0 \to \CR_I.
    \]
    Note that we have a functor
    \begin{equation}
      \label{eqn-j!}
      \begin{aligned}
        \j_I^!: \Gamma^\lax( \CR_I, \ul\bM|_{\CR_I} ) \to  
        \Gamma^\lax( \CR_{I,\circ}\times x_0, \ul\bM|_{\CR_{I,\circ}\times x_0} ) \simeq   \\
        \simeq 
        \Gamma^\lax( \CR_{I,\circ}\times x_0, \ul\bA|_{\CR_{I,\circ}}\boxt \bM_{x_0} ) \simeq \Gamma^\lax( \CR_{I,\circ}, \ul\bA|_{\CR_{I,\circ}} ) \otimes \bM_{x_0},
        \end{aligned}
    \end{equation}
    where the first equivalence is given by the factorization $\bA$-module structure on $\bM$.

  \sssec{}
    Note that $\ul{\bA}|_{\Ran_\circ^\untl}$ is a unital factorization category on the punctured curve $X_\circ$. We denote this object by
    \[
        \bA_\circ \in \mathbf{UntlFactCat}(X_\circ),
    \]
    to distinguish it from the object $\bA\in \mathbf{UntlFactCat}(X)$. Note that the latter category is denoted just by $\mathbf{UntlFactCat}$ in the rest of this appendix.

    \medskip
    By restriction, we also obtain an object 
    \[
        \CA_\circ \in \on{FactAlg}(\bA_\circ).
    \]

  \sssec{}
    Similarly, $\ul{\bM}|_{\Ran_\circ^\untl\times x_0}$ is a unital factorization $\bA_\circ$-module category \emph{at the point $\emptyset\in \Ran_\circ^\untl$}. Here we use the identification
    \[
        (\Ran_\circ^\untl)_{\emptyset} \simeq \Ran_\circ^\untl\times x_0,\; \ul{y}\mapsto (\ul{y},x_0).
    \]
    We denote this object by
    \[
        \bM_\circ \in \bA_\circ\mathbf{-mod}^\onfact_{\emptyset}.
    \]
    Via the correspondence in \secref{exam-module-at-empty-set}, $\bM_\circ$ is given by the DG category $\bM_{x_0}$. It follows that
    \begin{equation}
      \label{eqn-factmod-at-empty}
        \CA_\circ\on{-mod}^\onfact(\bM_\circ)_\emptyset:= \on{Fun}_\CA( \Vect^{\onfact_\emptyset}, \bM_\circ  ) \simeq \on{Fun}(\Vect, \bM_{x_0}) \simeq \bM_{x_0}.
    \end{equation}
    By construction, this is just the forgetful functor $\oblv_{\CA_\circ}$

  \sssec{}
    \label{sss: A box V factmod empty}
    The above equivalence can be proved in a more explicit way. 

    \medskip
    Given $\CM\in \CA_\circ\on{-mod}^\onfact(\bM_\circ)_\emptyset$, its fiber $\CM_\emptyset$ is an object in $(\bM_\circ)_\emptyset \simeq \bM_{x_0}$; conversely, given an object $V\in \bM_{x_0}$, the tensor product
    \[
        \ul{\CA}_\circ \boxt V \in \Gamma^\lax(\Ran_\circ^\untl, \ul\bA_\circ)\otimes \bM_{x_0} \simeq \Gamma^\lax( (\Ran_\circ^\untl)_{\emptyset}, \ul\bM_\circ )
    \]
    has a canonical factorization $\bA_\circ$-module structure. One can check these two constructions are inverse to each other.

  \sssec{}
    \label{sss adapt to induction}
    Given an object $V\in \bM_{x_0}$, we say it is \emph{adapted to $\CA$-induction} if it satisfies the following conditions:
    \begin{itemize}
      \item 
        For any marked finite set $I=I^\circ \sqcup\{0\}$, the partially defined left $\j_{I,!}$ to the functor \eqref{eqn-j!} is well-defined at the object $\ul\CA|_{\CR_{I,\circ}} \boxt V$, i.e., the following object exists:
        \[
          \j_{I,!} \big(  (\boxt_{i\in I} \ul\CA_\circ)_\disj \boxt V  \big) \in \Gamma^\lax( \CR_I, \ul\bM|_{\CR_I} ).
        \]
        In particular, we have an object
        \[
            \j_!(\ul\CA_\circ\boxt V) \in \Gamma^\lax( \Ran_{x_0}^\untl ,\ul\bM).
        \]
      \item
        For any marked finite set $I=I^\circ \sqcup\{0\}$, the canonical (Beck--Chevalley) morphism
        \[
            \j_{I,!} \big(  (\boxt_{i\in I} \ul\CA_\circ)_\disj \boxt V  \big) \to \big( (\boxt_{i\in I^\circ } \ul{\CA}_\circ) \boxt \j_!(\ul\CA_\circ\boxt V) \big)|_{\disj}
        \]
        is invertible.
    \end{itemize}

\sssec{}
  \label{sss fact structure on j_!}
    Let $V\in \bM_{x_0}$ be an object adapted to $\CA$-induction. We claim $\j_!(\ul\CA_\circ\boxt V)$ can be canonically upgraded to an object
    \[
        \j_!(\CA_\circ\boxt V) \in \CA\on{-mod}^\onfact(\bM)_{x_0}.
    \]

    \medskip 
    We will only construct the structural isomorphisms (see \secref{sss trad factmod object})
    \begin{equation}
      \label{eqn-structure-iso-j}
      \on{act}_I\big(  \big((\boxt_{i\in I^\circ} \ul{\CA})\boxt \j_!(\ul\CA_\circ\boxt V)\big) |_\disj \big) \xrightarrow{\simeq} \on{union}_I^!(\j_!(\ul\CA_\circ\boxt V))|_\disj,
    \end{equation}
    and leave the higher compatibilities to \cite{CFZ}.
    
    \medskip
    By assumption, we have
    \begin{equation}
      \label{eqn-structure-iso-j-1}
         \big((\boxt_{i\in I^\circ} \ul{\CA})\boxt \j_!(\ul\CA_\circ\boxt V)\big) |_\disj \simeq \big( (\boxt_{i\in I^\circ } \ul{\CA}_\circ) \boxt \j_!(\ul\CA_\circ\boxt V) \big)|_{\disj} \simeq \j_{I,!} \big(  (\boxt_{i\in I} \ul\CA_\circ)_\disj \boxt V  \big).
    \end{equation}

    \medskip
    On the other hand, we have a Cartesian square
    \[
      \xymatrix{
          \CR_{I,\circ} \times x_0
          \ar[r]^-{\j_I} \ar[d]_-{\on{union}_{I,\circ}}
           & \CR_I \ar[d]^-{\on{union}_I} \\
           \Ran_\circ^\untl \times x_0 \ar[r]_-\j
          &  \Ran_{x_0}^\untl
      }
    \]
    such that the vertical arrows satisfy properties similar to those in Lemma \ref{lem-union-local-iso}. This implies the Beck--Chevalley natural transformation
    \[
       \xymatrix{
          \Gamma^\lax(\CR_{I,\circ}\times x_0, \ul\bM|_{\CR_{I,\circ}\times x_0})
           \ar[d]_-{\on{union}_{I,\circ,*}}
           & \Gamma^\lax(\CR_I, \ul\bM|_{\CR_I}) \ar[d]^-{\on{union}_{I,*}} \ar[l]_-{\j_I^!} \\
           \Gamma^\lax(\Ran_\circ^\untl\times x_0 , \ul\bM|_{\Ran_\circ^\untl\times x_0 }) 
          &  \Gamma^\lax(\Ran_{x_0}^\untl, \ul\bM) \ar[l]^-{\j^!}
      }
    \]
    is invertible, where the vertical functors are right adjoint to the $!$-pullback functors. Passing to partially defined left adjoints, we obtain
    \begin{equation}
      \label{eqn-structure-iso-j-2}
        \on{union}_I^!(\j_!(\ul\CA_\circ\boxt V))|_\disj \simeq \j_{I,!} \big( \on{union}_{I,\circ}^!(  \ul\CA_\circ\boxt V )|_\disj\big)
    \end{equation}

    \medskip
    Via the isomorphisms \eqref{eqn-structure-iso-j-1} and \eqref{eqn-structure-iso-j-2}, the desired isomorphism \eqref{eqn-structure-iso-j} is given by applying $\j_{I,!}$ to 
    \[
        \on{act}_{I,\circ}\big((\boxt_{i\in I} \ul\CA_\circ)_\disj \boxt V \big) \xrightarrow{\simeq} \on{union}_{I,\circ}^!(  \ul\CA_\circ\boxt V )|_\disj,
    \]
    which is given by the factorizaiton $\ul{A}_\circ$-module structure on $\ul\CA_\circ\boxt V $ (see \secref{sss: A box V factmod empty}).

\sssec{}
  Modulo issues about homotopy coherence, it is clear
  \[
      \Maps_{\CA\on{-mod}^\onfact(\bM)_{x_0}}( \j_!(\CA_\circ\boxt V), \CM  ) \simeq \Maps_{\CA_\circ\on{-mod}^\onfact(\bM_\circ)_{\emptyset}}( \CA_\circ\boxt V, \CM|_{\Ran_\circ^\untl \times x_0} ).
  \]
  By \eqref{eqn-factmod-at-empty}, the RHS can be identified with
  \[
      \Maps_{\bM_{x_0}}(  V, \CM_{x_0} ).
  \]
  Hence we obtain the following result.

  \begin{prop}
    \label{prop fact ind}
      Let $\bA$, $\bM$, $\CA$ be as in \secref{ss induced modules}. Suppose $V\in \bM_{x_0}$ is adapted to $\CA$-induction (see \secref{sss adapt to induction}), then the partially defined left adjoint $\on{ind}_\CA$ to the forgetful functor
      \[
          \oblv_\CA: \CA\on{-mod}^\onfact(\bM)_{x_0} \to \bM_{x_0}
      \]
      is defined on $V$, and we have
      \[
          \ind_\CA(V) \simeq \j_!(\CA_\circ\boxt V),
      \]
      where the RHS is defined in \secref{sss fact structure on j_!}.
  \end{prop}

\ssec{The Ran operad}
\label{ss fact via operad}
  This subsection serves as an advertisement for \cite{CFZ}, where all the homotopy-coherent difficulties in this appendix (as well as those ignored in \cite[Appendix B, C]{GLC2}) will be treated by the methods developed in Lurie's \emph{Higher Algebra} \cite{Lu}.

  \sssec{}
    In \cite{CFZ}, we will rewrite the foundations of factorization structures\footnote{It is fair to say \cite{Ra} is the only homotopy-coherent foundation of factorization categories that exists in the literature. However, there are several disadvantages in Raskin's approach which makes it hard to prove results claimed in \secref{ss external fusion}.} using the language of \emph{generalized operads} developed in \cite{Lu}. 

  \sssec{}
  In \cite[Chapter 2]{Lu}, Lurie defined an $\infty$-operad as an $\infty$-cateogry $\CO^\otimes \to \on{Fin}_*$ over the category of marked finite sets that satisfies certain conditions. Roughly speaking, an $\infty$-operad is a colored operad (introduced by May in \cite{Ma}) enriched over the $\infty$-category of spaces, except that the collection of colors is allowed to form a category\footnote{It is the category $\CO:= \CO^\otimes \times_{\on{Fin}_*} \{0,1\} $.} rather than a space/set. Equivalently, an $\infty$-operad is a pseudo-tensor category (introduced by Beilinson--Drinfeld in \cite{BD1}), but the underlying category is allowed to be an $\infty$-category. 

  \sssec{}
  A generalized $\infty$-operad should be viewed as a \emph{family of $\infty$-operads} parameterized by some base category $\CC$ (see \cite[Sect. 2.3]{Lu}). Just like the usual theory of various types of algebras and monoidal categories can be developed using the corresponding operads, for any generalized $\infty$-operad $\CO^\otimes \to \on{Fin}_*\times \CC$, one can develop the notion of $\CO$-monoidal categories and $\CO$-algebras in them. 

  \medskip
  The main idea behind \cite{CFZ} is: there should exist a (classical) generalized operad $\Ran^\otimes$ parameterized by the category $\on{Aff}$ of affine schemes, such that
  \begin{equation}
    \label{eqn-laxfact-via-operad}
    \textrm{lax-factorization objects in }\mathcal{D} = \Ran\textrm{-algebras in }\mathcal{D},
  \end{equation}
  where $\mathcal{D}$ is any symmetric monoidal $(\infty,2)$-category. For instance,
  \[
    \textrm{lax-factorization DG categories} = \Ran\textrm{-algebras in }\mathbf{DGCat}.
  \]

  \sssec{}
  The construction of the generalized operad $\Ran^\otimes$ is easy. For any affine scheme $S$, we have a symmetric monoidal category $\Ran^\untl(S)$ with tensor products given by unions of finite sets. In particular, it corresponds to a (classical) $\infty$-operad $\Ran^\untl(S)^{\cup}$. We now define
  \[
      \Ran(S)^{\otimes} \subseteq \Ran^\untl(S)^{\cup}
  \]
  to be the 1-full subcategory containing of those morphisms that corrspond to \emph{disjoint} unions. Alternatively, we equip $\Ran^\untl(S)$ with a structure of pseudo-tensor categories, where a multi-map $\{x_i\}_{i\in I} \to y$ exists iff the points $\{x_i\}_{i\in I}$ are disjoint and $\sqcup x_i \subseteq y$.

  \medskip
  The above construction is contravariantly functorial in $S$. Hence we have a functor from $\on{Aff}^{\on{op}}$ to the category of (classical) $\infty$-operads. Now the generalized operad $\Ran^\otimes$ is defined to be the corresponding coCartesian fibration
  \[
      \Ran^\otimes \to \on{Aff}^{\on{op}}.
  \]

  \sssec{}
  Now for any $\on{Aff}^{\on{op}}$-family of symmetric monoidal $(\infty,2)$-categories $\bD^\otimes \to \on{Aff}^{\on{op}}$, we can \emph{define} a lax-factorization algebra object in $\bD$ to be a functor $\bA:\Ran^\otimes \to \bD^\otimes$ defined over $\on{Fin}_*\times  \on{Aff}^{\on{op}}$ such that 
  \begin{itemize}
    \item $\bA$ preserves inert morphisms;
    \item $\bA$ preserves coCartesian arrows over $\on{Aff}^{\on{op}}$.
  \end{itemize}
  We say $\bA$ is a (\emph{strict}) factorization algebra object in $\bD$ if in addtionally
  \begin{itemize}
    \item $\bA$ preserves coCartesian arrows over $\on{Fin}_*\times  \on{Aff}^{\on{op}}$.
  \end{itemize}
  One can check for $\bD=\mathbf{CrysCat}$, the above notions indeed recover (lax-)factorization DG-categories.

  \medskip 
  In fact, to treat both cases simultaneouly, we will introduce \emph{marked generalized $(\infty,2)$-operads}, which are pairs $(\bO^\otimes,E)$ such that $\bO^\otimes$ is a generalized $(\infty,2)$-operad and $E$ is a class of morphisms in $\bO^\otimes$. Then lax/strict factorization algebra objects are functors out of $(\Ran^\otimes, \on{lax}/\on{strict} )$, where the marked morphisms are coCartesian arrows over $\on{Aff}^{\on{op}}$ and $\on{Fin}_*\times  \on{Aff}^{\on{op}}$ respectively.

  \sssec{}
  The main advantage of $\Ran^\otimes$ is that it provides a natural way to deal with factorization \emph{module} structures. 

  \medskip
  In \cite[Chapter 3]{Lu}, for any \emph{coherent} $\infty$-operad $\CO^\otimes$ and a color $\mathfrak{m}\in \CO$, Lurie defined the notion of $\mathfrak{m}$-type $\CO$-modules for $\CO$-algebras, and provided a framework to deal with restrictions and relative tensor products of such modules, where the higher compatibilities for these constructions are encoded as certain fibrations of $\infty$-categories.

  \medskip
  In \cite{CFZ}, we will generalize the notion of coherence to \emph{marked} generalized $(\infty,2)$-operads, and prove:
  \begin{thm}
    $(\Ran^\otimes,\on{strict})$ is coherent as a marked generalized $(\infty,2)$-operads.
  \end{thm}

  \sssec{}
    As a result, we can deal with restrictions and relative tensor products of $\ul{x}$-type $(\Ran^\otimes,\on{strict})$-modules internal to any $\on{Aff}^{\on{op}}$-family of symmetric monoidal $(\infty,2)$-categories, such as $\mathbf{CrysCat}$, in a way similar to \cite[Chapter 3]{Lu}.

    \medskip
    Note that a color $\ul{x}$ in $\Ran^\otimes$ is exactly an affine point $\ul{x}:S\to \Ran^\untl$ for some affine scheme $S$. We will show, for example, the following two notions are the same:
    \begin{itemize}
      \item 
        A pair $(\bA,\bM)$ of a (strict) factorization category $\bA$ and its module $\bM$ at $\ul{x}$, as defined in this appendix;
      \item
        A pair $(\bA,\bM)$ of a $(\Ran,\on{strict})$-algebra in $\mathbf{CrysCat}$ and its $\ul{x}$-type module $\bM$, as defined in \cite{CFZ}.
    \end{itemize}

  \sssec{}
    We will also show that $(\Ran^\otimes,\on{lax})$ is not \emph{coherent}. This reflects the phenomenon that if one mimics the definitions in \secref{ss external fusion} and defines factorization multi-functors of \emph{lax-}factorization modules, then the compositions of these multi-functors are not well-defined.

    \medskip 
    Nevertheless, in \cite{CFZ}, we will define $\ul{x}$-type $(\Ran^\otimes,\on{lax})$-module objects and interpret the construction of the strictening functor \eqref{sss strictening fact} via \emph{monoidal envelopes} and \emph{operadic Kan extensions} (see \cite[Sect. 2.2.4 and Sect. 3.1.2]{Lu}).


\newpage

\begin{thebibliography}{999}

\bibitem[AS1]{AS1} F. Abellan and W. H. Stern, {\it 2-Cartesian fibrations I: A model for 1-bicategories fibred in 1-bicategories},
Applied Categorical Structures (2022).

\bibitem[AS2]{AS2} F. Abellan and W. H. Stern, {\it 2 Cartesian fibrations II: A Grothendieck construction for 1 bicategories},
arXiv:2201.09589.

\bibitem[Bogd]{Bogd} E.~Bogdnova, {\it Local systems with restricted variation on the formal punctured disc via factorization},
arXiv:2411.05297. 

\bibitem[BZGO]{BZGO} D.~Ben-Zvi, S.~Gunningham and H.~Orem, {\it Highest Weights for Categorical Representations}, 
IMRN {\bf 2020}, 9988--10004. 

\bibitem[AG]{AG} D. Arinkin and D. Gaitsgory, {\it Singular support of coherent sheaves, and the geometric Langlands conjecture},
Selecta Math. N.S. {\bf 21} (2015), 1--199.

\bibitem[AGKRRV]{AGKRRV} D. Arinkin, D.Gaitsgory, D. Kazhdan, S. Raskin, N. Rozenblyum and Y. Varshavsky,
{\it The stack of local systems with restricted variation and geometric Langlands theory with nilpotent singular support}, arXiv:2010.01906.

\bibitem[BD1]{BD1} A.~Beilinson and V.~Drinfeld, {\it Chiral algebras}, AMS {\bf 51} (2004).

\bibitem[BD2]{BD2} A.~Beilinson and V.~Drinfeld, {\it Quantization of Hitchin’s integrable system and Hecke eigensheaves}, available
at \url{https://math.uchicago.edu/~drinfeld/langlands/QuantizationHitchin.pdf}

\bibitem[CFZ]{CFZ} L.~Chen, Y.~Fu and Y.~Zhao, {\it The Ran operad}, in preparation.

\bibitem[DG]{DG} V.~Drinfeld and D.~Gaitsgory, {\it Compact generation of the category of $\mathrm {D} $-modules on the stack of $ G $-bundles on a curve}, Cambridge Journal of Mathematics 3, no. 1 (2015): 19-125.

\bibitem[FG]{FG} J.~Francis and D.~Gaitsgory, {\it  Chiral koszul duality}, Selecta Mathematica 18.1 (2012): 27-87.

\bibitem[Ga0]{Ga0} D.~Gaitsgory, {\it Construction of central elements in the affine Hecke algebra
via nearby cycles}, Invent. Math. {\bf 144} (2001), 253--280. 

\bibitem[Ga1]{Ga1} D.~Gaitsgory, {\it The Atiyah-Bott formula for the cohomology of the moduli space of bundles on a curve}, 
arXiv:1505.02331.

\bibitem[Ga2]{Ga2} D.~Gaitsgory, {\it Contractibility of the space of rational maps}, Invent. Math. {\bf 191} (2013), 91--196.

\bibitem[Ga3]{Ga3} D.~Gaitsgory, {\it Sheaves of categories and the notion of 1-affineness}, Contemporary Mathematics {\bf 643}  (2015), 1--99.


\bibitem[Ga4]{Ga4} D.~Gaitsgory, {\it Local and global versions of the Whittaker category}, PAMQ {\bf 16}  (2020), 775--904.

\bibitem[Ga5]{Ga5} D.~Gaitsgory, {\it Local and global Langlands conjecture(s) over function fields}, 
arXiv:2509.24902

\bibitem[GHL]{GHL} A. Gagna, Y. Harpaz and E. Lanari, {\it Cartesian Fibrations of ($\infty,2$)-categories}, arXiv:2107.12356

\bibitem[GL]{GL} D.~Gaitsgory and J.~Lurie, {\it Weil’s conjecture for function fields}, available at \url{https://math.ias.edu/~lurie}

\bibitem[GLC2]{GLC2} D.~Arinkin, D.~Beraldo, L.~Chen, J.~Faegerman, D.~Gaitsgory, K.~Lin, S.~Raskin and N.~Rozenblyum, \newline
{\it Proof of the geometric Langlands conjecture II: Kac-Moody localization and the FLE}, arXiv:2405.03648

\bibitem[GR0]{GR0} D.~Gaitsgory and N.~Rozenblyum,  {\it DG ind-schemes}, Contemporary Mathematics {\bf 610} (2014), 139--251.

\bibitem[GR1]{GR1} D.~Gaitsgory and N.~Rozenblyum, {\it A study in derived algebraic geometry, Vol. I:  
Correspondences and Duality}, {\bf 221} (2017), AMS, Providence, RI.

\bibitem[GR2]{GR2} D.~Gaitsgory and N.~Rozenblyum, {\it A study in derived algebraic geometry, Vol. I:  
Deformations, Lie Theory and Formal Geometry}, {\bf 221} (2017), AMS, Providence, RI.

\bibitem[Ho]{Ho} Quoc P.~Ho, {\it The Atiyah-Bott formula and connectivity in chiral Koszul duality}, arXiv:1610.00212

\bibitem[Lu0]{Lu0} J.~Lurie, {\it Higher Topos Theory}, Princeton University Press (2009).

\bibitem[Lu]{Lu} J.~Lurie, {\it Higher Algebra}, available at: \url{http://math.harvard.edu/~lurie}

\bibitem[Ma]{Ma} J.P.~May {\it The Geometry of Iterated Loop Spaces}, Lecture Notes in Mathematics (1972).

\bibitem[Ra]{Ra} S.~Raskin, {\it Chiral categories}, available at: \url{https://gauss.math.yale.edu/~sr2532/}



\bibitem[Ro]{Ro} N.~Rozenblyum, {\it Connections on moduli spaces and infinitesimal Hecke modifications}, arXiv:2108.07745 

\bibitem[Zhao]{Zhao} Y.~Zhao, {\it Half-integral levels}, arXiv:2312.11058. 


\end{thebibliography}
\end{document}